\documentclass[10pt]{ociamthesis}  
\usepackage{mlpackage}
\usepackage{todonotes}
\usepackage{subcaption} 

\newcommand*\mlp{\textsc{mlp}}
\newcommand*\nn{\textsc{nn}}
\newcommand*\nns{\textsc{nn}s}
\newcommand*\rnn{\textsc{rnn}}
\newcommand*\rnns{\textsc{rnn}s}
\newcommand*\tpu{\textsc{tpu}}
\newcommand*\tpus{\textsc{tpu}s}
\newcommand*\gpu{\textsc{gpu}}
\newcommand*\gpus{\textsc{gpu}s}
\newcommand*\amp{\textsc{amp}}

\newcommand*\api{\textsc{api}}
\newcommand*\ode{\textsc{ode}}
\newcommand*\odes{\textsc{ode}s}
\newcommand*\cpu{\textsc{cpu}}

\newcommand*\admm{\textsc{admm}}
\newcommand*\erm{\textsc{erm}}
\newcommand*\cnn{\textsc{cnn}}
\newcommand*\cnns{\textsc{cnn}s}

\newcommand*\lstm{\textsc{lstm}}
\newcommand*\gru{\textsc{gru}}
\newcommand*\qr{\textsc{qr}}
\newcommand*\lu{\textsc{lu}}
\newcommand*\svd{\textsc{svd}}
\newcommand*\rgd{\textsc{rgd}}
\newcommand*\sgd{\textsc{sgd}}
\newcommand*\gd{\textsc{gd}}
\newcommand*\vae{\textsc{vae}}
\newcommand*\vaes{\textsc{vae}s}
\newcommand*\fft{\textsc{fft}}

\newcommand*\gan{\textsc{gan}}
\newcommand*\adam{\textsc{adam}}
\newcommand*\adagrad{\textsc{adagrad}}
\newcommand*\rmsprop{\textsc{rmsprop}}
\newcommand*\scornn{\textsc{scornn}}
\newcommand*\scurnn{\textsc{scurnn}}
\newcommand*\urnn{\textsc{urnn}}
\newcommand*\eunn{\textsc{eunn}}
\newcommand*\exprnn{\textsc{exprnn}}
\newcommand*\dtriv{\textsc{dtriv}}
\newcommand*\mnist{\textsc{mnist}}
\newcommand*\pmnist{\textsc{p-mnist}}
\newcommand*\timit{\textsc{timit}}
\newcommand*\mse{\textsc{mse}}
\newcommand*\pde{\textsc{pde}}
\newcommand*\pdes{\textsc{pde}s}
\newcommand*\cuda{\textsc{cuda}}
\newcommand*\mitedu{\textsc{mit}}
\newcommand*\matlab{\textsc{matlab}}
\newcommand*\jax{\textsc{jax}}

\DeclareMathOperator{\rank}{rank}                        
\DeclareMathOperator{\UUaux}{U}                           
\DeclareMathOperator{\uuaux}{\mathfrak{u}}              
\DeclareMathOperator{\SLaux}{SL}                         
\DeclareMathOperator{\TEaux}{T}                          
\DeclareMathOperator{\GLaux}{GL}                         
\DeclareMathOperator{\GLpaux}{GL^{+}}                    
\DeclareMathOperator{\glaux}{\mathfrak{gl}}              
\DeclareMathOperator{\Oaux}{O}                           
\DeclareMathOperator{\SOaux}{SO}                         
\DeclareMathOperator{\soaux}{\mathfrak{so}}              
\DeclareMathOperator{\Symaux}{Sym}                       
\DeclareMathOperator{\Skewaux}{Skew}                     
\DeclareMathOperator{\Staux}{St}                         
\DeclareMathOperator{\Graux}{Gr}                         
\NewDocumentCommand{\UU}{ m }{ \UUaux\pa{#1} }

\NewDocumentCommand{\SL}{ m }{ \SLaux\pa{#1} }

\NewDocumentCommand{\TE}{ m }{ \TEaux\pa{#1} }
\NewDocumentCommand{\GL}{ m }{ \GLaux\pa{#1} }
\NewDocumentCommand{\GLp}{ m }{ \GLpaux\pa{#1} }

\NewDocumentCommand{\gl}{ m }{ \glaux\pa{#1} }

\NewDocumentCommand{\Ort}{ m }{ \Oaux\pa{#1} }

\NewDocumentCommand{\SO}{ m }{ \SOaux\pa{#1} }

\NewDocumentCommand{\solie}{ m }{ \soaux\pa{#1} }
\NewDocumentCommand{\uulie}{ m }{ \uuaux\pa{#1} }

\NewDocumentCommand{\Skew}{ m }{ \Skewaux\pa{#1} }
\NewDocumentCommand{\Sym}{ m }{ \Symaux\pa{#1} }
\NewDocumentCommand{\Symp}{ m }{ \Symaux^+\pa{#1} }

\NewDocumentCommand{\M}{ >{\SplitArgument{1}{,}}m}{%
    \RR^{\prodaux #1}%
}
\NewDocumentCommand{\Rank}{>{\SplitArgument{2}{,}}m}{
    \Rankaux #1%
}
\NewDocumentCommand{\Rankaux}{ m m m }{%
    \RR^{#1 \times #2}_{#3}%
}
\NewDocumentCommand{\St}{ >{\SplitArgument{1}{,}}m}{%
    \Staux\pa{\commasaux #1}%
}
\NewDocumentCommand{\Gr}{ >{\SplitArgument{1}{,}}m}{%
    \Graux\pa{\commasaux #1}%
}

\NewDocumentCommand{\commasaux}{ m m }{%
    \IfNoValueTF{#2}{ #1 }{ #1, #2 }%
}
\NewDocumentCommand{\prodaux}{ m m }{%
    \IfNoValueTF{#2}{ #1 \times #1 }{ #1 \times #2 }%
}

\DeclareMathOperator{\expm}{expm}                           
\DeclareMathOperator{\logm}{logm}                           
\newcommand\transaux{\intercal}                             
\newcommand\trans[1]{#1^\transaux}                          
\newcommand\conj[1]{#1^{\mathrm{H}}}                        
\newcommand\VV{V}                                           
\newcommand\HH{\mathcal{H}}                                 
\renewcommand\SS{\mathbb{S}}                                
\newcommand\loss{\ell}                                      
\DeclareMathOperator\ad{ad}                                 
\DeclareMathOperator\Ad{Ad}                                 
\newcommand\glie{\mathfrak{g}}                              
\newcommand\hlie{\mathfrak{h}}                              
\newcommand\mlie{\mathfrak{m}}                              
\newcommand\conn{\nabla}                                    
\newcommand*{\connflat}{\nabla^{\text{flat}}}               
\newcommand\lie{\mathcal{L}}                                
\DeclareMathOperator{\unif}{\mathcal{U}}                    
\DeclareMathOperator{\diag}{diag}                           
\newcommand\I{\mathrm{I}}                                   
\renewcommand*{\Im}{\mathrm{Im}}                            
\newcommand{\relu}{\code{relu}}                             
\newcommand{\XXin}{\mathcal{X}}                             
\newcommand{\YYout}{\mathcal{Y}}                            
\newcommand{\XX}{\mathsf{X}}                                
\newcommand{\ort}{\bot}                                     
\newcommand{\lv}[1]{{#1}^{\texttt{L}}}                      
\newcommand{\killing}[1]{{#1}^+}                            
\newcommand\pskew{\mathrm{skew}}                            
\newcommand\psym{\mathrm{sym}}                              
\DeclareMathOperator{\sigm}{sigm}                           
\DeclareMathOperator{\softmax}{softmax}                     
\newcommand{\hprod}{\mathbin{\odot}}                        
\DeclareMathOperator{\sn}{sn}                               
\DeclareMathOperator{\ct}{ct}                               
\DeclareMathOperator{\inj}{r_{inj}}                         
\DeclareMathOperator{\diam}{diam}                           
\newcommand*{\paral}[1]{#1^{\dag}}                          
\newcommand*{\normal}[1]{#1^{\perp}}                        
\newcommand*{\dgamma}{\dot{\gamma}}                         

\newcommand{\quot}[1]{\widetilde{#1}}                       
\newcommand{\total}[1]{\overline{#1}}                       
\newcommand{\triv}{\phi}                                    
\newcommand{\retr}{r}                                       
\newcommand\MM{M}                                           
\newcommand{\gm}{\textsl{g}}                                
\DeclareMathOperator{\Iso}{Iso}                             
\DeclareMathOperator{\segaux}{seg}                          
\newcommand{\seg}[1]{\segaux\pa{#1}}                        
\newcommand{\segint}[1]{\segaux^0\pa{#1}}                   
\newcommand{\segm}[1]{U_{#1}}                               
\DeclareMathOperator{\tcutaux}{Tcut}                        
\DeclareMathOperator{\cutaux}{cut}                          
\newcommand{\cut}[1]{\cutaux\pa{#1}}                        
\newcommand{\tcut}[1]{\tcutaux\pa{#1}}                      
\DeclareMathOperator{\conjlocaux}{conj}                     
\newcommand{\conjloc}[1]{\conjlocaux\pa{#1}}                
\newcommand{\conjpoint}[1]{r_{\conjlocaux}\pa{#1}}          
\newcommand{\shape}{II}                                     

\NewDocumentCommand{\liebrack}{s O{} >{\SplitArgument{1}{,}}m}{%
    \IfBooleanTF{#1}{\liebrackaux*#3}{\liebrackaux[#2]#3}%
}
\DeclarePairedDelimiterX{\liebrackaux}[2]{\lbrack}{\rbrack}{#1, #2}

\usetikzlibrary{arrows} 	
\tikzset{->, >=stealth', shorten >=1pt, auto, node distance=2cm, semithick, baseline=(current bounding box.center)}

\title{Geometric Optimisation on\\[1ex]Manifolds with Applications to\\[1ex]Deep Learning}
\author{Mario Lezcano Casado}
\college{Keble College}

\degree{Doctor of Philosophy}     
\degreedate{Hilary 2021}      

\begin{document}
\baselineskip=18pt plus1pt

\setcounter{secnumdepth}{3}
\setcounter{tocdepth}{3}

\hypersetup{pageanchor=false}
\maketitle                  
\begin{acknowledgements}
    I would like to begin by thanking my supervisors, Prof.\ Raphael Hauser and Prof.\ Vidit Nanda for their time, effort, and guidance throughout this thesis. Thank you for giving me the flexibility to delve into whatever mad topic I found interesting at the time. Discussions with Jaime Mendizabal and David Martínez-Rubio have been truly excellent. Their suggestions, clarifications, and probing questions pushed me to think in new ways and greatly improved the thesis. My collaborators Adam Golinski and Tom Rainforth helped me to explore the applications that this research has in other fields. Prof.\ Jose Manuel Gamboa Mutuberría has been a constant source of support and encouragement all these years. To all of them a warm thank you.

    A heartfelt thank you to my parents and my brother for their unconditional support, in particular during the final writing of this thesis.

    On the less academic side, I would like to thank everybody in and around La Katacumbia. Borondiño, Juan, Adri, Álvaro, Jorge, Recio, Miguel, David vinos, Sam, Felix, Steve, Adam, Sergio, Daniel, María, Laura and everyone at $11$ Iffley road\textellipsis{} You have made these years in the small city of Oxford incredibly enjoyable. From the parties by the river, to nights at Plush and Cellar, barbecues, swimming pool, squash, gigs at the Tap Social and even a very special summer solstice in Stonehenge followed by a weekend in Bristol that was not for the faint-hearted. In all fairness, a long list of things to have done in the small town of Oxford.

    I must single out Garima, Adri, and David, with whom I spent the best of lockdowns. Starting as an odd night in London that looked straight from a B film, it was followed by some of the best food I've had in my life---the arroz con leche, risotto, and banana bread come to mind---Smash-up, Hollow Knight, Cuphead, the month that we worked out\textellipsis{} For these things and many more, thank you for making this surreal year special.

    I would also like to thank Amigas y Conocidas, Roberto, Lorenzo, Cámara, Charles, Visierdo, Inés and those floating around it such as my beloved Juan. It is always fantastic to be back and be able to forget whatever I'm working on and just have a great night out, whether in Burgos or in Berlin.

    There are a few more that have made the last few years enjoyable---Irene and Elyem from la Caixa, Mendi, Paco, David, Aguirre, Aitor and Adrià from Comeis, Guillem and Sara, Debbie, Lina, Elena, Noémie, Manolo\textellipsis---with whom I have shared some truly amazing moments.

    To all of you: Thank you.
\end{acknowledgements}
\hypersetup{pageanchor=true}

\begin{romanpages}          
\tableofcontents            
\end{romanpages}            

\chapter{Introduction}
    \section{Why this Thesis?}
    We start this thesis by justifying the theoretical and practical approach that this thesis takes. We outline the general problems that we study on throughout it. After this, we summarise the results in the thesis and its structure.

    \subsection{A practical reason}
    Optimisation has always been a field with vast applications in the real world, bringing together the fields of analysis and constructive mathematics---and subsequently computer science---in order to be able to approximate in a principled way the minima of functions traditionally defined on some subset of the Euclidean space.

    A natural extension of this setting is that of optimisation on Riemannian manifolds. Most of the interesting sets of matrices used in the context of optimisation turn out to have a manifold structure. Optimisation on manifolds is both theoretically and practically challenging due to the inherent complexity of the objects involved. Even then, optimisation on matrix manifolds has proven to be rather useful in many subfields of machine learning and neural networks (\nn). Examples of interesting matrix manifolds in the context of gradient-based optimisation are the set of positive definite matrices in Bayesian statistics as kernels~\parencite{rasmussen2005gaussian}, orthogonal matrices within recurrent neural networks~\parencite{arjovsky2016unitary,helfrich2018orthogonal,lezcanocasado2019cheap}, \nns{} with structured linear layers via matrix factorisations such as the \qr{} or the \svd{} decomposition~\parencite{zhang2018stabilizing,kingma2018glow}, or invertible matrices in normalising flows~\parencite{berg2018sylvester} and \vaes~\parencite{tomczak2016improving}.

    In parallel, during the last $10$ years, neural networks have revolutionised many fields such as speech recognition~\parencite{hinton2012deep,saon2017english}, image recognition~\parencite{krizhevsky2012imagenet,huang2019gpipe}, and natural language processing~\parencite{collobert2011natural,ott2018scaling}. Initially, neural networks were a concatenation of linear layers, but with time, these architectures have been getting more and more complex with the introduction of architectures such as attention mechanisms~\parencite{bahdanau2015neural,vaswani2017attention}, neural Turing machines~\parencite{graves2014neural}, variational auto encoders (\vae)~\parencite{kingma2014auto} and generative adversarial networks (\gan)~\parencite{goodfellow2014generative}.

    This trend of building larger models on more advanced hardware, such as \gpus{} and \tpus{} has gone hand in hand with the development of large linear algebra libraries such as PyTorch, Tensorflow, or \jax. These libraries are similar to Numpy or Scipy, but also come with full \gpu{} and \tpu{} support, an autodifferentiation engine, support for automatic mixed-precision (\amp), and many other capabilities to help develop modern high-dimensional deep learning models that scale to datasets with millions of datapoints.

    At the time of writing, there is no library to put these two trends together in a practical way. The current libraries that implement optimisation on manifolds techniques are designed for researchers, not for practitioners. For this reason, the thesis aims to fulfil the following two objectives:

    \paragraph{GeoTorch: A library for optimisation on manifolds at scale}
    We design and implement a Python library to help the non-expert using all these powerful tools in a way that is efficient, extensible, and simple to incorporate into the workflow of the data scientist, practitioner, and applied researcher. The algorithms implemented in this library have been designed with usability and \gpu{} efficiency in mind, and they can be added to any PyTorch model with just one extra line of code.

    \paragraph{Orthogonal \rnns}
    We showcase the effectiveness of these tools on an application of optimisation on manifolds in the setting of time series analysis. In this setting, orthogonal and unitary optimisation is used to constraint and regularise recurrent models and avoid \emph{vanishing} and \emph{exploding gradient} problems.
    The algorithms designed for GeoTorch allow us to achieve state of the art results in the standard tests for this family of models.

    \subsection{A theoretical reason}
    The current approach to optimisation on manifolds rarely makes use of the abstract approach, favouring the embedded setting. This seems like a natural thing to do, as the manifolds of interest in Riemannian optimisation are, by definitions, manifolds embedded in $\M{n,k}$---\ie, matrix manifolds.
    When one has a matrix manifold, it is natural to use the structure of the manifold as a manifold embedded in a linear space and identify points on the manifold, vectors, actions by a (matrix) Lie group, isometries and similar objects as matrices or linear functions. This approach often puts aside the entire abstract toolbox from modern differential geometry, like fibred manifolds, group actions, principal connections, or Riemannian homogeneous spaces, in favour of matrix manipulations.

    Using matrices and algebraic manipulations rather than sections of the different associated bundles and group actions often hides the actual symmetries that make these manifolds tractable. It is these symmetries that allow us to find closed-form solutions to some differential equations on the manifold. Cartan made a similar point when talking about the abstract index formalism:

    \begin{quote}
        The great opportunities that the \emph{absolute differential calculus} of Ricci and Levi-Civita has given us must not prevent us from avoiding purely formal calculations, where the profusion%
        \footnote{It is of historical interest to note that this expression was popularised in English by Spivak, who jokingly translated it as ``debauch of indices'' in the first volume of his treaty on differential geometry. The word ``débauches'' in French has the same meaning as ``debauch'' in English in almost any situation, being this one of the few exceptions in which it simply means ``profusion'' with no extra connotation.}
    of indices masks an often very simple geometric reality. It is this reality that I set out to highlight here.\hfill(Original in French%
        \footnote{Les services éminents qu'a rendus et que rendra encore le Calcul différentiel absolu de Ricci et Levi-Civita ne doivent pas nous empêcher d'éviter les calculs trop exclusivement formels, où les débauches d'indices masquent une réalité géométrique souvent très simple. C'est cette réalité que j'ai cherché à mettre partout en évidence.})
        ~\parencite{cartan1946lecons}
    \end{quote}

    Identifying geometric objects on the manifold with matrices simplifies the background necessary to check whether a proof is true, but it often hides the true reasons why the proof is true. Drawing similarities with type systems in computer science, we can see that while it is true that dropping the types of a correct program does not make it incorrect, it certainly makes it more difficult to reason about. The same happens if we simply write points on the manifold or sections on the tangent bundle of the manifold as matrices: It hides the geometric meaning from these operations.

    In this sense, this more abstract approach, whenever applicable, can be used to explain the nature of some numerical computations. In particular, in the field of constrained optimisation, when the set of constraints is a Riemannian manifold with enough symmetries, we are in the best of scenarios as we can use all the tool set from differential, Riemannian, and comparison geometry to explain the behaviour of our algorithms. In this sense, this helps us in getting closer to the goal already set by Hamming in his book on numerical methods

    \begin{quote}
        The purpose of computing is insight, not numbers.\hfill\parencite{hamming1962numerical}
    \end{quote}

    With these points in mind, in this thesis we will use the language of differential geometry for studying the field of optimisation on manifolds. In particular, we will explore the following two research directions:

    \paragraph{Use abstract differential geometry to study optimisation problems over matrices}
    We use the more abstract language of global differential geometry and group actions to prove convergence bounds on certain optimisation problems on manifolds, postponing the instantiation of the results on concrete manifolds as much as possible.

    In doing so, we follow the approach started by Smith in his PhD thesis~\parencite{smith1993geometric}, which was written aimed for a numerical analysis audience in the influential paper~\parencite{edelman1998geometry}.

    \paragraph{Use comparison geometry to study optimisation problems on Riemannian manifolds}
    We use tools from comparison geometry to give bounds on quantities that are of interest in optimisation problems. In particular, we build on the work of~\parencite{kaul1976schranken} to give explicit bounds on the norm of the second derivative of the Riemannian exponential.

    \section{Contributions and Outline}
    This thesis puts together three fields that are not usually found together. These are global differential geometry, optimisation on manifolds, and deep learning. Researchers in each of these fields approach problems lying at the intersection with completely different tools, notation, and goals in mind. It is for this reason that we include at the beginning of the thesis two introductory chapters, one on classical optimisation, machine learning, and their ties to deep learning (\Cref{ch:neuralnetworks}), and another one on differential geometry from a global perspective (\Cref{ch:geometry}). All the results in these chapters are well-known to researchers from the fields of optimisation, machine learning and/or differential and Riemannian geometry, so it is only the exposition and comments throughout it that are original.

    In the last three chapters, we present the research carried by the author during these last years. All the results in these chapters are original except when specified otherwise. When a section is well-known to geometers, we carefully announce so either at the beginning of each section or before each result so that it is perfectly clear what is classic and what is new.

    \Cref{ch:fibred_manifolds_in_optimisation}. Its main theme is optimisation on manifolds, and its secondary theme is differential and Riemannian geometry. In this chapter, we study how the \textbf{theory of pullbacks along submersions} is a fruitful one (\Cref{sec:riemannian_submersions_optimisation}). This theory helps to simplify problems in optimisation from spaces with a difficult geometry to simpler ones. We also show how to use the theory laid out in~\Cref{ch:geometry} to \textbf{prove second order bounds on the pullback of a function by a Riemannian submersion} (\Cref{thm:submersion_submetry,prop:first_second_order_bounds_submersion}). We then go and look at the special case when we pullback an optimisation problem to a Euclidean space via a retraction. We call this the \textbf{static trivialisation framework}. We make special emphasis on the case when we use the Riemannian exponential map to pullback the problem to a tangent space (\Cref{sec:riemannian_exponential}), as this will be the main tool studied and used in~\Cref{ch:second_order_bounds} and implemented in~\Cref{ch:geotorch}. At the end, we introduce the \textbf{dynamic trivialisation framework} (\Cref{alg:dyn_triv_bundle}), which sees the Riemannian exponential as a map from $T\MM$ to $\MM$ and uses the vector bundle structure of $T\MM$ to solve the shortcomings of static trivialisations. These ideas were was presented at ICML 2019 and NeurIPS 2019~\parencite{lezcanocasado2019cheap,lezcanocasado2019trivializations}.

    \Cref{ch:second_order_bounds}. Its main theme is Riemannian and comparison geometry, and its secondary theme is optimisation on manifolds. In this chapter, we give bounds on $\norm{\Hess\pa{f \circ \exp_p}}$ in terms of first and second order bounds on $f$ and the curvature of the Riemannian manifold $(\MM, \gm)$. We start the chapter by doing a literature review of previous approaches to the subject, where we show how these bounds can be readily used in a number of papers that use second order bounds on retractions. We then offer a self-contained proof of Rauch's theorem, as we heavily rely on it throughout the second part of the chapter, and it is not very well-known in the field of optimisation. We also look at its implications in the area of convex optimisation on manifolds at the end of the section. After this, we use these results and a few more to develop novel \textbf{second order bounds on the Riemannian exponential} (\Cref{thm:second_order_bounds,thm:full_second_order_bounds}). From the first and second order bounds, we give \textbf{conditional convergence rates for \sgd{} on the function $f \circ \exp_p$} (\Cref{thm:static_trivializations}), and the equivalent result for the dynamic trivialisation algorithm (\Cref{thm:dynamic_trivializations}). The second order bounds and their application to get convergence rates for static and dynamic trivialisations are the main theoretical results of the thesis. We finish this chapter with a short \textbf{characterisation result for retractions that are Riemannian exponential maps} for some connection in terms of sprays on the second tangent bundle $TT\MM$ (\Cref{thm:retraction}).

    \Cref{ch:geotorch}. Its main themes are practical optimisation on manifolds and deep learning. In this chapter, we explain how to implement the ideas of static and dynamic trivialisations that were laid out in~\Cref{ch:fibred_manifolds_in_optimisation}. The exposition in this chapter is done by developing the example of optimisation with orthogonal constraints. We then present a literature review of the current state of the art of practical optimisation on manifolds. After this, we go on to describe the implementation details of library \textbf{GeoTorch} that we have developed to perform scalable optimisation on manifolds in large models (\Cref{sec:overview_geotorch}). We finish this chapter and the thesis showing a number of experiments to showcase the practical efficiency of dynamic trivialisations in the context of orthogonal \rnns  (\Cref{sec:exprnn,sec:experiments}).

    GeoTorch is available under \mitedu{} license at
\begin{center}
    \url{https://github.com/Lezcano/geotorch}
\end{center}

    The ideas in this chapter have been submitted to core PyTorch and have been accepted to be part of PyTorch 1.9.0. The work in~\Cref{sec:approximations_exponential} led to the addition of the matrix exponential into core PyTorch by Nikita Vedeneev. This implementation yielded a $\times 12$ speed-up over the previous fastest GPU implementation.

\clearpage
\chapter{Machine Learning and Neural Networks}\label{ch:neuralnetworks}
    In this chapter, we aim to contextualise the field of optimisation on manifolds and its applications to neural networks within the broader field of machine learning. As we will see, the field of theoretical machine learning draws from fields as disparate as probability in Banach spaces, optimisation in Hilbert spaces, computer science, and engineering. This makes it difficult to find consistent terminology for this field, where some concepts have been rediscovered several times during the last $50$ years. As such, this chapter has an introductory as well as a unifying character, giving a short introduction to all these ideas, their development throughout history and the kind of results that one encounters in each relevant field, while also introducing the nomenclature used in the fields of machine learning and deep learning, which will be used throughout the rest of the thesis. This is done in the hope that readers coming from other fields are not left with a text full of foreign terminology without a proper dictionary.

All the theoretical ideas in this chapter and many more can be found in the books~\parencites{vapnik1995nature}{nesterov2004introductory}.

\paragraph{Outline of the chapter.}
We will start by giving a short motivation for the ideas that underlay the current theoretical approach to offline machine learning, with the usual split between generalisation and optimisation. Then, we will provide a historical review of some classical results in the field of statistical learning theory and optimisation in $\RR^n$, to serve as an introduction and contextualisation to the convergence results that we will give in the thesis. After that, we will give specific examples of a popular family of models called neural networks, and we will show how to implement them in the \textbf{PyTorch} library. Throughout the course of this introduction to the practical side of machine learning, we will make sure to always draw connections to the theory outlined before. We hope this connection between theory and practice helps to put in context the ideas that will be introduced in the thesis, where practical remarks will often be interleaved with more theoretical ones to give algorithms that are solid on both ends.

\section{Offline Statistical Learning}\label{sec:offline_statistical_learning}
Before introducing what neural networks are, we shall first present the standard framework for offline statistical learning in its most basic setting.

\subsection{Framework and terminology}
\paragraph{Features and labels.}
Let $\XXin$ and $\YYout$ be two sets. We say that $\XXin$ is the \textbf{feature space} and $\YYout$ is the \textbf{label space}. For example, if we wanted to model whether the income of the parents affects whether their children end up studying at Oxford, we would have $\XXin \iso \RR_{\geq 0}$ as the set of all possible incomes (for example in pounds), and $\YYout \iso \set{0,1}$ which models whether certain children of theirs studied at Oxford or not.\footnote{The use of $\iso$ instead of $=$ here has a bit of a pedantic nature. Note that the set of all weights $\XXin$ \emph{is not} intrinsically equal to the abstract set $\RR_{\geq 0}$ but it is merely identified mathematically with it. Hence, $\XXin$ and $\RR_{\geq 0}$ are isomorphic, and not equal.}\footnote{The model and methodology presented here is an over-simplification which would not be of much use in the real world. In a real model, one would compare many other variables, together with a much more complex analysis than the one that we will present here.} In general, we could have $\XXin$ and $\YYout$ to be any pair of sets, discrete or continuous, without any structure.

\paragraph{Distribution and training set.}
We assume that there exists a random variable $(X, Y)$ on $\XXin \times \YYout$ that we want to study. We do not have direct access to $(X,Y)$, but we can sample from it.
One convenient way of thinking of the random variable $(X,Y)$ is as thinking that there exists certain random variable $X$ on $\XXin$ which induces a conditional one $Y \mid X$ on $\YYout$.

In the context of offline statistical learning, we assume that we have access to $m$ \iid{} samples $(X_1,Y_1), \dots, (X_m, Y_m) \in \XXin \times \YYout$ from $(X, Y)$. We call these samples the \textbf{training set}, and we assume $m$ to be fixed a priori. This is where offline learning differs from online learning. In online learning, we assume that we have access to a sampling oracle that gives us \iid{} samples of $(X,Y)$ on demand, so our algorithm may adapt as we get more and more samples.\footnote{Often these $m$ samples would be split into \textbf{training set}, \textbf{test set} and \textbf{validation set}. Although this practice is fundamental in practice, we will not talk about it here to simplify this introduction.}

In the example of the acceptance at Oxford, the distribution $(X,Y)$ would be the actual unknown distribution, having certain distribution on the income $\XXin$ and certain probability of being accepted at Oxford given certain fixed income $Y \mid X$. The training set could be $m$ samples that we collect via a survey.

\paragraph{Model and loss function.}
In order to model $Y \mid X$, we consider a family of functions $\deffun{f_\theta : \XXin -> \YYout;}$ indexed by some parameter $\theta \in \Theta$. Usually $\Theta$ will be a Euclidean space $\RR^n$ for some $n > 0$, but in more general settings, like those that we will consider in this thesis, $\Theta$ can be a manifold or a more general topological space. The family of functions $\set{f_\theta | \theta \in \Theta}$ is called the \textbf{model}. The last thing that we will need, is a \textbf{loss function} $\deffun{\loss : \YYout \times \YYout -> [0,\infty);}$, which serves as a proxy for a distance on $\YYout$, although it need not be symmetric---and often it is not.

Following the previous example of admissions to Oxford, we would have to model the binary output $\YYout = \set{0,1}$. To do so, one usually considers the relaxation $\YYout \iso \set{0,1} \subset [0,1]$. An example of a model in this case could be what is called \textbf{polynomial linear regression of degree $n$} for a fixed $n > 0$ together with a mapping $\deffun{\sigma : \RR -> [0,1];}$. This model can be mathematically described setting $\Theta \iso \RR^{n+1}$ and
\[
    \deffun{f_\theta : \XXin \times \RR^{n+1} -> [0,1];
            (x, \theta) -> f_\theta(x) = \sigma\pa[\Big]{\sum_{k=0}^n \theta_k x^k}}
\]
and we could choose $\sigm(x) \defi \frac{1}{1+\exp(-x)}$ as the mapping. This is called the \textbf{sigmoid function}.\footnote{The image of the sigmoid function is $(0,1)$ rather than the full $[0,1]$, which seems in contradiction with the fact that we are interested in modelling a random variable with values on the ends of $[0,1]$. This is done due to convenience, as $(0,1)$ is diffeomorphic to $\RR$ (through the sigmoid function, for example), while $[0,1]$ is not even homeomorphic to $\RR$.}

In classification models, it is customary to think of the map $f_\theta$ with image contained in $[0,1]$ as a random variable on $[0,1]$ modelling $Y | X$. As such, the usual loss function is the \textbf{cross-entropy} between the distribution induced by $f_\theta$ and the empirical distribution
\[
    \loss(x,y) =
    \begin{cases*}
        \log(x)   & \text{if $y = 1$} \\
        \log(1-x) & \text{if $y = 0$}
    \end{cases*}
\]
Note that, in this case, $y$ is assumed to be in $\set{0,1}$, while $x$ is assumed to be in $(0,1)$, which is the image of $f_\theta$.

\begin{remark}[Modelling a problem]
The choice of the model and the loss function heavily depends on the problem at hand, and currently, it is closer to an art than a science. For this reason, we will not explore these further theoretically, as the scope of doing so is limited. We will come back to them in the experiment sections, when we implement certain models to benchmark the algorithms that we will propose against those in the literature.
\end{remark}

\paragraph{Expected and empirical risk.}
With these objects in hand, we can define the \textbf{expected risk} or \textbf{population risk} as
\[
    r(\theta) := \E\cor{\loss\pa{f_\theta(X), Y}}.
\]
We would then like to approximate a solution to the \textbf{expected risk problem}
\begin{equation}
    \argmin_{\theta \in \Theta} \E\cor{\loss\pa{f_\theta(X), Y}}.
\end{equation}
Of course, we cannot directly compute $r(\theta)$ as we do not have direct access to the distribution of $(X, Y)$, but what we can do is to estimate it via the \textbf{empirical risk}
\[
    R_m(\theta) := \frac{1}{m}\sum_{i=1}^m \loss\pa{f_\theta\pa{X_i}, Y_i}.
\]
We can compute the empirical risk, as opposed to the population risk. It also has the desirable property of being an unbiased consistent estimator of $r(\theta)$. Using this estimator, we can define the \textbf{empirical risk minimisation objective} (\erm)
\begin{equation}\label{eq:erm}
    \min_{\theta \in \Theta} \frac{1}{m}\sum_{i=1}^m \loss\pa{f_\theta\pa{X_i}, Y_i}.
\end{equation}
Denote an \textbf{expected optimal value} for the expected risk problem and an \textbf{empirical optimal value} for the empirical risk problem by
\[
\theta^\ast \in \argmin_{\theta \in \Theta} r(\theta)
\qquad
\psi^\ast \in \argmin_{\theta \in \Theta} R_m(\theta).
\]
In general, we will not be able to compute $\psi^\ast$ exactly. We will use some optimisation method that gives a sequence of estimates $\psi_t$ given by some optimisation algorithm. Putting all these quantities together, we have the so-called \textbf{statistics-optimisation trade-off}.
\begin{proposition}[Statistics-optimisation trade-off]
    We have the following factorisation:
    \begin{equation}\label{eq:statistics_optimisation_tradeoff}
        r(\psi_t) - r(\theta^\ast) \leq
        \underbrace{\sup_{\theta \in \Theta}\pa{r(\theta) - R_m(\theta)} +
        \sup_{\theta \in \Theta}\pa{R_m(\theta) - r(\theta)}}_{\text{Statistical term}} +
        \underbrace{R_m(\psi_t) - R_m(\psi^\ast)}_{\text{Optimisation term}}.
    \end{equation}
\end{proposition}
\begin{proof}
    Interpolating we have
    \[
        r(\psi_t) - r(\theta^\ast) = \pa{r(\psi_t) - R_m(\psi_t)} + \pa{R_m(\psi_t) - R_m(\psi^\ast)} + \pa{R_m(\psi^\ast) - R_m(\theta^\ast)} + \pa{R_m(\theta^\ast) - r(\theta^\ast)}.
    \]
    Noting that the term in the third parenthesis is non-positive and taking suprema we get the result.
\end{proof}

The area of \textbf{statistical learning theory} studies how to give bounds for the statistical term for different classes of models,\footnote{More generally, statistical learning theory tries to bound $r(\psi)$ for an arbitrary point $\psi \in \Theta$. The point does not need to come from \erm, and the bounds may be different to the ones presented here. That being said, this is by far the most used approach.} and tries to give bounds on the number of samples $m$ needed to make the statistical term smaller than certain given quantity $\epsilon$. Given that the samples $(X_i, Y_i)$ are stochastic, the results in this field are given as results \textbf{with high probability}, that is, one proves that the statistical term is less than $\epsilon$ with probability at least $1-\delta$, where $\epsilon, \delta > 0$ are some parameters given a priori. The abstract idea behind these type of results is that of \textbf{concentration of measure}, quantifying for example the rate of convergence of certain central limit theorem. These usually come from fine high-dimensional versions of Markov's and Chebyshev's inequalities and other tail inequalities.

On the other hand, the area of \textbf{stochastic optimisation} studies how to design algorithms to approximate solutions to the problem in~\eqref{eq:erm}. Given an algorithm, one proves the convergence of $\psi_t$ as a function of the number of steps $t$ to a minimiser $\psi^\ast$. One often assumes some notion of convexity or some degree of regularity of the derivatives of the function being minimised---$R_m(\theta)$ in this case---to be able to prove this kind of results.

\subsection{A note on statistical learning theory}
The statistical term is studied in the \textbf{statistical learning theory}, often also referred to as \textbf{generalisation theory}, given that it studies how well an empirical optimum---for example $R_m(\psi^\ast)$ or any other estimator---approximates $r(\theta^\ast)$. In this approach, the quantity $R_m(\theta)$ is treated as random variable that depends on $(X_1, Y_1), \dots, (X_m, Y_m)$.

\paragraph{Dimensionless bounds.}
In practice, machine learning models are larger and larger. In particular, the parameter space $\Theta$ can have a dimension in the order of the billions. As such, the bounds that this theory hopes to get are \textbf{dimensionless}, that is, independent of the dimension of $\Theta$.\footnote{Note that this need not hold on the optimisation term. When computing bounds for the optimisation term later on, we will often put them in terms of the Lipschitz constant. This Lipschitz constant often depends on the dimension of the problem.} For this reason, this theory is based in the field of \textbf{high dimensional probability}. When one is able to get dimensionless bounds, it tends to be the case that the proofs and results can be adapted to work in Hilbert and Banach spaces, and as such, it should be no surprise that this theory is itself based in the theory of \textbf{probability in Banach spaces}. The reference book in this field is Ledoux and Talagrand's~\parencite{ledoux1991probability}.

\paragraph{Modelling the generalisation term.}
The approach to get bounds on the statistical term in~\eqref{eq:statistics_optimisation_tradeoff} comes from assuming a metric space $(\Theta, d)$ on the parameter space and, for a fixed $m > 0$, considering the stochastic process $X_\theta = r(\theta) - R_m(\theta)$ that depends on the random variables $(X,Y), (X_1, Y_1), \dots, (X_m, Y_m)$. The problem can be then translated to bounding above and below the quantity
\[
    \E \sup_{\theta \in \Theta} X_\theta.
\]
Generalising the problem by forgetting the particular structure of $X_\theta$, one looks to give upper and lower bounds on the expectation of a general stochastic process $X_\theta$ on a metric space $(\Theta, d)$.

The first result in this direction was the \textbf{chaining argument}. It was first developed by Kolmogorov in $1934$ but never published~\parencite{chentsov1956weak}, and it was then generalised by Dudley~\parencite{dudley1967sizes}, who proved \textbf{Dudley's entropy integral} under the assumption that $X_\theta$ is subgaussian, giving the upper bound
\[
    \E\sup_{\theta \in \Theta} X_\theta \leq 12 \intf[0][\infty]{\sqrt{\log N\pa{\Theta, d, \epsilon}}}{\epsilon}
\]
where $N(\Theta, d, \epsilon)$ is the \textbf{$\epsilon$-covering number} of the metric space $(\Theta, d)$, that is, the minimum number of balls of radius $\epsilon$ needed to cover the whole $\Theta$.

Lower bounds on this quantity were given in the theory of Gaussian processes by Slepian, Fernike, and Sudakov. These results assume $X_\theta$ to be a Gaussian process and consider the \textbf{natural metric} induced by it in $\Theta$, which is the $L^2$ metric
\[
    d(\theta_1, \theta_2) = \norm{X_{\theta_1} - X_{\theta_2}}_2 = \E\cor{\abs{X_{\theta_1} - X_{\theta_2}}^2}^{1/2}.
\]
Using this distance, a Gaussian process is, by definition, also subgaussian under this metric. These natural metrics allow giving not only lower bounds, but actually matching upper and lower bounds. This was done by Talagrand in the celebrated \textbf{majorising measure theorem}~\parencites{talagrand1987regularity}{talagrand1992simple}{talagrand1994constructions}. This result is quite technical, as it involves some fairly complex constructions on metric spaces. Intuitively, these structures can be seen as a refinement of the global quantity $N(\Theta, d, \epsilon)$ to be rather a finer local one.

Going back to statistical learning theory, one of the most basic models of learning to consider is the set
\[
    \Theta = \set{(-\infty,x] | x \in \RR}
\]
and $f_\theta$ to be the indicator function on the interval $\theta \in \Theta$. A natural metric for these functions is the $L^\infty$ metric, as it is a metric that makes the empirical measure $\mu_m = \sum_{i=1}^m \delta_{X_i}$ subgaussian. On the other hand, it is direct to see that the covering numbers $N(\Theta, \norm{-}_\infty, \epsilon) = \infty$ for any $\epsilon$. As such, all the previous bounds become trivial for the case of infinite sets of indicator functions. However, it is possible to give finite bounds for these functions via a process called \textbf{symmetrisation}, which can be summarised in the following chain of inequalities
\[
    \E\cor[\bigg]{\sup_{\theta \in \Theta}\sum_{i=1}^m\pa{f_\theta(X_i) - \E f_\theta}} \leq
    2\E\cor[\bigg]{\sup_{\theta \in \Theta}\sum_{i=1}^m\epsilon_k f(X_i)} \leq
    \sqrt{2\pi}\E\cor[\bigg]{\sup_{\theta \in \Theta}\sum_{i=1}^m g_k f(X_i)}
\]
where $\epsilon_k$ are \iid{} symmetric Bernoulli and $g_k$ are \iid{} $\Norm(0,1)$. The quantities in the first and second bound are called the \textbf{Rademacher complexity} of the set of functions $\Theta$ and the \textbf{Gaussian complexity} of $\Theta$ respectively. These inequalities tell us that bounding any of these two complexities is enough to give bounds on the empirical process, and as such, to give a priori bounds on the generalisation of the class of our class of models. Most of the theory in statistical machine learning centres around giving bounds for these two complexity quantities under different assumptions on the family $f_\theta$.

For a rather in-depth introduction to these and other topics from a probability theory perspective, we strongly recommend the book by Ramon van Handel~\parencite{vanhandel2014probability}. For a book with more applications towards the machine learning see~\parencite{vershynin2018high}. There is also the classical book by Vapnik which lays the foundations of the field of statistical learning~\parencite{vapnik1995nature}.

To finish this note on the statistical term, we shall give an example of the magnitude of this term for a simple family of classifiers.

\begin{proposition}
    Let $(X,Y)$ be a random variable on $\XXin \times \YYout \subset \RR^n \times \RR$. Let $B_n(r)\subset \RR^n$ be the ball centred at zero with respect to the $\ell^2$ norm. Consider two bounded sets $\Theta \subset B_n(r_\Theta)$, $\XXin \subset B_n(r_\XXin)$ for two fixed $r_\Theta, r_\XXin >0$. Consider the family of linear predictors $f_\theta(x) = \scalar{\theta, x}$ parametrised by $\Theta$. For a loss function $L$-Lipschitz on the first variable, we can bound the statistical term as
    \[
        \E\sup_{\theta \in \Theta}\pa{R_m(\theta) - r(\theta)} =
    \E\cor[\bigg]{\sup_{\theta \in \Theta}\sum_{i=1}^m\loss\pa{\scalar{\theta, X_i}, Y_i} - \E \loss\pa{\scalar{\theta, X}, Y}} \leq
        2 \frac{r_\XXin r_\Theta L}{\sqrt{m}}.
    \]
\end{proposition}
\begin{proof}
    See for example~\parencite[Lemma 26.9 and Lemma 26.10]{shalevshwartz2014understanding}.
\end{proof}

If one further assumes subgaussianity of the distribution of $(X,Y)$, it is then possible to convert these bounds in expectation to bounds \textbf{with high probability}, that is, one gets tail bounds on the distribution of $R_m(\theta)$.

\begin{remark}[Qualitative properties of the bound]
    The first thing to note is that, as we mentioned before, this bound is independent of the dimension $n$ of the parameter space $\Theta$ and the input space $\XXin$.

    Secondly, this bound tells us that, in the most basic case, the statistical term in~\eqref{eq:statistics_optimisation_tradeoff} scales as $\mathcal{O}\pa{\frac{1}{\sqrt{m}}}$. It is also possible to show that this bound is asymptotically tight for this family of classifiers. This bound suggests that it is enough to optimise up to a precision on the order of $\mathcal{O}\pa{\frac{1}{\sqrt{m}}}$.

    The asymptotics $\mathcal{O}\pa{\frac{1}{\sqrt{m}}}$ do not appear just in the case of linear classifiers, but also in many other architectures. In essence, they stem from the convergence rates given by the central limit theorem. These rates are asymptotically tight unless extra assumptions are considered. In typical scenarios, under more restrictive hypotheses, the rate of convergence can be improved to $\mathcal{O}\pa{\frac{1}{m}}$, which is usually referred to as \textbf{fast rate}.
\end{remark}

The take home from these ideas is that it should be just necessary to bound the optimisation term at a rate $\mathcal{O}\pa{\frac{1}{\sqrt{m}}}$, as this bound is already imposed by the statistical term in the statistics-optimisation factorisation.

\section{First Order Optimisation}\label{sec:first_order_optimisation}
\subsection{The complexity model}
In the rest of the thesis, we will just look at the optimisation term introduced in the previous section. More generally, we will be interested in studying algorithms to approximate a solution of problems of the form
\[
    \min_{x \in \XX} f(x)
\]
where $\XX$ is a subset of $\RR^n$.\footnote{The contributions in this thesis will be in the more abstract setting where $\XX$ is a differentiable manifold not necessarily embedded in $\RR^n$. We will restrict the presentation in this introductory section to subsets of $\RR^n$ for the sake of simplicity. It will be in~\Cref{ch:fibred_manifolds_in_optimisation} where we will give an introduction to the subject of optimisation on manifolds together with a literature review of the field in~\Cref{sec:previous_work}, once we have given a presentation of the geometric tools and notation in~\Cref{ch:geometry}. Even though we will not mention the manifold setting in this section, we will write most definitions in a coordinate-independent fashion, so that they readily generalise to the manifold setting without much effort.} Note that the set $\XX$ is what we called $\Theta$ in the context of modelling and machine learning.

As it is customary when developing a complexity theory for a class of algorithms, to be able even talk about what an algorithm is without getting lost in computational semantic considerations, it is necessary to introduce a computational model. A \textbf{computational model} in the theory of optimisation abstracts the idea of \emph{what is a function}, as a function in computer science can be defined in many possible ways and a definition of a function may be as concrete as to be language-dependant.

In optimisation, the computational model abstracts a function $f$ in terms of its value and derivatives at its points.

\begin{definition}
    A \textbf{$k$-th order oracle} $\mathscr{O}$ for a class $C^k(\XX)$ function $\deffun{f: \XX -> \RR;}$ is a mapping that, given a point $x \in \XX$, returns the $k$-th order information of $f$:
    \[
        \mathscr{O}(x) = (f(x), \grad f(x), \Hess f(x), \dots, f^{k)}(x)).
    \]
\end{definition}

An algorithm in this model may query the oracle for the value of the function and its derivatives at one or more points. Then, the algorithm will process this information together with the information from previous calls to output a point $x_t \in \XX$. This new point will be used to generate more queries to the oracle and so on. We note that, in general, there is no reason for which an algorithm should perform just one call to the oracle in every step. In fact, some very important algorithms require of two calls or even a logarithmic number of calls to the oracle in every step. On the other hand, in this thesis we will mostly consider algorithms that perform one call to the oracle in every step.

In this thesis, we will centre our attention on \textbf{first order methods}. These are methods that just have access to the value of the function and its gradient. This is the most common class of methods used in the context of neural networks since if a model has $n$ parameters---$\dim \XX = n$---the Hessian would have $n^2$ entries, which is prohibitively large even for small models. It is worth mentioning that, when implementing second order optimisation, one almost never computes the full Hessian. Instead, algorithms are specified in terms of \emph{Hessian-vector products} of the form $H^{-1}(v)$ where $v$ is a vector and $H$ is the Hessian of $f$---second order methods---or an approximation of it from first order information---quasi-Newton methods. Even then, computing Hessian-vector products tends to be quite costly when compared with first order methods both in time and memory.

Due to memory limitations given the large models used in deep learning, we will just look at first-order algorithms that do not depend on previous calls to the oracle. An algorithm $\mathcal{A}$ of this form is just a function
\[
    x_{t+1} = \mathcal{A}(x_t, \grad f(x_t)).
\]
These algorithms are inherently \textbf{local}, as they just process the information at the current point to make their next guess.

\begin{remark}[Optimisation algorithms as discretisations of \odes]
    A more formal definition of an algorithm being local comes by interpreting such an algorithm as a discretisation of a differential equation. For example, the gradient step is the forward-Euler discretisation of the differential equation
\[
    \dot{x} = \grad f(x).
\]
This is an autonomous differential equation which gives rise to a flow called the \textbf{gradient flow}. Since this is an invariant equation, this view can be generalised to Riemannian manifolds provided a suitable discretisation scheme. This more general definition of local algorithms allows for a constant number of calls to the oracle to be used in each step, rather than just one. The perspective of looking at the algorithm $\mathcal{A}$ as an integrator has yielded a unified way of looking at large families of optimisation algorithms under a common lens~\parencites{wibisono2016variational}{wilson2018lyapunov}{franca2020dissipative}.
\end{remark}

\begin{remark}[A note on stochastic optimisation]
    As we mentioned in the previous section, in machine learning, the objective function often depends on an unknown random variable $\xi$ so that the objective function takes the form\footnote{The random variable $\xi$ represents in the offline setting the product random variable of the dataset $(X_1, Y_1), \dots, (X_m, Y_m)$.}
    \[
        f(x) = \E\cor{f(x,\xi)}.
    \]
    In the offline learning setting, the random variable $\xi$ has finite support and represents the dataset so that $f(x) = R_m(x)$.

    The field that studies functions with this structure is called \textbf{stochastic optimisation}, and it was introduced in the seminal paper by Robbins and Monroe~\parencite{robbins1951stochastic}. Stochastic optimisation has become particularly popular in the last decade due to its applications in machine learning.

    The simplest stochastic algorithm approximates the gradient of $\E\cor{f(x, \xi)}$ using $B$ \iid{} samples $\xi_i\sim\xi$ to form a stochastic approximation to the direction of steepest descent of $f$
    \begin{equation}\label{eq:stoch_gd}
        x_{t+1} = x_t - \frac{\eta}{B}\sum_{i=1}^B\grad f(x_t, \xi_i).
    \end{equation}
    This is called \textbf{stochastic gradient descent} (\sgd). The samples $\xi_i$ are called the \textbf{minibatch}---or sometimes simply the \textbf{batch}---and the number $B$ is called the \textbf{minibatch size}. Many other algorithms have been developed in recent years for the stochastic optimisation in the online and offline setting, yielding algorithms widely used in practice such as \textbf{\adagrad}~\parencite{duchi2011adaptive}, and its heuristic variations \textbf{\rmsprop}~\parencite{hinton2012deep} and \textbf{\adam}~\parencite{kingma2015adam}. In parallel, the field of stochastic optimisation has grown in its own directions, for example by using this framework to set online decision problems~\parencite{bubeck2012regret}.

    In the rest of the thesis, we will develop the theory in the non-stochastic setting. On the other hand, the results that we will present are quite general, so it should not be too difficult to generalise the ideas to the stochastic setting. That being said, we will see in the experiments section that the theory developed here works surprisingly well in practice in the stochastic setting. We leave the theoretical development of these ideas open for future research.
\end{remark}

Having introduced the abstract computational model, we will spend the rest of the section introducing and giving examples of the two main classes of functions considered in the optimisation literature.

\subsection{Convex optimisation}
Convex optimisation studies the setting in which the function $f$ is convex and the set of constraints $\XX$ is a convex subset of $\RR^n$. This is the most classical scenario as there are many interesting problems in practice that are either convex or can be relaxed into a convex problem, while still reflecting most of the initial structure of the problem.

The convexity hypothesis is a quite restrictive one, but at the same time, it allows applying local algorithms to solve global problems over the whole space $\XX$. This comes from the fact that for a convex function on a convex space, every local minimiser is a global minimiser. Furthermore, if the function is strictly convex, there is a unique global minimiser. Convexity (resp.\ strict convexity) is a local property when the function is of class $C^2(\XX)$---it is equivalent to $\Hess f(x) \succeq 0$ (resp.\ $\Hess f(x) \succ 0$) at every point $x \in \XX$---so it makes tackling the global problem of optimisation by local methods feasible.

It would be rather naïve to try to give a comprehensive introduction to convex optimisation in a couple of pages given the vast literature on the subject. For example, when the function is linear or quadratic, one enters the fields of linear programming, conic programming and semidefinite programming among others, which study algorithms such as interior point methods and simplex methods, each of which decomposes in many other subfamilies. To avoid falling down that rabbit hole, we will just mention methods that can be applied in practice to general non-linear functions, even if we are just able to prove convergence for a subfamily of convex functions.

Furthermore, to simplify the exposition, we will assume that $f$ is at least of class $C^2(\XX)$. Virtually all the theory in convex optimisation can be reformulated without this hypothesis via the use of subgradients. On the other hand, the $C^2(\XX)$ regularity assumption often clarifies the exposition, and proofs in the $C^2(\XX)$ setting can often be generalised to the convex non-differentiable setting via standard algebraic arguments.

To exemplify the kind of results that one gets in this field, let us consider the simplest of the algorithms: Gradient descent. \textbf{Gradient descent} (\gd)---sometimes called \textbf{steepest descent}---accounts for following the direction of steepest descent for a time $\eta > 0$. Starting at a point $x_0 \in \XX$, the update rule is given by
\[
    x_{t+1} = x_t - \eta\grad f(x_t).
\]

The quantity $\eta$ is usually called the \textbf{step-size}, although in machine learning is often called the \textbf{learning-rate}. Now, this begs the question: What is a reasonable step-size?

In the same way that one does when giving bounds on the discretisation term of an \ode, to answer this question one has to put some assumptions on the function. The simplest assumption is that of smooth convexity.
\begin{definition}\label{def:bounded_hessian}
    A function $f \in C^2(\XX)$ is said to be of \textbf{$\alpha$-bounded hessian} if
    \[
        -\alpha \I_n \preceq \Hess f(x) \preceq \alpha \I_n\mathrlap{\qquad \forall x \in \XX.}
    \]

    If the function $f$ is also convex (\ie, $0 \preceq \Hess f(x)$), $f$ is said to be \textbf{$\alpha$-\textbf{smooth}}.
\end{definition}
In other words, a $C^2$ function is of $\alpha$-bounded Hessian if the eigenvalues of its Hessian lie in the interval $[-\alpha, \alpha]$, and it is $\alpha$-smooth if they lie in $[0, \alpha]$.

Expanding the function as a Taylor series along the straight line that connects two points $x,y \in \XX$ and bounding Lagrange's remainder we get the following bound around $x\in \XX$ for a function of $\alpha$-bounded Hessian
\begin{equation}\label{eq:bounded_hessian_taylor}
    f(y) \leq f(x) + \scalar{\grad f(x), y-x} + \frac{\alpha}{2}\norm{x-y}^2.
\end{equation}
Denote $\gamma(t) = x + t(y-x)$ for $t \in [0,1]$ the line connecting $x$ and $y$.
This bound can be interpreted geometrically as having an upper bound on $f \vert_\gamma$ by a quadratic function tangent to $f$ at $x$ with leading term $\frac{\alpha}{2}$.
Note that this bound is everywhere sharp for the distance function to a fixed point $f(x) = \frac{\alpha}{2}d\pa{x, x_0}^2$ as this is a quadratic function with Hessian equal to $\alpha\cdot\Id$.
For this function, in the unconstrained case $\XX = \RR^n$, it is possible to get to the minimum $x_0$ from any point in one gradient descent step letting $\eta = \frac{1}{\alpha}$.
Algebraically, we can see all these ideas by setting $y = x_{t+1} = x_t - \frac{1}{\alpha}\grad f(x_t)$ in~\eqref{eq:bounded_hessian_taylor} to get
\begin{equation}\label{eq:progress}
    f(x_t) - f(x_{t+1}) \geq \frac{1}{2\alpha}\norm{\grad f(x_t)}^2.
\end{equation}
Geometrically, this means that we improve the value of the function at least by $\frac{1}{2\alpha}\norm{\grad f(x_t)}^2$ after every step. On the other hand, by convexity and Cauchy-Schwartz, one may bound below the norm of the gradient by
\[
    \frac{f(x_t) - f(x^\ast)}{\norm{x_t - x^\ast}} \leq \norm{\grad f(x_t)}.
\]
Plugging this into~\eqref{eq:progress}, putting the resulting equation in terms of $\delta_t = f(x_t) - f(x^\ast)$ and showing that $\norm{x_t - x^\ast}$ is decreasing in $t$, one gets the inequality
\[
    \delta_{t+1} \leq \delta_t - \frac{1}{2\alpha\norm{x_0 - x^\ast}^2}\delta^2_t.
\]
and bounding the general term of this recurrence, the following full convergence result follows.
\begin{theorem}
    Assume that $\XX = \RR^n$ and $f$ is $\alpha$-smooth and has a minimiser $x^\ast$. Then, gradient descent with step $\eta = \frac{1}{\alpha}$ starting at a point $x_0$ converges as
    \[
        f(x_t) - f(x^\ast) \leq \frac{2\alpha\norm{x_0 - x^\ast}^2}{t+4}.
    \]
\end{theorem}
This result says that the optimisation term in the setting of convex $\alpha$-smooth functions decays as $\mathcal{O}\pa{\frac{1}{t}}$. A much more impressive result is the one developed by Nesterov in $1983$. Nesterov proposed an algorithm such that, under the same $\alpha$-smoothness hypothesis, converged at a rate of $\mathcal{O}\pa{\frac{1}{t^2}}$. This was a major result, as a few years before it had been proven that this rate was optimal for the set of convex $\alpha$-smooth functions. After that result by Nesterov, many other algorithms that converge at this fast rate have been developed. In general, an algorithm that converges faster than vanilla gradient descent for a class of functions is called an \textbf{accelerated algorithm}.

This type of results gives us a \textbf{worst-case analysis} of the optimisation oracle, as one may reword this as ``gradient descent converges to a point with a value at most $\epsilon$ away of the minimum in at most $\mathcal{O}(\frac{1}{\epsilon})$ steps''.

Of course one can easily generalise this algorithm to one where $\XX \subset \RR^n$ is a proper convex subset, provided that we have an efficient way to project from $\RR^n$ onto $\XX$. This is called \textbf{projected gradient descent}. In this case, a similar convergence analysis gives the same rates as in the unconstrained case.

    There are many other classes that are studied in the setting of convex optimisation, such as \textbf{$\alpha$-strongly convex} functions ($\Hess f \succeq \alpha \I_n$ for $\alpha > 0$), Lipschitz functions and combinations of these. Another line of research is that of studying convergence under weaker regularity assumptions than differentiability. In this case, one of the most commonly known algorithm given its versatility is Nemirovsky and Yudin's \textbf{mirror descent}~\parencite{nemirovsky1983problem}, which may be applied in the more general setting of Banach spaces. In general, one can find a vast zoo of algorithms with convergence guarantees in convex optimisation. To name a few, by replacing the penalty function in mirror descent by a non-linear function one gets \textbf{proximal methods}. If one can do linear optimisation over the domain, one may use \textbf{Frank-Wolfe}. In \textbf{coordinate gradient descent} and \textbf{block-coordinate gradient descent} one optimises one or more of the coordinates of the function at a time, leaving the others fixed during that iteration. In \textbf{alternated direction method of multipliers (\admm)} for constrained optimisation one tries to approximate a solution to equation given by the Lagrange multipliers iteratively updating the multipliers\textellipsis{} For a theoretical introduction to all these methods and others see~\parencites{nesterov2004introductory}{boyd2004convex}.

\subsection{Non-convex optimisation}
In the setting of non-convex optimisation, one drops completely the assumption of $f$ being convex. On the one hand, this gives a much wider range of applicability to the results developed by this theory. On the other, the results that this theory is able to prove are much weaker. In particular, in this theory, as the global properties given by convexity are dropped, one is likely to encounter critical points and local minima of $f$, to the point that, in general, not even finding but checking that a given point is a local minimum for a non-convex function is NP-hard~\parencite{murty1987some}. For this reason, in this setting, one often has to content oneself with approximating critical points of $f$, that is, points at which $\grad f(x) = 0$.

Of course, as it was the case with convex functions, if one wants to get reasonable bounds, it is necessary to restrict the attention to a subset of functions. Again, the class of functions of $\alpha$-bounded Hessian is suitable for this task. For these functions, we get the following standard result. We prove it here for completeness, as it will be the result that we will generalise to manifold in~\Cref{ch:second_order_bounds} (\cf, \Cref{thm:dynamic_trivializations}).

\begin{theorem}\label{thm:non_convex_thm}
    Let $\deffun{f : \RR^n -> \RR;}$ be a function of $\alpha$-bounded Hessian bounded below by $f^\ast \in \RR$. Gradient descent on $f$, with a step-size of $\frac{1}{\alpha}$ starting at a point $x_0$, converges as
    \[
        \min_{0 \leq k \leq t}\norm{\grad f(x_k)} \leq \sqrt{\frac{2\alpha\pa{f(x_0) - f^\ast}}{\pa{t+1}}}.
    \]
\end{theorem}
\begin{proof}
    By~\eqref{eq:progress}, we have that for a function of $\alpha$-bounded Hessian
    \[
        \frac{1}{2\alpha}\norm{\grad f(x_k)}^2 \leq f(x_k) - f(x_{k+1}).
    \]
    Summing over all the indices and telescoping
    \[
        \frac{1}{2\alpha}\sum_{k=0}^t \norm{f(x_k)}^2 \leq f(x_0) - f(x_{t+1}) \leq f(x_0) - f^\ast
    \]
    which yields the desired bound.
\end{proof}

In other words, if the function is of $\alpha$-bounded Hessian, gradient descent with constant step-size finds a point with gradient $\epsilon$-close in norm to a critical point of $f$ in at most $\mathcal{O}(\frac{1}{\epsilon^2})$ steps, as opposed to the $\mathcal{O}(\frac{1}{\epsilon})$ needed to approximate the global minimiser when the function is convex.

\begin{remark}[Projected gradient descent on non-convex functions]
    It is worth noting that the previous theorem assumes the problem to be unconstrained, that is, $\XX = \RR^n$. One could hope to extend this proof to work on constrained optimisation problems, as in the convex case. A moment's reflection shows that this is not easily achieved. In fact, non-convex constrained optimisation is notoriously difficult, as the projection does not respect the derivatives in any reasonable way, which destroys the convergence of the algorithm. As such, these two elements do not interact well together.
\end{remark}

Even though the proof for gradient descent for non-convex objectives assures that gradient descent will converge to a critical point, this critical point might not be a local minimum, but perhaps a saddle point or a local maximum. On the other hand, via an application of the stable manifold theorem from dynamical systems, it is possible to prove that if the initial point $x_0$ is sampled from a smooth distribution on $\RR^n$, gradient descent will almost always converge in the limit to a local minimiser~\parencite{lee2016gradient}, although it may take exponentially many steps in $\epsilon$~\parencite{du2017gradient}. These results inspired many others that tried to estimate Hessian information to escape saddle-points by perturbing the gradient~\parencite{allenzhu2018natasha2}.

In the last $5$ years, there has been renewed interest in the field of non-convex optimisation due to the popularity of neural networks. Before that, there had been very little progress since the years $1965$--$1975$~\parencite{nesterov2004introductory}. The new lower-bounds developed for first order (and higher-order) methods on several classes of non-convex functions---assuming certain first and second order bounds---have been of special interest~\parencites{cartis2010complexity}{carmon2020lower}. In parallel to this work, several accelerated algorithms for non-convex optimisation~\parencite{carmon2017convex} and algorithms in the stochastic and non-stochastic setting for functions of the form $f(x) = \frac{1}{n}\sum f_i(x)$, like the empirical risk. For a recent review of these results see the papers~\parencites{allenzhu2018natasha2}{allenzhu2018neon2}.

Other directions of research in this area are those coming from convex relaxations, where one is able to find a convex problem that has the same minimiser as the non-convex one, such as those coming from \textbf{geometric programming}, many linear programming problems, and some eigenproblems such as the \svd{} decomposition. Many of these non-convex problems that happen to have a fast way to solve them while being non-convex often have a hidden Riemannian structure that makes them \textbf{geodesically convex} under some convenient metric. We will explore in depth this kind of problems in~\Cref{ch:fibred_manifolds_in_optimisation}. For now, we will close this review section referencing what is, as the time of this writing, the largest compilation to our knowledge of papers and reviews in the area of provable non-convex optimisation
\begin{center}
    \url{https://sunju.org/research/nonconvex/}
\end{center}
The web-page is active as of $2020$, and contains lists of relevant papers written in the last $5$--$10$ years in most of the sub-fields of non-convex optimisation.

\section{Neural Network Architectures}\label{sec:2_architectures}
In this section, we introduce neural networks in the framework of statistical models. We give some examples of different classes of neural networks, and we show how to implement them in the Python library PyTorch. A good introductory reference for all these ideas and many more is the book~\parencite{goodfellow2016deep}.

\subsection{Neural networks}
\begin{definition}\label{def:neural_network}
    Let $\Theta \subset \RR^n$ be open and connected. A \textbf{neural network architecture} is a mapping $\deffun{f : \Theta \times \XXin -> \YYout;}$ that is differentiable in the first variable and that factorises into a composition of simpler functions $\deffun{f^{(i)}: \Theta_i \times \RR^{d_{i-1}} -> \RR^{d_i};}$ for $i=1, \dots, p$ with $\XXin \iso \RR^{d_0}$, $\YYout \subset \RR^{d_p}$ and $\Theta = \Theta_1 \times \cdots \times \Theta_p$
\[
    f_{\theta} = f_{\theta_p}^{(p)} \circ \dots \circ f_{\theta_1}^{(1)}.
\]
We will often omit the parameters of the functions and write $f^{(i)} = f_{\theta_i}^{(i)}$.
Each parametrised function $f^{(i)}$ is called a \textbf{layer}. We say that $f^{(1)}$ is the \textbf{input layer}, $f^{(i)}$ for $i = 2, \dots, p-1$ are the \textbf{hidden layers}, and $f^{(p)}$ is the \textbf{output layer}.\footnote{Sometimes the input and output layers are also called first and last hidden layers.} A network is said to be \textbf{deep} whenever it has three or more layers. \textbf{Deep learning} is the area that studies deep neural networks.
\end{definition}

Each of the layers in a neural network is designed to be a (parametrised) transformation that can be computed very efficiently. In particular, most of the layers used in practice are designed so that they can be heavily parallelised, making their implementation very fast on a \gpu. We will touch on this in more detail later in this section.

The size of deep neural networks has seen an exponential explosion in size in the last ten years, seeing models that have hundreds of layers and up to a hundred billion parameters the current largest one~\parencite{brown2020language}.

\begin{remark}
    In most standard neural networks, $\Theta$ is just a Euclidean space. As we did in the previous section, we will assume that $\Theta$ is isomorphic to a Euclidean space in this section for simplicity, but the reader should bear in mind that the direction of the rest of the thesis will be that of considering $\Theta$ to be a connected manifold---often a matrix manifold.
\end{remark}

\subsection{Feedforward networks}
The most classical neural network architecture is that of the feedforward network. These are models where each layer is an affine transformation of its input together with a coordinate-wise non-linear transformation.

\begin{definition}\label{def:linear_layer}
Let $\XXin \iso \RR^k$, $\YYout \iso \RR^d$ and $\Theta \iso \M{d, k}\times\RR^d$. A \textbf{linear layer of width $d$} is the affine transformation defined by
    \[
        \deffun{f : \Theta \times \XXin -> \YYout;
        ((A, b), x) -> f_{A,b}(x) = Ax + b}
    \]
    We say that $A$ is the matrix of \textbf{weights} and $b$ is the \textbf{bias}.

    Given a function $\deffun{\sigma : \RR -> \RR;}$, we define a \textbf{non-linearity} as the associated map $\deffun{\sigma :\RR^d -> \RR^d;}$ that comes from applying the map $\sigma$ on vectors in $\RR^d$ component-wise. We will denote both maps with the same symbol.

    A \textbf{feedforward network} is a neural network where every layer is a linear layer followed by a non-linearity, $f^{(i)}_{\theta_i} = \sigma \circ f_{A_i, b_i}$, $\theta_i = (A_i, b_i)$. Usually, the non-linearity in the last layer is chosen to be the identity---\ie, there is no non-linearity in the last layer.
\end{definition}

\paragraph{Strengths, weaknesses, and limitations of feedforward networks}
A feedforward network is one of the simplest neural network architectures used in practice. It is sometimes also referred to as \textbf{multi-layer perceptron (\mlp)}. The work of Cybenko~\parencite{cybenko1989approximation} and Hornik~\parencite{hornik1991approximation} shows that feedforward networks are \textbf{universal approximators}, in the sense that they are dense in the set of continuous functions supported on a compact subset of $\RR^k$. On the other hand, an important remark is in order: While these functions may approximate any continuous function arbitrarily well, the width of a shallow network necessary to get an $\epsilon$-approximator grows exponentially fast in $\epsilon$~\parencite{eldan2016power}. It is this one of the reasons why feedforward networks are often replaced in practice by more specialised layers that try to encode domain-specific biases. Examples of these are \cnns~\parencite{fukushima1982neocognitron}, attention~\parencite{bahdanau2015neural}, neural machine translation~\parencite{kalchbrenner2013recurrent} and transformer models~\parencite{vaswani2017attention} among many others. More generally, the field of \textbf{geometric deep learning} studies how to encode invariances present in the data so that the model structurally preserves them, via the use of equivariant layers~\parencite{cohen2016group}, graph neural networks~\parencite{scarselli2009graph}, and other architectures.

\subsection{Recurrent neural networks (\rnns)}\label{sec:rnns}
The simplest non-trivial data topology is that of \textbf{sequential data}.
This comes up naturally when processing an input with variable length like audio samples, time series, or sentences in a language. In these cases, every element of the sequence is usually encoded into a vector of a fixed dimension, but the number of these vectors is unknown a priori. Using computer science nomenclature, these objects are encoded in the free monoid $\pa{\RR^k}^\ast \defi \oplus_{t=0}^\infty \pa{\RR^k}^t$. Recurrent Neural Networks were designed to process this kind of sequential data, constituting a map from $\pa{\RR^k}^\ast$ into a finite-dimensional embedding space $\RR^d$.

\begin{definition}\label{def:rnn}
    Given a sequence of inputs $\pa{x_t}_{t=1}^T \in \pa{\RR^k}^\ast$, we define a \textbf{recurrent neural network} (\rnn) with hidden size $d > 0$ as
\[
    h_t = \sigma\pa{B h_{t-1} + Cx_t} \mathrlap{\qquad t=1, \dots, T}
\]
where $B \in \M{d}$ and $C \in \M{d,k}$ are parameters, $\sigma$ is some fixed non-linearity, and $h_0 \in \RR^d$ is some fixed initial vector. $C$ is called the \textbf{input kernel} and $B$ is the \textbf{recurrent kernel}. In most applications, the starting vector $h_0$ is set to zero.
\end{definition}

In classification tasks, where one needs an encoding of the whole sequence into a vector, one often chooses either $h_T$, or $\frac{1}{T} \sum h_t$ as the encoding of the whole sequence in $\RR^d$. The problem with these encodings is that they are biased towards the last elements of the sequence, given that these were the elements that were encoded the last. For that reason, one often uses a \textbf{bi-directional \rnn}, which accounts for running two separate \rnns, one on $\pa{x_t}_{t=1}^T$ and another one on the reversed sequence $\pa{x_{T-t}}_{t=1}^T$ and computes the embedding as the sum of the hidden states of the two.

The intermediate vectors $h_t$ are usually regarded as the \textbf{memory} of the \rnn{} at time $t$. Each of the applications of the recurrent step is called a \textbf{time-step}, and the dimension $t$ is called the \textbf{time dimension}.

\rnns{} are a rather straightforward generalisation of feedforward neural networks to sequences. On the other hand, they present some problems when the matrix $B$ is not well-conditioned. Consider the degenerate case where $\sigma = \Id$. If $B$ has eigenvalues of norm larger than $1$, the gradient of $B$ will explode in some directions for long sequences, while if $B$ has eigenvalues of norm less than $1$, the gradients of $B$ with respect to the first time-steps will be close to zero, so the gradient will incorporate almost no information of the initial elements of the sequence. These problems are regarded as \textbf{exploding gradient} and \textbf{vanishing gradient} problems~\parencite{bengio1994learning}. We discuss them in more detail in~\Cref{sec:motivation}.

The \textbf{Long-Short Term Memory cell} (\lstm)~\parencite{hochreiter1997long} is a more complex recurrent architecture that tries to heuristically alleviate the vanishing gradient and exploding gradient problems via a system referred to as \textbf{gating}.
The idea of this heuristic is to \emph{control the flow of the gradients along the \lstm}, by allowing to retain some information in the internal state, while giving mechanisms to delete other information---often thought of as \emph{forgetting}. The gating mechanism is composed of $4$ vectors: \textbf{input} $i$, \textbf{forget} $f$, \textbf{output}, and \textbf{output modulation} $g$.
The \lstm{} also separates the hidden state into an \textbf{internal state} $c_t$---sometimes referred to as \textbf{cell}---and the hidden state $h_t$.
If we denote by $\sigm(t) = \frac{1}{1+e^{-t}}$ the sigmoid function and by $\tanh$ the hyperbolic tangent---keeping in mind that the range of these functions are $(0,1)$ and $(-1,1)$ respectively---the \lstm{} is defined by the equations
\begin{align*}
    \begin{pmatrix}
        i \\ f \\ o \\ g
    \end{pmatrix}
     =
    \begin{pmatrix}
        \sigm \\ \sigm \\ \sigm \\ \tanh
    \end{pmatrix}
    \pa[\Big]{W
        \begin{pmatrix}
            x_t \\ h_{t-1}
        \end{pmatrix}
    }\\
    c_t = f \hprod c_{t-1} + i \hprod g \qquad h_t = o\hprod \tanh(c_t)
\end{align*}
where $\hprod$ denotes the component-wise multiplication and $W \in \M{4d, (d+k)}$ is a matrix of parameters.
There are other popular variations of this recurrent layer, such as the \textbf{Gated Recurrent Unit} (\gru)~\parencite{chung2014empirical}, but we will spare the reader their description, as we will not use them in the practical work in this thesis.

\subsection{Why affine layers?}\label{sec:2_affinelayers}
The layers that we have introduced until now are all variations of the same idea: Applying an affine transformation to the inputs and then apply a component-wise non-linearity. This is not a coincidence. This design comes from engineering considerations arising from the implementation of neural networks. Even though neural networks were first devised in the $80$'s~\parencite{fukushima1982neocognitron}, they were partly forgotten until the first decade of the XXI century~\parencites{ciresan2010deep}{krizhevsky2012imagenet}. Current interest has built upon two large engineering efforts that made neural networks accessible to users outside of academia. These two engineering advances are Graphics Process Units (\gpus) and auto-differentiation engines.

\subsubsection{Graphics Processing Units (\gpus)}
\gpus{} have been around for many years now, mostly tied to the design and video game industry. It has been just in the last $10$ years that we have seen the field of \gpu{} computing soar as a major tool in data analysis, allowing processing large quantities of datapoints in parallel.

A coarse but fairly accurate description of \gpus{} would come from thinking about them as pieces of hardware that perform operations massively in parallel much more quickly than a \cpu. In general, they can apply the same operation in parallel to many datapoints. While a \cpu{} can execute many different operations on individual pieces of data with a very low overhead between operations, a \gpu{} shines at executing the same operation in parallel on many datapoints. We can think of a \gpu{} as having a \cpu{} with a very large number of cores, where the cores are tied to execute the same operation.

If we look at a matrix-vector multiplication, for a matrix of size $d \times k$, this can be decomposed into $d$ inner-products on $\RR^k$, each of which can be computed in parallel. Furthermore, each inner-product accounts for adding $k$ different elements, which can be performed in $\log(d)$ time.

At the same time, when performing stochastic gradient descent as described in~\eqref{eq:stoch_gd}, in every step of the algorithm we need to compute the value of the function at $B$ datapoints, where $B$ is the minibatch size. This can also be performed in parallel, taking again advantage of the fast parallel processing given by \gpus. This has allowed neural networks to be able to scale the approximation of the minimisation problem to large datasets, having some of them in the order of $10^8$ datapoints~\parencites{deng2009imagenet}{kuznetsova2020open}{mahoney2006large}.

\paragraph{A word on efficiency and \gpus}
\gpus{} provide an efficient way to execute the same operation on many datapoints in parallel. Each of these highly parallel operations is called a \textbf{kernel}. Launching a kernel has a much larger overhead than executing an operation on a \cpu. For this reason, \gpus{} shine when reducing the amount of kernels launched during a program. We can think of executing a kernel as having a unitary fixed cost---that of the overhead---when executed on very few points. For this reason, computing a sequence of multiplications of sparse matrices can often be more efficiently implemented on a \cpu{} than on a \gpu, due to the lower overhead per operation of the \cpu. \gpus{} also suffer from a penalty coming from memory latency, as each minibatch has to be moved from \cpu{} to \gpu, which can result on a computational bottleneck in some applications. The take home from all these ideas is that one has to be careful when computing the efficiency of an algorithm implemented on a \gpu, taking into account the size of the data one is working with and how parallelisable are the operations that are being executed. This latter point will be of particular importance when devising a fast implementation of the methods proposed in this thesis in~\Cref{sec:approximations_exponential}.

\subsubsection{Auto-differentiation: Forward and backward pass}\label{sec:autodiff}
In the offline learning setting, as we detailed in~\Cref{sec:offline_statistical_learning}, we have access to a dataset $\set{(X_i, Y_i)}$ of $m$ pairs of features and labels. As we mentioned before, in deep learning, in many datasets we can have $m$ to be in the order of the millions. Computing the gradient of the empirical risk
\[
R_m(\theta) = \frac{1}{m}\sum_{i=1}^m \loss\pa{f\pa{X_i, \theta}, Y_i}
\]
would be infeasible, as every step would require to evaluate the function $f$ on millions of points. As an alternative, stochastic gradient descent is used. As we already mentioned, this algorithm chooses in every step a subset of $B$ elements $\set{(X_{a_i}, Y_{a_i})}$ from the dataset to form a minibatch and averages their losses to form
\begin{equation}\label{eq:empirical_batch_risk}
    \widehat{R}_B(\theta) \defi \frac{1}{B}\sum_{i=1}^B \loss\pa{f\pa{X_{a_i}, \theta}, Y_{a_i}}
\end{equation}
performing the update
\begin{equation}\label{eq:sgd_rule}
    \theta \gets \theta - \eta\grad \widehat{R}_B(\theta).
\end{equation}

In order to implement this algorithm in practice, it is necessary to be able to compute the gradients of the function $\theta \mapsto \loss\pa{f\pa{X_i, \theta}, Y_i}$. It was in the seminal work~\parencite{ciresan2010deep} where it was first documented that this naïve looking algorithm was able to perform astonishingly well at optimising highly non-convex neural networks. Furthermore, in this same work, they showed that this optimisation process could be accelerated several orders of magnitude by implementing it on a \gpu. In this first work, the authors implement themselves the necessary \gpu{} kernels. This is a very difficult task, as programming on a \gpu{} is notoriously technical.

In parallel to this work, there had been a line of research that involved the computation of the derivatives of functions automatically~\parencite{wengert1964simple}. These are called \textbf{auto-differentiation engines} or simply \textbf{autodiff engines}.

It was just then that Theano was born~\parencite{alrfou2016theano}. Theano was the first library marrying auto-differentiation and \gpus{} first released in 2007 as a library for the programming language Lua. An autodiff engine helps the user define differentiable functions: The user writes the function using methods provided by the library and the library will automatically compute its derivative at a given point. In some sense, it implements the chain rule---more formally, the adjoint method---for a wide range of functions and constructions in a programming language. In practice, this is a rather challenging problem.

    \subsection{An implementation of a feedforward network}
    Let us explain the rudiments of how all the ideas in this section come together via an example written on the PyTorch library, as this will be the library that we will use all throughout~\Cref{ch:geotorch}.\footnote{A disclaimer to the seasoned practitioner: This is intended to be a simple working example to showcase the ideas in this chapter. It does not try to be a performant one by any possible measure.}

    \begin{example}[\mnist{} classification with a feedforward network]
        Consider a two layer feedforward neural network with $\code{relu}(x) := \max(x, 0)$ non-linearities designed to classify greyscale images of handwritten digits. A dataset that provides this exact set-up is that of \mnist~\parencite{lecun1998mnist}. The \mnist{} dataset consists of $60.000$ labelled images of size $28\times 28$ represented by a floating-point number at every pixel describing the intensity of the grey (see~\Cref{fig:mnist_digits}). As such, $\XXin = \RR^{28 \times 28}$, $\YYout= \set{0,1, \dots, 9}$.

    \begin{figure*}[t]
        \centering
        \includegraphics[width=\columnwidth]{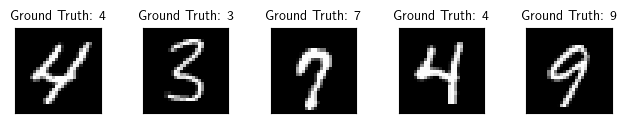}
        \caption{Example of digits present in the \mnist{} dataset.}%
        \label{fig:mnist_digits}%
    \end{figure*}

    The model will output a distribution on the set of $10$ elements by processing the $10$-dimensional output of the last layer with a $\softmax$ non-linearity, which maps a $10$ dimensional vector into the $10$-dimensional simplex by mapping its $k$-th coordinate to
    \[
        \softmax\pa{x}_k = \frac{e^{x_k}}{\sum_{i=0}^9 e^{x_i}}.
    \]
    In practice we will output $\log \circ \softmax$ for numerical stability reasons.

    \paragraph{Loss}
    The loss function will be the negative log-likelihood at a given label
    \[
        \deffun{\loss : \RR^{10} \times \set{0,1,\dots, 9} -> \RR;
                \pa{x, k} -> -\log\pa{\softmax\pa{x}_k}}
    \]

    \paragraph{Hyperparameters}
    We perform stochastic gradient descent with $\eta = 10^{-3}$ and a minibatch size of $128$. We will iterate over the whole dataset $80$ times. The hidden size of the two layer network will be $256$.

    \paragraph{Data processing}
    The dataset is normalised to have mean $0$ and standard deviation of $1$. This is a common preprocessing step. The data will also be first flattened by the model, that is, it will be converted from data in $\RR^{28 \times 28}$ into data in $\RR^{784}$.\footnote{By applying this transformation, we are effectively \emph{forgetting} the $2$D structure of the data. This is far from being an optimal way to process an image, but it is the simplest.}

    The procedure to fit this model to the dataset can be written in PyTorch 1.6 as shown in~\Cref{lst:feedforward}.

    \begin{figure}[t]
    \begin{minipage}{\linewidth}
\begin{lstlisting}[language=python,
                         escapechar=|,
                         caption={PyTorch code implementing a two layer feedforward network to classify \mnist},
                         label={lst:feedforward}
                         ]
import torch
import torch.nn as nn
import torch.nn.functional as F
from torchvision import datasets, transforms

# Hyperparameters
epochs = 80
batch_size = 128
learning_rate = 0.001

# Data processing |\label{lin:data_proc}|
# Centre the MNIST dataset, which has mean 0.1307 and standard deviation 0.3081
transform = transforms.Compose([
    transforms.ToTensor(),
    transforms.Normalize((0.1307,), (0.3081,))
    ])
dataset = datasets.MNIST('./data', train=True, download=True, transform=transform)
loader = torch.utils.data.DataLoader(dataset, batch_size=batch_size, shuffle=True) |\label{lin:data_proc_end}|

# Run the model on GPU if possible
device = torch.device("cuda" if torch.cuda.is_available() else "cpu") |\label{lin:device}|

# Instantiate model and optimiser (Stochastic Gradient Descent)
model = nn.Sequential(             |\label{lin:seq}|
            nn.Flatten(),
            nn.Linear(28*28, 256), |\label{lin:lin1}|
            nn.ReLU(),
            nn.Linear(256, 10),    |\label{lin:lin2}|
            nn.LogSoftmax(dim=1)
        ).to(device)               |\label{lin:move_model}|
optim = torch.optim.SGD(model.parameters(), lr=learning_rate) |\label{lin:sgd}|

# Training loop
for epoch in range(epochs):
    for idx, (batch_x, batch_y) in enumerate(loader): |\label{lin:loader}|
        batch_x, batch_y = batch_x.to(device), batch_y.to(device) |\label{lin:batch_device}|

        out = model(batch_x)            |\label{lin:forward}|
        loss = F.nll_loss(out, batch_y) |\label{lin:loss}|

        optim.zero_grad() |\label{lin:init}|
        loss.backward()   |\label{lin:backward}|
        optim.step()      |\label{lin:step}|
        print("Train Epoch: {} [{}/{}]\tLoss: {:.6f}"
                .format(epoch, idx*batch_size, len(dataset), loss))
\end{lstlisting}
    \end{minipage}
    \end{figure}

    Let us explain what is happening in this example.

    \paragraph{Data processing}
    In lines~\ref{lin:data_proc} to~\ref{lin:data_proc_end}, we download and normalise the dataset. In particular, we download it automatically, and we use the PyTorch convenience functions to apply a transform pipe to the points of this dataset. In this case, we are transforming every image so that the whole dataset has zero mean and unitary standard deviation. The values $0.1307$ and $0.3081$ are the empirical mean and variance of the dataset and were computed a priori.

    In this part we also wrap the dataset into a loader. The loader is an abstraction that samples batches of size \code{batch\_size} for us from the dataset. We can iterate over it like a list, as we do in~\cref{lin:loader}. We also set the \code{DataLoader} to shuffle the data every time we start iterating over it.

    In~\cref{lin:device}, we check whether the computer has a \gpu{} available. If so, we will be computing every operation in a \gpu. If there is no \gpu{} available, we fall back to the \cpu.

    \paragraph{Model specification}
    In~\cref{lin:seq}, we define a \code{Sequential} model. This model is one that splits exactly as a sequential composition of functions, as defined in~\Cref{def:neural_network}. We start by flattening the input, that is, applying the linear isomorphism $\RR^{28 \times 28} \iso \RR^{784}$. After that, we apply a linear transformation (\cf, \Cref{def:linear_layer}) together with the $\relu(x) = \max\pa{x, 0}$ non-linearity, where the $\max$ is taken coordinate-wise. Finally, in~\cref{lin:move_model} we move the model to \gpu{} if there was one available. If not, this line does not do anything. Being able to switch between \cpu{} and \gpu{} with a single line of code is one of the strengths of these modern machine learning libraries.

    In~\cref{lin:sgd}, we define the optimiser to be \sgd{} with a learning rate of $\eta = 10^{-3}$.

    \paragraph{Forward pass}
    The gist of the code starts in the training loop. The training loop iterates over the whole dataset $80$ times---also called \textbf{epochs}---in batches of size $128$. As we had mentioned, we can iterate over the \code{loader}, which returns in every step a batch of size $128$ of images and labels. In particular, we have that \code{batch\_x} has dimensions $(128, 28, 28)$ while \code{batch\_y} is a tensor of one dimension of size $128$ of integers. As we did with the model, if there was an available \gpu, we move the data from \cpu{} to \gpu.

    In~\cref{lin:forward}, we compute the image under our model of every point of the dataset. Using our previous notation, if we have a batch $\set{(X_i, Y_i)}$, this is equivalent to computing $\set{f_\theta(X_i)}$. As such, the output \code{out} is of size $(128, 10)$, where every vector \code{out[i]} is the component-wise logarithm of a point on the $10$-dimensional simplex.

    In~\cref{lin:loss}, we compute the average negative log-likelihood loss over the minibatch using the logarithmic probabilities and the true labels. As such, the variable \code{loss} is just a floating-point number encapsulated in a PyTorch class, that is, a tensor of dimension $1$ holding the value of the empirical risk of the minibatch $\widehat{R}_B(\theta)$ as defined in~\eqref{eq:empirical_batch_risk}.

    \Cref{lin:forward,lin:loss} do much more heavy work than it looks at first sight. When these lines are called, each operation in model, that is, every matrix multiplication, every addition of a bias vector, every coordinate-wise non-linearity, \etc, is recorded into a graph called the \textbf{computation graph} as it is being executed. This graph contains the dependencies between all the variables in the program as a dependency tree, where the parameters of the model are in the leaves and the inner nodes of this tree are operations putting them together. The computation of $f_\theta(X_i)$ and the creation of the computation graph is what is often called the \textbf{forward pass}.

    \paragraph{Backward pass}
    The computation of the gradients and the update of the parameters happens in the last three lines of the loop. As we have seen, in~\cref{lin:sgd} we have passed the parameters of our model to the optimiser. \Cref{lin:init} initialises the optimiser, zeroing-out the gradients of all the parameters that it controls, this is just a technical point. \Cref{lin:backward} computes the gradients traversing the computation graph backwards using the chain rule---in this context often called the \textbf{adjoint method}. Finally, \cref{lin:step} updates the variables with the update step from the optimisation algorithm (\cf, \Cref{eq:sgd_rule}).
\end{example}

\clearpage
\chapter{Differential and Riemannian Geometry}\label{ch:geometry}
    This section is intended as an introduction to some results and computations in differential geometry and Riemannian geometry that will be useful later on.
    We will skip the proofs of some constructions at the beginning for the sake of the exposition, as they are quite technical, but we will give references to what, in our opinion, are some particularly clean proofs of them.
    We will always assume all the objects to be in the smooth category, unless stated otherwise.

\paragraph{Outline of the chapter.}
The problem motivating this whole chapter is that of computing geodesics on a large class of manifolds---naturally reductive homogeneous spaces---as this will be of interest in the rest of the thesis. The chapter is organised going from more general objects to more particular ones, stopping at each layer looking into what it adds to the global picture. We finish the chapter in~\Cref{sec:matrix_groups} showing how to instantiate the constructions in the chapter in a number of manifolds commonly used in optimisation.

Even though all the ideas in this chapter are well-known to geometers, we are not aware of any source that contains all of them. It is for this reason that we include at the beginning of each section a short comment on both historical references and standard references. We also include the proofs necessary to derive the equations for the geodesics in the studied spaces.

\section{Fibred Manifolds}\label{sec:theory_submersions}
In this section we provide an introduction to the theory of submersions and Riemannian submersions. This can be seen as a first approach to the theory of Ehresmann connections on fibre bundles. The theory of connections on a fibred manifold as distributions on the total space was first developed in~\parencite{ehresmann1952connexions}, while most of the formulae relating the geometry of the base space to the geometry of the total space on a Riemannian submersion is due to O'Neill~\parencite{oneill1966fundamental}. A more detailed review of this subject can be found in~\parencite[Chapter 9]{besse2008einstein}.

    We start by recalling the definition of a submersion.
\begin{definition}
    Let $\deffun{\pi : \total{M} -> M;}$ be a map between manifolds. We say that $\pi$ is a \textbf{submersion} if $\deffun{\pa{\dif \pi}_p : T_p \total{\MM} -> T_{\pi(p)}\MM;}$ is surjective for every $p \in \total{\MM}$. In particular, if $\pi$ is surjective, this is equivalent to saying that $\deffun{\dif \pi : T\total{M} -> TM;}$ is surjective. We say that $\total{M}$ is the \textbf{total manifold}, and $M$ is the \textbf{base manifold}. We will assume submersions to be surjective unless stated otherwise.

    The \textbf{fibre over $x \in M$} is defined as $\total{M}_x \defi \pi^{-1}(x)$. By the implicit function theorem, the fibres of a submersion are manifolds of dimension $\dim\pa{\total{M}} - \dim\pa{M}$. A pair of manifolds together with a submersion between them is called a \textbf{fibred manifold}.
\end{definition}

\begin{notation}
    Even though the notation may suggest otherwise, the manifold $\total{M}$ is completely unrelated to the manifold $M$. We choose this notation for fibred manifold as it is possible to identify vector fields $X$ on $M$ with certain class of vector fields $\total{X}$ on $\total{M}$. This notation makes this identification easier to follow.
\end{notation}

    If all the fibres are diffeomorphic, then we get the concept of a space that locally looks like a product of the base space and its fibre. This idea is formalised through fibre bundles.
\begin{definition}
    Let $\deffun{\pi : E -> M;}$ be a fibred manifold. Then $\deffun{\pi : E -> M;}$ is a \textbf{fibre bundle over $M$ with fibre $F$} if for every $x \in M$ there is an open neighbourhood $U$ and a map $\deffun{\alpha : \pi^{-1}(U) -> F;}$ such that
    \[
        \deffun{(\pi, \alpha) : \pi^{-1}(U) -> U \times F;}
    \]
    is a diffeomorphism. We say that $(\pi, \alpha)$ is a \textbf{bundle chart}. For every $x \in M$ we have that  $\deffun{\alpha\mid_x : E_x -> F;}$ is a diffeomorphism.
\end{definition}
    All the fibred manifolds we will work with will be fibre bundles.

    A key property of fibred manifolds is that they admit local smooth sections. These are mainly used to show that objects defined on fibre bundles are smooth.
\begin{proposition}[Existence of local sections]\label{prop:existence_local_sections}
    Let $\deffun{\pi : \total{M} -> M;}$ a fibred manifold. Then, for every $p \in \total{M}$, there exists a neighbourhood $U \subset \total{M}$, and a smooth function $\deffun{\sigma : \pi(U) -> U;}$ such that $\pi \circ \sigma = \Id_{\pi(U)}$. This is called \textbf{a local section through $p$}.
\end{proposition}
\begin{proof}
    See for example~\parencite[Ch.~II Prop.~2.2]{lang1999fundamentals}
\end{proof}

    \begin{remark}[Fibred spaces]
        It is in fact possible to show via a simple application of the inverse function theorem that submersions are characterised by the property of having smooth sections. This allows us to generalise the concept of fibred manifolds to the category of topological manifolds simply by swapping smooth sections with continuous sections. This is of particular interest when studying stratified spaces, such as low-rank matrices
        \[
            \Rank{n,k,\leq r} = \set{X \in \M{n,k} | \rank(X) \leq r}.
        \]
        In this case, we have that this is not a manifold, but a collection of manifolds---the collection of the manifolds of fixed-rank matrices (\Cref{sec:fixed_rank})---called strata, with some homogeneity properties. These spaces appear naturally when considering decompositions such as the \qr{} decomposition, which is smooth on each of the strata of this topological manifold. Sadly, we will not have time to explore this connection between topology, numerical analysis, and optimisation on manifolds in this thesis.
    \end{remark}

    \begin{remark}[Global sections]
    Global sections need not exist in general. In particular, we may not find in general an immersed copy of the base space inside the total space. To see this, consider the following circle bundle over the round sphere:
    \[
        U\SS^2 \defi \set{(x,v) \in T\SS^2 | \norm{v} = 1}.
    \]
    It is not difficult to show that this is a fibre bundle over the sphere with fibre $\SS^1$. On the other hand, by the hairy ball theorem, it has no global sections, as a global section would give an everywhere non-zero vector field over the $2$-sphere. This is a particular example of a general behaviour of non-trivial principal bundles. Note that this is an example of a bundle with no global section, but we can still find an immersed $\SS^2$ inside its total space.
    \end{remark}

    The study of the existence or lack of existence of global sections in fibred spaces and fibre bundles is of paramount importance in the field of algebraic topology. An introductory exposition to fibre bundles in this context can be found in~\parencite{steenrod1951topology}.

    For a submersion, we define the \textbf{vertical bundle} as $\VV \total{M} \defi \ker\dif \pi \subset T\total{\MM}$. This is a vector subbundle of $T\total{\MM}$.

    If we also have a metric $\total{\gm}$ on $\total{M}$, we can define the \textbf{horizontal bundle} $\HH$ as the fibre-wise orthogonal complement with respect to $\total{\gm}$ of $\VV\total{M}$, that is,
    \[
        \HH_p \defi \pa{\VV_p \total{M}}^\ort =\pa{\ker (\dif \pi)_p}^\ort = \set{X \in T_p \total{M} | \total{\gm}(X, Y) = 0,\, \forall Y \in \VV_p \total{M}}.
    \]
    These two vector bundles form a split $T\total{M} = \VV \total{M} \oplus \HH$. For an element $\zeta \in T\total{M}$ we denote its split into its horizontal and vertical components by $\zeta = \zeta_{\VV} + \zeta_{\HH}$.

    Note that since $\pa{\dif \pi}_p$ is a surjective linear map, $\deffun{\pa{\dif \pi}_p : \HH_p -> T_{\pi(p)}M;}$ is a linear isomorphism of vector spaces.

    In particular, $\HH$ is a smooth choice of a complement of $\VV\total{M}$. Note that the only property that we have used of the metric $\total{\gm}$ is that it is non-degenerate, so that the complement $\pa{\VV\total{M}}^\perp$ together with $\VV\total{\MM}$ span all of $T_p \total{M}$. We can abstract these properties to define a general connection on a fibred manifold.
    \begin{definition}\label{def:connection}
        We say that a vector subbundle $\HH$ of $T\total{\MM}$ is an \textbf{Ehresmann connection} on a fibred manifold $\deffun{\pi : \total{M} -> M;}$ if it is a complement of the vertical bundle, that is, $T\total{M} = \VV\total{M} \oplus \HH$.
    \end{definition}

    The split given by an Ehresmann connection makes $\dif\pi\vert_{\HH}$ into a vector bundle morphism along $\pi$, in the sense that the following diagram commutes
    \begin{equation}\label{eq:bundle_isomorphism}
    \begin{tikzpicture}
        \node (gh1) {$\HH$};
        \node (gh2) [node distance=14ex, right of=gh1] {$TM$};
        \node (tm1) [node distance=10ex, below of=gh1] {$\total{M}$};
        \node (tm2) [node distance=14ex, right of=tm1] {$M$};
        \draw[->] (gh1) to node {$\dif \pi\vert_{\HH}$} (gh2);
        \draw[->] (gh1) to node [swap] {$\pi_{T\total{M}}\vert_{\HH}$} (tm1);
        \draw[->] (gh2) to node {$\pi_{TM}$} (tm2);
        \draw[->] (tm1) to node  {$\pi$} (tm2);
    \end{tikzpicture}
    \end{equation}
    and $\dif\pi\vert_{\HH}$ is a fibre-wise linear isomorphism.

    \begin{definition}
        Let $\deffun{\pi : \total{\MM} -> \MM;}$ be a fibred manifold with an Ehresmann connection $\HH$. A vector field $X$ on $M$ can be uniquely identified with a vector field $\total{X}$ on $\total{M}$ with values on $\HH$. We call this vector field the \textbf{horizontal lift} of $X$. In symbols,
    \[
        \total{X}_p \defi \pa{\dif \pi\vert_{\HH_p}}^{-1}\pa{X_{\pi(p)}}\mathrlap{\qquad \forall p \in \total{M}.}
    \]
    This lift gives a \textbf{projectable} vector field, meaning that $\dif\pi(\total{X}) = X$ is well-defined. This construction can be readily generalised to any tensor bundle of $TM$, being able to lift not only vector fields, but also any kind of tensor. In technical words, a connection induces connections on all the associated bundles.
    \end{definition}

    \begin{definition}
        Let $(\total{M}, \total{\gm}), (M, \gm)$ be two Riemannian manifolds and let $\deffun{\pi : \total{M} -> M;}$ be a submersion between them with induced Ehresmann connection $\HH$. We say that $\pi$ is a \textbf{Riemannian submersion} if for every $p \in \total{M}$
        \[
            \deffun{\dif \pi \vert_{\HH_p} : \HH_p -> T_{\pi(p)}M;}
        \]
        is a linear isometry. In other words,
        \[
            \total{\gm}_p(\zeta_1, \zeta_2) = \gm_{\pi\pa{p}}\pa{\pa{\dif \pi}_p(\zeta_1), \pa{\dif \pi}_p(\zeta_2)} \mathrlap{\qquad \forall \zeta_1, \zeta_2 \in \HH_p.}
        \]
        We may equivalently write this condition in terms of horizontal vectors as
        \[
            \total{\gm}_p(\total{u}, \total{v}) = \gm_{\pi\pa{p}}\pa{u, v} \mathrlap{\qquad \forall u, v \in T_{\pi(p)}M.}
        \]
    \end{definition}

    \begin{remark}[A Riemannian submersion does not restrict to a local isometry]
        By the definition above, together with the existence of local sections of $\pi$, one might be tempted to think that for a point $x \in M$ there exists a section $\deffun{s : \Sigma -> s(\Sigma);}$ through $p \in \pi^{-1}(x)$ such that $\pi\vert_{s(\Sigma)}$ is an isometry. This is not true as, by the local version of Frobenius theorem~\parencite[Ch.~VI Thm.~1.1]{lang1999fundamentals}, $\HH$ is only integrable---there exists a submanifold $N \subset \total{\MM}$ with $T_p N = \HH_p$ for $p \in N$---whenever it is involutive around $p$, that is, it is closed under the Lie bracket. This means that we may only find a section giving an immersed manifold $s(\Sigma)$ with tangent bundle $\dif s(\Sigma) \subset \HH$ whenever $\HH$ is involutive. But $\HH$ is often not involutive. In fact, the curvature of $\HH$ is defined as how much $\HH$ deviates from being involutive. Hence, we may only locally integrate $\HH$ when the connection is locally flat. In this case, $\total{M}$ is locally a Riemannian product of manifolds.
    \end{remark}

    If we have a fibred manifold $\deffun{\pi : \total{M} -> M;}$ with an Ehresmann connection $\HH$ and a metric $\gm$ on $M$, we may lift $\gm$ to a metric tensor on $\HH$ by setting $\total{\gm}_\HH \defi \pa{\dif\pi\vert_{\HH}}^\ast\pa{\gm}$ where $\pa{-}^\ast$ denotes the pullback
    \[
        \total{\gm}_{\HH_p}(u,v) \defi \gm_{\pi(p)}\pa{\dif \pi\pa{u}, \dif \pi\pa{v}} \mathrlap{\qquad \forall u,v \in \HH_p.}
    \]
    Using the decomposition $T\total{M} = \VV \total{M} \oplus \HH$, we can choose a metric tensor $\total{\gm}_\VV$ on $\VV \total{M}$ and define a metric on $\total{M}$ declaring the horizontal and vertical bundles to be orthogonal, that is, $\total{\gm} \defi \total{\gm}_\VV \oplus \total{\gm}_\HH$. This metric makes $\deffun{\pi : \total{M} -> M;}$ into a Riemannian submersion. We can summarise this discussion as follows:

    \begin{proposition}[Lifting a metric on a fibred manifold]\label{prop:submersion_into_riemannian_submersion}
        Let $\deffun{\pi : \total{M} -> M;}$ be a fibred manifold. Given a metric $\gm$ on $M$, an Ehresmann connection $\HH$ and a metric tensor $\total{\gm}_{\VV}$ on $\VV \total{\MM}$ we can form a metric that turns $\pi$ into a Riemannian submersion by setting $\total{\gm} = \total{\gm}_{\VV} \oplus \pa{\dif\pi\vert_{\HH}}^\ast(\gm)$.

        Conversely, if we have a metric $\total{\gm}$ on $\total{M}$ that is projectable---$\total{\gm}_{p_1}(\total{u},\total{v}) = \total{\gm}_{p_2}\pa{\total{u}, \total{v}}$ for every $p_1, p_2 \in \total{M}_x,\,u,v \in T_x M$---we may push it forward to form a metric on $M$ such that $\pi$ is a Riemannian submersion.
    \end{proposition}

    We will show in the next sections how to construct metrics on $\VV \total{\MM}$ in the particular case of principal bundles and homogeneous spaces, where the fibre has the structure of a Lie group.

    Riemannian submersions generalise the idea of an isometry between spaces of the same dimension. In particular, they provide a strong link between the geometry of $\total{M}$ and $M$.

    \begin{proposition}\label{prop:properties_riemannian_submersions}
        Let $\deffun{\pi : \total{M} -> M;}$ be a Riemannian submersion. We have that:
        \begin{enumerate}
            \item\label{it:sub_1} A geodesic in $\total{\MM}$ with horizontal initial conditions is horizontal everywhere. $\pi$ takes horizontal geodesics on $\total{M}$ to geodesics on $M$. In particular, if $(\total{M}, \total{\gm})$ is complete, so is $(M, \gm)$.
            \item\label{it:sub_2} Any geodesic on $M$ may be locally lifted to a geodesic on $\total{M}$ with horizontal initial conditions.
            \item\label{it:sub_3} $\pi$ is distance decreasing: $\total{d}\pa{p,q} \geq d\pa{\pi\pa{p}, \pi\pa{q}}$, $\forall p,q \in \total{M}$.
            \item\label{it:sub_4} \parencite{hermann1960sufficient} If $(\total{M}, \total{\gm})$ is complete, any curve in $M$ may be lifted globally to a curve in $M$. In particular, any geodesic in $(M, \gm)$ may be lifted globally to a horizontal geodesic of $(\total{M}, \total{\gm})$.
            \item\label{it:sub_5} \parencite{oneill1966fundamental} For $X, Y$ vector fields on $M$ we have
                \begin{align*}
                    \conn_{\total{X}}\total{Y} &= \total{\conn_X Y} + \lfrac{1}{2}\cor{\total{X}, \total{Y}}_{\VV} \\
                    [\total{X}, \total{Y}]_{\HH} &= \total{[X,Y]}.
                \end{align*}
        \end{enumerate}
    \end{proposition}
    \begin{proof}
        \Cref{it:sub_1} can be found in~\parencite[Ch.~7 Corol.~45]{oneill1983semi}. \Cref{it:sub_2} is a corollary of the formula for the connection in~\Cref{it:sub_5}. \Cref{it:sub_3} is a direct corollary of~\Cref{it:sub_1}. \Cref{it:sub_4} is proved in~\parencite[Prop.~$3.2$]{hermann1960sufficient}. \Cref{it:sub_5} is proved in~\parencite[Lemma $3$ and Lemma $1.2$]{oneill1966fundamental}
    \end{proof}

    According to this proposition, if we have a Riemannian submersion, and we know how to compute geodesics with horizontal initial conditions on the total space, we can compute geodesic on the base space. Furthermore, since $\dif\pi$ is surjective, any geodesic on the base space is of this form.

    Before finishing this section, we look at Riemannian submersions with totally geodesic fibres, as these will be rather important in the sequel. We start by defining a totally geodesic submanifold.

    \begin{definition}
        An immersed submanifold $(F, \gm_F)$ of a Riemannian manifold $(M, \gm_M)$ is \textbf{totally geodesic} if geodesics on $(F, \gm_F)$ are also geodesics on $(M, \gm_M)$.
    \end{definition}

    \begin{example}
        On a round sphere maximum circles are geodesic submanifolds but any other circle is not.
    \end{example}

    The deviation of an immersed submanifold from being a geodesic submanifold is given by the \textbf{shape tensor} sometimes also called the \textbf{second fundamental form}.
    \begin{definition}
        Given an immersed submanifold $(F, \gm_F)$ of a Riemannian manifold $(M, \gm_M)$ with Levi-Civita connections $\conn^F, \conn^M$ respectively, we define the \textbf{shape tensor} of $F$ for two vector fields $X,Y$ on $F$ as the $(2,1)$ tensor
        \[
            \shape(X,Y) = \conn^F_X Y - \conn^M_X Y.
        \]
    \end{definition}

    A simple computation involving Gauss' equation gives the following characterisation of totally geodesic submanifolds.
    \begin{proposition}\label{prop:totally_geodesic_fibres}
        A Riemannian submanifold $(F, \gm_F)$ of $(M, \gm)$ is totally geodesic if and only if $\shape \equiv 0$.
    \end{proposition}
    \begin{proof}
        See~\parencite[Ch.~4 Prop.~13]{oneill1983semi}.
    \end{proof}

    Totally geodesic submanifolds will be important when we talk about Riemannian submersions with totally geodesic fibres. Most of the Riemannian submersions that we will encounter in the optimisation context will have this property.

	In the next sections we will look at two problems: How to construct Riemannian submersions and how to compute geodesics on some total spaces of these submersions. Once we know how to do this, we may repeat this process forming a chain of submersions
    \[
        M_1 \xrightarrow{\pi_1} M_2 \xrightarrow{\pi_2} M_3 \xrightarrow{\pi_3} \dots,
    \]
    where, if we know how to compute geodesics on $M_1$, and every map is a Riemannian submersion, we know how to compute geodesics on all the other manifolds in the chain. This will be particularly useful for optimisation with orthogonal constraints as shown in~\Cref{sec:matrix_groups}, where we will see that the special orthogonal group, the Stiefel manifold, and the Grassmannian form such a tower. We will also leverage this in~\Cref{ch:geotorch} to implement optimisation on different manifolds by implementing it on one manifold, and then pushing the geodesics forward along Riemannian submersions.

    \section{Invariant Metrics and Principal Bundles}\label{sec:principal_bundles}
    In this section we will study manifolds together with a group action. In particular, we will be interested in studying conditions for the existence of an invariant metric under this action. This will take us naturally to the definition of principal bundles which, in turn, will provide us with a way to construct Riemannian submersions. Most of the theory presented in this section and the next one can be generalised to study connections on principal and associated bundles. The presentation in this section differs from the classical one in that we emphasise the Riemannian side of principal bundles as spaces on which we can find a metric invariant under a Lie group action that descends into the quotient of the manifold by the action.

Principal bundles were first proposed in full generality by Palais, who introduced the study of proper actions on manifolds, as a generalisation of actions by compact groups~\parencite{palais1961existence}. The results of this section can mostly be found across the books~\parencites{kobayashi1963foundations}{kobayashi1969foundations}, written by two of the fathers of the subject.

    We start by defining some notation and recalling some rudiments of Lie groups.
\begin{notation}
    Let $\deffun{\mu : G \times M -> M;}$ be an action of a Lie group $G$ on a manifold $M$. We will write $\mu_g(x) = \mu^x(g) = \mu(g,x)$ for the different partial maps defined by it. In particular, $\mu_g$ is a diffeomorphism with inverse $\mu_{g^{-1}}$.

    We will denote \textbf{left and right actions} of a Lie group onto itself by $L_g$ and $R_g$, and the \textbf{conjugacy action} of a Lie group by $c_g \defi L_g \circ R_{g^{-1}}$.

    For a Lie group $G, H$, we will denote their Lie algebras as $\glie, \hlie$. Given a vector $X \in \glie$, we denote the \textbf{left-invariant vector field} associated to $X$ by
    \[
        \lv{X}_g \defi \pa{\dif L_g}_e(X).
    \]
    Left-invariant vector fields provide a linear isomorphism between $T_g G$ and $\glie$ and they allow us to define left-invariant tensors on the whole manifold by defining them just on the Lie algebra. For example, if we have an inner product $\gm_e$ on the Lie algebra, we can define a left-invariant metric as
    \[
        \lv{\gm}_g\pa{\lv{X}_g, \lv{Y}_g} \defi \gm_e\pa{X, Y} \mathrlap{\qquad X,Y \in \glie.}
    \]

    We recall that the \textbf{Lie exponential} is defined as the integral of left-invariant vector fields, that is, if $v \in \glie$, $\expm(tX) = \gamma(t)$ we have that
    \[
        \dot{\gamma}(t) \defi \lv{X}_{\gamma(t)}.
    \]
    We will denote it by $\expm$, to differentiate it from the Riemannian exponential map.

    We denote the \textbf{orbit of a point $x \in M$} as
    \[
    G(x) \defi \set{\mu_g(x) | g \in G},
    \]
    and the \textbf{stabiliser of $x$}, also referred to as the \textbf{isotropy group at $x$}, as
    \[
    G_x \defi \set{g \in G | \mu_g(x) = x}.
    \]
	Note that for continuous actions $G_x = \pa{\mu^x}^{-1}(\set{x})$ is closed.
\end{notation}

\begin{remark}[Left and right actions]
    In the sequel we will be using left and right actions. The use of one or the other is merely due to convenience, since we may turn a left action $\mu_g$ into a right action by considering the map $g \mapsto \mu_{g^{-1}}$. We will reserve right actions for actions we use to quotient by, and left actions to talk about actions that act on this quotient.
\end{remark}

    We are interested in studying metrics that are invariant under the action of a group, as these metrics will descend into the quotient of the manifold by the group action.
    \begin{definition}
        Given a manifold $M$ together with a right action $\mu$ by a Lie group $G$, we say that a metric $\gm$ on $M$ is \textbf{$G$-invariant} if
        \[
            \gm_{\mu_g\pa{p}}\pa{\dif \mu_g(u), \dif \mu_g(v)} = \gm_p\pa{u, v} \mathrlap{\qquad \forall g \in G, p \in M, u, v \in T_p M,}
        \]
        in other words, the action $\mu_g$ is an isometry for every $g \in G$. Furthermore, if the group of isometries $\mu_G = \set{\mu_g | g \in G}$ is a closed subgroup of $\Iso(M, \gm)$, we say that \textbf{$G$ acts by isometries}.
    \end{definition}

    \begin{remark}[Effective actions]\label{rmk:effective}
    Note that $\mu_G$ is isomorphic to $G$ if $\mu_g = \Id$ implies that $g = e$. We call these \textbf{effective actions}. If an action is not effective, we can define the \textbf{ineffective kernel} of the action $N = \bigcap_{x \in M} G_{x}$. This is a normal subgroup of $G$, and in this setting $\mu_G$ is isomorphic to $G/N$ as Lie groups. In particular, the action of $G/N$ on the manifold $M$ is equivalent to that of $G$, so in most situations we can safely assume that the action is effective, although at times it is convenient to work with non-effective actions, when doing explicit computations in concrete manifolds.
    \end{remark}

    \begin{remark}[Closed subgroups of isometries]
        By a theorem of~\cite{myers1939group}, the isometry group $\Iso(M, \gm)$ is a Lie group. Hence, the requirement that $\mu_G$ is a closed subgroup of $\Iso(M, \gm)$ implies that $\mu_G$ is itself a Lie group. We will see in shortly that having a closed subgroup of isometries will allow us to \emph{quotient} by it---\ie, putting a unique smooth structure on the space of orbits such that it makes the projection a submersion---which may not be possible when the group is not closed.
    \end{remark}

    The first kind of actions that we will consider are proper actions. An action is said to be \textbf{proper} if the map from $G \times M$ to $M \times M$, $(g, x) \mapsto (\mu_g(x), x)$ is proper, in the sense that the preimage of a compact set is compact. Proper actions give us a way to construct $G$-invariant metrics. Even though this result goes back to Palais, it was recently shown that this metric can be chosen to be complete.
    \begin{theorem}[\cite{kankaanrinta2005proper}]\label{thm:proper_gives_isometric}
        Let $M$ be a manifold with a proper $G$-action $\mu$. Then, there exists a complete metric $\gm$ which makes $\mu_G$ into a closed subgroup of $\Iso(M, \gm)$.
    \end{theorem}

    The second kind of actions that we will look at are free actions. A \textbf{free action} is one that has no non-trivial fixed points. In other words, if $\mu_g(x) = x$ for an $x \in M$, then $g = e$. Another way of looking at this is that every isotropy group is trivial. Free isometric actions are actions by which we can quotient and obtain a smooth manifold.

    \begin{theorem}[Free slice theorem]\label{thm:free_slice}
        Let $G$ be a closed subgroup of $\Iso(M, \gm)$ that acts freely. Then, the set of orbits $M / G = \set{G(x) | x \in M}$ can be given a unique manifold structure and a Riemannian metric such that $\deffun{\pi : M -> M/G;}$ is a Riemannian submersion.
    \end{theorem}
    \begin{proof}
        See~\parencite[Theorem 5.6.21]{petersen2016riemannian}.
    \end{proof}

    Combining these two theorems we have the following corollary.
    \begin{corollary}\label{corol:invariant_metric_principal_bundle}
        Let $M$ be a manifold together with a free proper action. There exists a $G$-invariant metric on $M$ that descends into a metric on $M/G$ which makes $\deffun{\pi : M -> M/G;}$ into a Riemannian submersion.
    \end{corollary}

    Proper actions are a reasonably general setting on which to study invariant metrics.
    \begin{theorem}[Isometric actions are proper]\label{thm:isometric_is_proper}
        Let $(M, \gm)$ be a Riemannian manifold and $G$ be a closed subgroup of $\Iso(M, \gm)$. Then, its action on $M$ is proper.
    \end{theorem}
    \begin{proof}
        See~\parencite[Theorem 3.62]{alexandrino2015lie}.
    \end{proof}

    \begin{remark}[Isometric actions vs.\ effective and proper]
    \Cref{thm:isometric_is_proper} gives a converse to~\Cref{thm:proper_gives_isometric}. Between these two theorems, we have an equivalence between effective and proper $G$-actions on a manifold $M$ and closed subgroups $G$ of $\Iso\pa{M, \gm}$ where $\gm$ is a complete metric.
	\end{remark}

    \begin{remark}[Non-free isometric actions]
        In some situations, it is not strictly necessary for the action to be free but it certainly makes things easier. If the action is not free, one still has an open and dense subset of the manifold whose quotient by the action is a manifold. This is part of the content of the \textbf{principal orbit theorem}, and the start of the theory of Riemannian foliations~\parencite[Thm.\ $3.82$]{alexandrino2015lie}.
    \end{remark}

	Let us now define principal bundles, which are fibre bundles whose fibre is a Lie group, together with a compatible fibre-preserving action.
    \begin{definition}\label{def:principal_bundle}
        Let $\deffun{\pi : P -> M;}$ be a fibre bundle with fibre a Lie group $G$. We say that $P$ is a \textbf{principal bundle} if we have a free and transitive (on the fibres) fibre-preserving right action $\deffun{\mu : P \times G -> P;}$, and an atlas such that the bundle charts $\deffun{\alpha : \pi^{-1}(U) -> G;}$ are equivariant with respect to the action,
        \[
            \alpha(\mu(p,g)) = \alpha(p)g \mathrlap{\qquad \forall p \in \pi^{-1}(U), g \in G.}
        \]
        This makes the charts compatible with the action $\mu$ on $P$ and the action of multiplication on the right on $G$.
	\end{definition}

	We now show that principal bundles are the right abstraction to look at metrics that are $G$-invariant by which we can quotient.

    \begin{theorem}[Construction of principal bundles]
        Let $M$ be a manifold with a free and proper $G$-action. Then, $\deffun{\pi : M -> M/G;}$ has the structure of a principal bundle.
    \end{theorem}
    \begin{proof}
        See~\parencite[Theorem 3.34]{alexandrino2015lie}. A sketch of the proof goes as follows:

		Since the action is free and proper, by~\Cref{thm:free_slice} the map $\deffun{\pi : M -> M/G;}$ is a submersion, so every fibre $M_x \defi \pi^{-1}(\set{x})$ is a manifold. These fibres are, by definition, the orbits of the action and, since the action is free, they are all isomorphic to $G$, so a principal bundle is a fibre bundle with fibre $G$ and a fibre-preserving free and proper action. Adapting the proof of~\Cref{thm:free_slice} one proves that the bundle charts are equivariant.
    \end{proof}

	And we have the following converse.
    \begin{proposition}[Characterisation of principal bundles]\label{prop:construction_principal_bundles}
		The action $\mu$ of a principal bundle is free and proper.
    \end{proposition}
    \begin{proof}
        See~\parencite[Proposition 3.33]{alexandrino2015lie}. These properties follow from the fact that an action of a Lie group on itself is free and proper, and using the equivariance of the bundle charts.
    \end{proof}

	In particular, every principal bundle admits an invariant metric that descends into the base space and makes $\pi$ into a Riemannian submersion.

    Until now, we have argued that principal bundles are spaces that accept $G$-invariant metrics, and that if we have a $G$-invariant metric on the total space, it descends into a metric on the base space. We will now see some sort of converse to this result. In particular, we will see how to lift a metric from the base space into the total space such that the metric on $P$ is $G$-invariant. In general, we described in~\Cref{sec:theory_submersions} how, given a metric on a base space and a connection on the total space, we may lift it to a metric on the total space of a fibred manifold by declaring $\dif \pi$ to be an isometry and choosing a metric tensor on the vertical bundle. In the case of principal bundles, the fact that the fibres look like Lie groups allows us to simplify the process of choosing a metric tensor on $\VV P$ to choosing one on the Lie algebra.

    We start by defining a connection on a principal bundle.
    \begin{definition}\label{def:principal_connection}
        A \textbf{connection on a principal $G$-bundle} $\deffun{\pi : P -> M;}$ (or simply a \textbf{principal connection}) $\HH \subset TP$ is a $G$-equivariant Ehresmann connection, that is, a $G$-equivariant sub-bundle complementary to $\VV P$. In symbols,
        \[
            \pa{\dif \mu_g}_p(\HH_p) = \HH_{\mu(p,g)} \mathrlap{\qquad \forall p \in P, g \in G.}
        \]
    \end{definition}

    It is direct to see that a $G$-invariant metric on $P$ induces a connection on $P$ defining $\HH_p \defi \pa{\VV_p P}^\ort$, and it is exactly the $G$-invariance of the metric what gives the equivariance. In particular, since principal bundles accept $G$-invariant metrics, every principal bundle admits a principal connection. We may use these to lift a metric from $M$ to $P$ making $\deffun{\pi : P -> M;}$ into a Riemannian submersion.

\begin{theorem}[Lifting a metric on a principal bundle~{\parencite{vilms1970totally}}]\label{thm:lifting_principal_bundle}
    Let $\deffun{\pi : P -> M;}$ be a principal $G$-bundle. Given a metric $\gm$ on $M$, an inner product $\gm_e$ on $\glie$ and a principal connection $\HH$, there exists a unique metric $\total{\gm}$ on $P$ such that $\pi$ is a Riemannian submersion with totally geodesic fibres isometric to $(G, \lv{\gm_e})$---the left-invariant metric on $G$ associated to $\gm_e$---and horizontal distribution $\HH$.
    \end{theorem}
    \begin{proof}
        Given the principal connection $\HH$, we can put a metric on the horizontal space by pulling back the metric on $M$ as $\gm_{\HH} = \pa{\dif\pi\vert_{\HH}}^\ast(\gm)$. We can then declare $\HH$ and $\VV P$ to be orthogonal, and put a metric on $\VV P$ by transferring the right-invariant metric on $G$ to $P$ via the bundle charts. If $\deffun{\alpha : \pi^{-1}(U) -> G;}$ is a bundle chart, the metric is defined locally on $\pi^{-1}(U)$ as
        \[
            \total{\gm}\vert_{\pi^{-1}(U)} = \alpha^\ast(\lv{\gm_e}) \oplus \pa{\dif\pi\vert_{\HH \cap T\pi^{-1}(U)}}^\ast(\gm \vert_U).
        \]
        This is well-defined as, since the charts are right-equivariant, we have that for any two bundle charts $\alpha_U, \alpha_V$ with $U \cap V \neq \emptyset$,
        \[
            \alpha_V \circ \alpha_U^{-1} = L_g
        \]
        for some $g \in G$ (see \eg,~\cite[Lemma~$4.2.1$]{naber2011topology}).

        The fact that the fibres in this construction are totally geodesic submanifolds of $P$ is proved in~\parencite[Theorem~$3.5$]{vilms1970totally}.

        The uniqueness of this construction comes from the fact that the fibres are totally geodesic, so they have to be orthogonal to the connection subbundle.
    \end{proof}

    \begin{remark}[Associated bundles]
        The construction above goes through in the more general setting of associated bundles, as proved in the original paper. On the other hand, the simpler case of principal bundles is enough for what we will need in the sequel.
    \end{remark}

    Let us now define a family of vector fields which will be useful later.
    \begin{definition}
        For a principal $G$-bundle $\deffun{\pi : P -> M;}$ with action $\mu$ we define the \textbf{infinitesimal generators of the action} as
        \[
            \deffun{\xi : P \times \glie -> \VV P; (p, v) -> \xi_v(p) \defi \pa{\dif \mu^p}_e(v)}
        \]
    \end{definition}

    \begin{example}
        In the simplest case of the principal bundle $\deffun{\pi : G -> \set{e};}$, $\xi_v(g) = \pa{\dif L_g}_e(v)$. In other words, $\xi_v$ are just left-invariant vector fields on $G$, that is, $\xi_v = \lv{v}$.
    \end{example}

    The infinitesimal generators of the action take values on $\VV P$, since for a $p \in P$, $\pa{\dif \pi}_p \circ \pa{\dif \mu^p}_e = \dif\pa{\pi \circ \mu^p}_e = 0$, where we have used that the action $\mu^p$ is fibre-preserving, and thus $\pi \circ \mu^p$ is constant. The infinitesimal generators provide with bundle-isomorphism between $\VV P$ and $P \times \glie$, and hence, they show that the vertical bundle is always trivial.

    \begin{proposition}\label{prop:vertical_space}
        Let $\deffun{\pi : P -> M;}$ be a principal $G$-bundle, then $\deffun{\pa{\dif\mu^p}_e : \glie -> \VV_p P;}$ is a linear isomorphism, so that the vertical space factorises as $VP \iso P \times \glie$. We also have that the infinitesimal generators induce a Lie algebra homomorphism
        \[
            \xi_{[u,v]} = [\xi_u, \xi_v]\mathrlap{\qquad\forall u,v \in \glie.}
        \]
    \end{proposition}
    \begin{proof}
        For a $p \in P$, $\VV_p P$ can be identified with the tangent to the orbit of $G$ at $p$. This orbit is diffeomorphic to $G$, as the action is free, so $\VV_p P$ and $\glie$ have the same dimension. If we prove that $\pa{\dif \mu^p}_e$ is injective, then it is a linear isomorphism. Let $v \in \glie$ such that $\pa{\dif\mu^p}_e(v) = 0$ for $p \neq e$. This is equivalent to saying that
        \[
            \derivat{t}{0}\mu(p, \expm(tv)) = 0.
        \]
        Thus, $p$ is locally a fixed point of the map $\mu^{\expm(tv)}$, but since the action is free this only happens when $v=0$ so $\pa{\dif \mu^p}_e$ is injective.

        To prove that $\xi$ is an Lie algebra homomorphism, we have by equivariance that
        \[
            \mu_g \circ \mu^p (\expm(tv)) = \mu^{\mu_g(p)}(g^{-1}\expm(tv)g) \quad \forall v \in \glie, g \in G, p \in P
        \]
        and differentiating at $t = 0$ we get the following infinitesimal equivariance property:
        \begin{equation}\label{eq:infinitesimal_equivariance}
            \pa{\dif \mu_g}_p(\xi_v(p)) = \xi_{\Ad_{g^{-1}}(v)}(\mu_g(p))\mathrlap{\qquad \forall g \in G, p \in P, v \in \glie.}
        \end{equation}

        Finally, letting $g = \expm(-tu)$, substituting $p$ with $\mu\pa{\expm\pa{tu}, p}$ in~\eqref{eq:infinitesimal_equivariance}, and differentiating
        \[
            \derivat{t}{0}\pa{\mu_{\expm(-tu)}}_\ast\pa{\xi_v\pa{\mu\pa{\expm(tu), p}}} = \derivat{t}{0}\xi_{\Ad_{\expm\pa{tu}}(v)}(p)\mathrlap{\qquad \forall p \in P, u,v \in \glie.}
        \]
        The left-hand side is the definition of $[\xi_u, \xi_v](p)$, while the right-hand side is equal to $\xi_{[u,v]}(p)$, after using that the derivative of a linear map is the map itself.
    \end{proof}

    \begin{remark}[Lifting a metric on a principal $G$-bundle for $G$ compact]
        In the case when $G$ is compact, the construction of the metric on $\VV P$ in~\Cref{thm:lifting_principal_bundle} may be greatly simplified by the use of infinitesimal generators. As it was shown in~\Cref{eq:infinitesimal_equivariance}, $\pa{\dif \mu_g}_p$ may be evaluated using infinitesimal generators. In particular, it acts as the adjoint action on $\glie$. If we have an $\Ad(G)$-invariant inner product on $\glie$, we may then pull it back to a $G$-invariant metric tensor on $\VV P$. This metric is much simpler to work with and compute, as it does not make use of the bundle charts.

        As we will see in~\Cref{prop:left_G_right_H_invariant}, on compact Lie groups one may construct $\Ad(G)$-invariant metrics, so we may carry this construction whenever $G$ is compact. Luckily, we will just study the setting of $G$-compact in the sequel, as all Riemannian homogeneous spaces are of this form.
    \end{remark}

    \section{Homogeneous Spaces}
    We now come back to the computation of geodesics on different manifolds. The discussion in the previous section tells us that, by~\Cref{corol:invariant_metric_principal_bundle}, if one has a $G$-invariant metric on a principal bundle, then this metric descends to a metric on the quotient $P/G \iso M$, turning the projection into a Riemannian submersion. Since a Riemannian submersion maps horizontal geodesics into geodesics, if we know how to compute geodesics on $P$, then we can compute geodesics on $M$ just by projecting the horizontal ones.

    In this section, we will address the questions of how to implement these ideas in concrete families of manifolds. In particular, we will look at families of manifolds on which we can compute the geodesics easily, so that we can get the geodesics on the quotient simply by projecting onto it.

    The theory of (left-) invariant metrics on a space with a transitive action was pioneered by~\cite{nomizu1954invariant}, which was the first to realise that, if one has a transitive action, it can be used to translate geometric problems on a manifold to algebraic problems on a tangent space. An introduction to reductive spaces in terms of representations can be found in~\parencite[Ch.~X Sec.~$1$ and $2$]{kobayashi1969foundations}. An early introduction to the subject of reductive homogeneous spaces avoiding representations and delving deeper into the theory of the canonical connection of a reductive space can be found in~\parencite[Ch.~$1$]{kowalski1980generalized}. The theory of reductive structures was first systematically studied by Lichnerowicz in~\parencite{lichnerowicz1958geometrie}. A concise and self-contained introduction to a number of the ideas in this section and the next one can also be found in the excellent text~\parencite[Chapter $3$]{cheeger2008comparison}.

    We start by looking at the simplest example. Let $H$ be a closed subgroup of a Lie group $G$, consider the right action of $H$ on $G$ by right multiplications. It is not difficult to see that this action is proper and free, so we have our first specific family of principal bundles.
    \begin{proposition}\label{prop:homogeneous_space_is_principal_bundle}
        Let $H$ be a closed subgroup of a Lie group $G$ acting on $G$ via right multiplications. Then, $\deffun{\pi : G -> G/H;}$ is a principal $H$-bundle.
    \end{proposition}

    Manifolds that are a quotient of a Lie group by a closed subgroup are particularly important and are called \textbf{homogeneous spaces}.

    \begin{definition}
        A manifold $M$ is a $G$-\textbf{homogeneous space} if it is diffeomorphic to a manifold of the form $G/H$ with $H$ a closed subgroup of a Lie group $G$.
    \end{definition}

    The name of homogeneous spaces comes from the fact that these are exactly the manifolds that admit transitive actions, as we will show in a second.

    \begin{definition}
        An $G$-action is \textbf{transitive} on $M$ if for every $x,y \in M$, there exists a $g \in G$ such that $\mu_g(x) = y$.
    \end{definition}

    When $G$ is a Lie group acting on a smooth manifold, a transitive action $\mu_g$ allows us to move any point to any other point. In particular, since $\mu_g$ is a diffeomorphism (its inverse is $\mu_{g^{-1}}$), these manifolds look the same from a smooth point of view at every point.

    \begin{proposition}[Construction and characterisation of homogeneous spaces]\label{prop:construction_and_characterisation_homogeneous_spaces}
        Let $H$ be a closed subgroup of $G$. The $G$-action by left multiplications induced on the quotient $G/H$ is transitive, so $G/H$ is a $G$-homogeneous space.

        Conversely, let $M$ be a $G$-homogeneous space with action $\mu$, and fix a point $x_0 \in M$. Then, $M$ is diffeomorphic to the quotient of $G$ by the stabiliser of $x_0$, $H \defi G_{x_0}$ by right multiplications, through the diffeomorphism
        \[
            \deffun{\quot{\mu}^{x_0} : G / H -> M ; \pi(g) -> \mu^{x_0}(g)}.
        \]
    \end{proposition}

    \begin{remark}[Maps that descend to the quotient]
        A map $\deffun{\Phi : P -> P;}$ on a principal bundle map that commutes with the action $\mu_g$ descends into a map on the quotient $\quot{\Phi} \defi \pi \circ \Phi$. We will use this in the setting of the principal bundle of a homogeneous space $\deffun{\pi : G -> G/H;}$,  with $R_h$ for $h \in H$ as the bundle right-action and $\Phi = L_g$. The left-action then descends into a left-action $\quot{L}_g$ on $G/H$ since left and right multiplication on a Lie group always commute.
    \end{remark}

    For a homogeneous space $M$, there might be different groups that act transitively on it. Therefore, its representation as a quotient of Lie groups may not be unique. To see this, we shall first introduce a large family of examples for homogeneous spaces.

    \begin{definition}
        A Riemannian manifold $(\MM, \gm)$ is a \textbf{Riemannian homogeneous space} if its isometry group $\Iso(\MM, \gm)$ acts transitively on $\MM$.
    \end{definition}

    \begin{example}[Non-unique quotient representation of a homogeneous space]
        Let $\MM = \RR^n$ with the canonical metric, the full isometry group is generated by rotations, reflections, and translations, but the subgroup of translations acts transitively on $\RR^n$ and is a closed subgroup of $\Iso(\RR^n, \gm_{\RR^n})$. In other words, both the full isometry group and the subgroup of translations act transitively on $\RR^n$. Hence, if we denote the translations on $\RR^n$ by $\TE{n}$, we have two representations of $\RR^n$ as $\RR^n \iso \TE{n} \iso \Iso\pa{\RR^n, \gm_{\RR^n}} / \Ort{n}$, where we have used that the group of translations has a trivial stabiliser.
    \end{example}

    Riemannian homogeneous spaces have a particularly simple geometry. In particular, by definition, for any two points there exists an isometry that transforms one into the other. Hence, quantities that are invariant under isometries, such as the curvature tensor or the Levi-Civita connection, may be computed at all points just by computing them at one point and translating them using the isometries of the manifold.
    In other words, from the perspective of Riemannian geometry, these manifolds look the same at every point, so it is enough to study them at one of their points---or in a neighbourhood of one point---to understand their geometry at any given point.

    The stabilisers of Riemannian manifolds are always compact. This means that a Riemannian homogeneous space can always be represented as $G/H$ with $H$ compact.
    \begin{proposition}[Necessary condition for Riemannian homogeneous spaces]\label{prop:riemannian_homogeneous_stabiliser_is_compact}
        Let $(M, \gm)$ be a Riemannian manifold and let $G$ be a closed subgroup of $\Iso(M, \gm)$, then its stabiliser $H$ at a point $x_0 \in M$ is compact.
    \end{proposition}
    \begin{proof}
        By~\Cref{thm:isometric_is_proper} an action by isometries is proper, and since $H = \pa{\mu^{x_0}}^{-1}(\set{x_0})$, $H$ is compact.
    \end{proof}

    The converse is also true. In order to prove it, we first need the following important construction on Lie groups called \textbf{the averaging trick}.

    \begin{proposition}[Bi-invariant metric by a compact subgroup]\label{prop:left_G_right_H_invariant}
        Let $G$ be a Lie group and $H$ a compact subgroup. Then $G$ admits a left-$G$-invariant metric that is also right-$H$-invariant
    \end{proposition}
    \begin{proof}
        Let $\omega_\hlie$ be a non-zero alternating multilinear map on $\hlie$ in $k = \dim \hlie$ variables (\ie, $\omega_\hlie \in \bigwedge^k \hlie^\ast$). Since $\dim \bigwedge^k \hlie^\ast = 1$, this form is unique up to a multiplicative constant. We can then extend it to all of $H$ by pushing it forward along $R_h$, defining $\pa{\omega_H}_h \defi \pa{R_h}_\ast\pa{\omega_\hlie}$. This makes $\omega_H$ into a right-invariant top-dimensional form on $H$, and its associated measure is called the \textbf{right Haar measure on $H$}.

    Let $\scalar{\cdot, \cdot}$ be any inner product on $\glie$. We can turn this inner product into an $\Ad(H)$-invariant inner product on $\glie$ by averaging it over the elements of the group using the right Haar measure $\omega_H$
    \[
        \gm_\glie\pa{u, v} \defi \int_H \scalar{\Ad_h(u), \Ad_h(v)}\,\omega_H \mathrlap{\qquad u, v \in \glie.}
    \]
    Note that this integral is well-defined since $H$ is compact.
    The $\Ad(H)$-invariance follows from the right-invariance of $\omega$
    \[
        \gm_\glie\pa{\Ad_{h'}\pa{u}, \Ad_{h'}\pa{v}} = \int_H \scalar{\Ad_{hh'}(u), \Ad_{hh'}(v)}\,\omega_H = \gm_\glie\pa{u, v}\mathrlap{\qquad \forall u, v \in \glie, \forall h'\in H.}
    \]

    Finally, we can extend this inner product to a metric on all of $G$ by pushing it forward along $L_g$, $\gm_g \defi \pa{L_g}_\ast\pa{\gm_\glie}$. This makes $\gm$ into a left-$G$-invariant metric, but since it is $\Ad(H)$-invariant and it is left-$H$-invariant, it is right-$H$-invariant as well.
    \end{proof}

    We call a metric that is both left-invariant and right-invariant \textbf{bi-invariant}. When $G$ is itself compact, we may choose $H = G$ in the theorem above, getting the following result.
    \begin{corollary}[Bi-invariant metrics on compact Lie groups]\label{corol:compact_group_bi-invariant}
        A compact Lie group $G$ admits a bi-invariant metric.
    \end{corollary}

    \begin{remark}[Correspondence between invariant metrics and $\Ad(H)$-invariant inner products]
        It is easy to see by going in the opposite direction in the construction of~\Cref{prop:left_G_right_H_invariant} that left-$G$-invariant right-$H$-invariant metrics on $G$ are in one-to-one correspondence with $\Ad(H)$-invariant inner products on $\glie$.
    \end{remark}

    With this construction in hand, we may prove the converse of~\Cref{prop:riemannian_homogeneous_stabiliser_is_compact}.
    \begin{proposition}[Sufficient condition for Riemannian homogeneous spaces]\label{prop:sufficient_riemannian_homogeneous}
        Let $M$ be a manifold with a transitive action. If $M$ has a compact stabiliser $H = G_{x_0}$ then, there is a metric $\gm$ that makes $(M, \gm)$ into a Riemannian homogeneous space.
    \end{proposition}
    \begin{proof}
        Following the construction in~\Cref{prop:left_G_right_H_invariant}, we may construct a metric on $G$ such that it is left-$G$-invariant and right-$H$-invariant. Since it is right-$H$-invariant, it descends into a metric $\gm$ on $M \iso G/H$, and since it is left-$G$-invariant, the metric on $M$ is $G$-invariant, that is, $(M, \gm)$ is a Riemannian homogeneous space.
    \end{proof}

    \begin{remark}[A homogeneous space may not be Riemannian homogeneous]
        It is not the case that every homogeneous space admits a metric that makes it into a Riemannian homogeneous space. For example, the family of matrices of a fixed rank never admits a Riemannian homogeneous structure outside the trivial case, as noted in~\Cref{sec:fixed_rank}.
    \end{remark}

    As we have seen, on a Riemannian homogeneous space $M$, after fixing a point $x_0 \in M$, its isometry group may be endowed with a metric that is right-$H$-invariant for $H = G_{x_0}$. In other words, this metric descends into the quotient $G/H$, making it isometric to $M$ by construction. Furthermore, since this metric can be chosen to be left-$G$-invariant, it makes the diffeomorphism from~\Cref{prop:construction_and_characterisation_homogeneous_spaces} into an isometry and the map $\deffun{\pi : G -> G/H;}$ into a Riemannian submersion.

    As a Riemannian homogeneous space is a principal bundle with a metric that descends onto the quotient, we know that geodesics on $G/H$ are just the projection of geodesics on $G$ with a given initial horizontal vector. Hence, one piece that we are missing to be able to compute geodesics on Riemannian homogeneous spaces is to figure out what constitutes a horizontal vector for these metrics. In order to be able to do this, we will work with a class of Riemannian homogeneous spaces for which the horizontal bundle is particularly simple. We have already implicitly worked with these spaces during the proof of~\Cref{prop:left_G_right_H_invariant}, and now we will give them a name.

    We aim to be able to left-translate a tensor that descends from $T_e G$ to $T_{eH}G/H \iso T_{x_0}M$ onto all of $M$. Since we are taking the quotient by right action of the isotropy group $H$, any tensor that descends into the quotient should be right-$H$-invariant. Since we will be looking at tensors that are $\dif\quot{L}_g$-invariant, we are in particular interested in looking at tensors that are $\dif\quot{L}_h$-invariant for $h \in H$. These two things together take us into looking at tensors that are $\Ad(H)$-invariant. We can formalise this idea by looking at the quotient isotropy representation $\pa{\dif \quot{L}_h}_e \defi \pa{\dif \pi \circ \dif L_h}_e$. To compute it, we look first at the quotient action on the left
    \[
        \quot{L}_h(g) \defi \pi(L_h(g)) = hgH = hgh^{-1}H = \pi(c_h(g)) \mathrlap{\qquad g \in G, h \in H}
    \]
    and differentiating this relation at $g = e$ we get
    \[
        \pa{\dif\quot{L}_h}_e(v) = \pa{\dif\pi}_e\pa{\Ad_h(v)} \mathrlap{\qquad h \in H, v \in \glie.}
    \]
    In other words, we see that the differential of the diffeomorphisms $\quot{L}_h$ can be described via the adjoint representation of $H$ on $\glie$. Since $\hlie$ is clearly $\Ad(H)$-invariant, the study of tensors that are $\dif\quot{L}_h$-invariant heavily simplifies when $\glie$ splits into two complementary vector spaces that are $\Ad(H)$-invariant.

    \begin{definition}\label{def:reductive}
        We say that a homogeneous space $G/H$ is \textbf{reductive} if there is a linear complement $\glie = \hlie \oplus \mlie$ such that
    \[
        \Ad_h(\mlie) \subset \mlie \mathrlap{\qquad \forall h \in H.}
    \]
    \end{definition}

    \begin{remark}
    The subspace $\mlie$ does not need to be an ideal, as it may not be closed under the Lie bracket.
    \end{remark}

    We now give a more geometric interpretation of the usefulness of the reductive condition. But first we recall a basic geometric fact about Lie groups.

    \begin{remark}[Lie groups have trivial tangent bundle]
        Looking at the principal bundle $\deffun{\pi : G -> \set{e};}$ of $G$ acting on itself via right multiplication, we already saw in~\Cref{prop:vertical_space} that the infinitesimal generators of the action---left-invariant vector fields in this case---provide a vector bundle isomorphism $\deffun{\xi : G \times \glie -> \VV G;}$. Since in this case $\VV G = TG$, this proposition translates to the classical fact that left-invariant vector fields give a trivialisation of the tangent bundle of a Lie group, that is, $\xi(v) = \lv{v}$.
    \end{remark}

    \begin{remark}[A distinguished principal connection on a reductive homogeneous space]\label{rmk:distinguished_connection}
        Since $G$ is a principal $H$-bundle, a principal connection gives a split $TG = \HH \oplus \VV G$. We have that, through left-invariant vector fields, $\VV G \iso G \times \hlie$. What the reductive condition gives us is a distinguished principal connection $\HH_g \defi \pa{\dif L_g}_e\pa{\mlie}$. Since this is a principal connection, it descends into the quotient $G/H$ to give a bundle morphism along $\pi$ between $\HH \iso G \times \mlie$ and $T(G/H)$ which is an isomorphism on the fibres, as described in~\eqref{eq:bundle_isomorphism}. Furthermore, the equivariance allows us to identify an element in $T(G/H)$ with an element in $\mlie$ modulo a choice of an element in $h \in H$, coming from the fact that $\Ad_h(\mlie) = \mlie$, but $\Ad\vert_\mlie \neq \Id_\mlie$. In practice, when identifying an element $(x,v) \in T(G/H)$ with an element in $\mlie$, we will have to choose an element of the fibre $\pi^{-1}(x) \iso H$. This choice will not matter in the definition of the metric tensor on $G$, as the inner products that we consider on $\mlie$ will be $\Ad(H)$-invariant (\cf~\Cref{prop:left_G_right_H_invariant}).
    \end{remark}

    We have access to all these desirable properties in the setting of Riemannian homogeneous spaces, as every Riemannian homogeneous space is reductive.
    \begin{proposition}\label{prop:riemannian_is_reductive}
        Every Riemannian homogeneous space is reductive.
    \end{proposition}
    \begin{proof}
        A Riemannian homogeneous space can be represented as $G/H$ with $H$ a compact subgroup of $G$ (\Cref{prop:riemannian_homogeneous_stabiliser_is_compact}). Following~\Cref{prop:left_G_right_H_invariant}, fixing an inner product on $\glie$, we may average it into an $\Ad(H)$-invariant inner product on $\glie$. The choice $\mlie \defi \hlie^\perp$ is a reductive complement of $\hlie$.
    \end{proof}

    The language of reductive homogeneous spaces allows us to give a more concise version of the ideas already presented in~\Cref{prop:left_G_right_H_invariant}. In particular, it allows us to give a simpler version of the construction in~\Cref{thm:lifting_principal_bundle} in the context of reductive homogeneous spaces.

    \begin{theorem}[Lifting a metric on a reductive homogeneous space]\label{thm:lift_metric_reductive}
        Let $(M, \gm)$ with $M \iso G/H$ be a reductive homogeneous space and let $\gm_\hlie$ be an $\Ad(H)$-invariant inner product on $\hlie$ (which exists when $H$ is compact). There exists a unique metric $\total{\gm}$ on $G$ that makes $\deffun{\pi : G -> M;}$ into a Riemannian submersion with totally geodesic fibres isometric to $(H, \lv{\gm_\hlie})$.
    \end{theorem}
    \begin{proof}
    We may lift the metric $\gm$ into a left-$G$-invariant metric on the canonical distribution $\HH = \lv{\mlie}$ as done in~\Cref{thm:lifting_principal_bundle} by defining $\gm_{\HH} = \pa{\dif\pi\vert_{\HH}}^\ast\pa{\gm}$. At the same time, the $\Ad(H)$-invariant inner product on $\hlie$ may be left-transported as in the proof of~\Cref{prop:left_G_right_H_invariant} to a right-$G$ left-$H$-invariant product on $\VV G = \lv{\hlie}$. The metric $\total{\gm}$ comes from declaring the horizontal and vertical distributions to be orthogonal, that is, $\total{\gm} = \lv{\gm_\hlie} \oplus \gm_\HH$. This metric is left-$G$-invariant and right-$H$-invariant by construction and makes $\pi$ into a Riemannian submersion. The proof that the fibres of this construction are totally geodesic follows from~\Cref{thm:lifting_principal_bundle}.
    \end{proof}

    It is worth noting that a more general version of this theorem holds as first proved in the paper~\parencite{bergery1975certaines} (in French). If we have $K \subset H$ closed and compact subgroups of a Lie group $G$, we can still form a Riemannian submersion with totally geodesic fibres of the form
    \[
        \deffun{\pi : G/K -> G/H; gK -> gH}
    \]
    although in this case we end up with a fibre bundle with fibre the homogeneous space $H/K$, rather than a principal bundle. The details can be found in English in~\parencite[Thm.~$9.80$]{besse2008einstein}.

    \section{Naturally Reductive Homogeneous Spaces}\label{sec:left_invariant_metrics}
    In this section, we look at a particularly nice example of Riemannian homogeneous spaces for which the Levi-Civita connection of the metric coincides with the natural torsion-free affine connection on $G/H$. These are called \textbf{naturally reductive homogeneous spaces}, and were introduced in~\parencite{nomizu1954invariant}. An introductory treatment with a more modern notation can be found in~\parencite[Ch.~X Sec.~$3$]{kobayashi1969foundations}.

    We now have all the necessary tools to start computing geodesics on Riemannian homogeneous spaces. If we have a Lie group $G$ with a compact subgroup $H$ we can put a left-$G$-invariant and $\Ad(H)$-invariant metric on $G$. This metric descends into a left-invariant metric on $G/H$ making $\deffun{\pi : G -> G/H;}$ a Riemannian submersion. Hence, if we know how to compute geodesics for this metric on $G$, we know that the geodesics on $G/H$ are exactly projection of the geodesics on $G$ with initial conditions on $\mlie$, where $\mlie$ is the orthogonal complement of $\hlie$.

    The process to compute geodesics with initial conditions $(gH,v) \in T(G/H)$ would be as follows:
    \begin{enumerate}
        \item Identify $v$ with an element in $\mlie$ as follows: $v_{\mlie} \defi \pa{\dif L_{g^{-1}}}_g\pa{\pa{\dif\pi\vert_{\HH_g}}^{-1}(v)} \in \mlie$.
        \item Compute the geodesic $\gamma_e(t)$ on $G$ with horizontal initial conditions $(e, v_{\mlie})$.
        \item Transport $\gamma_e(t)$ to $T_g G$ using the isometry $L_g$: $\gamma_g \defi \pa{L_g}_\ast(\gamma_e)$.
        \item Project the geodesic back to the manifold: $\pi \circ \gamma_g$.
    \end{enumerate}

    Note that in the first step we have implicitly chosen an element in $\pi^{-1}(gH) \iso H$---which we have conveniently denoted by $g$---as the element in the fibre onto whose tangent space to lift $v$. As we mentioned before, this does not change the geodesics, as the metric is $\Ad(H)$-invariant.

    The only step missing in implementing the plan above is the second one. In other words, we just need examples of reductive homogeneous spaces on which we can compute horizontal geodesics. To do so, we first need the following proposition which will allow us to translate the Lie derivative on the manifold to the Lie derivative on the Lie algebra via left-invariant vector fields.
    \begin{proposition}\label{prop:left_invariant_algebra_homomorphism}
        Let $G$ be a Lie group. The map $X \mapsto \lv{X}$ is a Lie algebra homomorphism
        \[
            \lv{\liebrack{X,Y}} = \liebrack{\lv{X}, \lv{Y}} \mathrlap{\qquad \forall X,Y \in \glie.}
        \]
    \end{proposition}
    \begin{proof}
        Note that $\lv{X} = \pa{\dif L_g}_e(X) = \derivat{t}{0}g\expm(tX)$ comes from the differential of a right action evaluated at $g$. In other words, $\lv{X} = \xi(X)$ and the result follows from~\Cref{prop:vertical_space}.
    \end{proof}

    Now, we may compute the Levi-Civita connection for a left-invariant metric on $G$.
    \begin{proposition}\label{prop:lc_connection_lie_group}
        Let $\gm$ be a left-invariant metric on $G$. Then, the Levi-Civita connection at the identity takes the form
    \[
        \pa{\conn_{\lv{X}}\lv{Y}}_e = \lfrac{1}{2}\liebrack{X, Y} + U(X,Y)\mathrlap{\qquad \forall X, Y\in \glie,}
    \]
    where $U$ is the symmetric bilinear form on $\glie$ defined by
    \[
        2\gm\pa{U\pa{X,Y}, Z} =  -\gm\pa{X, \liebrack{Y,Z}} - \gm\pa{Y, \liebrack{X,Z}} \mathrlap{\qquad \forall X,Y,Z \in \glie,}
    \]
    or, more succinctly,
    \[
        U(X,Y) \defi -\lfrac{1}{2}\pa{\ad^\ast_X(Y) + \ad^\ast_Y(X)}\mathrlap{\qquad \forall X, Y \in \glie,}
    \]
    where $\ad^\ast_X$ is the adjoint with respect to the inner product on $\glie$.
    \end{proposition}
    \begin{proof}
        By the Koszul formula (see \eg,~\cite[Thm.\ 2.2.2]{petersen2016riemannian}), for any $X,Y,Z \in \glie$
    \begin{align*}
        2\gm\pa{\conn_{\lv{X}}\lv{Y}, \lv{Z}} &=
        \lv{X}\pa{\gm\pa{\lv{Y}, \lv{Z}}}
        + \lv{Y}\pa{\gm\pa{\lv{X}, \lv{Z}}}
        - \lv{Z}\pa{\gm\pa{\lv{X}, \lv{Y}}}\\
        &- \gm\pa{\lv{X}, [\lv{Y}, \lv{Z}]}
        - \gm\pa{\lv{Y}, [\lv{X}, \lv{Z}]}
        + \gm\pa{\lv{Z}, [\lv{X}, \lv{Y}]}.
    \end{align*}
    Since $\lv{X}, \lv{Y}, \lv{Z}$ are left-invariant fields, the function $p \mapsto \gm(\lv{X}, \lv{Y})_p$ is constant, so the first three terms are zero. Using~\Cref{prop:left_invariant_algebra_homomorphism} and evaluating at the identity we get the formula.
    \end{proof}

    \begin{remark}
        It is not true in general that $\conn_{\lv{X}}\lv{Y}$ is itself left-invariant due to the factor $U$. This is the reason for which we could compute the connection just at the identity.
    \end{remark}

    Following this proposition, if we have a Riemannian homogeneous space $G/H$, the vector bundle morphism $\deffun{\dif \pi : \mlie \times G -> T(G/H);}$ along $\pi$ induces the mapping
    \[
        X^\star \defi \dif\pi\pa{\lv{X}} \mathrlap{\qquad{X \in \mlie}}.
    \]
    With these vector fields, we may compute the connection at $eH \in G/H$.

    \begin{proposition}\label{prop:lc_connection_riemannian_homogeneous}
        Let $(G/H, \gm)$ be a Riemannian homogeneous space then
        \[
            \pa{\conn_{X^\star}Y^\star}_{eH} = \lfrac{1}{2}\pa{[X, Y]_\mlie}^\star_{eH} +\pa{U_\mlie(X,Y)}^\star_{eH}\mathrlap{\qquad \forall X, Y\in \mlie}
        \]
        where $\deffun{U_\mlie : \mlie \times \mlie -> \mlie;}$ is the projection of the bilinear form $U$ from~\Cref{prop:lc_connection_lie_group} onto $\mlie$.
    \end{proposition}
    \begin{proof}
        The result follows from the expression for the Koszul formula as in~\Cref{prop:lc_connection_lie_group} after noting that $\dif\pi\pa{\lv{X}} = \dif\pi\pa{\lv{\pa{X_\mlie}}}$ so that $\liebrack{X^\star, Y^\star} = \pa{\liebrack{X,Y}_\mlie}^\star$.
    \end{proof}

    This proposition suggests the following definition:
    \begin{definition}
        A \textbf{naturally reductive homogeneous space} is a Riemannian homogeneous space for which $U_{\mlie} \equiv 0$, or, equivalently, if the projection onto $\mlie$ of the adjoint $\deffun{\pa{\ad_X}_\mlie : \mlie -> \mlie;}$ is skew-symmetric for every $X \in \mlie$. In symbols,
        \[
            \scalar{\liebrack{X, Y}_\mlie, Z} = - \scalar{Y, \liebrack{X, Z}_\mlie}\mathrlap{\qquad \forall X,Y,Z \in \mlie.}
        \]
    \end{definition}

    \begin{remark}[Pseudo-Riemannian metrics on $G$]
        The definition of a naturally reductive space just depends on $\ad^\ast$, which, in turn, just depends on the choice of reductive structure $\glie = \hlie \oplus \mlie$ and a suitable inner product on $\mlie$. In particular, as we lift the metric from $\MM$ to $G$, we have a choice on the $\Ad(H)$-invariant bilinear map on $\hlie$. Throughout this chapter, we have assumed that this bilinear map was an inner product, but it is possible to have a pseudo-Riemannian left-$G$-invariant and right-$H$-invariant metric on $G$ that descends into a Riemannian metric on $G/H$. This setting includes examples such as the different non-compact Grassmannian and general non-compact symmetric spaces. For simplicity, we will not explore these spaces, as, to the best of our knowledge, these spaces are not of great interest in the setting of optimisation, with the notable exception of the hyperbolic space.
    \end{remark}

    \begin{remark}[A geometric motivation for naturally reductive geometric spaces]
        We saw in~\Cref{rmk:distinguished_connection} that reductive homogeneous spaces have a distinguished principal connection. For a homogeneous space, the associated bundle induced by the isotropy representation is the tangent bundle of $G/H$, that is, $G \times_H T_{x_0}M \iso T\pa{G/H}$. Hence, this principal connection defines a parallel transport system on $G/H$, which is given for a curve $\gamma(t)$ and an horizontal lift of it $g(t) \in G$ by
        \[
            \deffun{\mathbb{P}_{\gamma, t} : T_{\gamma(0)}G/H -> T_{\gamma(t)}G/H;
                X -> \dif\pa{\pi \circ L_{g(t)} \circ L_{g(0)^{-1}}}_{g(0)}
                \pa{\pa{\dif\pi\vert_{\HH_{g(0)}}}^{-1}(X)}.}
        \]
        The associated affine connection to this parallel transport system is called the \textbf{canonical connection}. This affine connection is given at the identity for a curve $\gamma(0) = eH$, $\dgamma(0) = X$ by
        \[
            \pa{\conn_{X^\star}Y^\star}_{eH} \defi \derivat{t}{0}\mathbb{P}^{-1}_{\gamma, t}\pa{Y^\star(t)} = \pa{\liebrack{X, Y}_\mlie}^\star_{eH}.
        \]
        This affine connection is left-invariant and complete, with geodesics of the form $\gamma\pa{t} = \pi\pa{g\expm\pa{tX}}$ for $X \in \mlie$ and $g \in G$ by construction. This connection has non-zero torsion given by the formula $T\pa{X,Y}_{eH} = \pa{\liebrack{X, Y}_\mlie}^\star_{eH}$ \parencite[Ch.~X Thm.~$2.6$]{kobayashi1969foundations}.

        For any affine connection $\conn$ we can define a torsion-free connection $\widetilde{\conn}$ with the same geodesics by subtracting a half of its torsion from it, $\widetilde{\conn} \defi \conn - \frac{1}{2}T$ \parencite{ambrose1960sprays}. By applying this construction to the canonical connection, given that the torsion is also left-invariant, we get a torsion-free left-invariant affine connection on $G/H$ with the same geodesics as the canonical connection, which is called the \textbf{natural torsion-free connection}. Now, on a Riemannian homogeneous space we have two distinguished torsion-free connections, namely the Levi-Civita connection associated to the metric and the natural torsion-free connection. Naturally reductive homogeneous spaces solve the question of when do these two connections agree. Since $\gm$ is left-invariant, its Levi-Civita connection $\conn^{\mathrm{LC}}$ will also be left-invariant. Since $\widetilde{\conn}$ is also left-invariant, we will have that $\widetilde{\conn} = \conn^{\mathrm{LC}}$ if and only if they agree at $eH$ on the geodesics of $\widetilde{\conn}$, that is, on the vector fields $X^\star$. But~\Cref{prop:lc_connection_riemannian_homogeneous} gives a formula for $\conn^{\mathrm{LC}}$ at $eH$, from which we deduce that $\widetilde{\conn} = \conn^{\mathrm{LC}}$ if and only if $U_\mlie \equiv 0$. In other words, the connection of a Riemannian homogeneous space $(M, \gm)$ is the natural torsion-free connection of its reductive structure if and only if $\pa{M,\gm}$ is naturally reductive. This was the original motivation for their introduction by Nomizu. Note that, in general, $\gm$ could be a pseudo-Riemannian metric, making $\pa{M, \gm}$ a pseudo-Riemannian homogeneous space. We restrict our attention to the Riemannian setting, as, for optimisation, we need a metric space structure on $M$ to be able to prove convergence results.
    \end{remark}

    \begin{remark}[Killing vector fields]
        Sometimes, the formulas for the connection in~\Cref{prop:lc_connection_riemannian_homogeneous} are given in terms of Killing vector fields on a Riemannian homogeneous space $(\MM, \gm)$. In that case, letting $G = \Iso(\MM, \gm)$, and fixing a point $p \in M$, it is possible to identify each $X \in \mlie$ with a Killing vector field $\killing{X}$ on $G/H$ that does not vanish at $p$. In this case, this assignment is a Lie algebra antihomomorphism $\killing{\liebrack{X,Y}} = -[\killing{X}, \killing{Y}]$, and the formula for the Levi-Civita connection of two Killing vector fields is the same but with a minus sign. This picture is particularly simple in the setting of naturally reductive homogeneous spaces as Killing vector fields are then given by the projection of \emph{right}-invariant vector fields on $G$, as will follow from~\Cref{prop:geodesics_naturally_reductive}. On the other hand, the left-invariant vector fields are more amenable for the computation of geodesics.
    \end{remark}

    The formula for the Levi-Civita connection of the projection of a left-$G$-invariant right-$H$-invariant metric on $G/H$ takes a particularly simple form when $(G/H, \gm)$ is naturally reductive.

    \begin{proposition}\label{prop:connection_naturally_reductive_spaces}
        Let $G$ be a Lie group and $H$ a compact Lie subgroup. Let $\gm$ be a left-$G$-invariant $\Ad(H)$-invariant metric on $G$ so that $(G/H, \pi_\ast(\gm))$ is a naturally reductive homogeneous space. Then
        \[
            \conn_{\lv{X}}\lv{Y} = \frac{1}{2}\lv{\liebrack{X,Y}} \mathrlap{\qquad \forall X,Y \in \mlie.}
        \]
    \end{proposition}
    \begin{proof}
        By the Koszul formula we have that for $X,Y,Z \in \glie$
        \[
            2\gm\pa{\conn_{\lv{X}}\lv{Y}, \lv{Z}} =
            - \gm\pa{\lv{X}, [\lv{Y}, \lv{Z}]}
            - \gm\pa{\lv{Y}, [\lv{X}, \lv{Z}]}
            + \gm\pa{\lv{Z}, [\lv{X}, \lv{Y}]}.
        \]
        By linearity, if we prove that the formula holds for $Z \in \mlie$ and for $Z \in \hlie$ then we are done.

        Assume that $X,Y \in \mlie$ and $Z \in \mlie$. By the naturally reductive condition and the left-invariance of $\gm$ the first two terms cancel getting
        \[
            2\gm\pa{\conn_{\lv{X}}\lv{Y}, \lv{Z}} = \gm\pa{\lv{[X, Y]}, \lv{Z}}
        \]
        so the formula holds on its horizontal part.

        Assume now that $Z \in \hlie$. Since $\gm$ is $\Ad(H)$-invariant, computing the derivative at zero of
        \[
            \gm\pa{\Ad_{\expm\pa{tZ}}(X),\Ad_{\expm\pa{tZ}}(Y)} = \gm\pa{X,Y},
        \]
        we get that $\ad_Z$ is skew-symmetric so the two first terms of the Koszul formula also cancel, and the formula holds on its vertical part as well.
    \end{proof}

    We are now ready to compute the geodesics for naturally reductive homogeneous spaces, and, in particular, for compact Lie groups with a bi-invariant metric.
    \begin{proposition}[Geodesics on naturally reductive spaces]\label{prop:geodesics_naturally_reductive}
        Let $G$ be a Lie group with a left-invariant metric such that $U \equiv 0$. The geodesic at the identity $\gamma_e$ with initial condition $X \in \glie$ is given by the one-parameter subgroups $\gamma_e(t) = \expm(tX)$.

        More generally, if $G/H$ is a naturally reductive homogeneous space, the geodesics at an arbitrary point $gH \in G/H$ are of the form
        \[
            \gamma_{gH}(t) = \pi(g\expm(tX_\mlie)) = g\expm(tX_\mlie)H.
        \]
    \end{proposition}
    \begin{proof}
        If $U \equiv 0$, we have $\conn_{\lv{X}}\lv{Y} = \lfrac{1}{2}\lv{\liebrack{X, Y}}$ (\Cref{prop:lc_connection_lie_group}). In particular, the vector field $\lv{X}$ is self-parallel. The integral curves of a left-invariant vector field is given by the Lie exponential, by definition of the Lie exponential.

        The same argument goes through after selecting $X,Y \in \mlie$ and applying the formula for the connection on two left-invariant vector fields on a naturally reductive homogeneous space (\Cref{prop:connection_naturally_reductive_spaces}). Since these are horizontal geodesics and $\deffun{\pi : G -> G/H;}$ is a Riemannian submersion we get the result.
    \end{proof}

    A large family of naturally reductive homogeneous spaces comes from having a Lie group $G$ together with a bi-invariant metric.
    \begin{definition}\label{def:normal_homogeneous_space}
        A \textbf{normal Riemannian homogeneous space} is a Riemannian homogeneous space $(\MM, \gm)$ with a presentation $G/H$ such that the metric is induced by a bi-invariant metric on $G$.
    \end{definition}

    Normal Riemannian homogeneous spaces are examples of naturally reductive spaces.
    \begin{proposition}\label{prop:normal_riemannian_are_homogeneous}
        Let $G$ be a Lie group with a bi-invariant metric, then
        \[
            \scalar{\ad_X(Y), Z} = -\scalar{Y, \ad_{X}(Z)} \mathrlap{\qquad \forall X, Y, Z \in \glie,}
        \]
        or, more compactly, $\ad_g^\ast = -\ad_g$ for every $g \in G$. In particular, a normal Riemannian homogeneous space is naturally reductive and its connection on left-invariant vector fields is given by
        \[
            \conn_{\lv{X}}\lv{Y} = \frac{1}{2}\lv{\liebrack{X,Y}} \mathrlap{\qquad\forall X,Y \in \glie.}
        \]
    \end{proposition}
    \begin{proof}
        Differentiate $\scalar{\Ad_{\expm\pa{tX}}(Y), \Ad_{\expm\pa{tX}}(Z)} = \scalar{Y, Z}$ at $t=0$. For the formula for the connection, consider the Koszul formula and use that $\ad^\ast_X = -\ad_X$ for every $X \in \glie$.
    \end{proof}

    \begin{remark}
        We already proved in~\Cref{corol:compact_group_bi-invariant} that compact groups admit bi-invariant metrics. Using the same construction, without the need of averaging, one proves that Abelian groups also admit a bi-invariant metric. A basic result in Lie group theory says that any Abelian group is the direct product of a torus with $\RR^n$. Hence, Lie groups of the form $G \iso K \times \RR^n$, for $K$ compact, admit a bi-invariant metric. In fact, these are the only groups with bi-invariant metrics.

\begin{theorem}[Classification of groups with bi-invariant metrics]
    A Lie group admits a bi-invariant metric if and only if it is isomorphic to a direct product $G \times H$, with $G$ compact and $H$ Abelian.
\end{theorem}
\begin{proof}
    See \parencite[Lemma~$7.5$]{milnor1976curvatures}.
\end{proof}

	In the case when $G$ is compact and semisimple, the bi-invariant scalar product is unique up to a scaling factor, and it is given by $-B$, where $B$ is the Killing form. These normal homogeneous spaces are called \textbf{standard Riemannian homogeneous spaces} and this metric is called the \textbf{canonical metric}.
    \end{remark}

    \section{Matrix Manifolds}\label{sec:matrix_groups}
    All the theory developed in this chapter comes together particularly nicely when $G$ is a matrix Lie group. In this case, if $G/H$ is a naturally reductive homogeneous space, we may simplify the computation of geodesics to the computation of the Lie exponential on $G$. It turns out that the Lie exponential is just the exponential of matrices, so this effectively transforms all the necessary computations into linear algebra problems. In this section we will see how some of the most important families of manifolds used in optimisation fall within the different categories of manifolds described before.

    We start by recalling the definition of a matrix Lie group.
    \begin{definition}\label{def:matrix_lie_group}
        A complex (resp.\  real) \textbf{matrix Lie group} is a closed subgroup of $\GL{n, \CC}$ (resp.\ $\GL{n, \RR}$).
    \end{definition}

    On matrix Lie groups, the Lie exponential is just the regular exponential of matrices.
    \begin{proposition}\label{prop:lie_exponential_is_exponential_matrices}
        The Lie exponential map on a real or complex matrix Lie group is given by the exponential of matrices.
    \end{proposition}
    \begin{proof}
        The matrix exponential $\gamma(t) = \expm\pa{tX}$ can be expressed as the solution of the matrix differential equation
        \[
            \gamma'(t) = X\gamma(t) \qquad \gamma(0) = \I_n
        \]
        for $X \in \CC^{n \times n} = \mathfrak{gl}(n, \CC)$. Since, for matrix groups, $\pa{\dif L_X}_{\I_n}(A) = XA$, we have that this is exactly the differential equation that defines the exponential map as an integral curve of a left-invariant vector field.
    \end{proof}

    Let us now show how the whole theory for computing geodesics is put together in the case of some commonly used manifolds in optimisation.

    \subsection{The Stiefel manifold}\label{sec:stiefel_manifold}
        We define the \textbf{Stiefel manifold} as the space of matrices with orthonormal columns
        \[
            \St{n,k} \defi \set{X \in \M{n,k} | \trans{X}X = \I_k}\mathrlap{\qquad 1 \leq k \leq n.}
        \]
        Throughout this section we will assume $n > k$ as, for $n = k$, $\St{n,k} \iso \Ort{n}$, which is disconnected. Hence, when doing orthogonal optimisation on $\M{n}$, we will just work on $\SO{n}$.

        Calling something a manifold does not make it a manifold, so let us show that it is in fact a homogeneous space by showing that it has a transitive action by a Lie group. This shows, in particular, that it is a smooth manifold (\Cref{prop:construction_and_characterisation_homogeneous_spaces}).

        Consider the left action of $\SO{n}$ on $\St{n,k}$ by multiplication on the left. This action is transitive, since for $U_1, U_2 \in \St{n,k}$ there exists a matrix $\total{U} \in \SO{n}$ such that $L_{\total{U}}(U_1) = U_2$. To see this, complete $U_1, U_2$ into two orthogonal matrices $\total{U}_1, \total{U}_2 \in \SO{n}$ and consider $\total{U} \defi \total{U}_1^{-1}\total{U}_2$.

        Let $x_0 = \begin{psmallmatrix}\I_k \\ 0_{n-k, k}\end{psmallmatrix} \in \St{n,k}$. The isotropy group of the action above at $x_0$---the elements of $\SO{n}$ that fix $x_0$---is
        \[
            H
            = \I_k \tensor \SO{n-k}
            = \set[\Big]{\begin{pmatrix}
                   \I_k & 0_{k, n-k} \\
                   0_{n-k, k} & Q
               \end{pmatrix}
               | Q \in \SO{n-k}}.
        \]
        Since $H \iso \SO{n-k}$, we have that $\St{n,k} = \SO{n} / \SO{n-k}$, where the projection $\deffun{\pi : \SO{n} -> \St{n, k};}$ maps an orthogonal matrix to its first $k$ columns (\Cref{prop:construction_and_characterisation_homogeneous_spaces}).\footnote{Note that here we are writing $\pi$ to refer to the map that is the composition of the canonical projection onto $\SO{n}/\SO{n-k}$ and the isotropy action at $x_0$ to avoid unnecessarily cluttering the notation.}

        Consider the inner product on $\solie{n} = \Skew{n}$ given by $\total{\gm}\pa{A, B}_e \defi \frac{1}{2}\tr\pa{\trans{A}B}$.\footnote{The factor of $\frac{1}{2}$ is there so that $\St{n,1}$ is isometric to the round sphere of radius $1$.}  Since this is $\Ad(\SO{n})$-invariant, it extends naturally to a bi-invariant metric on all of $\SO{n}$ via left-invariant vector fields, setting $\total{\gm}\pa{UA, UB}_U \defi \total{\gm}\pa{A,B}_e$ (\Cref{corol:compact_group_bi-invariant}). Since this metric on the total space is bi-invariant, $\St{n,k}$ is a normal Riemannian homogeneous space (\Cref{prop:normal_riemannian_are_homogeneous}).

        Denote by $\total{U} = \begin{pmatrix}U & U_\perp\end{pmatrix}$ the decomposition of a matrix $\total{U} \in \SO{n}$ into its first $k$ columns and the rest so that $\pi\pa{\total{U}} = U$. We want to compute the Riemannian exponential of a matrix $\Delta \in T_U\St{n,k}$. To do so, we will start by computing $\hlie$ and setting $\mlie = \hlie^\perp$ (\Cref{prop:riemannian_is_reductive})
        \[
            \hlie = 0_{k,k} \tensor \solie{n-k} \qquad
            \mlie =
            \set[\Big]{\begin{pmatrix}
                S & -\trans{A} \\
                A & 0_{n-k, n-k}
            \end{pmatrix} |
            S \in \solie{k}, A \in \M{(n-k), k}}.
        \]

        As the space is reductive, we can represent a tangent vector $\Delta \in T_U\St{n,k}$ by a vector in $X_\mlie \in \mlie$ as $\Delta = \pa{\dif\pi}_{\total{U}}(\total{U}X_\mlie)$, and the exponential map is given by
        \begin{equation}\label{eq:full_exponential_stiefel}
            \exp_{U}(\Delta) =
            \pi(\total{U}\expm(X_\mlie)) =
            \total{U}\expm
            \begin{pmatrix}
                S & -\trans{A} \\
                A & 0_{n-k, n-k}
            \end{pmatrix}
            P_{n,k},
        \end{equation}
        where $P_{n,k} \defi \begin{psmallmatrix} \I_k \\  0_{k, n-k} \end{psmallmatrix}$ is the matrix representation of the projection onto the first $k$ columns (\Cref{prop:geodesics_naturally_reductive}). The differential of $\pi$ at a matrix $B$ is then given by $\dif\pi(B) = BP_{n,k}$, so
        \[
            \Delta = \dif\pi(\total{U}X_\mlie) = US + U_\perp A \in T_U\St{n,k}.
        \]
        In view of this, we may represent a tangent space to the Stiefel manifold as
        \[
            T_U \St{n,k} = \set{ U S + U_\perp A | S \in \solie{k}, A \in \M{(n-k), k}}.
        \]

        \begin{remark}[The choice of completion $U_\perp$]
            Note that $\total{U} \in \pi^{-1}(U) \iso H$ (\Cref{prop:homogeneous_space_is_principal_bundle}). Hence, a choice of $U_\perp$ to complete $U$ into an orthogonal matrix $\total{U}$ accounts for a choice of $h\in H \iso \SO{n-k}$, which would transform $\mlie$ as $\Ad_h(\mlie)$. In this case, as we are fixing a basis of $\mlie$ and we are computing the basis of $T_U\St{n,k}$ in terms of it, this choice materialises as a change of basis on $T_U\St{n,k}$. This change of basis of the tangent space does not affect the computation of the geodesics or any other metric-related object, as the metric comes from an $\Ad(H)$-invariant inner product on $\glie$.
        \end{remark}

        We could directly use~\eqref{eq:full_exponential_stiefel} to compute the exponential map on $\St{n,k}$, but this formula has two problems. First, it requires the matrix exponential of the $n \times n$ matrix $X_\mlie$. As this matrix has rank at most $2k$, for small $k$, we would like to have a formula that scales down with $k$. Second, we need to compute $U_\perp$ if we only have access to $U$. This is equivalent to computing a \qr{} decomposition of an $n \times (n-k)$ matrix, which we would like to avoid.

        Assume for the rest of the section that $n > 2k$. We now present a method proposed in~\parencite{gallier2020differential} that avoids the first problem, but not the second. The idea is to decompose $A$ into its thin \qr{} decomposition
        \[
            A = Q_A R_A \qquad Q_A \in \St{n-k, k}, R_A \in \M{k}.
        \]
        We then compute
        \begin{align*}
            \total{U}
            \expm
            \underbrace{
            \begin{pmatrix}
                S & -\trans{Q_A R_A} \\
                Q_A R_A & 0_{n-k, n-k}
            \end{pmatrix}
            }_{n \times n}
            P_{n,k} &=
            \total{U}
            \begin{pmatrix}
                \I_k & 0_{k,k}\\
                0_{n-k,k} & Q_A
            \end{pmatrix}
            \expm
            \begin{pmatrix}
                S & -\trans{R_A} \\
                R_A & 0_{k, k}
            \end{pmatrix}
            \begin{pmatrix}
                \I_k & 0_{k,n-k}\\
                0_{k,k} & \trans{Q_A}
            \end{pmatrix}
            P_{n,k} \\
            &=
            \begin{pmatrix}U &  U_\perp Q_A \end{pmatrix}
            \expm
            \underbrace{
            \begin{pmatrix}
                S & -\trans{R_A} \\
                R_A & 0_{k, k}
            \end{pmatrix}
            }_{2k \times 2k}
            P_{n,k},
        \end{align*}
        where in the first equality we have used that, since $\expm$ is analytic,
        \[
            \expm(BX\trans{B}) = B\expm(X)\trans{B}\mathrlap{\qquad \forall B \in \St{n,k}, X \in \M{k}.}
        \]
        The cost of this formula is a thin \qr{} of an $(n-k) \times k$ matrix, the exponential of a $2k \times 2k$ matrix and the computation of $U_\perp$, which can be done through another thin \qr{} of an $n \times (n-k)$ matrix.

        We will now show how to avoid the computation of $U_\perp$. This may be achieved via a trick first introduced in~\parencite{edelman1998geometry}. In their paper they just outline the method and do not explicit the computations. We will see that, given the abstract theory presented before, the computations will follow naturally.

        The trick to avoid the computation of $U_\perp$ will come from finding a different parametrisation of tangent spaces of $\St{n,k}$. To do this, consider the natural embedding of $\St{n,k}$ into $\M{n,k}$. For a given matrix on the ambient space $C \in \M{n,k}$, and any point $U \in \St{n,k}$, the canonical inner product of $\M{n,k}$ gives a decomposition into the normal and tangential part of $C$
        \[
            C = \pi_U(C) + \pi^\perp_U(C).
        \]
        We will use the projection $\pi_U$ to parametrise the tangent space in terms of matrices in $\M{n,k}$. In order to do this, we shall compute formulas for these projections. By differentiating the formula $\trans{U}U = \I_k$ that defines $\St{n,k}$ we get an implicit formula for the tangent space of $\St{n,k}$
        \[
            T_U\St{n,k} = \set{\Delta \in \M{n,k} | \trans{U}\Delta \in \solie{k}}.
        \]
        It is direct to see that this implicit representation is equivalent to the explicit representation of the tangent space of $\St{n,k}$ given before. Now, consider a matrix $N$ in the normal space $N_U\St{n,k}$, where \emph{normal} is taken with respect to the canonical inner product of $\M{n,k}$. $N \in \M{n,k}$ is in the normal space if and only if $\scalar{N,\Delta} = 0$ for every $\Delta \in T_U\St{n,k}$, so we may choose a matrix of the form $N = UB$ for an arbitrary symmetric matrix $B$, as
        \[
            \scalar{\Delta, UB} = \tr\pa{\trans{\Delta}UB} = \scalar{\trans{U}\Delta, B} = 0\mathrlap{\qquad \forall \Delta \in T_U\St{n,k},}
        \]
        where we have used that symmetric and skew-symmetric matrices are orthogonal. Counting dimensions, we see that every matrix in the normal space is of this form, concluding that the normal space can be expressed as
        \[
            N_U\St{n,k} = \set{UB \in \M{n,k} | B \in \Sym{k}}.
        \]
        From here, we get the formulas for the normal and tangential projections from $\RR^n$
        \begin{align*}
            \pi^\perp_U(C) &= U\psym(\trans{U}C), \\
            \pi_U(C) &= C - \pi^{\perp}_U( C ) = U\pskew(\trans{U}C) + (\I_n - U\trans{U})C,
        \end{align*}
        where we have written $\psym(C) \defi \frac{1}{2}\pa{C + \trans{C}}$, $\pskew(C) \defi \frac{1}{2}\pa{C - \trans{C}}$ for the orthogonal projections onto the symmetric and skew-symmetric matrices.

        Using $\pi_U$, we can represent the matrix $\Delta \in T_U\St{n,k}$ as the projection of a matrix $C \in \M{n,k}$,
        \[
            \Delta = U\pskew(\trans{U}C) + (\I_n - U\trans{U})C.
        \]
        This decomposition is not unique, as there are many matrices $C\in\M{n,k}$ that project onto the same matrix $\pi_U(C)$. Comparing this decomposition with that in terms of $\mlie$, since both are orthogonal decompositions, we see that they are related by the equations
        \[
            U_\perp A = (\I_n - U\trans{U})C \qquad
            S = \pskew\pa{\trans{U}C}.
        \]

        Finally, we just have to apply the same \qr{} trick that we used before:
        \begin{align*}
            \total{U}
            \expm
            \underbrace{
            \begin{pmatrix}
                S & -\trans{A} \\
                A & 0_{n-k, n-k}
            \end{pmatrix}
            }_{n \times n}
            P_{n,k}
            &=
            \begin{pmatrix}
                U & \I_n
            \end{pmatrix}
            \begin{pmatrix}
                \I_k & 0_{k,k} \\
                0_{n,k} & U_\perp
            \end{pmatrix}
            \expm
            \begin{pmatrix}
                S & -\trans{A} \\
                A & 0_{n-k, n-k}
            \end{pmatrix}
            P_{n,k} \\
            &=
            \begin{pmatrix}
                U & \I_n
            \end{pmatrix}
            \expm
            \begin{pmatrix}
                \pskew\pa{\trans{U}C} & -\trans{\pa{(\I_n -U\trans{U})C}} \\
                (\I_n -U\trans{U})C & 0_{k, k}
            \end{pmatrix}
            P_{n,k} \\
            &=
            \begin{pmatrix}
                U & Q_C
            \end{pmatrix}
            \expm
            \underbrace{
            \begin{pmatrix}
                \pskew\pa{\trans{U}C} & -\trans{R_C} \\
                R_C & 0_{k, k}
            \end{pmatrix}}_{2k \times 2k}
            P_{n,k},
        \end{align*}
        where we have written $Q_C \in \St{n,k}, R_C\in\M{k}$ for the \qr{} decomposition of $(\I_n - U\trans{U})C\in\M{n,k}$. This formula is entirely in terms of $C$ and $U$, circumventing the computation of $U_\perp$. Furthermore, it just involves the exponential of a $2k \times 2k$ matrix. By doing this, we inadvertently lose the differentiability of the process, as the \qr{} decomposition is not smooth for matrices that are not full rank. This will render these tricks useless for the applications in this thesis, but they are rather interesting in their own sake nonetheless.

        We summarise this section in the following proposition.
        \begin{proposition}[Exponential map on the Stiefel manifold]
            Let $\total{U} \in \SO{n}$ and $U = \pi\pa{\total{U}} \in \St{n,k}$ be points on the total space and the manifold respectively. Let $X_\mlie =
            \begin{psmallmatrix}
                S & -\trans{A} \\
                A & 0_{n-k, n-k}
            \end{psmallmatrix}
            \in \mlie$ be the matrix on the Lie algebra that defines the tangent vector (matrix) $\Delta = \pa{\dif \pi}_{\total{U}}\pa{\total{U}X_\mlie} \in T_U\St{n,k}$. The Riemannian exponential map with respect to the canonical metric is given by
        \[
            \exp_{U}(\Delta) =
            \pi(\total{U}\expm(X_\mlie)) =
            \total{U}\expm
            \begin{pmatrix}
                S & -\trans{A} \\
                A & 0_{n-k, n-k}
            \end{pmatrix}
            P_{n,k},
        \]
        where $P_{n,k} \defi \begin{psmallmatrix} \I_k \\  0_{k, n-k} \end{psmallmatrix}$ is the matrix representation of the projection onto the first $k$ columns.

        If we represent $X_\mlie$ using a matrix $C \in \M{n,k}$ so that $U_\perp A = (\I_n - U\trans{U})C$ and $S = \pskew\pa{\trans{U}C}$, using a \qr{} decomposition of $Q_C R_C = (\I_n - U\trans{U})C \in \M{n,k}$, we can write the exponential map as
            \[
                \exp_{U}(\Delta) =
                \begin{pmatrix}
                    U & Q_C
                \end{pmatrix}
                \expm
                \begin{pmatrix}
                    \pskew\pa{\trans{U}C} & -\trans{R_C} \\
                    R_C & 0_{k, k}
                \end{pmatrix}
                P_{n,k}.
            \]
            Furthermore, this formula is only differentiable with respect to $C$ whenever $(\I_n - U\trans{U})C$ has maximal rank $k$.
        \end{proposition}

    \subsection{The Grassmannian}\label{sec:grassmannian}
    The Grassmannian $\Gr{n,k}$ is the set of all $k$-dimensional subspaces of $\RR^n$. Its construction as a Riemannian homogeneous space is essentially the same as in the case of the Stiefel manifold, so we will not give as many details here as we gave in the previous section.

    The homogeneous space structure of the Grassmannian, analogously to the case of the Stiefel manifold, comes from considering the isotropy group of the transitive left $G = \SO{n}$ action on the class of $x_0$ with $x_0 = \begin{psmallmatrix}\I_k \\ 0_{n-k, k}\end{psmallmatrix} \in \St{n,k}$. In this case $H = \operatorname{S}\pa{\Ort{k} \tensor \Ort{n-k}}$ and $\Gr{n,k} \iso \SO{n} / \operatorname{S}\pa{\Ort{k} \times \Ort{n-k}}$, where
    \[
        \operatorname{S}\pa{\Ort{k} \tensor \Ort{n-k}}\defi
        \pa{\Ort{k} \tensor \Ort{n-k}} \cap \SO{n}.
    \]
    Considering the $\Ad(G)$-invariant inner product $\scalar{A_1, A_2} = \frac{1}{2}\tr\pa{\trans{A_1}A_2}$ on $\glie$, we get
    \[
        \hlie = \solie{k} \tensor \solie{n-k} \qquad
        \mlie =
        \set[\Big]{\begin{pmatrix}
            0_{k,k} & -\trans{A} \\
            A & 0_{n-k, n-k}
        \end{pmatrix} |
        A \in \M{(n-k), k}}
    \]
    and left-translating it into a bi-invariant metric on $\SO{n}$, it descends into a metric on $\Gr{n,k}$ that makes $\Gr{n,k}$ into a normal Riemannian homogeneous space.\footnote{The Grassmannian with this metric is not only a normal Riemannian homogeneous space but also a symmetric space, but we will not use this fact here.} The Riemannian exponential is then given by
    \[
        \exp_{[U]}(A_{\mlie}) = \pi\pa{\total{U}\expm(A_{\mlie})}
    \]
    for a $\total{U} \in \SO{n}$ with $\pi(\total{U}) = [U]$. The elements of the Grassmannian are typically represented by a representative of their class $U \in \St{n,k}$.

    We can stablish an isomorphism between $\Gr{n,k}$ and $\Gr{n,n-k}$ by sending a space to its complement. Hence, we will assume in what follows that $n \geq 2k$.

    It is possible to compute the exponential on the Grassmannian by performing just one \svd{} of an $(n-k) \times k$ matrix, as noted in~\parencite{edelman1998geometry}. This is similar to the \qr{} trick performed in the Stiefel manifold case. Let $A \in \M{(n-k), k}$, and let $A_\mlie \in \mlie$ be the associated matrix on $\mlie$. If $A = U_A\Sigma_A\trans{V_A}$ is a \svd{} of $A$ with $U_A \in \St{n-k, k}, \Sigma_A \in \M{k}, V_A \in \SO{k}$, then
    \[
        A_\mlie =
        \begin{pmatrix}
            0_{k,k} & -\trans{A} \\
            A & 0_{n-k, n-k}
        \end{pmatrix}
        =
        \begin{pmatrix}
            V_A & 0_{k, n-k} \\
            0_{n-k, k} & U_A
        \end{pmatrix}
        \begin{pmatrix}
            0_{k,k} & -\trans{\Sigma_A}\\
            \Sigma_A & 0_{k,k}
        \end{pmatrix}
        \begin{pmatrix}
            \trans{V_A} & 0_{k, n-k} \\
            0_{n-k, k} & \trans{U_A}
        \end{pmatrix}.
    \]
    Now, representing a $k$-subspace $[U] \in \Gr{n,k}$ giving a base $U \in \St{n,k}$ of it, if $U_\perp \in \St{n,n-k}$ is a completion of $U$ into an orthogonal matrix,
    \[
        \exp_{[U]}\pa{A_\mlie} =
        \begin{pmatrix}
            UV_A & U_\perp U_A
        \end{pmatrix}
        \begin{pmatrix}
            \cos(\Sigma_A) \\
            \sin(\Sigma_A)
        \end{pmatrix}
        \trans{V_A},
    \]
    where $\Sigma_A$ is a diagonal matrix, so the sine and cosine are applied element-wise to the diagonal of $\Sigma_A$.

    Note that, if we consider a matrix $C \in \M{n,k}$, we may use the same trick as we did for the Stiefel manifold to avoid having to compute $\total{U}$. Indeed, consider the \svd{} $U_C\Sigma_C\trans{V_C} = \pa{\I_n -U\trans{U}}C$. We have that for an element $X = \pa{\I_n - U\trans{U}}C \in T_{[U]}\Gr{n,k}$,
    \[
        \exp_{[U]}\pa{X} =
        \begin{pmatrix}
            UV_C & U_C
        \end{pmatrix}
        \begin{pmatrix}
            \cos(\Sigma_C) \\
            \sin(\Sigma_C)
        \end{pmatrix}
        \trans{V_C}.
    \]

    In \gpu{} implementations, this method is hardly every faster than the naïve computation of the exponential of matrices, given that computing the \svd{} is orders of magnitude slower than computing a matrix exponential. Even then, it is still possible to perform the \qr{} trick described in the previous section and compute the Riemannian exponential map on the Grassmannian by computing a \qr{} decomposition $Q_C R_C = \pa{\I_n -U\trans{U}}C$, so that
    \[
        \exp_{[U]}\pa{X} =
        \total{U}\expm
        \begin{pmatrix}
            0_{k, k} & -\trans{A} \\
            A & 0_{n-k, n-k}
        \end{pmatrix}
        P_{n,k}
        =
        \begin{pmatrix}
            U & Q_C
        \end{pmatrix}
        \expm
        \begin{pmatrix}
            0_{k, k} & -\trans{R_C} \\
            R_C & 0_{k, k}
        \end{pmatrix}
        P_{n,k}.
    \]
    This trick for the Grassmannian was not presented in~\parencite{edelman1998geometry}.

    We summarise this discussion in a proposition.
    \begin{proposition}[Exponential map on the Grassmannian manifold]
        Let $\total{U} \in \SO{n}$ and $[U] = \pi\pa{\total{U}} \in \Gr{n,k}$ be points on the total space and the manifold respectively with a representative $U \in \St{n,k}$. Let $X_\mlie =
        \begin{psmallmatrix}
            0_{k,k} & -\trans{A} \\
            A & 0_{n-k, n-k}
        \end{psmallmatrix}
        \in \mlie$ be the matrix on the Lie algebra that defines the tangent vector (matrix) $\Delta = \pa{\dif \pi}_{\total{U}}\pa{\total{U}X_\mlie} \in T_U\Gr{n,k}$. The Riemannian exponential map with respect to the canonical metric is given by
    \[
        \exp_{[U]}(\Delta) =
        \pi(\total{U}\expm(X_\mlie)) =
        \total{U}\expm
        \begin{pmatrix}
            0_{k,k} & -\trans{A} \\
            A & 0_{n-k, n-k}
        \end{pmatrix}
        P_{n,k}
    \]
    where $P_{n,k} \defi \begin{psmallmatrix} \I_k \\  0_{k, n-k} \end{psmallmatrix}$ is the matrix representation of the projection onto the first $k$ columns.

    If we represent $X_\mlie$ using a matrix $C \in \M{n,k}$ so that $U_\perp A = (\I_n - U\trans{U})C$ and $\pskew\pa{\trans{U}C} = 0$, using a \qr{} decomposition and an \svd{} of $Q_C R_C = U_C\Sigma_C\trans{V_C} = \pa{\I_n - U\trans{U}}C \in \M{n,k}$, we can write the exponential map as
        \[
            \exp_{[U]}(\Delta) =
            \begin{pmatrix}
                UV_C & U_C
            \end{pmatrix}
            \begin{pmatrix}
                \cos(\Sigma_C) \\
                \sin(\Sigma_C)
            \end{pmatrix}
            \trans{V_C} =
            \begin{pmatrix}
                U & Q_C
            \end{pmatrix}
            \expm
            \begin{pmatrix}
                \pskew\pa{\trans{U}C} & -\trans{R_C} \\
                R_C & 0_{k, k}
            \end{pmatrix}
            P_{n,k},
        \]
        respectively. Furthermore, this formula is only differentiable with respect to $C$ whenever $(\I_n - U\trans{U})C$ has maximal rank $k$, and in the case of the \svd{} whenever $X_\mlie$ rank has, additionally, non-repeated singular values.
    \end{proposition}
    \subsection{Symmetric positive definite matrices}
    The case of the symmetric positive definite matrices $\Symp{n}$ is slightly different from the previous two. We include all the details here as it is surprisingly difficult to find this derivation in the literature.\footnote{In this case, one may take a shortcut since $\Symp{n}$ is not only a naturally reductive homogeneous space, but also a symmetric space. One may then use the theory of symmetric spaces to compute the geodesics, but this somewhat hides the ideas of the construction of a Riemannian submersion that underlay this computation.}

    This manifold has a transitive left action by the non-compact group of orientation preserving invertible matrices $\GLp{n}$ given by
    \[
        \deffun{\mu : \GLp{n} \times \Symp{n} -> \Symp{n};
        P, A -> PA\trans{P}}
    \]
    The action is well-defined, since $\scalar{Av,v} > 0$ for every $v \in \RR^n \backslash \set{0}$ if and only if $\scalar{PA\trans{P}v,v} = \scalar{A\trans{P}v, \trans{P}v} > 0$ for every $v \in \RR^n \backslash \set{0}$ given that $\trans{P} \in \GL{n}$. It is direct to see that the action is transitive, since any matrix in $A \in \Symp{n}$ is diagonalisable as $A = Q\Sigma\trans{Q}$ for $Q \in \SO{n}$, so any matrix can be taken to the identity as $\mu(\Sigma^{-1/2}\trans{Q}, A) = \I_n$.

    The isotropy group of this action is $H = \SO{n}$. Considering the canonical scalar product rescaled\footnote{This rescaling makes up for a constant that appears when differentiating the action $\mu$.}
    \[
        \sscalar{X,Y} \defi 4\scalar{X,Y} = 4\tr\pa{\trans{X}Y}\mathrlap{\qquad X, Y \in \gl{n} \iso \M{n}}
    \]
    we get that the split $\glie = \mlie \oplus \hlie$ is the usual split of a matrix into its symmetric and skew-symmetric parts via the projection with respect to the canonical product, that is, $\hlie = \solie{n}\iso\Skew{n}$ and $\mlie = \Sym{n}$. Furthermore, since this is an $\Ad(H)$-invariant product, we may extend it to a product on $\GLp{n}$ via left translations as
    \[
        \total{\gm}\pa{B_1, B_2}_P \defi \sscalar{P^{-1}B_1, P^{-1}B_2} \qquad \mathrlap{P \in \GLp{n}, B_1, B_2 \in T_P\GLp{n}.}
    \]
    In this case, the isotropy representation defining the isomorphism in~\Cref{prop:construction_and_characterisation_homogeneous_spaces}, rather than being a projection onto some components as it the case for the Stiefel manifold and the Grassmannian, is given by the map
    \[
        \deffun{\mu^{\I_n} : \GLp{n}/\SO{n} -> \Symp{n} ;
                P\cdot\SO{n} -> P\trans{P}.}
    \]
    Therefore, the submersion from $\GLp{n}$ onto $\Symp{n}$ associated to the action $\mu$ is given by
    \[
        \deffun{\mu^{\I_n} \circ \pi : \GLp{n} -> \Symp{n} ; P  -> P\trans{P}.}
    \]
    Differentiating at the identity, we get the linear isomorphism between $\mlie$ and a tangent space of $\Symp{n}$
    \[
        \deffun{\dif\pa{\mu^{\I_n} \circ \pi \circ L_{A^{1/2}}}_{\I_n} : \mlie -> T_A\Symp{n};
                \Delta -> 2A^{1/2}\Delta A^{1/2}.}
    \]
    The construction of the metric on a naturally reductive homogeneous space goes on to make this map into a linear isometry to turn $\mu^{\I_n} \circ \pi$ into a Riemannian submersion. For this to happen, since this map has inverse
    \[
        \deffun{\eta_A : T_A\Sym{n} -> \mlie; \Delta -> \frac{1}{2}A^{-1/2}\Delta A^{-1/2}}
    \]
    we have that the product on $T_A\Symp{n}$ that turns the projection into a submersion is
    \[
        \gm\pa{\Delta_1, \Delta_2}_A \defi \sscalar{\eta_A(\Delta_1), \eta_A(\Delta_2)} = \scalar{A^{-1}\Delta_1, A^{-1}\Delta_2}\mathrlap{\qquad\Delta_1, \Delta_2 \in T_A\Symp{n}.}
    \]
    Applying now the formula for the geodesics, we have that for an element $\Delta \in T_A\Symp{n} \iso \Sym{n}$
    \begin{align*}
        \exp_A(\Delta)
        &= \pa{\mu^{\I_n} \circ \pi}\pa{A^{1/2}\expm\pa{\tfrac{1}{2}A^{-1/2}\Delta A^{-1/2}}}\\
        &= A^{1/2}\expm\pa{A^{-1/2}\Delta A^{-1/2}}A^{1/2}\\
        &= A\expm\pa{A^{-1}\Delta},
    \end{align*}
    where in the last equality we have used that the exponential and the conjugation by a matrix commute.

    We summarise these computations in the following proposition.
    \begin{proposition}[Exponential map on \texorpdfstring{$\Symp{n}$}{PSD{(n)}}]
        Let $A \in \Symp{n}$ and $\Delta \in T_A\Symp{n} \iso \Sym{n}$ be a point and a vector respectively. The Riemannian exponential map with respect to the metric $\gm\pa{\Delta_1, \Delta_2}_A = \scalar{A^{-1}\Delta_1, A^{-1}\Delta_2}$ is given by
        \[
        \exp_A(\Delta)
        = A^{1/2}\expm\pa{A^{-1/2}\Delta A^{-1/2}}A^{1/2}
        = A\expm\pa{A^{-1}\Delta}.
        \]
    \end{proposition}

    \subsection{Fixed-rank matrices}\label{sec:fixed_rank}
        The manifold $\Rank{n,k,r}$ of $n \times k$ matrices of fixed rank $r$ is a manifold that is often considered in the context of low-rank optimisation. This manifold is a homogeneous manifold as it has a transitive left action from $\GL{n} \times \GL{k}$ given by
        \[
            \deffun{\mu : \pa{\GL{n} \times \GL{k}} \times \Rank{n,k,r} -> \Rank{n,k,r};
                ((A,B), X) -> AXB^{-1}.}
        \]
        It is clear that the action is well-defined by looking at $X$ as a linear map from $\RR^k$ to $\RR^n$ looking at $A,B$ as linear isomorphisms on these spaces. It is transitive, as, after an \svd{}, any matrix $X \in \Rank{n,k,r}$ may be taken to the matrix $x_0 = \begin{psmallmatrix} \I_r & 0_{r, k-r} \\ 0_{n-r, r} & 0_{n-r, k-r} \end{psmallmatrix}$. The isotropy group of this action is then given by
        \[
            H = \set[\Big]{
            \begin{pmatrix}
                C & A_1 \\
                0_{n-r,r} & A_2
            \end{pmatrix},
            \begin{pmatrix}
                C & 0_{r, k-r} \\
                B_1 & B_2
            \end{pmatrix}
            | C \in \GL{r}, \begin{matrix}
                                A_1 \in \M{r, \pa{n-r}}, & B_1 \in \M{\pa{k-r}, r}\\ 
                                A_2 \in \GL{n-r}, & B_2 \in \GL{k-r}
                            \end{matrix}
            }
            \subset \GL{n} \times \GL{k}
        \]
        so we have that $\Rank{n,k,r} \iso \pa{\GL{n}\times \GL{k}} / H$.

        As opposed to the previous examples, this manifold is much more difficult to work with, as it does not admit the structure of a Riemannian homogeneous space. In particular, it is not reductive outside of the case $\Rank{n,n,n} \iso \GL{n}$~\parencite[Prop.~$5.7$]{munthekaas2016integrators}. This is one of the reasons why doing low-rank optimisation is particularly challenging, as we do not have at our disposal all the tools developed above to compute the geodesics.

\clearpage
\chapter{Fibred Manifolds in First Order Optimisation}\label{ch:fibred_manifolds_in_optimisation}
In this chapter we show how the theory of Riemannian submersions and the ideas developed in~\Cref{ch:geometry} can be used to prove convergence results for optimisation problems on some manifolds in terms of problems on simpler manifolds by looking at submersions and Riemannian submersions. We also present a general treatment of the \textbf{dynamic trivialisation framework}. This framework was originally published in the paper~\parencite{lezcanocasado2019trivializations} at NeurIPS 2019. The proofs of convergence for this framework and an in-depth literature review are postponed to~\Cref{ch:second_order_bounds}.

\paragraph{Outline of the chapter.}
In~\Cref{sec:optimisation_on_manifolds}, we set up the notation for the rest of the chapter, and we include a short historical introduction to the field. In~\Cref{sec:first_order_structure}, we give some reasons for why submersions are of interest in the context of first-order optimisation, and we discuss the relation between Riemannian submersions and retractions. In~\Cref{sec:riemannian_submersions_optimisation} we prove all the necessary tools to give convergence rates on a base manifolds in terms of convergence rates on a total space, when these two are connected by a Riemannian submersion. In~\Cref{sec:change_metric}, we take a closer look at the \textbf{static trivialisation framework} when using the Riemannian exponential map to pullback the problem to a Euclidean space. In~\Cref{sec:dyn_triv}, we introduce the \textbf{dynamic trivialisation framework}, and outline the practical advantages that it provides. These advantages will be put in practice all throughout~\Cref{ch:geotorch} during the design of the library for optimisation on manifolds GeoTorch presented in that chapter.

\section{Optimisation on Manifolds}\label{sec:optimisation_on_manifolds}
Consider the problem
\[
    \min_{x \in \XX} f(x)
\]
where $\XX \subset \RR^n$. In this section we will consider the setting in which $\XX$ has a differentiable manifold structure, and we will denote this set by $\MM$. More generally, we will consider the intrinsic view of the problem, where $\MM$ is just an abstract differentiable manifold, writing the problem as
\begin{equation}\label{eq:minimisation_problem}
    \min_{x \in \MM} f(x).
\end{equation}
In order to be able to perform first order optimisation, we will need some extra structure on $\MM$ such as a Riemannian metric or an affine connection. We will consider this extra structure as we need it through the chapter. Even more, we will see that the point that makes manifold optimisation simpler than general constrained optimisation is the existence of a group of symmetries. In the cases when this group of symmetries is large enough, we will be able to make use of the tools presented in the last chapter. This approach will allow us to prove convergence results for this particularly well-behaved set of constrained optimisation problems.

The idea of taking advantage of the Riemannian structure of an embedded manifold, seen as a curved space locally approximated by their tangent spaces was first introduced in~\parencite{luenberger1972gradient}. Luenberger sees a manifold as a set of constraints $\MM = \set{x\in \RR^n | h(x) = 0}$ for a map $\deffun{h : \RR^n -> \RR^m;}$ for which zero is a regular value. He then presents two extrinsic projected gradient descent algorithms: Orthogonal projection onto the manifold, and projecting following the perpendicular to the tangent space, seeing the tangent space as an affine subspace of the ambient space. He also considers the intrinsic algorithm of gradient descent along geodesics. About these different perspectives, Luenberger notes:
\begin{quote}
    Developing a convergence analysis for the gradient projection method leads quickly to horrendous approximations and almost endless assumptions. Such analyses hardly compare with the relative neatness and elegance of the corresponding analysis of the unconstrained case.
\end{quote}
but he also adds:
\begin{quote}
    It should not be inferred that this paper suggests that when implementing such a scheme one should try to move along a geodesic. The geodesic method is proposed for theoretical purposes only. It should be noted, however, that all the methods for selecting the descent curve discussed above are asymptotically equivalent in the sense that the curves on $\MM$, determined by gradient projection or the geodesic leaving in the direction of the projected gradient, are all tangent at the starting point. This means that as the step size goes to zero the methods all move approximately along the same curve. This in turn implies that the algorithms all possess the same asymptotic convergence properties. In all cases the simplest procedure should be used. For computation the simplest is the second of the projection methods; for analysis it is certainly the geodesic method.
\end{quote}

In other words, following a geodesic in the direction of steepest descent or following a curve that is tangent to the geodesic is asymptotically equivalent. This idea has driven the research in the field for the last $50$ years under the name of \textbf{retractions}.
\begin{definition}\label{def:retraction}
    A map $\deffun{\retr : T\MM -> \MM;}$ is a \textbf{retraction} if for every $p \in \MM$, the map $\deffun{\retr_p : T_p\MM -> \MM;}$ satisfies
    \[
        \retr_p(0) = p\qquad\text{and}\qquad\pa{\dif \retr_p}_0 = \Id.
    \]
\end{definition}
By construction, a curve defined by a retraction $\retr_p(tv)$ for $t \in [0,1]$ is tangent at $p$ to the geodesic with initial conditions $(p,v) \in T\MM$. In particular, the Riemannian exponential map is a retraction, and as such, retractions can be thought of as a first order approximation to the Riemannian exponential map on $(\MM, \gm)$ for any metric $\gm$.

Note that in the definition of retraction there is no condition on $\MM$ being an embedded manifold of a Euclidean space nor there is any mention to a metric structure on $\MM$. For this reason, to be able to talk about notions such as distances or directions of steepest descent, it is convenient to consider a Riemannian metric on $\MM$ so that $(\MM, \gm)$ is a Riemannian manifold. Armed with these two tools, it is possible to define what has been the most popular first order algorithm for optimisation on manifolds: \textbf{Gradient descent along a retraction}
\begin{equation}\label{eq:grad_descent}
    x_{t+1} = \retr_{x_t}\pa{-\eta\grad f(x_t)}.
\end{equation}
This algorithm has been instantiated for most manifolds of interest in the area of optimisation by finding convenient retractions for each of them. For an introduction to these methods and their applications, we refer the reader to the  book~\parencite{absil2009optimization}.

Even though we needed a Riemannian structure on $\MM$ to be able to define the gradient in~\eqref{eq:grad_descent}, the choice of a retraction is independent of this Riemannian metric. This is the main reason underlying Luenberger's remark that proving convergence for these methods is far from clean. As Luenberger mentions, things considerably simplify when $\retr$ is chosen to be the Riemannian exponential map. In this case, it is possible to give some meaningful theorems, as the exponential map may be studied through the properties of the metric tensor. We will explore these ideas further in~\Cref{ch:second_order_bounds}.

Concurrently with Luenberger's work, Brockett brought techniques from differential geometry to numerical analysis to study differential equations on Lie groups, with applications to control theory~\parencite{brockett1972system}. His approach was much more abstract, considering intrinsic properties of the manifolds, rather than looking at them as manifolds embedded in $\RR^n$. It was his PhD student Smith who, in his thesis, applied these techniques for the first time in the context of optimisation on manifolds, generalising for the first time conjugate gradient descent and the Newton's method to abstract manifolds~\parencite{smith1993geometric}. These ideas were written aimed at a numerical analysis reader in the influential paper~\parencite{edelman1998geometry}.

We follow this latter approach, concentrating mostly on the global perspective as done in~\Cref{ch:geometry}, given that we believe that it strongly emphasises the underlying geometry of the problem.

We postpone to~\Cref{sec:previous_work} a full literature review, as presenting a literature review before presenting the algorithmic framework may make it unnecessarily difficult to follow.

\section{Fibred Manifolds and First Order Structure}\label{sec:first_order_structure}
Consider a map $\deffun{\triv : \total{\MM} -> \MM;}$ where $\total{\MM}$ is another differentiable manifold. We may transform problem~\eqref{eq:minimisation_problem} into
\begin{equation}\label{eq:pullback_problem}
    \min_{p \in \total{\MM}} f(\triv(p)).
\end{equation}
We define this to be a \textbf{pullback problem} associated to the original problem~\eqref{eq:minimisation_problem}. We will also refer to this pullback problem as a \textbf{static trivialisation} of the initial problem.\footnote{The ``trivialisation'' part of the name comes from the fact that, in the case when $\dim \total{\MM} = \dim \MM$, which will be quite important later on when $\total{\MM} = T_p\MM$, if $\triv$ does not have critical points on a neighbourhood, then it defines a local trivialisation around that point, which we may use to simplify gradient computations and use optimisers from $\RR^n$. These ideas were studied in the context of discretisation of differential equations on Lie groups, where one has access to global trivialisations of the tangent bundle~\parencites{magnus1954exponential}{iserles1999solution}{iserles2000lie}. The ``static'' part comes opposed to the \textbf{dynamic trivialisations}, which will be introduced in~\Cref{sec:dyn_triv}.}

We state now the two main examples that first drew our attention to this family of problems.
\begin{example}[Homogeneous spaces]
    Consider a homogeneous space $\deffun{\pi : G -> G/H;}$. The function $\pi$ allows pulling back an optimisation problem from a space which may have a potentially difficult topology to a Lie group, where we may exploit its group structure. This was a recurrent idea in~\Cref{ch:geometry} and in this chapter we will see how we may apply this to optimisation on manifolds.
\end{example}

\begin{example}[Tangent space]
    Consider $\total{\MM} = T_p\MM$ for a fixed $p \in \MM$. We will see that, if we consider a complete Riemannian manifold $(\MM, \gm)$ with non-positive sectional curvature, we may choose $\triv = \exp_p$ so that the map $\deffun{\exp_p : T_p\MM -> \MM;}$ is a submersion. In the general setting, even though it is not a global submersion, we can still control the problems that may arise for the manifolds of interest using standard tools from geometry, as we will see in~\Cref{sec:change_metric}. The failure of $\exp_p$ at being a global submersion will naturally lead to the concept of \textbf{dynamic trivialisations} studied in~\Cref{sec:dyn_triv}.

    If we pullback the problem to a Euclidean space, there are some practical advantages. The main advantage is that on a Euclidean space we have many different optimisation techniques which exploit its vector space structure: Momentum methods, adaptive methods, quasi-Newton methods, \textellipsis{} We will see that this gives a great advantage over traditional optimisation methods when dealing with non-convex optimisation problems in the deep learning context in \Cref{sec:experiments}. We will enumerate a number of other different advantages that this approach presents at the end of~\Cref{sec:dyn_triv}.
\end{example}

We now turn our attention to the study the general properties of static trivialisations in terms of those of the original problem.

If we are going to compare the pullback problem with the original problem, we should at least ask for $\triv$ to be \textbf{surjective}. If $\triv$ is surjective, then both minimisation problems are equivalent
\[
    \min_{p \in \total{\MM}} f(\triv(p)) = \min_{x \in \MM} f(x).
\]
In this setting, if the original problem has a unique minimiser $x^\ast \in \MM$ then, any minimiser of the pullback problem $p^\ast \in \MM$ will be related to $x^\ast$ through $\triv$ as $\triv(p^\ast) = x^\ast$. In general, if $\triv$ is surjective, $\triv$ will map global minimisers of the pullback problem to global minimisers of the original problem. In other words, surjective maps do not modify the $0^{\text{th}}$ order structure of the problem.

Given that in our study we will be looking at first order optimisation methods, it would be reasonable to ask $\triv$ to also preserve the first order structure of the $f$ in some sense. For example, if we want to parametrise $\RR_{>0}$ in terms of $\RR$, we could choose $\triv(t) = \log\pa{1+\exp(t)}$ as the parametrisation. On the other hand, we could also consider a function like the one seen in~\Cref{fig:surjective}. Both of them are smooth surjective mappings from $\RR$ into $\RR_{>0}$, but it should be clear that it may not be the best idea to use the latter one, as it will create new local minima or saddle points in pullback problem that did not exist in the original one.

\begin{figure*}[!tbp]
\centering
  \begin{minipage}[b]{0.5\textwidth}
      \centering
      \includegraphics[width=0.9\columnwidth]{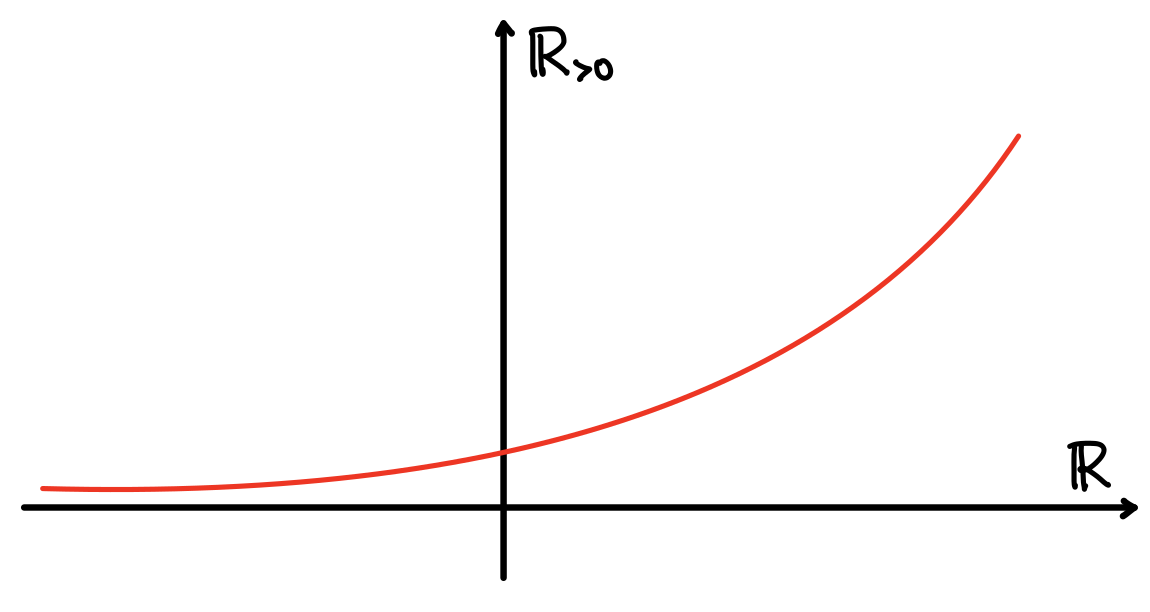}
  \end{minipage}%
  \begin{minipage}[b]{0.5\textwidth}
      \centering
      \includegraphics[width=0.9\columnwidth]{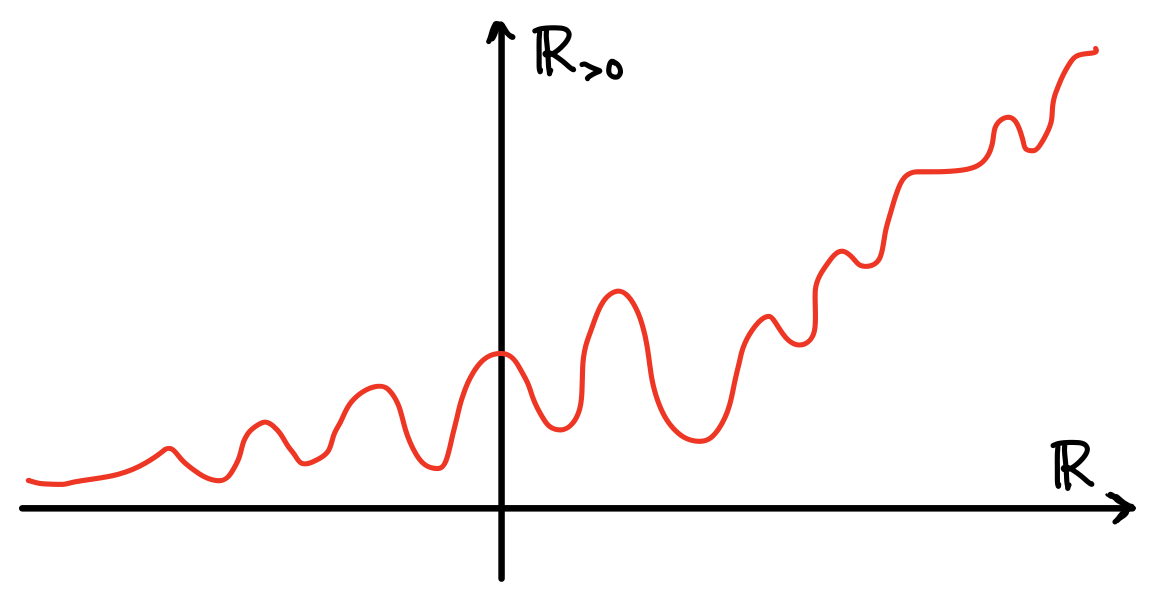}
  \end{minipage}
  \caption{Two surjective maps from $\RR$ to $\RR_{>0}$ with vastly different properties when used in optimisation problems.}%
  \label{fig:surjective}
\end{figure*}

More abstractly, by the chain rule we have:
\[
    \dif\pa{f \circ \triv} = \dif f \circ \dif \triv.
\]
In local coordinates around a point, this is just a vector-matrix multiplication. Now, if we want $\dif\pa{f \circ \triv}$ to have critical points if and only if $\dif f$ has critical points, this may only happen if $\dif \triv$ has a right-inverse at every point, that is, if $\dif \triv$ is surjective at every point. In other words, since we wanted $\triv$ to be also surjective, we have that the maps that do not create new critical points are exactly \textbf{submersions}.

\begin{proposition}[Submersions do not create spurious critical points]\label{prop:submersions_dont_create}
    Let $\deffun{\triv : \total{\MM} -> \MM;}$ be a surjective map between manifolds. The following are equivalent:
    \begin{enumerate}
        \item The map $\triv$ is a submersion.
        \item For every function $f$ on $\MM$ and every point $p \in \total{\MM}$, $p$ is a critical point of $f \circ \triv$ if and only if $\triv\pa{p}$ is a critical point of $f$.
    \end{enumerate}
\end{proposition}

Note that~\Cref{prop:submersions_dont_create} does not mean that the critical points of $f$ and $f \circ \triv$ are in bijection through $\triv$, but rather that for every point critical point $x \in \MM$ we get a fibre of critical points $\total{\MM}_x \defi \triv^{-1}(x)$. In other words, the critical points of $f \circ \triv$ are somewhat ``concentrated in the same way as those from $f$'', so that $\triv$ does not create new spurious critical points.

Another way of looking at this idea is through the lens of local sections. If $\triv$ is a submersion, for any two points such that $\triv(p) = x$, there exists a map $\deffun{\sigma : \triv(U) -> U;}$ on a neighbourhood $U$ of $p$ such that $\triv \circ \sigma = \Id_{\triv(U)}$ (\Cref{prop:existence_local_sections}). In particular, if $x$ is an isolated critical point of $f$ on a neighbourhood $W$, then $p$ will be an isolated critical point of $f \circ \triv$ on the immersed manifold $\Sigma = U \cap \sigma\pa{W} \subset \total{\MM}$. Note that this submanifold $\Sigma$ need not be open in $\total{\MM}$, so it will not be a neighbourhood of $p$ whenever $\dim(\total{\MM}) > \dim(\MM)$. In this case, the fibre will not be discrete---will have dimension $\dim(\total{\MM}) - \dim(\MM)$---and the point $p$ will not be an isolated critical point of $f\circ \triv$ on $\total{\MM}$.

Let us look at all these ideas in a simple example to help visualise them.
\begin{example}[Submersion onto the circle]
    Let $\deffun{\pi : \RR^2 \backslash \set{0} -> \SS^1;}$ be the orthogonal projection onto the unit circle in $\RR^2$. The differential of this map at a point projects a vector onto its tangent component, so $\pi$ is a submersion. We can see several local sections of this submersion at a point $p \in \RR^2 \backslash \set{0}$ in~\Cref{fig:sections}. This fibred manifold is particularly simple, as the total space splits as a product of the base space and the fibre $\RR^2 \backslash \set{0} \iso \SS^1 \times \RR_{> 0}$. This means that this is a trivial fibre bundle, so it also has global sections, as seen in~\Cref{fig:sections}.

\begin{figure*}[!tbp]
\centering
  \begin{minipage}[b]{0.5\textwidth}
      \centering
      \includegraphics[width=0.8\columnwidth]{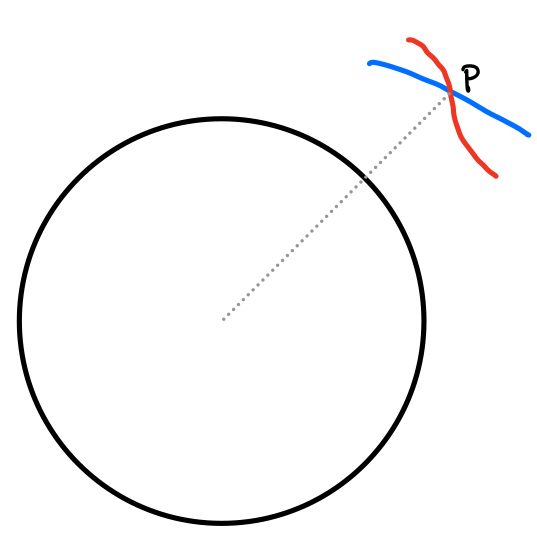}
  \end{minipage}%
  \begin{minipage}[b]{0.5\textwidth}
      \centering
      \includegraphics[width=0.8\columnwidth]{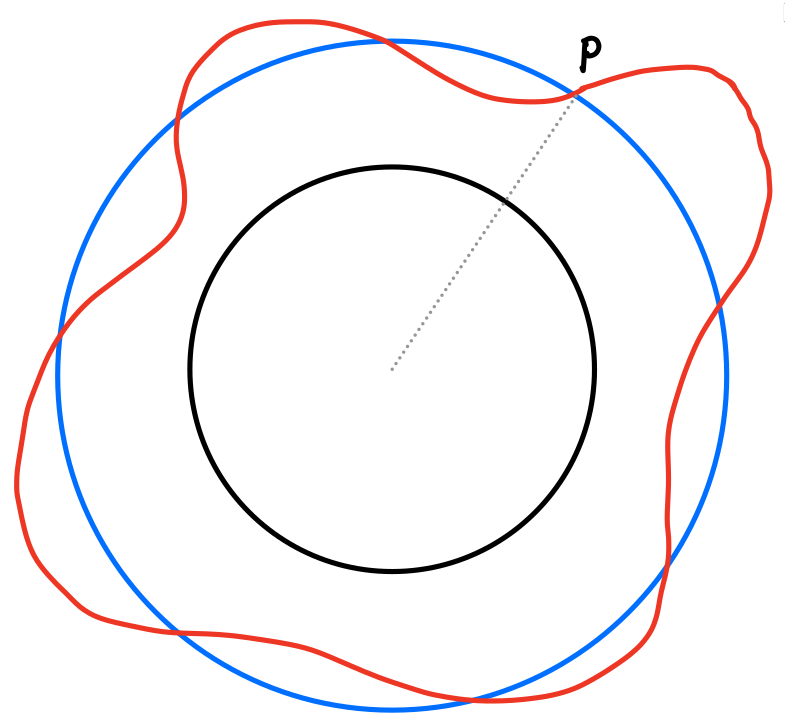}
  \end{minipage}
  \caption{Examples of local and global sections of the fibre bundle $\deffun{\pi : \RR^2 \backslash \set{0} -> \SS^1;}$.}%
  \label{fig:sections}
\end{figure*}

    Consider a function $\deffun{f : \SS^1 -> \RR;}$ with an isolated critical point $x \in \SS^1$. We may pullback this function to $\RR^2 \backslash \set{0}$ using $\pi$. In this case, a critical point $x \in \SS^1$ is transformed into a whole fibre of critical points of $f \circ \pi$. On the other hand, none of the fibres next to that one have any critical points.

    Another slightly different submersion onto the circle that is sometimes of help is that given from its universal cover. If we have a function $f$ on the circle, we may pull it back to a $1$-periodic function on the real line via the map $\deffun{\rho : \RR -> \SS^1 \subset \CC;}$ given by $\rho(t) = e^{2\pi it}$. This submersion is different in nature to the previous one given that, since $\dim \RR = \dim \SS^1$, it has a discrete fibre. As such, this submersion maps isolated critical points to isolated critical points.

    This second example hints at how submersions help to define functions on manifolds via functions on other more amenable spaces. We will exploit this idea in the sequel when working with homogeneous spaces $\deffun{\pi : G -> G/H;}$.
\end{example}

\begin{remark}[Higher order generalisation]
    We have motivated the use of submersions in the context of first order optimisation via the existence of $1^{\text{st}}$ order right-inverses. It is worth noting that the idea of a function being surjective can be phrased in categorical terms as the function having a $0^{\text{th}}$ order right-inverse, that is, an inverse in the category of sets. This should be not surprising, as the difference between $f$ and the pullback $f \circ \triv$ is just a right composition.

    This hints at how to extend these ideas and the ideas presented in the sequel to higher order methods. On the other hand, this extension gets quite technical very quickly as, to be able to talk about higher-order derivatives either one has work with jet bundles, or one has to fix connections---or Riemannian metrics---on $\total{\MM}$ and $\MM$ to be able to talk about $\conn\dif \triv$. It is for this reason that we leave the higher order treatment for future research.
\end{remark}

There are also examples of very simple submersions that are widely used in the machine learning and deep learning community.
\begin{example}[Submersions in machine learning]
    Let $\MM \iso \RR^n$ and consider component-wise functions like the hyperbolic tangent $\deffun{\tanh : \RR^n -> (-1,1)^n;}$ or the sigmoid function $\deffun{\sigma : \RR^n -> (0,1)^n;}$. These examples are not only submersions but also diffeomorphisms. We already saw their use in the \lstm{} model in~\Cref{sec:rnns}.
\end{example}

For now, we have looked at the use of submersions just through the lens of differential geometry. We still need to talk about metric structures to be able to make sense of notions relevant to convergence and gradients. Before doing so, let us fix some notation and recall some standard constructions.

Assume that we have two metrics $\gm, \total{\gm}$ on $\MM$ and $\total{\MM}$ respectively. For these metrics, denote by $\deffun{\alpha : T\MM -> T^\ast \MM;}$ (resp.\ $\beta$) the isomorphism between the tangent and cotangent bundle given by the metric, $\alpha(X) \defi \gm(X, -)$. For a given map $\deffun{\psi : \total{\MM} -> \MM;}$ we denote its fibre-wise dual as
\[
\deffun{\dif \psi' : T^\ast \MM -> T^\ast \total{\MM};
        \eta -> \eta \circ \dif \psi}
\]
and its fibre-wise adjoint with respect to $\gm$ and $\total{\gm}$ as $\deffun{\dif \psi^\ast : T\MM -> T\total{\MM};}$, which is defined for a vector field $X$ on $\MM$ and a vector field $Y$ along $\psi$ as
\[
    \total{\gm}\pa{X, \pa{\dif\psi}^\ast(Y)} \defi \gm\pa{\dif\psi(X), Y}.
\]

\begin{proposition}\label{prop:dual_adjoint}
    Let $\deffun{\psi : \total{\MM} -> \MM;}$ be a smooth map---not necessarily a submersion---between Riemannian manifolds with metrics $\total{\gm}$ and $\gm$ respectively. The following relation holds
    \[
    \beta \circ \dif\psi^\ast = \dif\psi' \circ \alpha.
    \]
\end{proposition}
\begin{proof}
    Let $Y$ be a vector field on $\total{\MM}$ and $X$ be a vector field along $\psi$,
    \[
        \pa{\dif\psi' \circ \alpha}(X)(Y) = \gm\pa{X, \dif\psi(Y)} = \total{\gm}\pa{\dif\psi^\ast(X), Y} = \pa{\beta \circ \dif\psi^\ast}(X)(Y).\qedhere
    \]
\end{proof}

Using this proposition, we can compute the gradient with respect to the new parametrisation. This is sometimes referred to as \textbf{the adjoint method} in numerical analysis.
\begin{corollary}[The adjoint method]\label{corol:3_grad_abstract}
    Let $\deffun{\psi : \total{\MM} -> \MM;}$ be a smooth map between Riemannian manifolds and $f$ be a function on $\MM$. Then,
    \begin{equation}\label{eq:grad_abstract}
        \grad\pa{f \circ \psi} = \dif\psi^\ast(\grad f).
    \end{equation}
\end{corollary}
\begin{proof}
    This is direct using the previous proposition
    \[
        \grad\pa{f \circ \psi} := \beta^{-1}\pa{\dif \pa{f \circ \psi}} = \pa{\beta^{-1} \circ \dif\psi'}\pa{\dif f} = \dif\psi^\ast\pa{\grad f}.\qedhere
    \]
\end{proof}

If $\deffun{\triv : \total{\MM} -> \MM;}$ is a Riemannian submersion, since $\dif \triv\vert_\HH$ is a linear isometry and $\dif \triv\vert_{\VV \total{\MM}} = 0$, we have that
\[
    \dif \triv \circ \dif \triv^\ast = \Id_{T\MM}.
\]
In other words, the map $\dif \triv^\ast$ takes values in the horizontal bundle $\HH$ (\Cref{def:connection}) and is a linear inverse of $\dif \triv\vert_\HH$---the non-zero part of $\dif \triv$---and a linear right-inverse of $\dif \triv$. This formula together with~\Cref{prop:dual_adjoint} is particularly useful to check when a map $\triv$ is a Riemannian surjection on low dimensional manifolds if we have access to a local representation of $\gm$ and $\total{\gm}$, as we do not need to compute the horizontal space associated to $\total{\gm}$.

It is also worth noting that, while the right-hand side of~\eqref{eq:grad_abstract} seems to depend on the choice of metric on $\MM$, it does not. Even then, when computing gradients recursively in autodiff frameworks, it is convenient to fix a metric on every intermediate manifold and compute the adjoints with respect to these fixed metrics, so that we can compose the gradient given by some layer with the adjoint with respect to these metrics of the next layer. Of course, in most situations this metric is just the canonical metric on $\RR^n$, so computing the adjoint simply accounts for transposing the Jacobian.

Let us now see how submersions provide a natural way to define a retraction on $\MM$ in terms of one on $\total{\MM}$. Assume that we have a retraction $\total{r}$ on $\total{\MM}$, that is, a method to perform optimisation on $\total{\MM}$. This induces discrete dynamics on $\MM$ as the pushforward of the dynamics on $\total{\MM}$ for a step-size $\eta > 0$
\begin{align*}
    p_{t+1} &= \total{r}_{p_t}(-\eta\grad f(p_t))\\
    x_{t+1} &= \psi(p_{t+1})
\end{align*}

This allows us to define the step map on $\MM$.
\begin{definition}\label{def:retraction_step}
    Let $\deffun{\psi : \total{\MM} -> \MM;}$ be a smooth map between Riemannian manifolds and let $\total{r}$ be a retraction on $\total{\MM}$. We define the \textbf{step map of $\total{r}$ along $\psi$ at $p \in \total{\MM}$} as
    \[
        \deffun{\Psi_p^{\total{r}} \defi \psi \circ \total{r}_p \circ \pa{\dif\psi}_p^\ast : T_{\psi(p)}\MM -> \MM;}.
    \]
\end{definition}

In this way, we can write the dynamics on $\MM$ as
\[
    x_{t+1} = \Psi_{p_t}^{\total{r}}\pa{-\eta\grad f\pa{x_t}}
\]
where $p_t$ is defined as before. A question that arises now is whether the map $\deffun{\Psi_p^{\total{r}} : T\MM -> \MM;}$ is a retraction for every $p \in\total{\MM}$.
\begin{proposition}\label{prop:retraction_step}
    Let $\deffun{\psi : \total{\MM} -> \MM;}$ be a smooth map between Riemannian manifolds and let $\total{r}$ be a retraction on $\total{\MM}$. Then $\Psi^{\total{r}}_p$ is a retraction for every $p$ if and only if $\psi$ is a---not necessarily surjective---Riemannian submersion.
\end{proposition}
\begin{proof}
    It is clear that $\Psi^{\total{r}}_p(x,0) = x$. For the second condition, differentiating $\Psi_p^{\total{r}}$, we have that the map is a retraction if and only if
    \[
        \dif\psi_p \circ \dif\psi^\ast_{p} = \Id_{T_{\psi(p)}\MM}
    \]
    or equivalently, whenever $\psi$ is a not necessarily surjective submersion.
\end{proof}

In the simplest case, when the dimension of $\total{\MM}$ and $\MM$ agree, this proposition says that the gradient is preserved just by local isometries, which is exactly what one would expect.

In the general case, we have seen in~\Cref{prop:submersion_into_riemannian_submersion} that, if we have a submersion $\deffun{\triv : \total{\MM} -> \MM;}$ and a metric on $\MM$, we may lift this metric to a metric on $\total{\MM}$ that makes $\deffun{\triv : \total{\MM} -> \MM;}$ into a Riemannian submersion after choosing an Ehresmann connection on $\total{\MM}$ and a metric tensor on $\VV\total{M}$. As such, \Cref{prop:retraction_step} shows that this process gives dynamics induced by a gradient descent performing optimisation on $\total{\MM}$ rather than on $\MM$ after lifting the gradient on every step to a horizontal gradient on $\total{\MM}$ via $\pa{\dif \triv}^\ast$.

On the other hand, there are metrics on $\total{\MM}$ that do not descend into a metric on $\MM$---those metrics that are not constant on the fibres. This shows that, through this method, we can get new dynamics that cannot be achieved just by performing gradient descent with a retraction on $\MM$ with respect to any metric.

There is a third interesting setting. If $\triv$ is a diffeomorphism, and we have two distinguished metrics $\total{\gm}, \gm$, but these metrics do not make $\triv$ into a Riemannian submersion, optimising $f \circ \triv$ accounts for optimising $f$ with respect to the pushforward metric $\triv_\ast(\total{\gm})$ on $\MM$. Therefore, precomposing by $\triv$ accounts for a change of metric on $\MM$. More generally, we may find this scenario when we have a metric on $\total{\MM}$ that does descend into a metric on $\MM$.

To summarise, we have three scenarios of particular interest:
\begin{enumerate}
    \item $\deffun{\pi : \total{\MM} -> \MM;}$ is a Riemannian submersion
    \item It is not a Riemannian submersion, but the metric on $\total{\MM}$ descends into a metric on $\MM$
    \item The metric on $\total{\MM}$ does not descend into a metric on $\MM$.
\end{enumerate}

We will study the first case in~\Cref{sec:riemannian_submersions_optimisation} and the second one in~\Cref{sec:change_metric}.

The third case is particularly interesting in the setting of linear networks, as it has been shown in~\parencite{arora2018optimization} that overparametrisation in the linear setting yields dynamics that fall in this setting---dynamics that cannot be achieved by Riemannian gradient descent---and furthermore, they achieve acceleration-like convergence. The follow-up work~\parencite{bah2019learning} looks at these results from the lens of Riemannian geometry. This allows the authors to prove stronger results than those in the original paper. Even though we think that the ideas presented in this thesis would help to draw a more geometric picture of the behaviour described in these papers, we will not pursue this research direction in this thesis.

Before delving into the first two points, we continue the simple example of the circle to showcase some of these scenarios.

\begin{example}[Projection onto the circle]
    We continue with the example of the orthogonal projection $\deffun{\pi : \RR^2 \backslash \set{0} -> \SS^1;}$. Both the total space and the base space involved in this submersion have a distinguished metric, \ie, that induced by their identification as subsets of $\RR^2$. These metrics do not turn $\pi$ into a Riemannian submersion. This is clear, since Riemannian submersions map horizontal geodesics to geodesics, and the projection of geodesics of $\RR^2 \backslash \set{0}$ with horizontal initial conditions---straight lines orthogonal to the fibre---just cover half of $\SS^1$, while the round metric on $\SS^1$ is complete.

    The standard metric on $\RR^2 \backslash \set{0}$ in polar coordinates is represented as $\dif r^2 + r^2\dif\theta^2$, while the differential of $\pi$ is given by $\dif \pi = \partial_t \tensor \dif \theta$, where $\partial_t$ is the line element on $\SS^1$. The pullback of the canonical metric on the circle $\dif t^2$ onto the horizontal bundle spanned by $\partial_\theta$ that makes $\dif\pi \vert_\HH$ into a linear isometry is given by $\pi^\ast\pa{\dif t^2} = \dif \theta^2$---rather than $r^2\dif \theta^2$, as the metric has. As described in~\Cref{sec:theory_submersions}, we find that a family of metrics that make $\pi$ into a Riemannian submersion are those of the form $g(r, \theta)^2\dif r^2 + \dif \theta^2$ for a strictly positive function. For example, for $g \equiv 1$ we get a cylinder.

    The horizontal geodesics for any metric of this form starting at $(r, \theta)$ are just circles of radius $r$ and therefore, they project onto geodesics on $\SS^1$, as expected. Note that this construction of a family of metrics that make a submersion into a Riemannian submersion is exactly that described in~\Cref{prop:submersion_into_riemannian_submersion}.

    The standard metric on $\RR^2 \backslash \set{0}$ does not descend into a metric on $\SS^1$ and as such, the dynamics given by optimising $f \circ \pi$ with respect to this metric would not be given by any retraction on the circle.
\end{example}

This example is not too interesting, as the metric that makes the projection into a Riemannian submersion is the one that ``copies'' the metric onto each circle of the total space. This is due to the fact that the bundle is trivial. When the bundle is not trivial, this construction gives much more interesting examples, as we will see in the next section.

\section{Riemannian Submersions}\label{sec:riemannian_submersions_optimisation}
Riemannian submersions are the simplest of all the previous scenarios. They are a natural generalisation of isometries, and thus, they are as well-behaved as one could expect. Riemannian submersions allow us to prove a fairly general convergence theorem, which can be interpreted as saying that if we know how to optimise in the total space $(\total{\MM}, \total{\gm})$ and we have a Riemannian submersion $\deffun{\triv : \total{\MM} -> \MM;}$, it is sensible to consider the optimisation algorithm along the step-map, as described in the previous section.

\begin{theorem}\label{thm:submersion_submetry}
    Let $\deffun{\triv : \total{\MM} -> \MM;}$ be a Riemannian submersion, $\total{r}$ a retraction on $\total{\MM}$ and $f$ a function on $\MM$. If for an initial point $p_0 \in \total{\MM}$ and a step $\eta > 0$ the sequence
    \[
        p_{t+1} = \total{r}_{p_t}\pa{-\eta\grad\pa{f \circ \triv}\pa{p_t}}
    \]
    converges to a critical point $p^\ast$ of $f \circ \triv$ on $\total{\MM}$ as
    \[
        d_{\total{M}}(p_t, p^\ast) \leq \frac{C}{t^r}
    \]
    for some constants $C, r > 0$, the sequence of points
    \begin{equation}\label{eq:proj_algorithm}
        x_t = \triv(p_t)
    \end{equation}
    converges to the critical point $x^\ast \defi \triv(p^\ast)$ of $f$ on $\MM$ at the same rates:
    \[
        d_M(x_t, x^\ast) \leq \frac{C}{t^r}.
    \]
\end{theorem}
\begin{proof}
    Since $p^\ast$ is a critical point of $f \circ \triv$ and $\triv$ is a submersion, $x^\ast$ is a critical point of $f$ (\Cref{prop:submersions_dont_create}). The result then follows from the fact that a Riemannian submersion is distance decreasing (\Cref{prop:properties_riemannian_submersions}).
\end{proof}

What this theorem says is that if we have convergence rates on $\pa{\total{\MM}, \gm}$ for a family of functions $\total{\mathcal{C}}_{\pa{\total{\MM}, \total{\gm}}}$, and the pullback by the Riemannian submersion $\triv$ respects that family of functions---if $f \in \mathcal{C}_{\pa{\MM, \gm}}$ then $f \circ \triv \in \total{\mathcal{C}}_{\pa{\total{\MM}, \total{\gm}}}$---then, we have the same convergence rates for $\mathcal{C}_{\pa{\MM, \gm}}$ by applying the algorithm to $f \circ \triv$ and using $\triv$ to project the iterates. The families of functions do not even need to be ``the same'' on both manifolds.

A tighter but more abstract way of saying this is that if we know how to prove convergence of an algorithm for the pullback of a family of functions $\triv^\ast\pa{\mathcal{C}_{\pa{\MM, \gm}}}$, we may pushforward the algorithm to solve problems in $\mathcal{C}_{\pa{\MM, \gm}}$.

We will now see how to implement this idea for the class of functions of $\alpha$-bounded Hessian functions on a manifold (\Cref{def:bounded_hessian}) in terms of functions that are $L$-Lipschitz and of $\alpha$-bounded Hessian. We therefore set
\begin{align*}
    \mathcal{C}_{\pa{\MM,\gm}}^{L, \alpha}
    &= \set{ f \in C^2(\MM)| \norm{\grad f} \leq L,\ \norm{\Hess f} \leq \alpha} \\ 
    \total{\mathcal{C}}_{\pa{\total{\MM},\total{\gm}}}^{\alpha'}
   &= \set{ f \in C^2(\total{\MM})| \norm{\Hess f} \leq \alpha'}
\end{align*}
where the bounds on the gradient and Hessian are to be regarded uniformly point-wise.

    In what follows, we will show that when $\deffun{\triv : \total{\MM} -> \MM;}$ is a Riemannian submersion with some further properties, we actually have the equality
    \[
    \triv^\ast(\mathcal{C}_{\pa{\MM,\gm}}^{L, \alpha}) =
    \total{\mathcal{C}}_{\pa{\total{\MM},\total{\gm}}}^{\alpha'}
    \]
    for a concrete value of $\alpha' > 0$ that just depends on $L$, $\alpha$, and $\triv$. We will explicitly bound the $\triv$-dependant quantity explicitly for some manifolds of interest.

    We will devote all of~\Cref{ch:second_order_bounds} to the study of convergence rates for functions in the class $\total{\mathcal{C}}_{\pa{\total{\MM},\total{\gm}}}^{\alpha'}$ in terms of the curvature of $(\total{\MM}, \total{\gm})$. This, together with~\Cref{thm:submersion_submetry} yields convergence rates on some homogeneous spaces in terms of their curvature.

    We start by recalling the usual bounds on the second order information on the pullback of a function. If a function $f$ is $L$-Lipschitz of $\alpha$-bounded Hessian and if we are have bounds for $\triv$ such as $\norm{\dif \triv} \leq L_\triv$ and $\norm{\conn\dif\triv}\leq \alpha_\triv$, we have that
    \[
        \norm{\conn\dif\pa{f\circ\triv}} \leq \alpha L^2_\triv + L \alpha_\triv.
    \]
    As such, all we have to do to implement the plan above is to give first and second order bounds when $\triv$ is a Riemannian submersion.

    We start by looking at a general submersion with totally geodesic fibres, like those presented throughout~\Cref{ch:geometry}. Bounds similar to these can be found in~\parencite{oneill1966fundamental,vilms1970totally}.
\begin{proposition}[First and second order bounds of a submersion]\label{prop:first_second_order_bounds_submersion}
    Let $\deffun{\triv : \total{\MM} -> \MM;}$ be a Riemannian submersion with totally geodesic fibres with Levi-Civita connections $\total{\conn}, \conn$ respectively. We have that
    \begin{gather}
        \norm{\dif \triv} \leq 1 \\
        \norm{\conn \dif \triv\vert_{T_p\total{\MM}}}
        =2\max_{\substack{X\in \HH_p, V \in \VV_p\total{\MM} \\ \norm{X}=\norm{V}=1}} \norm{\pa{\total{\conn}_X \tilde{V}}_{\HH_p}}\label{eq:norm_hessian_submersion}
    \end{gather}
    Here $Y_{\HH_p}$ denotes the horizontal part of each vector $Y\in T_p\total{\MM}$ and $\tilde{V}$ is any smooth local extension of $V \in \VV_p\MM$ to a vector field around $p$.
\end{proposition}
\begin{proof}
    The first bound is direct as for $X \in T\total{\MM}$
    \[
        \norm{\dif\triv(X)} = \norm{\dif\triv\pa{X_{\HH} + X_{\VV}}} = \norm{\dif\triv\pa{X_{\HH}}} = \norm{X_{\HH}} \leq \norm{X}.
    \]
    To prove the second formula, recall that for vector fields $X,Y$ on $\total{\MM}$ the Hessian of $\triv$ is defined as
    \begin{equation}\label{eq:hessian_submersion}
        \conn \dif \triv(X,Y) \defi \conn_{\dif \triv(X)}\pa{\dif \triv(Y)} - \dif\triv\pa{\total{\conn}_X Y}.
    \end{equation}
    We will split the computation into its horizontal and vertical components.

    If $V \in \VV\total{\MM}$ is a vertical vector, the first term of $\conn \dif \triv(V,V)$ in~\eqref{eq:hessian_submersion} is zero since $\triv$ is constant on the fibres so that $\dif \triv(V) = 0$. On the other hand, the hypothesis that the fibres are totally geodesic means that $\total{\conn}_V V \in \VV\total{\MM}$ (\Cref{prop:totally_geodesic_fibres}) so that $\conn \dif\triv(V,V) = 0$.

    If $X \in \HH$ is a horizontal vector field, let $\total{\gamma}$ be the horizontal geodesic with $\dot{\total{\gamma}}(0) = X$, and let $\gamma = \triv(\total{\gamma})$ be projection onto a geodesic on $\MM$. Then, evaluating~\eqref{eq:hessian_submersion} along $\total{\gamma}$ we have
    \[
        \conn \dif \triv\pa{\dot{\total{\gamma}},\dot{\total{\gamma}}} = \conn_{\dot{\gamma}}\dot{\gamma} - \dif\triv\pa{\total{\conn}_{\dot{\total{\gamma}}} \dot{\total{\gamma}}} = 0.
    \]

    If $V$ is a vertical vector field and $X \in \HH_p$, then
    \begin{equation}\label{eq:non_zero_part}
        \conn \dif \triv(X,V) =  - \dif\triv\pa{\total{\conn}_X V}
    \end{equation}
    and
    \[
        \norm{\conn \dif \triv(X,V)\vert_{T_p\total{\MM}}} = \norm{\pa{\total{\conn}_X V}_{\HH_p}}.
    \]

    Putting the horizontal and vertical case together, if $X \in T\total{\MM}$ is an arbitrary vector, we can decompose it into its horizontal and vertical part so that $\norm{\conn \dif \triv(X,X)} = 2\norm{\conn \dif \triv(X_{\HH},X_{\VV})}$. Using that $\Phi(X,Y) \defi \conn\dif\triv(X,Y)$ is a symmetric bilinear form, by polarisation, the triangle inequality and Cauchy--Schwarz
    \[
        4\norm{\Phi(u,v)} \leq \norm{\Phi \vert_{\operatorname{diag}}}\pa{\norm{u+v}^2 + \norm{u-v}^2} = 4\norm{\Phi \vert_{\operatorname{diag}}}
    \]
    where $\Phi\vert_{\operatorname{diag}}(u) \defi \Phi(u,u)$ so that
    \[
        \norm{\conn \dif \triv\vert_{T_p\total{\MM}}}
        =\max_{\substack{u,v\in T_p\total{\MM} \\ \norm{u}=\norm{v}=1}} \norm{\pa{\conn\dif \triv}_p(u,v)}
        =\max_{\substack{u\in T_p\total{\MM} \\ \norm{u}=1}} \norm{\pa{\conn\dif \triv}_p(u,u)}
        =2\max_{\substack{X\in \HH_p, V \in \VV_p\total{\MM} \\ \norm{X}=\norm{V}=1}} \norm{\pa{\conn\dif \triv}_p(X,V)}.\qedhere
    \]
\end{proof}

For a normal Riemannian homogeneous space (\Cref{def:normal_homogeneous_space}) these bounds can be simplified to an algebraic computation on the Lie algebra.
\begin{proposition}\label{prop:hessian_pi_homogeneous}
    If $(M,\gm)$ is a normal Riemannian homogeneous space with $\deffun{\triv : G -> \MM;}$ the canonical projection,
    \[
        \norm{\conn \dif \triv} = \max_{\substack{X\in\mlie, Y\in\hlie\\\norm{X}=\norm{Y}=1}}\norm{\liebrack{X,Y}}
    \]
    where $\hlie$ and $\mlie$ are as in~\Cref{def:reductive}
\end{proposition}
\begin{proof}
    We start by noting that, since $(M, \gm)$ is a normal Riemannian homogeneous space, $\triv$ is a Riemannian submersion with totally geodesic fibres (\Cref{thm:lift_metric_reductive}), so that we are in the setting of~\Cref{prop:first_second_order_bounds_submersion}.

    In order to compute the norm of the Hessian of $\triv$ at an arbitrary point in $G$, given that the only non-zero part is that given by~\Cref{eq:non_zero_part}, it is enough to compute the derivative of a vertical vector field with $V_p = V \in T_p\MM$ in the direction of a horizontal vector field. Furthermore, as we are working with a Riemannian homogeneous space, the norm of the Hessian of $\triv$ will not depend on the point. For this reason, it will be enough to compute it at the identity $e \in G$. But at the identity we can use that if $V$ is a vertical vector, $V \in \hlie$ and $\lv{V}$ is a vertical vector field. The result follows from the formula for the connection on a normal Riemannian homogeneous space (\Cref{prop:normal_riemannian_are_homogeneous}) and the fact that due to the reductive condition $[\mlie, \hlie]\subset\mlie$ we do not need to project onto $\mlie$.
\end{proof}

Now, we may apply these results to the Stiefel manifold and the Grassmannian since, as described in~\Cref{sec:matrix_groups}, they are both normal Riemannian homogeneous spaces.
\begin{proposition}
    Let $(\MM, \gm)$ be either the Stiefel manifold $\St{n,k}$ or the Grassmannian manifold $\Gr{n,k}$ as a bundle $\deffun{\pi : \SO{n} -> \MM;}$ with the metric induced by the scalar product $\scalar{A,B} = \frac{1}{2}\tr\pa{\trans{A}B}$. Consider a function $\deffun{f : \MM -> \RR;}$ that is $L$-Lipschitz and of $\alpha$-bounded Hessian. We then have that $f \circ \pi$ is of $(\sqrt{2}L+\alpha)$-bounded Hessian on $\SO{n}$.
\end{proposition}
\begin{proof}
    The chain rule and Cauchy--Schwarz give the bound
    \[
        \norm{\Hess\pa{f \circ \pi}} \leq \norm{\Hess f}\norm{\dif \pi}^2 + \norm{\dif f}\norm{\conn\dif\pi} \leq \alpha + L\max_{\substack{X\in\mlie, Y\in\hlie\\\norm{X}=\norm{Y}=1}}\norm{[X,Y]}
    \]
    where we have used~\Cref{prop:first_second_order_bounds_submersion,prop:hessian_pi_homogeneous} and the bounds on $f$.

    Since the total space $\SO{n}$ of the Stiefel manifold and the Grassmannian is a matrix manifold, the bracket on its Lie algebra is just the commutator of skew-symmetric matrices. As shown in~\cite[Lemma $2.5$]{ge2014ddvv} for the Frobenius norm $\frobnorm{A} = \sqrt{\tr\pa{\trans{A}A}}$ we have the bound
    \[
        \frobnorm{[A,B]} \leq \frobnorm{A}\frobnorm{B}\mathrlap{\qquad\forall A,B \in \solie{n}.}
    \]
    Using that $\norm{A} = \frac{1}{\sqrt{2}}\frobnorm{A}$ we get that
    \[
        \norm{[A,B]} \leq \sqrt{2}\norm{A}\norm{B}\mathrlap{\qquad\forall A,B \in \solie{n}}
    \]
    and the result follows.
\end{proof}

\begin{remark}
    The result in~\cite{ge2014ddvv} also gives the equality cases for this bound of the norm. It is direct to see that no two matrices $[m,h]$ with $m \in \mlie, h \in \hlie$ achieve this bound. On the other hand, consider the case $\Gr{8,4}$ and the matrices
    \[
        A =
        \begin{pmatrix}
            0 & -\lambda & 0 & 0 \\
            \lambda & 0 & 0 & 0 \\
            0 & 0 & 0 & -\lambda \\
            0 & 0 & \lambda & 0 \\
        \end{pmatrix}
        \qquad
        B =
        \begin{pmatrix}
            0 & 0 & 0 & \lambda \\
            0 & 0 & \lambda & 0 \\
            0 & -\lambda & 0 & 0 \\
            -\lambda & 0 & 0 & 0 \\
        \end{pmatrix}
    \]
    so that
    \[
    m =
    \begin{pmatrix}
        0_{4,4} & -\trans{A} \\
        A & 0_{4, 4}
    \end{pmatrix}\in\mlie
    \qquad
    h =
    \begin{pmatrix}
        0_{4,4} & 0_{4, 4}\\
        0_{4,4} & B
    \end{pmatrix}\in\hlie
    \]
    then we see that
    \[
        \norm{\liebrack{m,h}} = 2\lambda^2 \leq \norm{m}\norm{h} = 4\lambda^2
    \]
    so the bound in the Hessian of $\pi$ is pessimistic by at most a factor of $2$. We have the same upper-bound for the constants for $\St{n,k}$ and $\Gr{n,k}$ for $n\geq 8$ and $k \geq 4$, via the totally geodesic immersions of $\St{8,4}$ and $\Gr{8,4}$ into the larger ones. Similar results can be obtained whenever $n < 8$ or $k < 4$.
\end{remark}

\section{Static Trivialisations}\label{sec:change_metric}
Let us now explore the static trivialisation framework from \Cref{eq:pullback_problem} for the setting when $\dim\pa{\total{\MM}} = \dim\pa{\MM}$. In this case, we will be forward a metric from $\total{\MM}$ along $\triv$ into a new metric on $\MM$. Here we will present the framework in the fully general case. For a---perhaps illuminating---example in the setting of optimisation with orthogonal constraints see~\Cref{sec:exp_param}.

\subsection{Diffeomorphisms}
The simplest case comes when we have access to a diffeomorphism between the spaces $\deffun{\triv : \total{\MM} -> \MM;}$. If we have a metric $\total{\gm}$ on $\total{\MM}$, we may then form the pushforward metric $\triv_\ast\pa{\total{\gm}}$ that makes $\triv$ into an isometry. If also have a retraction $\total{r}$ on $\total{\MM}$, and a metric $\gm$ on $\MM$, we may define a step map on $\MM$ as
\[
    \deffun{\triv_\ast\pa{\total{r}} : T\MM -> \MM ; (x,v) -> \pa{\triv \circ \total{r}_{\triv^{-1}(x)} \circ \pa{\dif\triv}_{\triv^{-1}(x)}^{\ast}}\pa{v}}.
\]
This construction is essentially the same as that described in~\Cref{def:retraction_step}. The difference is that, in this case, given that we have a diffeomorphism, the fibre is trivial and the map does not depend on the lifted point. On the other hand, \Cref{prop:retraction_step} still applies, and we have that this will be a retraction if and only if $\triv$ is a local isometry, that is, if $\triv_\ast\pa{\total{\gm}} = \gm$.

The setting of having a diffeomorphism is not particularly interesting, as the existence of diffeomorphisms between manifolds is too strong a topological constraint. In the following section, we will see how the Riemannian exponential map is a submersion for some manifolds and, in general, it is a diffeomorphism in a \emph{large enough} set. This allows us to safely pull back problems from a manifold with a potentially difficult topology to a Euclidean space.

\subsection{The Riemannian exponential}\label{sec:riemannian_exponential}
Consider the pullback problem along the Riemannian exponential
\[
    \min_{v \in T_p\MM} f\pa{\exp_p\pa{v}}.
\]
We will assume throughout this section that we have a complete metric $\gm$ on a connected manifold $\MM$. In this setting, $\exp_p$ is surjective and
\[
    \min_{x\in \MM} f(x) = \min_{v \in T_p\MM} f(\exp_p(v)).
\]

Note that problem~\eqref{eq:pullback_problem} does not come equipped a priori with a choice of metric. As such, we have the freedom to select one that makes our life easier. Complete metrics exist for any differentiable manifold~\parencite{nomizu1961existence} but the geometry of a complete metric may not be simple in our manifold of choice. In particular, solving the second order non-linear differential equation to compute the Riemannian exponential map may be rather challenging if the metric does not have enough symmetries. Luckily, as we saw in~\Cref{sec:matrix_groups} most of the manifolds that we are interested on in optimisation are rather amenable in this aspect, with the notable exception of fixed rank matrices (\Cref{sec:fixed_rank}).

    \subsubsection{Hadamard manifolds}\label{sec:hadamard}
    A reasonable question could be: When is the exponential map a diffeomorphism, so that we can pushforward the flat metric on the tangent space onto a metric onto the manifold? Hadamard manifolds have this desirable property.
    \begin{definition}
        A connected and complete Riemannian manifold $(\MM, \gm)$ is \textbf{Hadamard} if its sectional curvature is everywhere non-positive.
    \end{definition}

    The importance of Hadamard manifolds comes from the Cartan--Hadamard theorem.
    \begin{theorem}[Cartan--Hadamard theorem]\label{thm:cartan_hadamard}
        In a Hadamard manifold $(\MM, \gm)$, for every $p \in \MM$, the map $\exp_p$ is a submersion. In particular, if the manifold is simply connected, $\exp_p$ is a diffeomorphism.
    \end{theorem}

    We provide a proof of this classical theorem as a corollary of the expository results in~\Cref{sec:first_order_bounds}. Examples of manifolds in this family that are of interest in optimisation are the hyperbolic space or the space of symmetric positive definite matrices with the left-invariant coming from its representation as a symmetric space $\Symp{n} \iso \GLp{n} / \SO{n}$. After a choice of frame---a linear isomorphism between $\RR^n$ and $T_p\MM$---we may effectively work on $\MM$ as if it was $\RR^n$ provided that we know how to compute the adjoint of the differential of the exponential map. This is something that we will explore in~\Cref{sec:adjoint_exponential}. For Hadamard manifolds, we may pullback the minimisation problem from a curved space into a Euclidean space on which we can use more sophisticated optimisation methods, such as momentum or adaptive methods.

    \begin{remark}[Numerical stability of the Riemannian exponential map on negatively curved manifolds]
    For manifolds that are strictly negatively curved, the exponential map grows exponentially fast in the sense described in Rauch's theorem (\Cref{thm:rauch}). For this reason, when using the exponential map to pullback a problem to a Euclidean space on a Hadamard manifold, one often encounters numerical stability issues which might have to be approached separately via a careful numerical implementation of the algorithms (see \eg,~\cite{yu2019numerically}).
    \end{remark}

    \begin{remark}[Other manifolds for which the Riemannian exponential is a diffeomorphism]
        Hadamard manifolds are a large family of manifolds on which the exponential map is a diffeomorphism but they are not the only ones. If the exponential map is a diffeomorphism at a point, the manifold may still have parts of positive curvature---\eg, at the vertex of a paraboloid. In fact, there are examples of manifolds that are not Hadamard where the exponential map at any point is a covering map~\parencite{gulliver1975variety}.
    \end{remark}

    \subsubsection{The general case}
    For a general Riemannian manifold, the exponential may not be a covering map due to the existence of conjugate points. Although we will present an introduction to conjugate points in~\Cref{sec:jacobi_fields}, we recall here one possible definition.

    \begin{definition}
        A point \textbf{$q$ is conjugate to a point $p$} if the map $\exp_p$ is not regular at $q$, \ie, its differential at a vector $v \in \exp^{-1}_p\pa{\set{q}} \subset T_p\MM$ is not full rank.
    \end{definition}

    Even though the Riemannian exponential is not in general a submersion, it is a local diffeomorphism around zero. The Gauss' lemma gives that the differential of the exponential map at zero is the identity. As such, by the inverse function theorem, $\exp_p$ is a diffeomorphism on a neighbourhood of $0 \in T_p\MM$. An obvious question now is: How large is this neighbourhood?

    This is of course a well studied question in Riemannian geometry. We summarise here some relevant results from the literature.
    \begin{definition}\label{def:cut_locus}
        Let $(\MM, \gm)$ be a complete Riemannian manifold. We define the \textbf{segment domain}, and its interior as
        \begin{align*}
            \seg{p} &\defi \set{v \in T_p\MM | \deffun{\exp_p(tv) : [0,1] -> \MM;} \text{ is length minimising}}\\
            \segint{p} &\defi \set{tv \in T_p\MM | t \in [0,1), v \in \seg{p}}
        \end{align*}
        We also define the \textbf{tangent cut locus at $p$} and the \textbf{cut locus at $p$} as
        \begin{gather*}
            \tcut{p} \defi \seg{p} \backslash \segint{p}\\
            \cut{p} \defi \exp_p\pa{\tcut{p}}
        \end{gather*}
    \end{definition}
    By Hopf--Rinow, $\exp_p\pa{\seg{p}} = \MM$. More generally, the segment domain and the cut locus have the following properties.
    \begin{proposition}[Properties of the cut locus]\label{prop:properties_cut_locus}
        The segment domain has the following properties:
        \begin{enumerate}
            \item $\segint{p}$ is the interior of  $\seg{p}$. In particular, $\segint{p}$ is open
            \item $\exp_p$ is a diffeomorphism on $\segint{p}$
            \item If $v \in \tcut{p}$, then either $\exp_p$ fails to be injective at $x = \exp_p(v)$, or its differential fails to be surjective, or both
        \end{enumerate}
    \end{proposition}
    \begin{proof}
     See for example \parencite[Section $5.7.3$]{petersen2016riemannian}.
    \end{proof}
    In plain words, $\segint{p}$ is a maximal star-shaped open neighbourhood of zero on $T_p\MM$ on which the exponential map is a diffeomorphism.

    Due to the third property of the cut locus, we see that points in the cut locus are those that are either joined by two geodesics of the same length starting at $p$ (cut points) or conjugate points of $p$. We will look at how to bound the first occurrences of the latter and its applications to convex geometry on manifolds in~\Cref{sec:law_of_cosines}.

    We may now ask what is the topology of the cut locus. By~\Cref{prop:properties_cut_locus}, it is a closed set in $\MM$, but we do not know how big is it. The cut locus is a remarkably slippery object of study given that, in general, it is not differentiable. For example, for surfaces, the cut locus of a point has structure of a local tree, in the sense that every point in the cut locus has a neighbourhood $W$ such that the connected component of $W \cap \cut{p}$ on which $z$ lies is a tree~\parencite{myers1935connections}. Here by tree we mean that any two points can be joined by a unique injective curve. This tree, on the other hand, may be infinite of a fractal-like shape. This was later generalised in~\parencite{buchner1977simplicial}, where it was proved that locally, the cut locus of an $n$-dimensional manifold is homeomorphic to an $(n-1)$-dimensional simplicial complex.

    Regarding the size of the cut locus, we have the following theorem by Itoh and Tanaka.
    \begin{theorem}[\cite{itoh1998dimension}]\label{thm:dimension_cut_locus}
        Let $\MM$ be a connected and complete Riemannian manifold of dimension $n$. For a point $p \in \MM$ the Hausdorff dimension of $\tcut{p}$ is either $0$ or $n-1$, and the Hausdorff dimension of $\cut{p}$ is an integer strictly smaller than $n$.
    \end{theorem}

    Let us put now this result in the language of measures, which is more familiar to the optimisation community.
\begin{corollary}\label{corol:measure_zero}
    $\tcut{p}$ has Lebesgue measure zero on $T_p\MM$ and $\cut{p}$ has measure zero with respect to the Borel measure on $\MM$. If $\MM$ is orientable, $\cut{p}$ has measure zero with respect to the measure induced by the Riemannian volume form.
\end{corollary}

We can argue that, although the cut locus can introduce problems in practice, the problematic set is not too large in a measure-theoretic sense. On the other hand, even though the cut locus is of measure zero, it might have quite a wild topology. For $\MM$ compact, the cut locus is connected as $\MM \backslash \set{p}$ is a deformation retract onto $\cut{p}$ given by flowing along geodesics emanating from $p$. On the other hand, in the non-compact case, the cut locus may have countably many components~\parencite{chavel1970class}. Luckily, as we have mentioned at the beginning of this section, the manifolds on which we are interested in optimisation are particularly well-behaved, so this is never a problem in practice, as we will see empirically in~\Cref{sec:experiments}.
It is also worth noting that, when setting up the static trivialisation framework, we will often sample a point $p \in \MM$ according to some measure and then start at the point $v_0 = 0 \in T_p\MM$. For this reason, we will start the optimisation at the farthest point from the conjugate locus, which is advantageous in practice.

    \subsubsection{The adjoint of the Riemannian exponential}\label{sec:adjoint_exponential}
In order to implement the pullbacks along the exponential map in first order optimisation, it is necessary to compute the adjoint of the differential of the exponential map. For general Riemannian manifolds, this can be computed in terms of the radial vector field defined by geodesics. We will not use this result, but we believe that it may be of interest.
\begin{proposition}[\cite{mcalpin1965infinite}]\label{prop:adjoint_riemannian_exponential}
    Let $(\MM, \gm)$ be a Riemannian manifold and let $(p,v) \in T\MM$ such that $\gamma(t) = \exp_p\pa{tv}$ does not have conjugate points for $t \in [0,1]$ then
    \[
        \pa{\dif \exp_p}^\ast_{v} = \pa{\dif \exp_{\gamma(1)}}_{-\dot{\gamma}(1)}
    \]
\end{proposition}

A proof of this result can be found in~\parencite[IX, Lemma 3.4]{lang1999fundamentals}.

Let us now compute the adjoint of the differential of the exponential on naturally reductive homogeneous spaces. For these spaces, we may simplify its computation to the computation of the adjoint of the Lie exponential on $G$.
\begin{proposition}\label{prop:adjoint_differential_naturally_reductive}
    Let $(G/H, \gm)$ be a naturally reductive homogeneous space and fix a point and a vector $(x,v) \in T\pa{G/H}$. Let $\deffun{\zeta_g = \pa{\dif \pi}_g \circ \pa{\dif L_g}_e\vert_\mlie : \mlie ->  T_{gH}(G/H);}$ be the linear isometry given by the naturally reductive structure. Then, letting $v_\mlie \defi \zeta_g^{-1}(v)$ for a $g \in \pi^{-1}(x)$ and denoting $\tilde{g} \defi \expm\pa{v_\mlie}$,
    \begin{align*}
        \pa{\dif \exp_p}^\ast_v &= \zeta_g \circ \pa{\dif\expm}^\ast_{v_\mlie} \circ \pa{\dif L_g}^\ast_{\tilde{g}} \circ \pa{\dif \pi}^\ast_{g\tilde{g}}
    \end{align*}
\end{proposition}
\begin{proof}
    The formula follows by differentiating the expression for the Riemannian exponential in~\Cref{prop:geodesics_naturally_reductive}
    \[
        \exp_x(v) = \pa{\pi \circ L_g \circ \expm \circ \zeta^\ast_g}(v)
    \]
    and taking adjoints.
\end{proof}

When $G$ is a matrix Lie group, we may give a more computationally concrete formula for the adjoint of the Riemannian exponential. In this case, $L_g$ is just left multiplication, which makes computing its adjoint with respect to the canonical basis direct. To compute the adjoint for the exponential, we use the fact that in this case, for matrix Lie groups, the Lie exponential is just the exponential of matrices (\Cref{prop:lie_exponential_is_exponential_matrices}) and, in particular, an analytic map. For an analytic matrix map, we have a particularly simple formula for the adjoint. As we are working on $\M{n}$, we will identify $\dif L_X$ with $L_X$ to simplify the notation.\footnote{After obtaining this result, we thought that this should be folklore in some areas like functional analysis or numerical analysis. In fact, this result can be found without proof in~\parencite[p.66]{higham2008functions}.}

\begin{theorem}\label{thm:analytic}
    Consider an analytic matrix function
    \[
        \deffun{\psi : \M{n} -> \M{n}; X -> \sum_{n=0}^\infty \frac{a_n}{n!}X^n}
    \]

    We then have that, for the canonical inner product on $\M{n}$
    \[
        \pa{\dif \psi}^\ast_X = \pa{\dif \psi}_{\trans{X}}\mathrlap{\qquad X \in \M{n}}.
    \]
\end{theorem}
\begin{proof}
    We can compute the differential of $\psi$ as
    \[
        \pa{\dif \psi}_X(E) = \sum_{n=0}^\infty \pa[\Big]{\frac{a_n}{n!} \sum_{i=0}^n X^i E X^{n-i}}.
    \]
    By linearity, it is enough to compute the adjoint of functions of the form $X \mapsto X^i E X^{n-i}$.

    Observe that the adjoint of the linear map given by left multiplication $L_A(X) = AX$ is exactly $L_{\trans{A}}$
    \[
        \scalar{L_A(X), Y} \defi \tr\pa{\trans{\pa{AX}}Y} = \tr\pa{\trans{X}\trans{A}Y} = \scalar{X, L_{\trans{A}}(Y)}.
    \]
    In the case of right multiplication, we also get $R^\ast_A = R_{\trans{A}}$.

    Finally, we just have to apply this formula to the functions $L_{X^i}(E) = X^i E$ and $R_{X^{n-i}}(E) = EX^{n-i}$, and noting that $X \mapsto X^i E X^{n-i} = L_{X^i}(R_{X^{n-i}}(E))$ we get the result.
\end{proof}

\begin{remark}[Complex setting]
    The generalisation of this result to complex functions is direct, just computing the differential of the analytic function with conjugate coefficients in its Taylor series. In this case, one can interpret this theorem by saying that ``the adjoint of the differential is the differential of the conjugate at the adjoint'', noting the two different meanings of the word `` adjoint'' in the sentence.

    One can also formulate the theorem for a holomorphic function defined just on an open subset $U \subset \CC$, and define the function on matrices on the set of matrices such that their spectrum is contained in $U$, hence making sense also of functions like $\log(X)$.
\end{remark}

Putting this together with~\Cref{prop:adjoint_differential_naturally_reductive}, we are able to compute the adjoint of the Riemannian map for manifolds such as the Stiefel manifold, the Grassmannian.

\begin{theorem}\label{thm:adjoint_exp}
    Let $(G/H, \gm)$ be a naturally reductive homogeneous space with $G$ a subgroup of $\UU{n}$ (resp.\ $\SO{n}$) and $\gm$ the quotient metric induced from the canonical metric on $\CC^{n \times n}$ (resp.\ $\M{n}$). With the notation in~\Cref{prop:adjoint_differential_naturally_reductive},
    \[
        \pa{\dif \exp_p}^\ast_v = \zeta_g \circ \pa{\dif\expm}_{\conj{v_\mlie}} \circ \pa{\dif L_g}^\ast_{\tilde{g}} \circ \pa{\dif \pi}^\ast_{\tilde{g}}
    \]
    where $\conj{-}$ denotes the conjugate transpose.
\end{theorem}

In~\Cref{sec:approximations_exponential}, we will explore how to numerically approximate both the exponential and its adjoint efficiently on a \gpu.

The conjugate locus on naturally reductive homogeneous spaces has a particularly simple structure, given that it is just the projection as the conjugate locus of the Lie group. If this is a matrix Lie group then the conjugate locus can be computed by algebraic means, as this is set of points on which the differential of the exponential of matrices fails to be surjective. Using the formula for the differential of the exponential and looking at its eigenvalues, it is easy to show the following.

\begin{theorem}\label{prop:domain_of_definition}
    Let $G$ be a matrix Lie group. The exponential of matrices is analytic, with analytic inverse on a bounded open neighbourhood of the origin given by
    \[
        U = \set{A \in \glie | \abs{\Im\pa{\lambda_i(A)}} < \pi}.
    \]
\end{theorem}
\begin{proof}
    See for example~\parencite[Proposition 7']{rossmann2002lie}.
\end{proof}

In particular, this gives us a neighbourhood of radius $\pi$---or in general a whole band---around the identity where the exponential is a diffeomorphism.

\begin{remark}[Symmetric spaces]
    For compact groups with a bi-invariant metric, where at the identity the Lie exponential and the Riemannian exponential coincide, this gives a tight bound for the conjugate radius. These computations may be extended to symmetric spaces by means of Jacobi fields. On a symmetric space $(G/H, \gm)$, as the curvature tensor is covariantly constant, the Jacobi equation---\ie, the equation that the differential of the exponential satisfies (\cf, \Cref{prop:jacobi_equation})---has constant coefficients. As such, it is possible to express it in terms of the Lie bracket on $G$. When $G$ is compact, using the roots of $G$ and $H$, it is possible to compute the largest radius on which the Riemannian exponential is a submersion. The details can be found in~\parencite{rauch1966geodesics}
\end{remark}

\begin{remark}[A note on the Lie exponential]
    The Lie exponential has the advantage over the general Riemannian exponential that it is straightforward to compute for matrix Lie groups, while the Riemannian exponential requires looking for a suitable metric on which the geodesics are easy to compute. On the other hand, once we have a suitable Riemannian exponential, it is possible to give convergence results with respect to the different metrics involved in the Riemannian submersion $\deffun{\pi : G -> \MM;}$. Furthermore, the metrics that we work with are often complete, so if $\MM$ is connected, the Riemannian exponential is surjective, and it is a diffeomorphism onto all the manifold but a set of measure zero---the conjugate locus (\Cref{corol:measure_zero}). On the other hand, for some groups such as $\GL{n, \RR}$ or $\SL{n, \RR}$ the Lie exponential is not surjective, which is problematic.
\end{remark}

To be able to use maps that are not surjective, and to mitigate the problem of converging to a critical point of the parametrisation, we introduce the framework of dynamic trivialisations.

\section{Dynamic Trivialisations}\label{sec:dyn_triv}
In this section, we will exploit vector bundle structure of the tangent bundle to alleviate the problems of lack of surjectivity or of convergence to a critical point of a static trivialisation. We will assume, as it is often the case, that we have access to a retraction $\deffun{\retr : T\MM -> \MM;}$.

In this setting, to approximate a minimiser of $f$, we may choose a point $p \in \MM$ and optimise on a fixed tangent space $T_p\MM$, performing gradient descent on the function $f \circ \retr_p$. We will call $p$ the \textbf{basepoint}. If we simply fix the basepoint $p$, this is exactly the setting of static trivialisations, where we just pulled back the problem by a fixed map. On the other hand, it may be the case that either $\deffun{\retr_p : T_p\MM -> \MM;}$ is not surjective or that it has critical points. In either case, given that we do not have access to a map but to a whole family of maps indexed by $p \in \MM$, we could change the basepoint if we detect that we are converging to a critical point of $\retr_p$ or if we see that we are not exploring all the manifold.

To formalise this, let $p_0 \in \MM$ be some starting point. Consider the iterates in $T_{p_0}\MM$
\[
    v_{0, k+1} = v_{0, k} - \eta_{0,k}\grad\pa{f \circ \retr_{p_0}}(v_{0, k})
\]
with $v_{0,0} = 0$. Then, if after $K$ iterations we observe that we are too close from a critical point of $\retr_{p_0}$, or we consider that we are not exploring a part of $\MM$ due to $\retr_{p_0}$ not being surjective, then we take $p_1 = \retr_{p_0}(v_{0, K})$, set $v_{1,0} = 0$ and continue the optimisation process on $T_{p_1}\MM$.

These ideas can naturally be translated into~\Cref{alg:dyn_triv_bundle} given a boolean stopping rule $\code{stop}$. We call it the \textbf{dynamic trivialisations framework}. The term ``dynamic'' comes from the fact that we choose an appropriate basepoint dynamically, according to the stopping rule $\code{stop}$ into which we can embed some prior knowledge about $\MM$ and structure of the problem.
\begin{algorithm}[H]
    \caption{Dynamic trivialisation framework}%
    \label{alg:dyn_triv_bundle}
    \begin{algorithmic}[1]
        \Require A starting point $p_0 \in \MM$, a boolean statement $\code{stop}$, a map $\deffun{\retr : T\MM -> \MM;}$, a function $\deffun{f : \MM -> \RR;}$
        \For {$i = 0, \ldots$}
            \State $v_{i, 0} = 0$
            \For {$k = 1, \dots$}
                \State $v_{i, k} = v_{i, k-1} - \eta_{i,k-1}\grad \pa{f \circ \retr_{p_i}} \pa{v_{i,k-1}}$\Comment{Optimise on $T_{p_i}\MM$}
                \If {\code{stop}}\Comment{Stopping condition}
                    \Break
                \EndIf
            \EndFor
            \State $p_{i+1} = \retr_{p_i}(v_{i,k})$\Comment{Update pullback point}
        \EndFor
    \end{algorithmic}
\end{algorithm}

The procedure described in~\Cref{alg:dyn_triv_bundle} accounts for changing the metric on $\MM$ every time the stopping rule fires. An example of a stopping rule to detect that we are converging to a critical point of $\phi$ with a stopping rule of the form
\[
    \code{stop} \equiv \frac{\norm{\grad\pa{f \circ \retr_{p_i}}\pa{v_{i, k}}}}{\norm{\grad f\pa{\retr_{p_i}\pa{v_{i,k}}}}} < \epsilon.
\]

\begin{figure*}[!tbp]
\centering
  \begin{minipage}[b]{0.5\textwidth}
      \centering
      \includegraphics[width=.696\columnwidth]{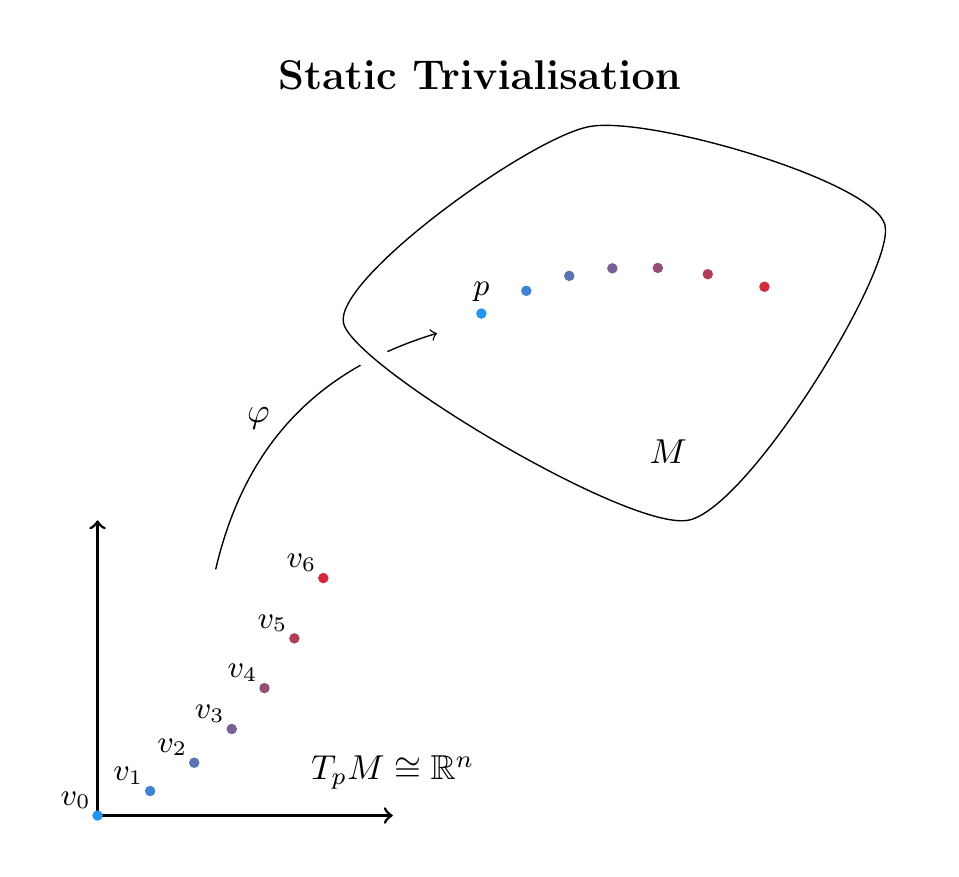}
  \end{minipage}%
  \begin{minipage}[b]{0.5\textwidth}
      \centering
      \includegraphics[width=\columnwidth]{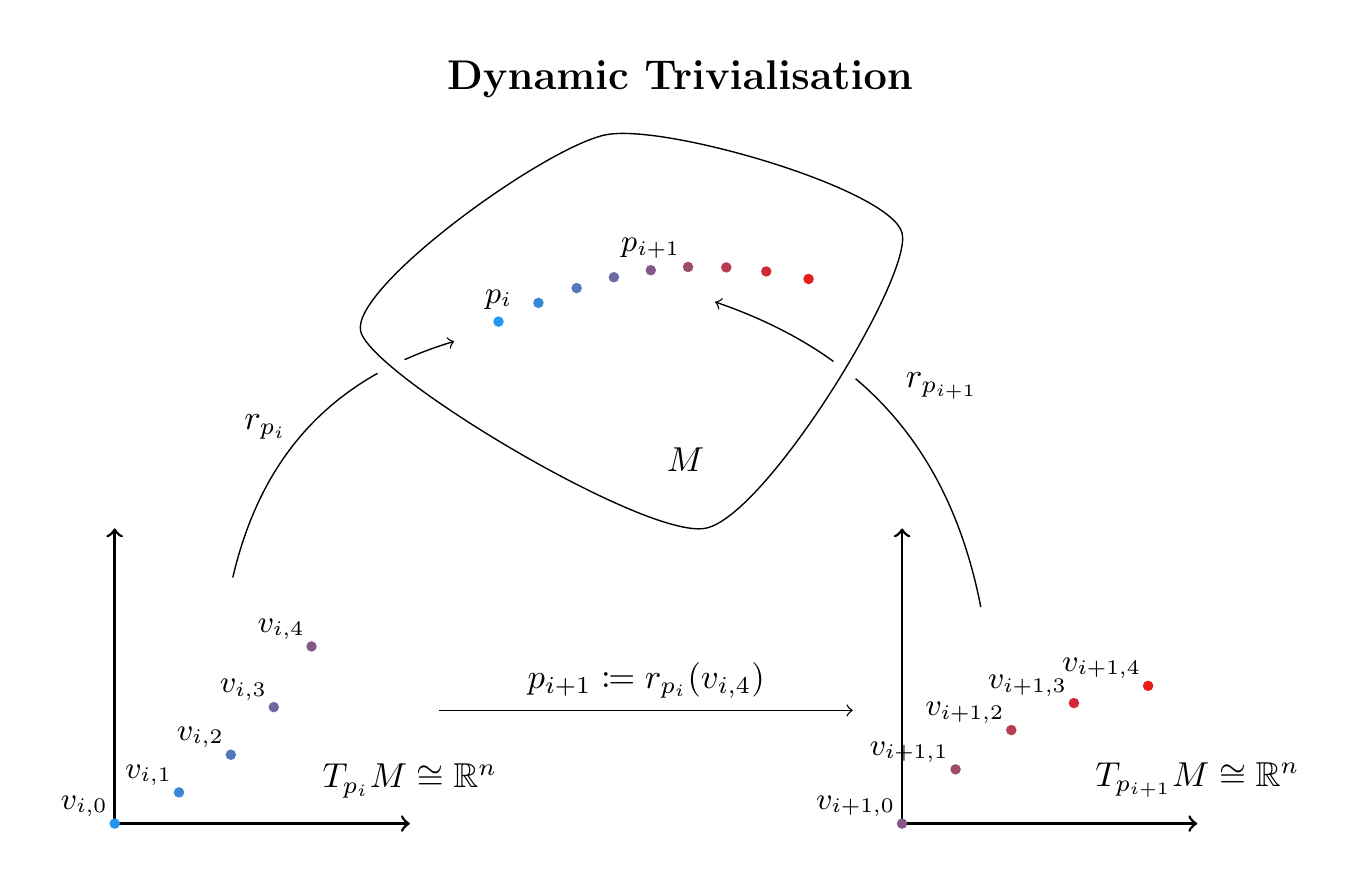}
  \end{minipage}
  \caption{Example of the pullback method vs.\ the dynamic trivialisation procedure changing the basepoint after $4$ steps.}%
  \label{fig:optimisation}
\end{figure*}

This algorithm has two interesting limiting cases.

\paragraph{Generalisation of pullbacks.}
If $\code{stop}\equiv \code{False}$, that is, if we never change the basis $p_0$, it reduces to the static trivialisation framework with the map $\retr_{p_0}$.

\paragraph{Generalisation of Riemannian gradient descent.}
In the case $\code{stop} \equiv \code{True}$, we are changing the basis of the trivialisation on every step. In this case, this method recovers exactly Riemannian gradient descent using $\retr$ as a retraction. To see this, just note that by the chain rule and the definition of a retraction
\[
    \dif \pa{f \circ \retr_{p_i}}_0
    = \pa{\dif f}_{\retr_{p_i}(0)} \circ \pa{\dif \retr_{p_i}}_0
    = \pa{\dif f}_{p_i}.
\]
From this it follows that
\[
    \grad \pa{f \circ \retr_{p_i}}(0) = \grad f \pa{p_i}
\]
so the update rule simplifies for a learning rate $\eta_{i,0} = \eta_i > 0$ can be rewritten as
\[
    v_{i, 1} = -\eta_i\grad f(p_i) \qquad p_{i+1} = \retr_{p_i}(-\eta_i\grad f(p_i))
\]
and $p_{i+1}$ are exactly the iterates given by doing Riemannian gradient descent using the retraction $\retr$.

For this reason, we can see dynamic trivialisations as an interpolation between the pullback method using $\retr_{p_0}$ and Riemannian gradient descent along $\retr$.

\begin{remark}[Retractions vs.\ more general maps]
    We ask for the map $\deffun{\retr : T\MM -> \MM;}$ to be a retraction for simplicity, but this is not strictly necessary. If we do not have that $\retr_p(0) = p$ then, when changing from $\retr_{p_i}$ to $\retr_{p_{i+1}}$ we would just have to take the vector on $T_{p_{i+1}}\MM$ that is mapped through $\retr_{p_{i+1}}$ to the point at which we are at the moment. Note that it is not necessary for $\pa{\dif \retr_p}_0$ to be the identity either, but simply to be full rank so that it defines a new metric locally. In practice, it would also be necessary for the neighbourhood on which the retraction is non-singular to be large enough, as it is the case with the Riemannian exponential. For other retractions, this is a property that should be studied case by case.
\end{remark}

Dynamic trivialisations have several practical advantages over \rgd.

\paragraph{Errors do not accumulate}
Consider usual Riemannian descent on $\SO{n}$ with a bi-invariant metric along the Riemannian exponential map (\cf, \Cref{sec:stiefel_manifold}). Unrolling the update rule, we get that after $t$ steps, we will be at the point
\[
    x_t = x_0\expm\pa{-\eta\grad f(x_0)} \cdot \ldots \cdot \expm\pa{-\eta\grad f(x_{t-1})}.
\]
This computation is exact in theory, but in practice, one has to compute the exponential of matrices at every step, which will lead to some numerical error in every step.  In particular, every matrix $\expm\pa{-\eta\grad f(x_i)}$, when represented in memory, will not be orthogonal but just \emph{very close} to orthogonal. These errors will accumulate, which means that the matrix $x_t$ will deviate from $\SO{n}$ as $t$ increases. This is problematic in practice since, as we will showcase in the experiment section~\Cref{sec:experiments}, optimisation on manifolds is often used in practice to regularise very badly conditioned optimisation problems. In these problems, the Lipschitz constant of the function being optimised is rather large, and projecting back to the manifold often affects very negatively the convergence of the algorithm.

Consider now the static trivialisation framework with the Riemannian exponential on a complete and connected Riemannian manifold. After choosing a frame, we are effectively pulling back the problem to $\RR^n$---\eg, to $\Skew{n}$ in the case of $\SO{n}$. As this is a vector space, its elements can be represented in memory exactly by the use of a frame $\deffun{\zeta : T_p\MM -> \RR^n;}$. It is for this reason that the static trivialisation framework does not present this accumulation of errors. On the other hand, by construction, if we have a surjective map from a Euclidean space onto a manifold that is not simply connected, it will necessarily have critical points, as it was the case with the conjugate locus for the Riemannian exponential map.

Dynamic trivialisations allow choosing a sweet-spot between these two situations depending on the problem. For example, consider that we know a priori the conjugate radius of the manifold in hand, which is often the case. With this information, we may run a coarse estimate based on the norm of the gradients so far and the step size of whether the current iterate may be too close to a singular point of the exponential map. If this cheap check passes, we may then run an actual check on the manifold to see if the current point is on the conjugate locus---on $\SO{n}$ this would account for doing an \svd{} of a matrix and checking the pairwise distance between eigenvalues values. If we are then $\epsilon$-close to the conjugate locus of $p$, we may then change the basepoint and so on.

The flexibility of the framework allows for the quantities involved in the estimation to be tuned in an online manner, where one may try to estimate some of the hyperparameters involved in the stopping rule according to the change of the norm of the gradient when one changes the point, or even second order information.

\paragraph{Gradient computations}
It is very common to find a section in papers on optimisation on manifolds where the equation for the Riemannian gradient on the manifold is derived for the manifold considered in the paper. These computations are always standard, and it often accounts for finding a projection from some matrix space onto a representation of the tangent space of the manifold and computing its adjoint. The computation of the adjoint of this linear map is purely mechanic, but it is often rather cumbersome, even more so in the case of the Riemannian Hessian.

The static and dynamic trivialisation frameworks implement this via composition of functions, after choosing a suitable frame $\deffun{\zeta : T_p\MM -> \RR^n;}$ on the tangent space. After this choice, one gets the Riemannian gradient without any extra effort, as the adjoint method is automatically computed by the autodiff engine, hence saving a considerable amount of effort. We will show some concrete examples of this in~\Cref{sec:exp_param}.

\paragraph{Modularity by design}
Another benefit of this framework in this context is that it takes full advantage of the modularity of autodiff frameworks. Consider that we want use the Lie exponential to perform optimisation on a Lie subgroup of a Lie group on which we have already implemented the Lie exponential. If that is the case, we know that the Lie exponential on $H$ will be just the Lie exponential on $\glie$ restricted to $\hlie$. As such, we just need to consider the map that embeds $\hlie$ into $\glie$ and precompose by this parametrisation to get the exponential at $H$.  Naturally reductive homogeneous spaces provide the same modularity with the Riemannian exponential map. We will explore this further in~\Cref{sec:overview_geotorch} when looking at the class system in GeoTorch.

\paragraph{Use of different optimisers}
Arguably the most important thing that this approach brings to the table is that, as one is effectively working on $\RR^n$, the update rule in~\Cref{alg:dyn_triv_bundle} need not be that coming from gradient descent but could be that from other optimiser, such as \adam{} or \adagrad. We will see that these optimisers bring a considerable advantage over traditional gradient descent in the context of neural networks in the experiments (\Cref{sec:experiments}).

\begin{remark}[A note on adaptive optimisers and change of basepoint]
    When using adaptive optimisers, one has to be careful when performing a change of basepoint. Adaptive optimisers have extra parameters where they keep an approximation of the curvature of the function at the current point. As such, when changing the basepoint, as we are effectively changing the metric, this notion of curvature changes as well. While it is possible to correct for this, it is too expensive to do in practice. Therefore, even though one may leave this estimator hoping that it is a \emph{good} starting point for the estimation of the curvature with respect to this new metric, sometimes destabilises the training. As a rule of thumb, it is preferable to just zero out these estimations and start the estimation from scratch when using adaptive optimisers.

    Of course the use of these optimisers is just recommended together with a stopping rule that does not change the basepoint very often like that given by the closeness to the conjugate locus in the case of the exponential. In the general case where one may change the basepoint after every optimisation step, as one does to recover Riemannian gradient descent, there is no hope in that this scheme optimiser would yield a sensible update rule when used with an adaptive optimiser.
\end{remark}

\begin{remark}[Vector bundles]
    This whole construction may be carried for a general vector bundle $\deffun{\pi : E -> \MM;}$ with a map $\deffun{\retr : E -> \MM;}$. We chose to carry the presentation in the simpler setting of $E = T\MM$, since we do not aware of any setting where using a more general bundle is of practical interest.
\end{remark}

\clearpage
\chapter{First and Second Order Bounds on the Riemannian Exponential}\label{ch:second_order_bounds}
In this chapter we present the main theoretical results of the thesis: Novel second order bounds on the Riemannian exponential map in terms of the curvature of the manifold. With these we then go and prove some conditional convergence results for the static and dynamic trivialisation framework when using the Riemannian exponential map. This kind of results are the first of their kind, to the best of our knowledge. The results presented in this chapter are as general as they can be, and they can be used to give concrete bounds for different manifolds, as many papers in the literature on optimisation on manifolds give their results in terms of second order bounds of retractions.

\paragraph{Outline of the chapter.}
We start by setting-up the problem and giving a literature review, linking previous theoretical approaches to the one given here. In~\Cref{sec:background}, we set the notation, and we recall a few results from differential and Riemannian geometry that will be used throughout the chapter. In~\Cref{sec:first_order_bounds}, we give a self-contained proof of first order bounds on the Riemannian exponential map, which go by the name of Rauch's theorem. We include it as we believe that this proof is not very well-known in the optimisation on manifolds community. We emphasise the maximal domains on which the inequalities hold, which is often neglected in the literature. In~\Cref{sec:second_order_bounds}, we prove the second-order bounds, following the proof structure given in~\parencite{kaul1976schranken} but heavily simplifying the proof and the final bounds. In~\Cref{sec:concrete_bounds}, we look at the form that these bounds take for some manifolds used in optimisation. In~\Cref{sec:bounded_hessian}, we show how to use these bounds to prove conditional convergence of the static and dynamic trivialisation framework for functions of bounded Hessian on manifolds of bounded geometry. To finish the chapter, we include a short section giving a simple algebraic way to tell when is a given retraction the exponential map of a connection.

\section{Introduction}
\subsection{The problem and overview of assumptions}
We are interested in approximating critical points of a---possibly non-convex---function defined on a manifold
\[
    \min_{x \in M} f(x).
\]
In particular, we are interested in giving global convergence bounds for such problems. These problems form a particularly well-behaved subset of the field of constrained optimisation. Examples of the importance of this family of problems are ubiquitous in engineering, statistics, and machine learning. For example, they have found applications in image processing, using Grassmannians for tracking subspaces; analysis of time series in deep learning with the orthogonal group to avoid vanishing and exploding gradient problems; Bayesian statistics and machine learning with the manifold of positive definite matrices in kernel methods and metric learning; the space of low-rank and fixed-rank matrices for low-rank models; and the hyperbolic space for word embeddings, among many others.

The approach that we investigate in this chapter is that of using the exponential map to pullback the problem from the whole manifold to a fixed tangent space, converting the problem into an unconstrained one over $T_pM \iso \RR^n$
\begin{equation}\label{eq:problem}
    \min_{v \in T_pM} f(\exp_p(v)).
\end{equation}
For this to be a sensible method, the exponential map has to be surjective, so we will assume $(M, \gm)$ to be \textbf{connected and complete} throughout the chapter.

    For some manifolds, such as the hyperbolic space or the space of symmetric positive definite matrices (and more generally, Hadamard spaces), one may find (geodesically) convex functions which can be optimised efficiently~\parencite{bacak2014convex}. On the other hand, when the manifold is compact, as is the case for the sphere, or when dealing with orthogonality constraints~\parencite{edelman1998geometry}, there exists no convex function besides the constant ones, so we inevitably have to deal with non-convex problems.

    To this end, we will consider a function $f$ which is of \textbf{$\alpha$-bounded Hessian} (\ie, with bounded Hessian in operator norm) and we ask ourselves whether $f \circ \exp_p$ is of $\alpha'$-bounded Hessian for any $\alpha'$. The answer to this problem will be: Not in general. For example, for the hyperbolic space, the exponential map grows as $\cosh(t)$, which has unbounded first and second derivatives. For this reason, we will have to content ourselves with proving tight bounds on the growth of the Hessian of $f \circ \exp_p$ on a neighbourhood of a given radius.

    When it comes to the assumptions on the manifold, we will look at the family of manifolds of \textbf{bounded first order geometry}.
    \begin{definition}[Bounded first order geometry]\label{def:bounded_first_order}
        We say that a Riemannian manifold $(M, \gm)$ has \textbf{$(\delta, \Delta, \Lambda)$-bounded geometry} if it has uniformly bounded injectivity radius $\inj > 0$ and we have bounds on the sectional curvature and the derivative of the curvature tensor $R$
        \[
            \delta \leq \sec \leq \Delta \qquad \norm{\conn R} < \Lambda.
        \]
    \end{definition}
    These manifolds have found many uses in the area of \pdes{} on manifolds and geometric analysis. In particular, any compact manifold with any metric is of bounded geometry. On the other hand, this family does not impose topological restrictions on $M$ as any differentiable manifold admits a complete metric of bounded geometry~\parencite{greene1978complete}. We will introduce these manifolds in a bit more depth in~\Cref{sec:bounded_geometry}. For a much more in-depth introduction to this family of manifolds see the thesis~\parencite{eldering2012persistence}.

    \paragraph{Proof strategy.}
    We will implement the following strategy:
    \begin{enumerate}
        \item Give first and second order bounds for $\exp_p$ that depend on the curvature of the manifold (\Cref{sec:first_order_bounds,sec:second_order_bounds})
        \item Instantiate these bounds on some manifolds of interest (\Cref{sec:concrete_bounds})
        \item Use the first and second order bounds for $\exp_p$ to estimate a bound on the norm of the Hessian of $f \circ \exp_p$ on a bounded domain (\Cref{sec:bounded_hessian})
        \item Use these bounds to prove the convergence of the dynamic trivialisation framework (\Cref{sec:bounded_hessian})
    \end{enumerate}

\subsection{Previous work}\label{sec:previous_work}
\paragraph{General pullbacks in constrained optimisation.}
The idea of pulling back a problem from a space with a complex geometry to a simpler one via a particular family of maps lies at the core of the area of constrained optimisation. Mirror descent~\parencite{nemirovsky1983problem} and Lagrange multipliers are examples of methods that allow simplifying a constrained optimisation problem into one in a simpler space. It is also a recurrent theme in the literature on optimisation on manifolds. We have the Burer--Monteiro method~\parencites{burer2003nonlinear}{burer2005local} which parametrises the space of low-rank matrices via the multiplication of two tall matrices. Other approaches via different factorisations have also been explored, for example via an \svd~\parencite{vandereycken2013low} or other factorisations like the polar factorisation~\parencite{mishra2014fixed}. Similar parametrisations have been explored for the algebraic variety of matrices of rank at most $k$---low-rank optimisation---and manifolds such as positive semidefinite matrices of a fixed rank~\parencites{vandereycken2013riemannian}{massart2020quotient}. The idea behind these algorithms is that of parametrising a space in terms of simpler ones, effectively pulling back the problem to an easier one and then mapping back the solution of the simpler problem to the complex space via the parametrisation map.

\paragraph{Pullbacks along the exponential in numerical analysis.}
Instances of~\eqref{eq:problem} can be found all throughout the area of numerical analysis. For example, in the study of the centre of mass in the manifold of $\Symp{n}$ matrices~\parencites{arsigny2006log}{arsigny2006geometric}; continuous dynamics on Lie groups, and their discretisations~\parencites{magnus1954exponential}{iserles1999solution}{iserles2000lie}; optimisation on compact Lie groups~\parencite{lezcanocasado2019cheap}; optimisation on non-compact Lie groups (via the Lie exponential)~\parencite{mahony2002geometry}; and probabilistic methods on manifolds for their use in machine learning~\parencite{falorsi2019reparameterizing}. More recently, these ideas were unified into the framework of \textbf{static trivialisations}~\parencite{lezcanocasado2019trivializations}.

The main idea behind pulling back an optimisation problem from a manifold to a flat space is that of heavily simplifying the problem at hand, going from a manifold which might have a complex topology and geometry, to one that is flat and has trivial topology. Of course, that comes at a cost, as this change of topology makes the new function $f \circ \exp_p$ have more critical points. Luckily, for any differentiable manifold and any smooth metric, this set of problematic points, called the \textbf{conjugate locus}, has measure zero by a theorem of Sard~\parencite{sard1965hausdorff}. Since the conjugate locus is a subset of the \textbf{cut locus}, it is far from the point $p$. This is because the cut locus can vaguely be regarded as ``the set that is opposite to the point $p$''. For example, on a sphere, the cut locus of a point $p$ is its antipode and on a cylinder, the cut locus of a point is the whole line opposite to that point. As such, the cut locus is, in some sense, \emph{as far from $p$ as possible}, and so is the conjugate locus. At the expense of this technicality, one may heavily simplify the problem of optimising a function on a space with a non-trivial topology to a Euclidean problem. We will expand on this technical point in~\Cref{sec:background}.

An idea surprisingly similar to this one was outlined in~\parencite[Section $7$]{manton2015framework} in the context of second order optimisation and Newton methods on manifolds. Even though the strategy is very similar to that of dynamic trivialisations, the nature of the proofs developed in that paper and the theoretical background are essentially different. It would be of interest to look further into the connections between the approach presented here and that presented in that paper, as they can surely benefit of each other.

\paragraph{Pullbacks along the exponential map: Trivialisations.}
If we pullback the problem to the tangent space at each iteration of the algorithm along a retraction, we recover Riemannian Gradient Descent (\rgd) along that retraction~\parencite{boumal2019global}. It was then noted in~\parencite{lezcanocasado2019trivializations} that one may change the point $p$ in~\eqref{eq:problem} to avoid converging to a point in the conjugate locus of $p$. In fact, one may change the point according to an arbitrary stopping rule giving the following meta-algorithm:

\begin{enumerate}
    \item Compute gradient steps of the function $f\circ \exp_{p_i}$ to obtain iterates $v_{i,k}\in T_{p_i}M$ for $k = 1, 2, \dots$
    \item If we are close to a point in the conjugate locus of $p_i$, we set $p_{i+1} = \exp_{p_i}(v_{i,k})$ and repeat
\end{enumerate}

This is the algorithm presented in~\Cref{alg:dyn_triv_bundle} choosing the Riemannian exponential map rather than a general retraction~\parencite{lezcanocasado2019trivializations}
\begin{algorithm}[!htp]
    \caption{Dynamic trivialisation framework}%
    \label{alg:dyn_triv}
    \begin{algorithmic}[1]
            \Require A starting point $p_0 \in M$, a boolean statement $\code{stop}$
            \For {$i = 0, \ldots$}
                \State $v_{i, 0} = 0$
                \For {$k = 1, \dots$}
                    \State $v_{i, k} = v_{i, k-1} - \eta_{i,k-1}\grad \pa{f \circ \exp_{p_i}} \pa{v_{i,k-1}}$\label{line:update}\Comment{Optimise on $T_{p_i}M$}
                    \If {\code{stop}}\Comment{Stopping condition}
                        \Break
                    \EndIf
                \EndFor
                \State $p_{i+1} = \exp_{p_i}(v_{i,k})$ \Comment{Update pullback point}
            \EndFor
    \end{algorithmic}
\end{algorithm}

Two possible stopping rules stand out over the others. If we let $\code{stop} \equiv \code{False}$, then we recover the \textbf{static-trivialisation framework} of pulling back the function to a fixed tangent space as in~\eqref{eq:problem}. On the other hand, if $\code{stop} \equiv \code{True}$, then the algorithm is exactly Riemannian gradient descent, as mentioned before. The stopping rule of being too close to the conjugate locus could be then written as
\[
    \code{stop} \equiv \frac{\norm{\grad\pa{f \circ \exp_{p_i}}\pa{v_{i, k}}}}{\norm{\grad f\pa{x_{i,k}}}} < \epsilon.
\]

The ideas of pulling back the problem at the current point of the optimisation and then taking a few steps on the tangent space have been recently explored in~\parencites{criscitiello2019efficiently}{sun2019escaping} in the context of escaping saddle points for optimisation on manifolds.

    \paragraph{Retractions with bounded second and third order derivatives.}
    To be able to prove convergence of the dynamic trivialisation framework, we will give novel bounds on the norm of the Hessian of the exponential map in terms of the curvature of the manifold. These bounds are one example of second order bounds for a large family of retractions. These bounds lie at the heart of the convergence results of Riemannian methods via retractions, where it is often referred to as the \textbf{$L$-smoothness of the retraction}. Examples of work under this assumption are adaptive methods on matrix manifolds~\parencite{kasai2018riemannian}, quasi-Newton methods on manifolds~\parencite{huang2015broyden}, stochastic methods with variance reduction~\parencite{kasai2018riemannian}, and general convergence of Newton methods in Riemannian manifolds~\parencite{ferreira2002kantorovichs}, among many others. The bounds developed in this chapter give exact rates of growth for the norm of the Hessian of the retraction in terms of the curvature of the manifold when the retraction is the exponential map, which falls exactly in the framework of these papers and many others.

    Third order and higher order bounds are used when trying to converge to second order optima and escaping saddle points~\parencite{criscitiello2019efficiently} and when adapting Newton methods to manifolds~\parencite{agarwal2020adaptive}, and other higher order methods. The computations of higher order bounds is analogous to that of second order bounds. In this chapter, we present all the necessary techniques to give $n$-th order bounds for a manifold of $n$-bounded geometry (\Cref{def:bounded_geometry}), although we do not perform these computations explicitly. Bounds of this flavour were first developed in~\parencite{eichhorn1991boundedness}, where the rate of growth of the $n$-th derivative of the Christoffel symbols in normal coordinates is estimated.

\section{Differential Geometry: Conventions and Notation}\label{sec:background}
    In this section, we establish the notation used to refer to some recurrent objects, such as the distance function, the segment domain, mixed derivatives and pullbacks of connections. We also use it to recall some definitions from differential geometry that are not that commonly seen in the area of optimisation, such as the definition of the Lie derivative and covariant derivative of tensors.

    \paragraph{The manifold}
    We will always work on a connected and complete Riemannian manifold $(M, \gm)$ of dimension at least $2$ and at least $C^4$ regularity. We will denote the sectional curvature on the plane defined by two vectors $u, v \in T_pM$ as $\sec(u,v)$. When we write $\sec \leq \Delta$ for a constant $\Delta \in \RR$, we mean that for every $p \in M$, $u,v \in T_pM$, $\sec(u,v) \leq \Delta$. We will implicitly the natural identification $T_{\pa{p,v}}(T_pM) \iso T_pM$ for every $v \in T_pM$.

    \paragraph{Distance function, segment domain, conjugate and cut locus}
    We summarise here some well-known results on some objects associated to the geometry of a Riemannian manifold. Most of these were already presented more formally in~\Cref{def:cut_locus,prop:properties_cut_locus}.

    For a point $p \in M$, we define the \textbf{segment domain} as
    \[
        \seg{p} \defi \set{v \in T_pM | \exp_p(tv) \text{ is length minimising for }t \in [0, 1]}
    \]
    and we denote its interior by $\segint{p}$. The set $\segint{p}$ is a star-shaped open neighbourhood of $0$ in $T_pM$.
    We will also write
    \[
        \segm{p} \defi \exp_p(\segint{p}).
    \]
    $\segm{p}$ is sometimes referred as a \textbf{normal neighbourhood of $p$}, and its complement is called the \textbf{cut locus} $\cut{p} \subset M$. We have that the map
    \[
    \deffun{\exp_p \vert_{\segint{p}} : \segint{p} -> \segm{p};}
    \]
    is a diffeomorphism. In particular, on these sets, we have that $\exp_p^{-1}$ exists and that $\dif\exp_p$ is full rank.

    A conjugate point $\exp_p(v) = q$ of $p$ is one such at which $\pa{\dif \exp_p}_v$ is not full rank. We denote the set of all these points $\conjloc{p}$. We have that $\conjloc{p} \subset \cut{p}$. By a theorem of Sard, $\conjloc{p}$ has measure zero in $M$~\parencite{sard1965hausdorff}. Even more, $\cut{p}$ has Hausdorff dimension at most $n-1$~\parencite{itoh1998dimension}. As such, $\segm{p}$ is an open radially convex neighbourhood of $p$ that covers almost all the manifold.

    We will write $r(x) \defi d(x,p)$ for the distance to a fixed point $p$. This function is differentiable on $\segm{p} \backslash \set{p}$, with gradient the unit radial vector field emanating from $p$. In particular, we have that for a geodesic $\gamma$ starting at $p$, $\grad r\vert_\gamma = \dgamma$. The cut locus of $p$ is exactly the set of points other than $p$ at which the distance function $r$ is not differentiable.

    \paragraph{Einstein convention}
    Whenever we refer to coordinates $\set{x^i}$ on all of $\segm{p}$, we will always assume that these are the normal coordinates given by $\exp_p^{-1}$ on $T_pM$. We will denote $\partial_i \defi \frac{\partial}{\partial x^i}$ for short. We will use Einstein's summation convention that if an index appears as a super-index and a sub-index in a formula, it means that we are summing over it. For example, for a vector field $X$ in local coordinates we write
\[
    X = \sum_{i=1}^n X^i \partial_i = X^i \partial_i.
\]

    \begin{remark}
        Almost all the results in this chapter can be developed working just on a neighbourhood of a geodesic rather than on the whole $\segm{p}$. We will make sure of make this explicit in each result throughout the chapter.
    \end{remark}

    \paragraph{Connections and derivations}
    We will write $D_X f \defi \dif f(X)$ for the directional derivative of a function along a vector field $X$ to disambiguate with the gradient of a function, which we will denote by $\grad f$ its gradient. We denote by $\conn$ the Levi-Civita connection on $\segm{p}$ and by $\connflat$ the pushforward of the flat connection on $T_pM$ along $\exp_p$. In the same way that we do for the norms, we will abuse the notation and also denote by $\conn$ the associated connections defined by $\conn$ on associated bundles.

    We recall that connections on tensor bundles are defined so that the Leibnitz rule holds. For example, for the Hessian and three vector fields $X, Y_1, Y_2$, the Leibnitz rule reads
    \[
    D_X\pa{\Hess f\pa{Y_1,Y_2}} = \pa{\conn_X\Hess f}\pa{Y_1, Y_2} + \Hess f\pa{\conn_X Y_1, Y_2} + \Hess f\pa{Y_1, \conn_X Y_2}
    \]
    so $\conn\Hess f$ is given by the $(0,3)$ tensor
    \[
     \pa{\conn_X\Hess f}\pa{Y_1, Y_2} \defi
    D_X\pa{\Hess f\pa{Y_1,Y_2}} - \Hess f\pa{\conn_X Y_1, Y_2} - \Hess f\pa{Y_1, \conn_X Y_2}.
    \]

    We can do the same for a tensor $T$ of type $(0,s)$. We define its covariant derivative $\conn T$ as the tensor of type $(0, s+1)$ such that the Leibnitz rule holds
    \[
        \pa{\conn_X T}(X_1, \dots, X_s) \defi D_X\pa{T\pa{X_1, \dots X_s}} - \sum_{i=1}^s T\pa{X_1, \dots, \conn_X X_i, \dots, X_s}.
    \]
    For example, if we have in local coordinates the differential of a function is a $(0, 1)$ tensor $\dif f = f^i \dif x_i$. We may compute its Hessian in local coordinates as the $(0,2)$ tensor given by
    \[
        \conn \dif f = \conn(f^i\dif x_i) = \dif f^i \tensor \dif x_i + f^i \conn \dif x_i
    \]
    where we have used that the connection on functions is just the differential, by definition.

    The Lie derivative of tensors is defined in the same way. For a $(0, s)$ tensor $T$, the Lie derivative of $T$, $\lie_X T$ in the direction of a tensor $X$ is defined as the $(0, s)$ tensor
    \[
        \pa{\lie_X T}(X_1, \dots, X_s) \defi D_X\pa{T\pa{X_1, \dots X_s}} - \sum_{i=1}^s T\pa{X_1, \dots, \lie_X X_i, \dots, X_s}.
    \]

    \paragraph{Pullback connections}
    When dealing with a smooth curve $\deffun{\gamma : [0,r] -> M;}$, we will sometimes want to derive a vector field $X$ along it. We will write $\conn_{\partial_t}X$ for the \textbf{covariant derivative along $\gamma$}, that is, the \textbf{pullback connection along $\gamma$}.

    We will sometimes write the covariant derivative along $\gamma$ as $\dot{X} \defi \conn_{\partial_t} X$. This comes from the fact that if we choose a parallel frame $\set{e_i}$ along $\gamma$, that is $\conn_{\partial_t} e_i = 0$, we have that by the Leibnitz rule
    \[
        \dot{X} = \conn_{\partial_t}\pa{X^ie_i} = \conn_{\partial_t}\pa{X^i}e_i + X^i\conn_{\partial_t}\pa{e_i} = \dot{X}^ie_i
    \]
    where $\dot{X}^i$ are just the regular derivatives of the coordinate functions. With this notation, the usual equation for geodesics simply reads $\ddot{\gamma} = 0$.

    If instead of a curve we have an embedded surface
    \[
    \deffun{c : [0,r] \times [-\epsilon, \epsilon] -> M;
            (t,s) -> c(t,s)}
    \]
    we will analogously write $\partial_s$ for the vector field in the direction of the second component. Suppose that we have a vector field $J$ along $\gamma(t) \defi c(t,0)$ given by
    \[
        J(t) = \frac{\partial c}{\partial s}(t, 0) = \dif c(\partial_s) \vert_{\gamma}
    \]
    where $\partial_s$ is the coordinate vector field on the second component of $[0,r] \times [-\epsilon, \epsilon]$. We will abuse the notation and write for a vector field $X$ along $\gamma$
    \[
        \conn_J X \defi \conn_{\partial_s} X
    \]
    in the same way that we may write the equation for the geodesics as
    \[
        \conn_{\dgamma}\dgamma \defi \conn_{\partial_t}\dgamma = 0.
    \]

    Sometimes, we will also pullback a connection along the exponential map $\exp_p$. For the distance function $r$ to $p$, we have the radial vector field on $\segm{p} \backslash \set{p} \subset M$ given by $\grad r$. By Gauss's lemma, the pullback of this vector field to $\segint{p} \backslash \set{0} \subset T_pM$ via the exponential map is exactly the radial vector field on the tangent space $\partial_r$, that is
    \[
        \dif \exp_p(\partial_r) = \grad r.
    \]
    With this notation, we can write the equation for the geodesics emanating from $p$ on all of $\segm{p} \backslash \set{p}$ using the pullback connection along $\exp_p$ as
    \[
        \conn_{\partial_r} \grad r = 0.
    \]

    \paragraph{Hessian and iterated Hessian}
    If we want to evaluate the $\pa{0, 2}$ Hessian twice on the same vector, we may do so by simplifying this problem to just performing a Euclidean derivative. Let $\gamma$ be the geodesic such that $\dgamma(0) = X$
    \begin{equation}\label{eq:hess_derivative_along_geodesic}
        \Hess f(X, X) = \pa{f \circ \gamma}''(0) = D_{\dgamma(0)}\scalar{\grad f, \dgamma} = \scalar{\conn_{\dgamma(0)} \grad f, \dgamma(0)} = \conn\dif f(X,X).
    \end{equation}
    In other words, $\conn \dif f$ serves as a coordinate-free definition of the $(0,2)$ Hessian of $f$. In the same way, $\conn \grad f$ is the $(1,1)$ Hessian of $f$.\footnote{Note that in the expression $\conn \grad f$, the first $\nabla$ represents a connection, while the second represents the gradient of $f$. This distinction will always by clear, as we represent the directional derivative on functions by the notation $D_X f$.}

    We also recall the definition of the iterated Hessian of a function as a $(0,2)$ tensor. This is also best introduced in its $(1,1)$ form by using that the $(1,1)$ Hessian is given by $\Hess f(X) = \conn_X\grad f$ so that
    \[
        \Hess^2 f(X) \defi \pa{\Hess f \circ \Hess f}(X) = \conn_{\conn_X \grad f}\grad f.
    \]
    Its $(0, 2)$ version is consequently defined as
    \begin{equation}\label{eq:def_iterated_hessian}
        \Hess^2 f(X,Y) \defi \scalar{\Hess f(\Hess f(X)), Y} = \Hess f(\Hess f(X), Y) = \scalar{\Hess f(X), \Hess f(Y)}
    \end{equation}
    where we have used that the Hessian is symmetric. This formula shows that $\Hess^2 f(X,Y)$ is also symmetric.

    \section{First Order Bounds for the Exponential Map}\label{sec:first_order_bounds}
In this section, we give bounds on the norm of the differential of the exponential map of a manifold with bounded sectional curvature. In particular, if the sectional curvature of any plane of $M$ is bounded above and below by $\delta \leq \sec \leq \Delta$, for a ball $B_p(r)$, the bounds will be of the form
\[
    f_\Delta(r)\norm{w} \leq \norm{\pa{\dif \exp_p}_{v}(w)} \leq f_\delta(r)\norm{w} \mathrlap{\qquad \forall w \in T_pM, \norm{v} \leq r}
\]
for suitable functions $\deffun{f_{\delta}, f_\Delta : [0, r] -> \RR^+;}$.

First order bounds for the exponential map have been known, at least for surfaces, since the times of Cartan~\parencite{cartan1928lecons}, and are well-known in areas such as comparison geometry or \pdes, but they are not so known in the area of optimisation on manifolds. The tight bounds that we will present here are often referred to as \textbf{Rauch's theorem} as he was the first to show these bounds in the positive curvature-case~\parencite{rauch1951contribution} and later in the general case~\parencite{rauch1959geodesics}.

There are quite a few proof techniques of these theorems. There exist proofs via the second variation formula of the energy~\parencite{spivak1999comprehensive}, through bounds on the norm via the Jacobi equation~\parencite{jost2017riemannian}, by bounding Riccati equation on self-adjoint operators and integrating the result~\parencite{eschenburg1994comparison}, or by the study of the distance function to a given point~\parencite{gromov1981structures}. We choose to present here this fourth approach.

The approach presented here was first introduced by Gromov in the context of volume bounds, in what's called now the Bishop--Gromov theorem. This approach has the advantage of yielding upper and lower bounds for positive and negative curvatures at the same time. Other approaches need different techniques for the upper and lower bounds, or they just give some weaker version of the theorem with constraints on the sign of the curvature. This approach also has a more geometric flavour as it accounts for bounding the principal curvatures of geodesic balls on the manifold. The general strategy of this proof is to decompose certain homogeneous second order differential equation into a Riccati equation and a system of first order equations, bound the Riccati equation, and then integrate the result to get the bounds on the differential of the exponential. The view on Jacobi fields and parallel vector fields is from~\parencite{cheeger2008comparison}, while the approximation to Rauch's theorem using curvature equations is from~\parencite{petersen2016riemannian}. We provide different proofs of the Riccati comparison lemmas, as those in the book are incorrect. We emphasise throughout the exposition whether the theorems work just on the segment domain or up to the first conjugate point, as this point is often neglected in previous expositions and it is of interest in optimisation.

    \subsection{Jacobi fields}\label{sec:jacobi_fields}
    We begin by introducing what will be the main tool that we will just throughout the chapter: Jacobi fields. Jacobi fields are certain vector fields along a given geodesic that describe the behaviour of the differential of the exponential map along this geodesic. We will see that they are the solution of a certain differential equation, and we will use this equation to give the bounds on the differential of the exponential.

    Before defining what Jacobi fields are and deducing this differential equation we will need a lemma. This lemma that can loosely be interpreted as \emph{the Lie derivative commutes with the differential of a smooth map}.

    \begin{lemma}\label{lemma:lie_derivative_pushforward}
        For an immersion $\psi$ and two vector fields $X, Y$ on $M$ there exist local extensions $\widetilde{X}, \widetilde{Y}$ along $\psi$---that is, $\dif \psi(X_x) = \widetilde{X}_{\psi(x)}$ for $x \in M$---and
        \[
            \dif \psi([X,Y]) = [\widetilde{X}, \widetilde{Y}].
        \]
    \end{lemma}
    \begin{proof}
        See for example~\parencite[1. Lemma 22 and 1. Lemma 33]{oneill1966fundamental}.
    \end{proof}

    With this lemma in hand, we are ready to show that the differential of the exponential satisfies certain differential equation.

    \begin{proposition}[First order Jacobi equation]\label{prop:riccati_jacobi}
        Let $(M, \gm)$ be a complete Riemannian manifold. Let $\deffun{\gamma : [0, r] -> M;}$ be a geodesic without conjugate points with initial unit vector $v \defi \dgamma(0)$ and a vector $w \in T_pM$ such that $w \perp v$. Consider the following variation of $\gamma$ in the direction of $w$
        \[
            \deffun{c : [0, r] \times (-\epsilon, \epsilon) -> M ; (t, s) -> \exp_p(t(v+sw))}
        \]
        where $\epsilon$ is small enough so that $c(-,s)$ has no conjugate points for every $s \in (-\epsilon, \epsilon)$.

        Define the vector field $J$ along $\gamma$ as
        \[
            J(t) \defi \dif c(\partial_s)\vert_\gamma = \frac{\partial c}{\partial s}(t, 0) = \pa{\dif \exp_p}_{tv}(tw).
        \]
        This vector field solves the following differential equation on $\gamma$
        \[
            \lie_{\grad r} J \vert_{\gamma} = 0
        \]
        or equivalently, using the pullback connection
        \begin{equation}\label{eq:riccati}
            \conn_{\dgamma} J = \conn_J \grad r.
        \end{equation}
    \end{proposition}
    \begin{proof}
        The fact that $\gamma$ is in $\segm{p}$ is equivalent to saying that $tv \in \segint{p}$ for $t \in [0, r]$. From this we get that the $\epsilon$ in the definition exists as $\segint{p}$ is open. Since $c$ defines an embedded surface, $c$ is a diffeomorphism onto its image.

        Since $\partial_t, \partial_s$ are coordinate vector fields of the surface defined by $c$, we have that
        \[
            [\partial_t, \partial_s] = 0.
        \]
        Using that $c$ is a diffeomorphism, by~\Cref{lemma:lie_derivative_pushforward} together with the fact that the Lie derivative restricted to a submanifold is the Lie derivative of the restrictions, we have that
        \[
            [\dif c\pa{\partial_t},\dif c\pa{\partial_s}] = \dif c\pa{[\partial_t, \partial_s]} = 0.
        \]
        We get the Lie formulation of the Jacobi equation by noting that $\dif c(\partial_t) \vert_\gamma = \grad r \vert_\gamma$. The equation using the pullback connection is just a reformulation of the Lie derivative using that the Levi-Civita connection is torsion-free.
    \end{proof}

    \begin{remark}
        Note that, even though the differential equation involves a Lie bracket, we just need to define $J$ along the flow of $\grad r$, which is just a geodesic $\gamma$. This equation can be extended into a \pde{} on all of $\segm{p} \backslash \set{p}$, looking for vector fields on this domain such that
        \[
            \lie_{\grad r}W = 0.
        \]
        One such a solution is clearly given by the radial vector field $W = \grad r$.

        To find other solutions, we imitate the construction that we did in~\Cref{prop:riccati_jacobi} and define $W$ on a sphere around $p$ and extend it radially to all of $\segm{p}$ as
        \[
            W(x) = \dif\exp_p\pa[\Big]{r(x)W\pa[\Big]{\lfrac{x}{r\pa{x}}}}.
        \]
        It is clear by~\Cref{prop:riccati_jacobi} that this vector field solves the first order Jacobi equation on all of $\segm{p} \backslash \set{p}$.
    \end{remark}

        We will now use this equation to derive the linear version of the Jacobi equation.

    \begin{proposition}[Jacobi equation]\label{prop:jacobi_equation}
    Let $\deffun{\gamma : [0, r] -> M;}$ be a geodesic with initial values $\gamma(0) = p$, $\dot{\gamma}(0) = v$. Then $J(t) = \pa{\dif \exp_p}_{tv}(tw)$ is a vector field along $\gamma$ that solves the following second order homogeneous linear differential equation
    \[
        \ddot{J} + R(J, \dgamma)\dgamma = 0.
    \]
    We call this equation the \textbf{Jacobi equation}.
\end{proposition}
\begin{proof}
    For a vector $w \in T_p\MM$, we may differentiate~\eqref{eq:riccati} to get
    \[
        \conn_{\partial_t} \conn_{\partial_t} J = \conn_{\partial_t} \conn_J \grad r
    \]
    and using the definition of the curvature tensor and the fact that $[J, \grad r] = 0$ we get
    \[
        \conn_{\partial_t} \conn_{\partial_t} J + R(J, \grad r)\grad r = 0.
    \]
    We then get the Jacobi equation by restricting this equation to $\gamma$, since $\grad r \circ \gamma = \dgamma$.

    Using an orthonormal parallel frame $\set{e_i}$ along $\gamma$ with $e_1 = \dgamma$, and expressing a solution of this equation as $J^ie_i$, where $\deffun{J^i : [0, r] -> \RR;}$, we see that, in fact, this is a second order linear differential equation on $[0, r]$
    \[
        \ddot{J}^i + R_j^iJ^j = 0 \mathrlap{\qquad i = 1, \dots, n}
    \]
    with coefficients
    \[
        R_j^i =
        \begin{cases}
            0 \qquad &\text{if } i = 1 \\
            R(e_j, \dgamma, \dgamma, e_i) \qquad &\text{if } i = 2, \dots, n.
        \end{cases}\qedhere
    \]
\end{proof}

\begin{definition}
    We say that a vector field $J$ along a geodesic $\gamma$ is a \textbf{Jacobi field} if it satisfies the Jacobi equation.
\end{definition}

\begin{remark}[Basic properties of Jacobi fields]
    By elementary theory of differential equations, the Jacobi equation has $2n$ independent solutions, defined by the initial values $(J(0), \dot{J}(0)) \in T_pM \times T_pM$. By construction, we have found $n$ independent solutions given by
    \[
        J(t) = \pa{\dif \exp_p}_{tv}(tw) \mathrlap{\qquad \forall w \in T_pM.}
    \]
    These correspond to the initial values $J(0) = 0$, $\dot{J}(0) = w$. Geometrically, they correspond to a family of geodesics $c$ that fixes the initial point, that is, $c(0, s) = p$.

    It is also direct to show that $\dgamma$ is another solution to this equation with initial values $J(0) = v$, $\dot{J}(0) = 0$. The other $n-1$ independent solutions correspond to variations of the geodesic $\gamma$ that do not fix the initial point $p$. These solutions will not be important for our analysis.

    It starts being clear now the importance of Jacobi fields. If we want to control the norm of the differential of $\exp_p$ at a point $rv \in \segint{p}$, for $r > 0$, $\norm{v} = 1$, we may define the Jacobi field along the geodesic $\deffun{\gamma : [0,r] -> M;}$ with initial conditions $(p,v)$ and we have that
    \[
        \pa{\dif\exp_p}_{rv}(w) = \frac{J(r)}{r}.
    \]
    In order to bound the norm of the differential of the exponential, we just need to bound the norm of the solutions of the Jacobi equation.

    If the vector $w$ is parallel to $v$, as we saw in the proof of~\Cref{prop:jacobi_equation}, the Jacobi field takes the form $J(t) = t\norm{w}\dgamma(t)$. As such, in this direction we have an exact solution of the Jacobi equation, and we can compute its norm exactly as $\norm{J(t)} = t\norm{w}$.

    If the vector $w$ is perpendicular to $\dgamma(0)$, then $J(t)$ is perpendicular to $\dgamma(t)$ for every $t \in [0, r]$, as the equation for $J^1$ would be given by
    \[
        \ddot{J}^1 = 0 \mathrlap{\qquad (J^1(0), \dot{J}^1(0)) = (0,0).}
    \]

    Geometrically, the previous proposition says that Jacobi fields just rotate around the vector field defined by $\dgamma$. In symbols, this means that if we split a Jacobi field along $\gamma$ into its radial and normal part as
    \[
        \paral{J} = \scalar{J, \dgamma}\dgamma \qquad \normal{J} = J - \paral{J}
    \]
    if $\paral{\dot{J}}(0) = 0$ then $\paral{J}(t) = 0$ for every $t \in [0, r]$.

    The results discussed in this remark are commonly known as the \textbf{Gauss's lemma}.

    Jacobi fields are also very closely related to conjugate points and the conjugate locus. A point $q = \exp_p(v)$ is conjugate to $p$ if the exponential at $v$ is not full rank. Another way of looking at this definition is by defining a point $q$ conjugate to $p$ if there exists a Jacobi field connecting $p$ and $q$ such that it is zero at $p$ and $q$. We will see when studying Rauch's theorem that the estimates for Jacobi fields will work on every point on the manifold but the conjugate locus, as these points will be exactly the singularities of the maps that we work with. For a unit vector $v$, we will denote by $\conjpoint{v} \in (0, \infty]$ the smallest time $t$ at which the geodesic $\gamma(t) = \exp_p(tv)$ is conjugate to $p$, that is, the point $\gamma(\conjpoint{v})$ is conjugate to $p$.
\end{remark}

    \subsection{Parallel vector fields}
    In the same way that Jacobi fields with initial condition $J(0) = 0$ and $\dot{J}(0) \perp \dgamma(0)$ are the vector fields such that $\lie_{\grad r} J = 0$, we have parallel vector fields.

    \begin{definition}
        We say that $E$ is a \textbf{parallel vector field} on $\segm{p} \backslash \set{p}$ if it satisfies
        \[
            \conn_{\grad r} E = 0.
        \]
    \end{definition}

    Parallel vector fields can be easily described along a geodesic as parallel transporting a vector $w \in T_{\gamma(0)}M$ along it. Analogously, we can also define a parallel vector field on all of $\segm{p} \backslash \set{p}$ by specifying a vector field on a sphere and extending it to all of $\segm{p}$ parallel transporting this vector field along geodesics.

    \subsection{Rauch's theorem}
    We now go back to the study of the norm of the differential of the exponential. The strategy that we will follow will be that of bounding the derivative of the norm of the exponential. For that end, let $J$ be a Jacobi field along a geodesic $\gamma$ such that $J(0) = 0$, $\dot{J}(0) \perp \dgamma(0)$. As we already saw before, this vector field takes the form
    \[
        J(t) = \pa{\dif \exp_p}_{t\dgamma(0)}\pa{t\dot{J}(0)}.
    \]
    The derivative of its norm is given by
    \[
        \frac{\dif}{\dif t}\norm{J} = \frac{\scalar{\dot{J}, J}}{\norm{J}}.
    \]
    At first sight it looks like we have not achieved much, as we still have a term involving the norm of $J$ on the right-hand side, but it turns out that it can be conveniently rewritten using~\eqref{eq:riccati} as
    \[
        \scalar{\dot{J}, J} = \scalar{\conn_{\grad r} J, J} = \scalar{\conn_J \grad r, J} = \Hess r(J, J)
    \]
    so we get
    \[
        \frac{\dif}{\dif t}\log{\norm{J}} = \Hess r\pa[\Big]{\frac{J}{\norm{J}}, \frac{J}{\norm{J}}}.
    \]

    Our plan will be to bound the Hessian of the distance to $p$ on vectors of norm $1$ to get bounds on the log-derivative of the norm of the Jacobi fields and then integrate these to get bounds on $\norm{J}$. To this end, we start by computing formulas that relate the Hessian of the distance function to the curvature tensor.

    A geometric interpretation of this approach comes after noting that the Hessian of the distance function is exactly the second fundamental form (or shape operator) of the distance function. As such, this quantity can be interpreted geometrically as the variation of the curvature of spheres.

    \begin{proposition}[Radial Curvature Equations]\label{prop:radial_curvature}
        Let $(M, \gm)$ be a Riemannian manifold. For a point $p \in M$, we have on $\segm{p}$
        \begin{align}
            \pa{\lie_{\grad r}\Hess r}(X, Y)& - \Hess^2 r(X, Y) = - R(X, \grad r, \grad r, Y)\label{eq:curv1} \\
            \pa{\conn_{\grad r}\Hess r}(X, Y)& + \Hess^2 r(X, Y) = - R(X, \grad r, \grad r, Y)\label{eq:curv2}.
        \end{align}
    \end{proposition}
        where $\Hess^2$ denotes the \textbf{iterated Hessian} defined in~\eqref{eq:def_iterated_hessian}.
    \begin{proof}
        Since for a geodesic $\gamma$ we have that $\grad r \vert_\gamma = \dgamma$, the gradient of $r$ has $\conn_{\grad r} \grad r = 0$. From this we get
        \begin{align*}
            - R(X, \grad r, \grad r, Y) &= \scalar{\conn_{\grad r} \conn_X \grad r, Y} + \scalar{\conn_{\lie_X \grad r} \grad r, Y}\\
                                        &= D_{\grad r} \scalar{\conn_X \grad r, Y} - \scalar{\conn_X \grad r, \conn_{\grad r} Y} + \Hess r(\lie_X \grad r, Y)\\
                                        &= D_{\grad r} \pa{\Hess r(X, Y)} - \Hess r\pa{X, \conn_{\grad r} Y}
                                        - \Hess r(\lie_{\grad r} X, Y).
        \end{align*}
        Equation \eqref{eq:curv1} follows after expanding the first term as
        \[
            D_{\grad r} \pa{\Hess r(X,Y)} = \pa{\lie_{\grad r}\Hess r}(X, Y) + \Hess r\pa{\lie_{\grad r}X, Y} + \Hess r\pa{X, \conn_{\grad r} Y - \conn_Y \grad r}
        \]
        and equation \eqref{eq:curv2} follows after expanding it as
        \[
            D_{\grad r} \pa{\Hess r(X,Y)} = \pa{\conn_{\grad r}\Hess r}(X, Y) + \Hess r\pa{\conn_{\grad r}X, Y}  + \Hess r\pa{X, \conn_{\grad r} Y}.\qedhere
        \]
    \end{proof}

    \begin{remark}[A Riccati-type equation]
        Note that the previous proposition holds for every vector field $Y$. As such, we may write it in its $(1,1)$ form
        \[
            \pa{\conn_{\grad r} \Hess r}(X) + \Hess^2 r(X) + R(X, \grad r)\grad r = 0.
        \]
        Denoting by $R_2(X) \defi R(X, \grad r)\grad r$ the $(1,1)$ curvature tensor in this equation, and the shape operator for the geodesic balls as $S \defi \Hess r$, this equation can be rewritten in the radial direction as a Riccati equation along gamma on self-adjoint $(1,1)$ tensors
        \[
            \dot{S} + S^2 + R_2 = 0.
        \]
        After a choice of a parallel frame along the geodesic, it can be regarded as a differential equation on self-adjoint matrices.

        As one does in one dimension splitting a second order differential equation into a first order equation and a Riccati equation, here we have split the Jacobi equation into first order symmetric matrix equation---more formally, an equation on self-adjoint endomorphisms along a geodesic---and the first order equation defining Jacobi fields (\Cref{eq:riccati})
        \[
            \dot{J} = S(J).
        \]
        A careful analysis of this Riccati equation yields another particularly clean proof of Rauch's theorem, at the expense of the use of more abstract methods, as presented in~\parencite{eschenburg1990comparison}.

        In contrast, we will use parallel and Jacobi vector fields to simplify these matrix equations.
    \end{remark}

    \begin{proposition}\label{prop:radial_curvature_evaluated}
        Let $\gamma$ be a geodesic, and let $J$ be a Jacobi field along it such that $J(0) = 0$, $\dot{J}(0) \perp \dgamma(0)$. Let $E$ be a parallel vector field along $\gamma$ that is normal to $\dgamma$. We have the following differential equations along $\gamma$
        \begin{align}
            \frac{\dif}{\dif t}\pa{\Hess r(J, J)}& - \Hess^2 r(J, J) = - \sec(J, \dgamma)\scalar{J, J}\label{eq:curv1_jacobi} \\
            \frac{\dif}{\dif t}\pa{\Hess r(E, E)}& + \Hess^2 r(E, E) = - \sec(E, \dgamma)\label{eq:curv2_parallel}.
        \end{align}
    \end{proposition}
    \begin{proof}
        We just evaluate~\eqref{eq:curv1} and~\eqref{eq:curv2} on $J, E$ and use that
        \[
            \lie_{\grad r} J = 0 \qquad \conn_{\grad r} E = 0.
        \]
        To go from the Riemannian tensor to the sectional curvature on the right-hand side, we just recall that if a Jacobi field or a parallel vector field is normal to $\dgamma$ at one point, it is normal to $\dgamma$ at every point.
    \end{proof}

    Given the signs of the $\Hess^2$ in this proposition, it is now a bit clearer how Jacobi fields may be used to give lower bounds and parallel vector fields to give upper bounds on the Hessian of the distance function.

    We now define some generalised trigonometric functions, which will be useful in the sequel.
    \begin{definition}[Generalised trigonometric functions]
        For a constant $\kappa \in \RR$, we define the \textbf{generalised sine function} $\sn_\kappa$ as the solution to the differential equation
        \[
            \ddot{x} + \kappa x = 0 \qquad x(0) = 0,\,\dot{x}(0) = 1.
        \]
        In particular, we have
        \[
            \sn_\kappa(t) \defi
            \begin{cases}
                \frac{\sin (\sqrt{\kappa}t)}{\sqrt{\kappa}}    \qquad &\text{if } \kappa > 0\\
                t                                              \qquad &\text{if } \kappa = 0\\
                \frac{\sinh (\sqrt{-\kappa}t)}{\sqrt{-\kappa}} \qquad &\text{if } \kappa < 0
            \end{cases}
        \]
        Define $\pi_\kappa$ as the first positive zero of $\sn_\kappa$ if it has one, and $\infty$ otherwise
        \[
            \pi_\kappa \defi
            \begin{cases}
                \frac{\pi}{\sqrt{\kappa}}  &\qquad \text{if } \kappa > 0 \\
                \infty                     &\qquad \text{if } \kappa \leq 0
            \end{cases}
        \]

        Finally, define the \textbf{generalised cotangent} $\ct_\kappa$ as
        \[
            \ct_\kappa(t)
            \defi
            \frac{\sn'_\kappa(t)}{\sn_\kappa(t)}
            =
            \begin{cases}
              \sqrt{\kappa}\cot(\sqrt{\kappa}t)    &\mathrlap{\qquad \text{if } \kappa > 0}\\
              \frac{1}{t}                          &\mathrlap{\qquad \text{if } \kappa = 0}\\
              \sqrt{-\kappa}\coth(\sqrt{-\kappa}t) &\mathrlap{\qquad \text{if } \kappa < 0}
            \end{cases}
        \]
        This function is differentiable on $(0, \pi_\kappa)$.
    \end{definition}

    \begin{remark}[Riccati equation]
        Note that the equation $\ddot{x} + \kappa x = 0$ is the associated second order equation to the Riccati equation $\dot{\rho} + \rho^2 + \kappa = 0$. The relation between these two equations comes from the change of variables $\rho = \frac{\dot{x}}{x}$. This is exactly the Riccati equation that the shape operator for the distance function solves in constant curvature, as pointed out before. Given that $\sn_\kappa$ solves the second order equation, $\ct_\kappa$ solves the Riccati equation.
    \end{remark}

    \begin{remark}[Qualitative behaviour]
        From a qualitative point of view, it is worth noting that the family of functions $\sn_\kappa(t)$ is strictly increasing in $\kappa$. We also have that $\sn_\kappa$ has a zero at $0$ for every $\kappa$ and another one at $\pi_\kappa$ for $\kappa > 0$, so
        \begin{align*}
            \lim_{t \to 0^+}\ct_\kappa(t) &=\infty\phantom{-}\mathrlap{\qquad\kappa \in \RR} \\
            \lim_{t \to \pi_\kappa^-} \ct_\kappa(t) &= -\infty \mathrlap{\qquad\kappa > 0.}
        \end{align*}
    \end{remark}

    All we need to do now is to relate the solutions of the Jacobi equation in Riccati as given by the Hessian of the distance function in~\Cref{prop:radial_curvature_evaluated} with its solutions in constant curvature. To do so, we need the following elementary comparison lemma for functions of real variable. It can be thought of as a version of Sturm's comparison theorem for Riccati equations, and it is proved in a similar way.
    \begin{lemma}[Comparison Lemma for Riccati Equations]\label{lemma:comparison}
        Let $\deffun{\rho_1, \rho_2 : (a,b) -> \RR;}$ be two differentiable functions such that
        \[
            \dot{\rho}_1 + \rho_1^2 \leq \dot{\rho}_2 + \rho_2^2.
        \]
        If we have that $\rho_1(t_0) \geq \rho_2(t_0)$ for a $t_0 \in (a,b)$, then $\rho_1 \geq \rho_2$ on $(a, t_0]$.
    \end{lemma}
    \begin{proof}
        Let $u = \rho_1 - \rho_2$ and $\zeta = \dot{\rho}_2 - \dot{\rho}_1 + \rho_2^2 - \rho_1^2$. We have by hypothesis that $u(t_0) \geq 0$ and $\zeta \geq 0$ on $(a,b)$, and it is enough to prove that $u \geq 0$ on $(a, t_0]$ provided that $u(t_0) \geq 0$.

        We can write the differential inequality as a differential equation in terms of $u$ and $\zeta$ as
        \[
            \dot{u} = -(\rho_1 + \rho_2)u -\zeta.
        \]
        and we can solve this differential equation using the method of variation of parameters. We compute the general solution of the homogeneous system
        \[
            \dot{v} = -(\rho_1 + \rho_2)v
        \]
        as $v = Cv_1$ for $v_1 \defi e^{-\int \rho_1 + \rho_2} > 0$ and a constant $C \in \RR$.

        To get a particular solution of the inhomogeneous equation, we may use variation of parameters. Let $u = \eta v_1$ and, plugging it into the inhomogeneous equation, we see that the function $\eta$ satisfies $\dot{\eta} = -\zeta v_1^{-1}$. As $\zeta$ and $v_1^{-1}$ are positive, we conclude that $\eta$ is decreasing.

        Finally, we get the general solution by adding the particular solution and the general solution to the homogeneous system $u = v_1(C + \eta)$. Since, by hypothesis, $u(t_0) \geq 0$, we get that $C + \eta(t) \geq 0$ for $t \in (a, t_0]$ so that $u(t) \geq 0$ on $(a, t_0]$.
    \end{proof}
    \begin{proposition}[Riccati Comparison Estimate]\label{prop:riccati_comparison_estimate}
        Let $\deffun{\rho : (0, b) -> \RR;}$ be a differentiable function and fix $\kappa \in \RR$.
        \begin{enumerate}
        \item
            If $\dot{\rho} + \rho^2 \leq -\kappa$ then $\rho(t) \leq \ct_\kappa(t)$ on $(0,b)$. Furthermore, $b \leq \pi_\kappa$

        \item
            If $\dot{\rho} + \rho^2 \geq -\kappa$ and $\lim\limits_{t \to 0^+}\rho(t) = \infty$ then $\ct_\kappa(t) \leq \rho(t)$ on $(0, \min\set{b, \pi_\kappa})$.
        \end{enumerate}
    \end{proposition}
    \begin{proof}
        For the first inequality, assume by contradiction that $\rho\pa{t_0} > \ct_\kappa(t_0)$ for some $t_0 \in (0, b)$. By continuity, we can choose $0 < \epsilon < t_0$ such that
        \[
            \rho\pa{t_0} \geq \ct_\kappa\pa{t_0-\epsilon}.
        \]
        By~\Cref{lemma:comparison}, since $\ct_\kappa\pa{t - \epsilon}$ solves the differential equation $\dot{x} + x^2 + \kappa = 0$, we have that
        \[
            \rho\pa{t} \geq \ct_\kappa\pa{t-\epsilon} \mathrlap{\qquad \forall t \in (\epsilon, t_0).}
        \]
        But taking limits as $t$ tends to $\epsilon$ on the right, the right-hand side goes to infinity, which is a contradiction. In this case, we also have that $b \leq \pi_\kappa$ for $\rho$ to be differentiable, as $\lim\limits_{t \to \pi_\kappa^-} \ct_\kappa(t) = -\infty$ for $\kappa > 0$.

        We get the second inequality after interchanging the roles of $\rho$ and $\ct_\kappa$ in the previous argument. In this case, we need $t \in (0, \min\pa{b, \pi_\kappa})$ for $\ct_\kappa(t)$ to be differentiable as $\lim\limits_{t \to \pi_\kappa^-} \ct_\kappa(t) = -\infty$ for $\kappa > 0$.
    \end{proof}

    We now have everything we need to prove Rauch's theorem.

    \begin{theorem}[Rauch's theorem]\label{thm:rauch}
        Let $(M, \gm)$ be a Riemannian manifold with bounded sectional curvature $\delta \leq \sec \leq \Delta$. Fix a point $p \in M$ and define $r$ as the distance to $p$. For any other point in a ball $x \in B_p(\pi_\Delta) \backslash \set{p}$, and any $w \in T_xM$ we have
        \[
            \ct_\Delta(r(x))\norm{\normal{w}}
            \leq
            \pa{\Hess r}_x(w, w)
            \leq
            \ct_\delta(r(x))\norm{\normal{w}}.
        \]
        The upper bound also holds for every $x \in M \backslash \pa{\cut(p) \cup \set{p}}$. These bounds are tight on spaces of constant curvature.
    \end{theorem}
    \begin{proof}
        Since $x \not\in \cut{p}$, there exists a unique distance minimising geodesic between $x$ and $p$. Let $\deffun{\gamma :[0,r] -> M;}$ be that geodesic.

        If $w \in T_xM$ is a radial vector, assume that $w = \dgamma(r)$---the case $w = -\dgamma(r)$ follows by linearity. We then have that, by definition of the Hessian, since $r\vert_\gamma = \dgamma$,
        \[
            \Hess r(\dgamma, \dgamma) = \scalar{\conn_{\dgamma}r, \dgamma} = \scalar{\conn_{\dgamma}\dgamma, \dgamma} = 0,
        \]
        so the Hessian of the distance function is zero in the radial direction.

        For the rest of the bounds, we will take advantage of the simple form that the curvature equations take when evaluated on Jacobi and parallel vector fields, as computed in~\Cref{prop:radial_curvature_evaluated}.

        For the lower bound we consider a Jacobi field along $\gamma$ such that $J(r) = w$ and define $\rho(t) = \Hess r(\frac{J}{\norm{J}}, \frac{J}{\norm{J}})$. Computing its derivative
        \[
            \dot{\rho}
            = \frac{\dif}{\dif t} \frac{\Hess r(J, J)}{\scalar{J, J}}
            = \frac{\frac{\dif}{\dif t}\Hess r(J, J) \norm{J}^2 - 2\Hess r(J, J)^2}{\norm{J}^4}
            = \frac{\frac{\dif}{\dif t}\Hess r(J, J)}{\norm{J}^2} - 2\rho^2.
        \]
        and using the formula~\eqref{eq:curv1_jacobi} for the derivative of the Hessian on a Jacobi vector field, we get
        \[
            \dot{\rho} + 2\rho^2 - \Hess^2 r\pa[\Big]{\frac{J}{\norm{J}},\frac{J}{\norm{J}}} = -\sec(J, \dgamma).
        \]
        Using Cauchy--Schwarz, we can bound the iterated Hessian for any vector field $X$ of norm $1$
        \[
            \Hess^2 r(X,X) = \scalar{\conn_X\grad r, \conn_X\grad r} \geq \scalar{\conn_X \grad r, X}^2 = \Hess r(X,X)^2
        \]
        so taking $X = \frac{J}{\norm{J}}$ together with the upper bound on the sectional curvature we get the differential inequality
        \[
            \dot{\rho} + \rho^2 \geq - \Delta.
        \]
        Finally, since $\norm{J(0)} = 0$ and $\norm{\dot{J}(0)} \neq 0$, we have that $\lim\limits_{t \to 0^+} \rho(t) = \infty$, so we may use the second part of~\Cref{prop:riccati_comparison_estimate} to finish the lower bound.

        For the upper bound consider a parallel vector field $E$ such that $E(r) = w$ and define $\rho = \Hess r(E, E)$. We may then rewrite~\Cref{eq:curv2_parallel} for the derivative of $\Hess r$ on a parallel vector field as
        \[
            \dot{\rho} + \rho^2 = -\sec(E, \dgamma) \leq -\delta
        \]
        and we finish by applying the first part of~\Cref{prop:riccati_comparison_estimate}.
    \end{proof}

    \begin{remark}[Bounds on Jacobi fields]
        For the upper bound, we never used the fact that $x \not\in \cut{p}$. We merely used that the geodesic $\gamma$ does not have any conjugate points on $[0,r]$, to be able to use~\Cref{prop:riccati_comparison_estimate}, as if $\norm{J(t)} = 0$, then the Hessian of the distance function at that point would be infinite.

	As stated, the theorem is as general as it can be, as the distance function $r$ is not differentiable on $\cut{p}$. On the other hand, if it is stated in terms of Jacobi fields, one can further generalise it to Jacobi fields along geodesics without conjugate points, since both the proof of the theorem and that of~\Cref{prop:radial_curvature} can be done in terms of a geodesic and short variations of geodesics. We will use this more general version in the following theorem.
    \end{remark}

    \begin{theorem}[First order bounds for the exponential map]\label{thm:first_order_bounds}
        Let $(M, \gm)$ be a Riemannian manifold with bounded sectional curvature $\delta \leq \sec \leq \Delta$. Fix a point $p \in M$, for any unit vector $v\in T_pM$ and $r \in [0, \pi_\Delta]$,
        \[
            \min\set[\Big]{1, \frac{\sn_\Delta(r)}{r}}\norm{w}
            \leq
            \norm{\pa{\dif \exp_p}_{rv}(w)}
            \leq
            \max\set[\Big]{1, \frac{\sn_\delta(r)}{r}}\norm{w} \mathrlap{\qquad \forall w \in T_pM.}
        \]
        The upper bound also holds for $r < \conjpoint{v}$.

        These bounds are tight on spaces of constant curvature.
    \end{theorem}
    \begin{proof}
        The bound by $1$ above and below comes from the Gauss's lemma, as the exponential is a radial isometry. From this, and using linearity, we just have to prove the bound for a vector $w$ in the normal direction with $\norm{w} = 1$.

        All we have to do now is to transform the bounds on the Riccati equation on bounds on the Jacobi equation. Define $\gamma$ as the geodesic starting at $p$ with $\dgamma(0) = \frac{v}{\norm{v}}$, and choose a Jacobi field along $\gamma$ such that $J(r) = w$, that is, choose $\dot{J}(0) = \frac{w}{r}$. As we noted at the beginning of this section, we have that
        \[
        \Hess r\pa[\Big]{\frac{J}{\norm{J}}, \frac{J}{\norm{J}}}
        = \frac{\scalar{\dot{J}, J}}{\norm{J}^2}
        = \frac{\lfrac{\dif}{\dif t}\norm{J}}{\norm{J}}
        = \frac{\dif}{\dif t}\log\pa{\norm{J}}
        \]
        so Rauch's theorem may be rewritten on $(0, \pi_\Delta)$ as
        \[
            \frac{\dif}{\dif t}\log\pa{\sn_\Delta(t)}
            \leq
            \frac{\dif}{\dif t}\log\pa{\norm{J}}
            \leq
            \frac{\dif}{\dif t}\log\pa{\sn_\delta(t)}.
        \]
        Integrating, using that $\norm{J(0)} = \sn_{\Delta}(0) = \sn_{\delta}(0) = 0$, and computing the limit using l'Hôpital, we get
		\[
            \sn_\Delta(t) \leq \norm{J} \leq \sn_\delta(t).
		\]
		and since $J(r) = \pa{\dif \exp_p}_{rv}(rw)$,
		\[
            \norm{\pa{\dif \exp_p}_{rv}(w)} = \frac{\norm{J(r)}}{r}.
		\]

        As it was the case in the proof of Rauch's theorem, we just used that $\gamma$ does not have conjugate points on $[0,r]$, rather than the stronger $rv \in \segint{p}$. As such, just by noting that the proof does not rely on the definition of $\Hess r$, and instead, can be carried out in terms of Jacobi fields, we get the finer result for the upper bound.
    \end{proof}

    \begin{example}
        To see how the definition of the upper bounds are an improvement over just considering $rv \in \segint{p}$, consider the flat torus ($\mathbb{S}^1 \times \mathbb{S}^1$ with the product metric). For this manifold, $\segint{p} \subset T_pM$ is a square centred at $0$. On the other hand, since for the flat torus we have that $\conjloc{p} = \emptyset$, we get that the upper bound holds on all $T_pM$.
    \end{example}

    We make explicit the following strengthening of the Cartan--Hadamard theorem that was implicitly present in the statement of the last theorem.
    \begin{theorem}\label{thm:generalization_cartan_hadamard}
        If $(M, \gm)$ has $\sec \leq \Delta$, then the exponential map $\exp_p$ is not singular on $B_p\pa{ \pi_\Delta}$, so $B_p\pa{\pi_\Delta} \subset M \backslash \conjloc{p}$ and $\pi_\Delta \leq \conjpoint{v}$ for every point $p \in M$ and every unit vector $v \in T_pM$.
    \end{theorem}

    \subsection{The law of cosines on manifolds}\label{sec:law_of_cosines}
    Rauch's theorem has found countless applications in geometry, as those given by Rauch himself, who proved with it a weak version of the sphere theorem, or as used by Gromov or Karcher to give bounds the volume form of a Riemannian manifold, or to control the size of the fundamental group of a given manifold. Here we deviate slightly from the presentation to remark one that has been found particularly useful in the context of optimisation on manifolds.

    \begin{theorem}[Law of cosines]\label{thm:law_of_cosines}
        Let $(M, \gm)$ be a Riemannian manifold with bounded sectional curvature $\delta \leq \sec \leq \Delta$. Let $p$ be a point on $M$ and let $x, y \in B_p\pa{R}$ for $R \leq \pi_\Delta$ such that the minimising length geodesic $\gamma$ that connects them also lies in this ball. Define the angle $\alpha = \angle pxy$. We then have
        \begin{align*}
            d(y, p)^2 &\leq \zeta_{1, \delta}(R)d(x, y)^2 + d(x, p)^2 - 2d(x, y)d(x, p)\cos(\alpha) \\
            d(y, p)^2 &\geq \zeta_{2, \Delta}(R)d(x, y)^2 + d(x, p)^2 - 2d(x, y)d(x, p)\cos(\alpha)
        \end{align*}
        where
        \begin{alignat*}{2}
            \zeta_{1,\kappa}(r)
            &\defi \max\set{1, r\ct_\kappa(r)}
            &= \begin{cases}
                1                                    &\text{if } \kappa \geq 0 \\
                \sqrt{-\kappa}r\coth(\sqrt{-\kappa}r)&\text{if } \kappa < 0
            \end{cases}
            \\
            \zeta_{2,\kappa}(r)
            &\defi \min\set{1, r\ct_\kappa(r)}
            &= \begin{cases}
                \mathrlap1\hphantom{\sqrt{-\kappa}r\coth(\sqrt{-\kappa}r)}  &\text{if } \kappa \leq 0 \\
                \sqrt{\kappa}r\cot(\sqrt{\kappa}r)                          &\text{if } \kappa > 0
            \end{cases}
        \end{alignat*}
        Furthermore, the first bound holds for any $x,y \in M$ such that there exists a distance minimising geodesic contained in $\segm{p}$ that connects them. These bounds are tight on spaces of constant curvature.
    \end{theorem}
    \begin{proof}
         Define $f_p = \frac{1}{2}r^2$. We start by relating the Hessian of $f_p$ to the Hessian of $r$ using the Leibnitz rule twice
        \[
            \Hess f_p = \conn\pa{r \dif r} = \dif r \tensor \dif r + r\Hess r.
        \]
        As such, for a radial vector $w$ of norm $1$, since $\Hess r(w,w) = 0$, this equation simplifies to
        \[
            \Hess f_p(w,w) = \dif r(w)^2 = \scalar{\grad r, w}^2 = 1.
        \]
        For normal vectors we may use the bounds on the Hessian of $r$. It is now evident that the functions $\zeta_{1, \delta}$ and $\zeta_{2, \Delta}$ are defined to be upper and lower bound on the Hessian of $f_p$ acting on vectors of norm $1$.

        Let $\deffun{\gamma : [0, t] -> M;}$ be a distance-minimising geodesic such that $\gamma(0) = x$, $\gamma(t) = y$ and $\gamma(s) \in \segm{p}$ for every $s \in [0, t]$. We will prove the first inequality, as the proof of the second one is analogous. By the bounds above, using~\eqref{eq:hess_derivative_along_geodesic}, we have that
        \[
            \Hess f_p(\dgamma(s), \dgamma(s)) = (f_p \circ \gamma)''(s) \leq \zeta_{1, \delta}(R)
        \]
        for every $s \in [0, t]$, since $f_p$ is differentiable on $\segm{p}$. Integrating this inequality we get
        \[
            \intf[0][t]{(f_p \circ \gamma)''(s)}{s} = (f_p \circ \gamma)'(t) - \scalar{\grad f_p(\gamma(0)), \dgamma(0)} \leq \zeta_{1, \delta}(R)t.
        \]
        Integrating once again we have that
        \[
            (f_p \circ \gamma)(t) - (f_p \circ \gamma)(0) - \scalar{\grad f_p(\gamma(0)), \dgamma(0)}t \leq \zeta_{1, \delta}(R)\frac{t^2}{2}.
        \]
        which, after substituting the values of $f_p$ and $\gamma$, gives
        \[
            d(y, p)^2 \leq \zeta_{1, \delta}(R)d(x,y)^2 + d(x,p)^2 + 2d(x,y)\scalar{\grad f_p(x), \dgamma(0)}.
        \]
        Finally, since $\grad f_p(x) = r(x)\grad r(x) = -\exp_x^{-1}(p)$ we get the desired inequality.
    \end{proof}

    \begin{remark}[Obstructions to a global law of cosines]
        One could ask whether it is possible to extend this local result on $\segm{p}$ to a global result for geodesic triangles whose sides are distance minimising, as one does in Toponogov's theorem. This is, in general, not possible. Consider the flat torus and a point $p$ in it and an equilateral triangle on a generating circle. The sides of this triangle are distance minimising of length $x$---one third the circumference of the circle---and the angles of this triangle will be $\alpha = \pi$. Since for the flat torus we have that $\delta = 0$, the first inequality would imply that $3x^2 \leq 2x^2$.
    \end{remark}

    The upper bound for the case $\delta \leq 0$ has found multiples uses in the optimisation literature in the context of Hadamard manifolds (\cf, \Cref{sec:hadamard}). Rauch's theorem allows for a particularly clean proof of Cartan--Hadamard theorem, showing that these manifolds are diffeomorphic to $\RR^n$ through the Riemannian exponential map at any given point.
    \begin{corollary}[Cartan--Hadamard theorem]
        A simply connected Hadamard manifold is diffeomorphic to $\RR^n$.
    \end{corollary}
    \begin{proof}
        By~\Cref{thm:generalization_cartan_hadamard}, since $\sec \leq 0$, $\pi_\Delta = \infty$ and the differential exponential map at a point $p \in M$ is always full rank. As such, by the inverse function theorem, it is a covering map, and $\RR^n$ is the universal cover of $M$. Since $M$ is simply connected, $M$ is diffeomorphic to its universal cover.
    \end{proof}

    It is now clear that simply connected Hadamard manifolds are a class of particularly well-behaved manifolds, in that the $\cut{p} = \emptyset$ for every $p \in M$, so $\segm{p} = M$ and $\segint{p} = T_pM$, and all the theorems in~\Cref{sec:first_order_bounds} hold globally (\ie, for every $v \in T_pM$).

     If one then considers a simply connected Hadamard manifold with sectional curvature bounded below by $\delta < 0$, then one may apply the upper bounds in~\Cref{thm:law_of_cosines} on the whole manifold. This idea was heavily exploited in~\parencite{bonnabel2013stochastic} to get convergence rates for stochastic gradient descent, and in~\parencite{zhang2016riemannian} to prove convergence both stochastic and non-stochastic setting. In~\parencite{ferreira2019iteration}, the authors used the lower bounds whenever $\delta < 0$ to prove convergence rates for certain step-size schedulers. In~\parencite{bento2017iteration}, the authors use the law of cosines on manifolds of either non-negative or non-positive curvature to prove convergence rates for subgradient methods and proximal-point methods. In~\parencite{agarwal2020adaptive}, the lower bounds are used to prove convergence rates for certain second order method.

     In the previous presentation, we have showed that the Hadamard restriction, that is, being simply connected and with sectional curvature bounded above, is not really necessary. In particular, if we have a positive bound on the curvature, we can still apply this kind of results, at the expense of working on certain neighbourhood of the initial point. The same happens if we remove the condition of the manifold being simply connected. In this case, the first-order bounds on the exponential still hold globally, but the law of cosines only works on a ball of radius one half the length of the shortest geodesic loop.

    \section{Second Order Bounds for the Exponential Map}\label{sec:second_order_bounds}
    In this section, we give bounds on the Hessian of the differential of the exponential map. This kind of bounds were first given in the paper~\parencite{kaul1976schranken}. The proof can just be found in German, so we will give a self-contained presentation. Our proof uses a comparison lemma developed by Kaul, and then streamlines all the other technical tools, considerably simplifying the proof and obtaining tighter explicit bounds.

    \subsection{The differential equation}
    We start by giving a technical remark on the nature of the Hessian of the exponential map. The differential of the exponential for a point $(p,v) \in TM$ is a linear map of the form
    \[
        \deffun{\pa{\dif\exp_p}_v : T_pM -> T_{\exp_p(v)}M;}.
    \]
    For this reason, it can be seen as a section of the bundle $T^\ast_pM \tensor \exp_p^\ast(TM)$ over $T_pM$. We have connections $\connflat$ and $\conn$ on $T^\ast_pM$ and $\exp_p^\ast(TM)$ respectively. Any two connections induce a connection on a tensor product bundle so that the Leibnitz rule holds. In the concrete example of the exponential map, this Leibnitz rule for vector fields $\overline{W}_1, \overline{W}_2$ on $\segint{p}$, takes the form
    \[
        \conn_{\dif \exp_p(\overline{W}_1)} \pa{\dif \exp_p(\overline{W}_2)}  = \pa{\conn_{\overline{W}_1} \dif \exp_p}(\overline{W}_2) + \dif\exp_p\pa{\connflat_{\overline{W}_1}\overline{W}_2}.
    \]
    Since $\exp_p$ is a diffeomorphism between $\segint{p}$ and $\segm{p}$, writing $W_i = \dif \exp_p(\overline{W}_i)$ for the pushforward of a vector from $\segint{p}$ to $\segm{p}$, we see that the Hessian of the exponential map can be put in terms of the Christoffel symbols in normal coordinates as
    \begin{equation}\label{eq:conn_minus_conn_flat}
        \pa{\conn_{\overline{W}_1} \dif \exp_p}(\overline{W}_2) = \conn_{W_1} W_2 - \connflat_{\overline{W}_1}\overline{W}_2 = \Gamma(W_1, W_2) = \Gamma_{i,j}^kW_1^iW_2^j\partial_k.
    \end{equation}
    In other words, the Hessian of the exponential map is exactly the Christoffel symbols of the connection in normal coordinates at $p$ evaluated at the pushforward of the vector along $\exp_p$. In particular, this shows that the Hessian of the exponential map is a symmetric bilinear map of the form
    \[
        \deffun{\pa{\conn\dif\exp_p}_v : T_pM \times T_pM -> T_{\exp_p(v)}M;}.
    \]

    As we did for the first order bounds, our strategy to bound this quantity will be to deduce a differential equation for $\conn \dif \exp_p$ and then bound the norm of the solutions of this equation. In this case, we will have to differentiate a second time through a second variation of a geodesic. Given that this second derivative will not be in the direction of the geodesic, it will not be enough to have vector fields along the geodesic---we will have to have them defined also in the direction in which we want to differentiate them.

    \begin{proposition}\label{prop:diff_equation_second_order}
        Let $(M, \gm)$ be a Riemannian manifold and let $\gamma$ be a geodesic with initial unit vector $v \in T_pM$. For two vectors $w_1, w_2 \in T_pM$ perpendicular to $v$, the vector field along $\gamma$
        \[
            K(t) \defi \pa{\conn \dif \exp_p}_{tv}(tw_1, tw_2)
        \]
        satisfies the following second order inhomogeneous linear differential equation along $\gamma$
        \[
            \ddot{K} + R(K, \dgamma)\dgamma + Y = 0 \mathrlap{\qquad K(0) = 0, \dot{K}(0) = 0}
        \]
        where $Y$ is the vector field along $\gamma$ given by
        \[
            Y \defi
              2 R\pa{J_1, \dgamma}\dot{J}_2
            + 2 R\pa{J_2, \dgamma}\dot{J}_1
            + \pa{\conn_{\dgamma} R}\pa{J_2, \dgamma}J_1
            + \pa{\conn_{J_2} R}\pa{J_1, \dgamma}\dgamma
        \]
        and $J_1, J_2$ are the Jacobi fields along $\gamma$ with initial conditions $J_i(0) = 0$,  $\dot{J}_i(0) = w_i$.
    \end{proposition}
    \begin{proof}
        Define the geodesic variation
        \[
            \deffun{c : [0, r] \times \pa{-\epsilon, \epsilon} \times \pa{-\epsilon, \epsilon} -> M;
                    (t,s_1, s_2) -> \exp_p\pa{t\pa{v + s_1w_1 + s_2w_2}}}
        \]
        Consider the first variation along $\gamma(t) = \exp_p(tv)$ in the direction of $s_1$ given by the family of Jacobi fields
        \[
            \tilde{J}_1(t, s)= \left.\frac{\partial c}{\partial s_1}\right\vert_{s_1=0}(t, s) = \pa{\dif \exp_p}_{t(v + sw_2)}(tw_1).
        \]
        We denote by $J_i(t)$ the Jacobi field along $\gamma$ with initial conditions $(0, w_i)$. In particular, we have that $J_1(t) = \tilde{J}_1(t,0)$, so $\tilde{J}_1$ is an extension of $J_1$ in the direction of $J_2$ so that $\conn_{J_2}\tilde{J}_1$ is well-defined. Moreover, we have that $\connflat_{J_2}\tilde{J}_1 = 0$ as $\tilde{J}_1$ is constant in the direction of $J_2$, that is, in the direction of $stw_2$ for $s \in (-\epsilon, \epsilon)$ and a fixed $t$.
        For this reason, the vector field $K$ is the second variation along $\gamma$ in the directions $w_1, w_2$
        \[
            K(t) =
                \pa{\conn \dif \exp_p}_{tv}\pa{tw_1, tw_2} =
                \conn_{J_2} \tilde{J}_1 =
                \left.\frac{\partial \tilde{J}_1}{\partial s}\right\vert_{s=0}(t).
        \]

        Furthermore, $\tilde{J}_1$ satisfies the Jacobi equation in all its domain with initial conditions $(0, \frac{1}{\norm{v+sw_1}})$ so that
        \[
            \conn_{\grad r}\conn_{\grad r}\tilde{J}_1 + R(\tilde{J}_1, \grad r) \grad r = 0.
        \]
        We may then derive a differential equation for $K$ along $\gamma$ by differentiating this equation in the direction of $J_2$
        \begin{equation}\label{eq:first_deriv}
            \conn_{J_2} \conn_{\grad r} \conn_{\grad r} \tilde{J}_1 + \conn_{J_2}\pa{R(\tilde{J}_1, \grad r)\grad r} = 0.
        \end{equation}
        Note that, as we just are interested in deriving a differential equation for $K$ along $\gamma$, we will use that $\grad r \vert_\gamma = \dgamma$ and $\tilde{J}_1 = J_1$ whenever we have that a quantity just depends on the values of the vector fields involved along $\gamma$, rather than in a neighbourhood  around $\gamma$ in the direction of $J_2$.

        We may put the first term of~\eqref{eq:first_deriv} in terms of $K$ by repeatedly using the definition of the curvature tensor and the fact that since $[J_2, \grad r]\vert_\gamma = 0$, we have that $\conn_{J_2}\grad r = \conn_{\dgamma}J_2 = \dot{J}_2$
        \begin{align*}
            \conn_{J_2} \conn_{\grad r} \conn_{\grad r} \tilde{J}_1
            &= \conn_{\dgamma} \conn_{J_2} \conn_{\grad r} \tilde{J}_1 + R(J_2, \dgamma)\conn_{\dgamma} J_1 \\
            &= \conn_{\dgamma} \conn_{\dgamma} K + \conn_{\dgamma}\pa{R\pa{J_2, \dgamma}J_1} + R(J_2, \dgamma)\dot{J}_1 \\
            &= \conn_{\dgamma} \conn_{\dgamma} K + \pa{\conn_{\dgamma} R}\pa{J_2, \dgamma}J_1
                + R\pa{\dot{J}_2, \dgamma}J_1
                + 2R\pa{J_2, \dgamma}\dot{J}_1.
        \end{align*}
        We can expand the second term as
        \[
            \conn_{J_2}\pa{R(\tilde{J}_1, \grad r)\grad r}
            = \pa{\conn_{J_2} R}\pa{J_1, \dgamma}\dgamma
            + R\pa{K, \dgamma}\dgamma
            + R\pa{J_1, \dot{J}_2}\dgamma
            + R\pa{J_1, \dgamma} \dot{J}_2.
        \]
        Putting everything together and using the symmetries of the curvature tensor and the first Bianchi identity we get the differential equation.

        For the initial conditions, we have that
        \[
            K(0)=\pa{\conn \dif \exp_p}_0(0,0) = 0.
        \]
        Differentiating the Hessian in the direction of $\dgamma = \frac{\dif}{\dif t}$ we have that
        \[
        \dot{K}(0) = \pa{\conn\conn \dif \exp_p}_0(v, 0, 0) + \pa{\conn \dif\exp_p}_0(w_1, 0) + \pa{\conn \dif\exp_p}_0(0, w_2) = 0.\qedhere
        \]
    \end{proof}

    \subsection{Manifolds of bounded geometry}\label{sec:bounded_geometry}
    By~\Cref{prop:diff_equation_second_order}, we have that the Hessian of the exponential map solves a Jacobi-like differential equation with two main differences: It is inhomogeneous and depends on the covariant derivative of the curvature tensor.

    The inhomogeneity of the differential equation will stop us from simplifying the differential equation into two first order differential equations, as we did for the Jacobi equation. In contrast, this time we will have to deal directly with the second order equation. Luckily, we have already developed most of the tools to do so.

    The fact that involves the covariant derivative of the curvature tensor is a more intrinsic difference. Since the sectional curvature completely defines the curvature tensor, bounds on the sectional curvature can be translated into bounds on the norm of the curvature tensor. One question that naturally arises is whether these $0$-th order bounds also give first order bounds. As one may expect, this is not the case.

    \begin{example}[Manifold of bounded $0$-th order and unbounded $1$-st order]
        Consider a rotationally symmetric surface of the form $\dif r^2 + \rho(r)^2\dif \theta^2$ on the cylinder $(0,1) \times \SS^1$ for a function $\rho > 0$. These metrics have a particularly simple formula for the sectional curvature
        \[
            \sec(\partial_r, \partial_\theta) = \frac{R(\partial_r, \partial_\theta, \partial_\theta, \partial_r)}{\rho^2}= -\frac{\ddot{\rho}}{\rho}.
        \]
        All we have to do now is to choose a positive function $\rho$ with bounded second derivative and arbitrarily large third derivative, for example $\rho(r) = 2 + r^5 \sin(1/r)$. This metric has $\abs{\sec} < 12$, but
        \[
        \pa{\conn_{\partial_r}R}\pa{\partial_r, \partial_\theta, \partial_\theta, \partial_r}
        = D_{\partial_r}\pa{R\pa{\partial_r, \partial_\theta, \partial_\theta, \partial_r}}
        -2R\pa{\partial_r, \conn_{\partial_r}\partial_\theta, \partial_\theta, \partial_r}
        = \dot{\rho}\ddot{\rho} - \rho \dddot{\rho}
        \]
        which is unbounded as $r$ goes down to $0$.
    \end{example}

    This example motivates the following definition.
    \begin{definition}[Manifold of bounded geometry]\label{def:bounded_geometry}
        Let $(M, \gm)$ be a Riemannian manifold. We say that a manifold has \textbf{$k$-bounded geometry} if for every $i = 0, \dots, k$ there exist constants $C_i \geq 0$ such that
        \[
            \norm{\conn^i R} \leq C_i
        \]
        and the injectivity radius is uniformly bounded below by a positive constant
        \[
            \inj \defi \inf_{p \in M}\inj(p) > 0.
        \]
    \end{definition}

    \begin{remark}
        A manifold has $0$-bounded geometry if it has bounded sectional curvature above and below and positive injectivity radius. As one may expect, given that the $1$-bounded geometry is given by the Christoffel symbols, which are given by directional derivatives of the metric $\gm$ in normal coordinates, the condition on the boundedness of the derivatives of the curvature tensor is equivalent to asking for the entries of the metric tensor to be bounded in normal coordinates. This was first proved in~\parencite{eichhorn1991boundedness}.
    \end{remark}

    It is clear that any compact manifold with any given metric will be of $k$-bounded geometry for every $k \in \mathbb{N}$. A natural question would be whether this condition puts any constraint in the topology of the manifold. This is not the case.
    \begin{theorem}[{\cite{greene1978complete}}]
        Every differentiable manifold admits a complete metric of bounded geometry.
    \end{theorem}

    Given that we usually have access to upper and lower bounds on the sectional curvature, and given the form that takes the covariant derivatives in~\Cref{prop:diff_equation_second_order}, we give the following definition, which is just a generalisation of that given in the introduction.
    \begin{definition}\label{def:weird_bounded_geometry}
        We say that a Riemannian manifold $(M, \gm)$ has \textbf{$(\delta, \Delta, \Lambda)$-bounded geometry} if $\inj > 0$ and
        \begin{gather*}
            \delta \leq \sec \leq \Delta\\
            \norm{\pa{\conn_x R}(y, x)y + \pa{\conn_y R}\pa{y, x}x} \leq 2\Lambda\norm{x}^2\norm{y}^2 \mathrlap{\qquad \forall p\in M\quad \forall x,y \in T_pM.}
        \end{gather*}
    \end{definition}

    \begin{remark}
        If a manifold is of $(\delta, \Delta, \Lambda)$-bounded geometry according to~\Cref{def:bounded_first_order} it is of $(\delta, \Delta, \Lambda)$-bounded geometry according to~\Cref{def:weird_bounded_geometry}, but not necessarily the other way around. In the proofs, we will just use~\Cref{def:weird_bounded_geometry}, as it is all we need. This definition also simplifies the computations when one wants to estimate the actual constant $\Lambda$ for a concrete manifold.
    \end{remark}

    \begin{remark}[Injectivity radius]
        We will not use the fact that these manifolds have uniformly bounded injectivity radius in this chapter, but this fact turns out to be crucial when proving other flavour of theorems in optimisation of manifolds that involve $\exp^{-1}_p$ or parallel transport along geodesics or a retraction.
    \end{remark}

    \subsection{Curvature bounds}
    We start by recalling the following lemma that gives a closed formula for the curvature tensor of manifolds of constant sectional curvature. In its more general form, it says that the Riemannian curvature tensor is completely specified by the sectional curvature.

    \begin{lemma}[Riemann, 1854]\label{lemma:constant_curvature}
        Let $(M, \gm)$ be a Riemannian manifold of constant sectional curvature $\kappa \in \RR^n$. The curvature tensor takes the form
        \[
            R(x,y)z \defi R_\kappa(x, y)z = \kappa\pa{\scalar{z, y}x - \scalar{z, x}y} \mathrlap{\qquad \forall x,y,z \in T_pM.}
        \]
    \end{lemma}

    To be able to bound the terms in the differential equation, we will need some estimates on the norm of the curvature tensor in terms of the sectional curvature. We will use $4$ inequalities that can be regarded as a generalisation of Berger's lemma~\parencite{berger1960sur}. The first three inequalities are announced---the last one with a worse constant---in~\parencite{kaul1976schranken} citing for their proof a monograph by Karcher that was never published. We provide original proofs for these inequalities. The third inequality appears in~\parencite{karcher1970short}, but with a different proof. The last inequality is entirely from~\parencite{karcher1970short}.

    \begin{proposition}\label{prop:curvature_bounds}
        Let $(M, \gm)$ be a Riemannian manifold with bounded sectional curvature $\delta \leq \sec \leq \Delta$ and define
        \[
            \epsilon = \frac{\Delta - \delta}{2} \qquad \mu = \frac{\Delta + \delta}{2} \qquad K = \max\set{\abs{\Delta}, \abs{\delta}}.
        \]
        Given a point $p \in M$, we have the following inequalities
        \begin{align}
            \abs{R(x, y, y, w) - R_\mu(x, y, y, w)} &\leq \epsilon\norm{x}\norm{y}^2\norm{w}\mathrlap{\qquad \forall x, y, w \in T_pM} \label{eq:inequality_R_3}\\
            \abs{R(x, y, z, w) - R_\mu(x, y, z, w)} &\leq \frac{4}{3}\epsilon\norm{x}\norm{y}\norm{z}\norm{w}\mathrlap{\quad \forall x, y, z, w \in T_pM}\label{eq:inequality_R_4}
        \end{align}
        The constants $1$ and $\frac{4}{3}$ are tight.

        Furthermore, if $e \in T_pM$ is a unitary vector, defining $\normal{R}(x,y)z$ as the component of the curvature tensor perpendicular to $e$, we have that for every $y,z$ perpendicular to $e$
        \begin{align}
            \norm{\normal{R}(e,y)z} &\leq \frac{4}{3}\epsilon\norm{y}\norm{z} \label{eq:bound_R_normal}\\
            \norm{R(e,y)z} &\leq 2K.\label{eq:bound_R_all}
        \end{align}
    \end{proposition}
    \begin{proof}
        We may assume that the vectors $x,y,z,w$ are of norm $1$ by linearity. For~\eqref{eq:inequality_R_3}, since the curvature tensor is skew-symmetric in the first two components and the last two components, we may assume that $x,w$ are perpendicular to $y$. Fix a vector $y$ and consider the bilinear form
        \[
            T_y(x,w) = R(x,y,y,w) - R_\mu(x,y,y,w).
        \]
        This application is symmetric, and as such, it attains its maximum at an eigenvector $x_1$ of norm one orthogonal to $y$ so that for every $x$
        \[
            \abs{T_y(x,w)} \leq \abs{R(x_1,y,y,x_1) - R_\mu(x_1,y,y,x_1)} = \abs{\sec(x_1, y) - \mu} \leq \epsilon.
        \]

        For the second inequality, by linearity we can assume that all the vectors have the same norm. By polarisation, we have that
        \begin{align*}
            6R(x,y,z,w) &= R(x, y+z, y+z, w) - R(x, y-z, y-z, w) \\
                        &-R(y, x+z, x+z, w) + R(y, x-z, x-z, w).
        \end{align*}
    We then apply the triangle inequality to the expression for $(R-R_\mu)(x,y,z,w)$ together with~\eqref{eq:inequality_R_3} to get
        \[
            6\abs{(R-R_\mu)(x,y,z,w)} \leq \epsilon\pa[\Big]{\norm{x}\norm{w}\pa{\norm{y+z}^2 + \norm{y-z}^2} + \norm{y}\norm{w}\pa{\norm{x+z}^2+\norm{x-z}^2}}
        \]
        and using the parallelogram law together with the fact that all the vectors have the same norm, we get that
        \[
            \norm{x+y}^2 + \norm{x-y}^2 = 2(\norm{x}^2 + \norm{y}^2) = 4\norm{x}\norm{y}
        \]
        and the second inequality follows.

        These two inequalities are tight on $\CC P^n$ seen as a real manifold~\parencite{karcher1970short}.

        The bound on the normal part of the curvature tensor follows directly from~\eqref{eq:inequality_R_4}, since $R_\mu(e,y,z,w) = 0$ whenever $y,z,w$ are perpendicular to $e$, so choosing $w$ as the unitary vector in the direction of $\normal{R}(e,y)z$ we get that
        \[
            \norm{\normal{R}(e,y)z} = \abs{R(e,y,z,w)} \leq \frac{4}{3}\epsilon\norm{y}\norm{z}.
        \]

        The last inequality is proved in~\parencite{karcher1970short}.
    \end{proof}

    \subsection{A comparison lemma}\label{sec:comparison_lemma}
    We shall proceed in a very similar way to how we did in the case of the first order equation. We will simplify the differential equation to one in one dimension, and there, we will use a comparison theorem for functions of real variable. The only difference is that, in this case, we will not be able to simplify the computations to a Riccati equation. We start by proving the second order version of the Riccati comparison lemma but for the Jacobi equation. This lemma is stated in~\parencite{kaul1976schranken}, but the reference given is in German and does not prove this inequality. We provide here a proof for completeness.
    \begin{lemma}[Jacobi comparison lemma]\label{lemma:comparison_edo}
        Fix a real number $\kappa \in \RR$, and let $\deffun{f, g : [0, r] -> \RR^+;}$ be two functions such that
        \[
            \ddot{f} + \kappa f \leq \ddot{g} + \kappa g \mathrlap{\qquad f(0) = g(0), \dot{f}(0) \leq \dot{g}(0).}
        \]
        Then $f \leq g$ on $[0, \min\set{r, \pi_{\kappa}}]$.
    \end{lemma}
    \begin{proof}
        Let $h = g - f$ and $\zeta = \ddot{h} + \kappa h \geq 0$. Let us show that the solution to the linear inhomogeneous initial value problem
        \[
            \ddot{h} + \kappa  h = \zeta\mathrlap{\qquad h(0) = 0,\ \dot{h}(0) \geq 0}
        \]
        is indeed positive on the given interval. Solving the equation, we find that the solution is given by
        \begin{align*}
            h(x)
            &= \sn_\kappa(x)\intf[0][x]{\sn'_\kappa(t)\zeta(t)}{t} - \sn_\kappa'(x)\intf[0][x]{\sn_\kappa(t)\zeta(t)}{t} + \dot{h}(0)\sn_\kappa(x)\\
            &= \intf[0][x]{\sn_\kappa(x-t)\zeta(t)}{t} + \dot{h}(0)\sn_\kappa(x).
        \end{align*}
        where we have used the trigonometric identity
        \[
            \sn_\kappa(x-t) = \sn_\kappa(x)\sn'_\kappa(t) - \sn'_\kappa(x)\sn_\kappa(t).
        \]
        From this we see that $h(x) \geq 0$ on $[0, \min\set{r, \pi_\kappa}]$ as $\sn_\kappa(x-t)$, $\zeta(t)$, $\dot{h}(0)$ and $\sn_\kappa(x)$ are positive in this interval.
    \end{proof}

    We can now show how to estimate the inhomogeneous differential equation for the Hessian of the exponential, taking advantage of the fact that the initial conditions are both zero.
    \begin{proposition}[Kaul, 1976]\label{prop:comparison_second_order}
        Let $(M, \gm)$ be a Riemannian manifold with bounded sectional curvature $\delta \leq \sec \leq \Delta$.
        Let $\deffun{\gamma : [0, r] -> M;}$ be a geodesic, and let $X, Y$ be vector fields along $\gamma$ with $X, Y \perp \dgamma$ such that
        \[
            \ddot{X} + R(X, \dgamma)\dgamma = Y \mathrlap{\qquad X(0) = 0, \dot{X}(0) = 0.}
        \]
        Assume that there exists a continuous $\eta$ function such that $\norm{Y} \leq \eta$ on $[0,r]$. Then, we have that $\norm{X} \leq \rho$ on $[0, \min\set{r, \pi_{\frac{\Delta + \delta}{2}}}]$, where $\rho$ is the solution of
        \[
            \ddot{\rho}+\delta\rho = \eta \mathrlap{\qquad \rho(0) = 0, \dot{\rho}(0) = 0.}
        \]
    \end{proposition}
    \begin{proof}
        We define again the quantities
        \[
            \epsilon = \frac{\Delta - \delta}{2} \qquad \mu = \frac{\Delta + \delta}{2}.
        \]
        Fix a $t_0 \in [0, r]$ and let $E$ be the parallel vector field along $\gamma$ such that $E(t_0) = \frac{X(t_0)}{\norm{X(t_0)}}$. The function $f = \scalar{X, E}$ satisfies the differential equation
        \begin{align*}
            \ddot{f} + \mu f
            &= \scalar{\ddot{X} + \mu X, E}\\
            &= \scalar{Y - R\pa{X, \dgamma}\dgamma + R_{\mu}\pa{X, \dgamma}\dgamma, E}\\
            &\leq \norm{Y} + \epsilon \norm{X}
        \end{align*}
        where we have used~\Cref{lemma:constant_curvature} in the first equality and~\eqref{eq:inequality_R_3} for the bound.

        Define $g$ as the solution to the differential equation
        \[
            \ddot{g} + \mu g = \norm{Y} + \epsilon\norm{X} \mathrlap{\qquad g(0) = \dot{g}(0) = 0.}
        \]
        Using that $f(t_0) = \norm{X(t_0)}$ and that $f(0) = \dot{f}(0) = 0$, by~\Cref{lemma:comparison_edo}, we get that  $\norm{X(t_0)} \leq g(t_0)$, and since the definition of $g$ does not depend on the chosen $t_0$, we may do this for any $t_0 \in [0,r]$ getting that $\norm{X} \leq g$ on $[0, \min\set{\pi_\mu, r}]$. Now, Using that $\norm{X} \leq g$, we get that $g$ satisfies the inequality
        \[
            \ddot{g} + \delta g \leq \norm{Y}\leq \eta.
        \]
        So, for the solution of
        \[
            \ddot{\rho}+\delta\rho = \eta \mathrlap{\qquad \rho(0) = 0, \dot{\rho}(0) = 0.}
        \]
        we have that $g \leq \eta$ on $[0, \min\set{\pi_\delta, r}] \supseteq [0, \min\set{\pi_\mu, r}]$, getting the result.
    \end{proof}

    \subsection{A second order version of Rauch's theorem}
    We are now ready to give bounds on the Hessian of the exponential map. We first note that, since our goal is to bound the norm of the Hessian of $f \circ \exp_p$, as this Hessian is symmetric, it attains its maximum at an eigenvector. As such, we just need to bound the quantity
    \[
        \pa{\conn\dif \exp_p}_v(w,w) \mathrlap{\qquad v \in \segint{p}, w \in T_pM.}
    \]
    For this reason, we will start by giving bounds on the diagonal of the Hessian of the exponential map. We will then see that, since the Hessian is a symmetric bilinear map, we can leverage these bounds to give bounds on the full Hessian for any pair of vectors $w_1, w_2 \in T_pM$.

    We give the second order bounds on the exponential map for a manifold of $(\delta, \Delta, \Lambda)$-bounded geometry (\cf~\Cref{def:weird_bounded_geometry}).

    \begin{theorem}[Second order bounds for the exponential map]\label{thm:second_order_bounds}
        Let $(M, \gm)$ be a Riemannian manifold with $(\delta, \Delta, \Lambda)$-bounded geometry. For a geodesic $\deffun{\gamma : [0,r] -> M;}$ with initial unit vector $v$, and a vector $w \in T_pM$, we have that
        \begin{itemize}
        \item If $w$ is radial to $\dgamma(0)$,
            \[
                \pa{\conn\dif\exp_p}_{rv}(w, w) = 0.
            \]
        \item If $w$ is normal to $\dgamma(0)$, the radial part of the Hessian is bounded as
            \[
                \pa[\Big]{\frac{1}{r} - \frac{\sn_{4\delta}(r)}{r^2}}\norm{w}^2
                \leq
                \scalar{\pa{\conn\dif \exp_p}_{rv}\pa{w,w}, \dgamma(r)}
                \leq
                \pa[\Big]{\frac{1}{r} - \frac{\sn_{4\Delta}(r)}{r^2}}\norm{w}^2
            \]
            for $r < \pi_\Delta$ for the upper bound and $r < \conjpoint{v}$ for the lower bound.

            The normal part of the Hessian is bounded for $r < \pi_{\frac{\Delta + \delta}{2}}$ as
            \[
                \norm{\normal{\pa{\conn\dif \exp_p}_{rv}\pa{w,w}}}
                \leq
                \rho(r) \norm{w}^2
            \]
            where
            \[
                \rho(t) =
                \lfrac{8}{9r^2}\sn_\delta\pa[\big]{\lfrac{t}{2}}^2
                \pa[\big]{3\Lambda\sn_\delta\pa[\big]{\lfrac{t}{2}}^2
                +2\pa{\Delta - \delta}\sn_\delta\pa{t}}.
            \]
        \end{itemize}
        Both bounds and their radii are tight in spaces of constant curvature.

        Furthermore, the radius of convergence for the normal part is tight for $\SO{n}$.
    \end{theorem}
    \begin{proof}
        As in~\Cref{prop:diff_equation_second_order}, we write
        \begin{align*}
            \tilde{J}(t, s)&= \pa{\dif \exp_p}_{t(v + sw)}\pa[\Big]{\frac{t}{r}w}\\
            J(t) &= \tilde{J}(t,0) \\
            K(t) &=\pa{\conn \dif \exp_p}_{tv}\pa[\Big]{\frac{t}{r}w, \frac{t}{r}w}.
        \end{align*}
        By linearity of the Hessian, it is enough to prove the result for $\norm{w} = 1$.

        If $w$ is radial, in~\Cref{prop:diff_equation_second_order} we have that $Y = 0$, so $K$ is a solution of the equation
        \[
            \ddot{K} + R(K, \dgamma)\dgamma = 0 \mathrlap{\qquad K(0) = 0, \dot{K}(0) = 0}
        \]
        and since it is a homogeneous second order linear equation with zero as the initial condition, its solution is $K(t) = 0$ for $t \in [0, r]$.

        If $w$ is normal, we need to bound the quantity $\conn_J \tilde{J}$. Note that this is a vector field along $\gamma$, but to take the derivative in the direction of $J$ we need to have $\tilde{J}$ defined in that direction as well. We start by bounding its radial part
        \begin{equation}\label{eq:radial_part_hessian}
            \scalar{\conn_J \tilde{J}, \grad r}
            = D_J\scalar{\tilde{J}, \grad r}- \scalar{\tilde{J}, \conn_J \grad r}
            = D_J\scalar{\tilde{J}, \grad r}- \Hess r(J, J).
        \end{equation}
        We can compute the first term directly. By Gauss's lemma we can simplify this derivative to one on $T_pM$
        \[
            D_J\scalar{\tilde{J}, \grad r}(\gamma(t))
            = \frac{1}{r}\left.\frac{\dif}{\dif s}\right\vert_{s=0}\scalar[\Big]{\frac{t}{r}w, \frac{t\pa{v+sw}}{\norm{t\pa{v +sw}}_{T_pM}}}_{T_pM}
            = \frac{t}{r^2}.
        \]
        Using the bounds on the Hessian of the distance function given in~\Cref{thm:rauch} we can bound the second term in~\eqref{eq:radial_part_hessian}. Evaluating these two quantities at $r$ and using the bounds on the size of the Jacobi fields together with the trigonometric identity
        \begin{equation}\label{eq:trig}
            \sn_\kappa(t)\sn'_\kappa(t) = \frac{\sn_\kappa(2t)}{2} = \sn_{4\kappa}(t)
        \end{equation}
        we get the bounds on the tangential part of the Hessian of the exponential.

        Finally, for its normal part, consider the differential equation given by~\Cref{prop:diff_equation_second_order}. Since $\scalar{K, \dgamma}' = \scalar{\dot{K},\dgamma}$, the radial (resp.\ normal) part of the derivative is the derivative of the radial (resp.\ normal) part. For this reason, we have that
        \[
            \pa{\normal{K}}'' + R(\normal{K}, \dgamma)\dgamma = -\normal{Y}.
        \]
        Therefore, we just have to bound $\norm{\normal{Y}}$ to be able to use~\Cref{prop:comparison_second_order} and finish.

        We start by giving a bound on the norm of $\dot{J}$. Using that $\dot{J} = \conn_J\grad r = \Hess r(J)$,
        \begin{equation}\label{eq:bound_derivative_jac}
            \norm{\dot{J}} \leq \norm{\Hess r} \norm{J}.
        \end{equation}
        We can then bound the norm of $\normal{Y}$ as
        \begin{align*}
            \norm{\normal{Y}}
            &\leq
             \norm{\pa{\conn_{\dgamma} R}\pa{J, \dgamma}J +
             \pa{\conn_J R}\pa{J, \dgamma}\dgamma}
              + 4 \norm{\normal{R}\pa{J, \dgamma}\dot{J}}\\
            &\leq 2\Lambda \norm{J}^2 + \frac{8(\Delta - \delta)}{3}\norm{J}\norm{\dot{J}}\\
            &\leq \pa[\Big]{2\Lambda + \frac{8(\Delta - \delta)}{3}\ct_\delta(t)}\norm{J}^2\\
            &\leq \pa[\Big]{2\Lambda + \frac{8(\Delta - \delta)}{3}\ct_\delta(t)}\frac{\sn_\delta(t)^2}{r^2}\\
            &= \frac{2}{r^2}\pa[\Big]{\Lambda\sn_\delta(t)^2 + \frac{2(\Delta-\delta)}{3}\sn_\delta(2t)}.
        \end{align*}
        where we have used~\eqref{eq:bound_R_normal}---given that since $J$ is perpendicular to $\dgamma$, so is $\dot{J}$---to bound the normal part of the curvature tensor. We have also used the first order bound on the differential of the exponential for perpendicular initial conditions to bound the norm of the Jacobi fields and~\eqref{eq:bound_derivative_jac} to bound their derivative. In the last equality we have used~\eqref{eq:trig} again.

        The result follows by noting that $\rho$ is the solution to the differential equation
        \[
            \ddot{\rho} + \delta \rho = \eta \mathrlap{\qquad \rho(0) = \dot{\rho}(0)=0}
        \]
        where $\eta$ is the bound on $\norm{\normal{Y}}$ and applying~\Cref{prop:comparison_second_order}.

        We will see that the radius for the bound of the normal part is tight in the case of $\SO{n}$ in~\Cref{ex:son}.
    \end{proof}

    The first thing to note is that these bounds go to zero as $r$ tends to zero. This is exactly what we expect, as the Christoffel symbols in normal coordinates vanish at the origin.

    The bounds in this theorem provide a notable improvement compared to the best bounds previously known (\cf~\parencite{kaul1976schranken}). For one, these bounds are tighter. We also have that these bounds are explicit, compared to the previous bounds, which were given in terms of a solution of a differential equation that did not have an explicit integral. We also simplified the technical tools necessary to get to these bounds.

    We finish this section by giving a bound on the full Hessian of the exponential.
    \begin{theorem}[Bounds on the Full Hessian]\label{thm:full_second_order_bounds}
        Let $(M, \gm)$ be a Riemannian manifold with $(\delta, \Delta, \Lambda)$-bounded geometry. For a geodesic $\deffun{\gamma : [0,r] -> M;}$ with initial unit vector $v$, $r < \pi_{\frac{\Delta + \delta}{2}}$, and any two vectors $w_1, w_2 \in T_pM$, we have that
        \[
            \norm{\pa{\conn\dif \exp_p}_{rv}\pa{w_1,w_2}} \leq
                \lfrac{8}{3r^2}\sn_\delta\pa[\big]{\lfrac{r}{2}}^2
                \pa{\Lambda\sn_\delta\pa[\big]{\lfrac{r}{2}}^2
                +2\max\set{\abs{\Delta}, \abs{\delta}}\sn_\delta\pa{r}}\norm{w_1}\norm{w_2}.
        \]
        Furthermore, the radius $\pi_{\frac{\Delta + \delta}{2}}$ is tight for $\SO{n}$.
    \end{theorem}
    \begin{proof}
        To bound the full Hessian we first see that it is just enough to bound its diagonal part. For any symmetric bilinear form and any two vectors we have the polarisation formula
        \[
            4\Phi(u,v) = \Phi(u+v, u+v) - \Phi(u-v, u-v).
        \]
        By linearity, we may assume that $\norm{u} = \norm{v} = 1$. Taking absolute values and applying the triangle inequality and Cauchy--Schwarz, we get that
        \[
            4\norm{\Phi(u,v)} \leq \norm{\Phi \vert_{\operatorname{diag}}}\pa{\norm{u+v}^2 + \norm{u-v}^2} = 4\norm{\Phi \vert_{\operatorname{diag}}}
        \]
        where $\norm{\Phi \vert_{\operatorname{diag}}}$ is the operator norm of the application $u \mapsto \Phi(u,u)$. For this reason, it is enough to bound the map $u \mapsto \pa{\conn\dif\exp_p}_{tv}\pa{u,u}$.

        Let $w$ be a vector normal to $v$. As we did in~\Cref{thm:second_order_bounds}, we consider the differential equation for
        \[
            K(t) =\pa{\conn \dif \exp_p}_{tv}\pa[\Big]{\frac{t}{r}w, \frac{t}{r}w}.
        \]
        This is exactly the same differential equation that we had for $\normal{K}$, only that rather than having to bound the normal part of the curvature tensor, we bound the full curvature tensor. For that we use~\eqref{eq:bound_R_all}, and following with the bounds as we did in~\Cref{thm:second_order_bounds} and solving the resulting equation we get that for $w$ perpendicular to $v$, the norm of $K$ is bounded by the solution of the differential equation
        \[
            \ddot{\rho} + \delta \rho
        = \frac{2}{r^2}\pa{\Lambda\sn_\delta(t)^2 + 2\max\set{\abs{\Delta},\abs{\delta}}\sn_\delta(2r)}
             \mathrlap{\qquad \rho(0) = \dot{\rho}(0)=0}
        \]
        which is solved by
        \[
            \rho(t) =
                \lfrac{8}{3r^2}\sn_\delta\pa[\big]{\lfrac{t}{2}}^2
                \pa{\Lambda\sn_\delta\pa[\big]{\lfrac{t}{2}}^2
                +2\max\set{\abs{\Delta}, \abs{\delta}}\sn_\delta\pa{t}}.\qedhere
        \]
    \end{proof}

    \begin{remark}[Tighter bounds]
        Another way to obtain bounds on the full Hessian of the exponential would be to take the bounds from~\Cref{thm:second_order_bounds} for the normal and parallel part of the Hessian and using the Cauchy--Schwarz inequality to get bounds of the form
        \[
            \norm{\pa{\conn\dif \exp_p}_{rv}\pa{w_1,w_2}} \leq
                \sqrt{\sigma^2 + \rho^2}\norm{w_1}\norm{w_2}.
        \]
        Where $\rho$ is the bound of the normal part defined in~\Cref{thm:second_order_bounds} and
        \[
            \sigma = \frac{1}{r^2}\max\set{\abs{r-\sn_{4\delta}(r)}, \abs{r-\sn_{4\Delta}(r)}}
        \]
        that is, $\sigma$ is the bound on the norm of the parallel part of the Hessian.

        This bound, although tighter in most specific examples, might be more difficult to manipulate and lacks the simplicity of that presented in the theorem above.
    \end{remark}

\section{Concrete Second Order Bounds}\label{sec:concrete_bounds}
\subsection{Constant curvature}
    The simplest family to evaluate these bounds in is that of the spaces of constant curvature such as the sphere, the hyperbolic plane, the Euclidean space or the flat torus.

    For a manifold of constant curvature $\kappa$, since the curvature is constant, the derivative of the curvature tensor is zero. If we further assume that their injectivity radius is positive\footnote{This is not really used to prove the bounds, but it is part of the definition of bounded geometry.}, as is the case of the hyperbolic space and the sphere, we have that they have $(\kappa, \kappa, 0)$-bounded geometry. Instantiating the bounds for these spaces we see that
    \begin{align*}
        \scalar{\pa{\conn\dif \exp_p}_{rv}\pa{w,w}, \dgamma(r)}
        &=
        \pa[\Big]{\frac{1}{r} - \frac{\sn_{4\kappa}(r)}{r^2}}\norm{w}^2 \\
        \norm{\normal{\pa{\conn\dif \exp_p}_{rv}\pa{w,w}}}
        &= 0
    \end{align*}
    for $r < \pi_\kappa$. As a first sanity check, we see that for the flat case $\kappa = 0$, the radial part is exactly equal to zero, as the whole Hessian of the exponential map is everywhere zero in this case---the second derivative of an affine function is zero.

    To check that this solution is actually correct, we can solve the differential equation in~\Cref{prop:diff_equation_second_order} for these spaces for
    \[
        K(t) \defi \pa{\conn \dif \exp_p}_{tv}(tw_1, tw_2).
    \]
    Using~\Cref{lemma:constant_curvature}, together with the fact that, since the sectional curvature is constant, $\conn R = 0$, and the trigonometric identity
    \[
        \sn_\kappa(t)\sn'_\kappa(t) = \frac{\sn_\kappa(2t)}{2} = \sn_{4\kappa}(t)
    \]
    we have that the differential equation for the constant curvature case is given by
    \[
        \ddot{K}(t) + \kappa \normal{K}(t) = 4\kappa\sn_{4\kappa}(t)\dgamma(t)\mathrlap{\qquad K(0) = 0,\ \dot{K}(0) = 0}
    \]
    We can split this equation into its normal and radial part. The normal part is clearly zero, since the right-hand side is zero, and the remaining equation is linear with zero as the initial condition. For the radial part, setting $x = \scalar{K, \dgamma}$ we get
    \[
        \ddot{x}(t) = 4\kappa\sn_{4\kappa}(t)\mathrlap{\qquad x(0) = 0,\ \dot{x}(0) = 0.}
    \]
    Finally, $\frac{x(r)}{r^2} = \frac{r - \sn_{4\kappa}(r)}{r^2}$ is exactly the value announced before.

\subsection{Locally symmetric spaces}
Locally symmetric spaces define a large family of particularly well-behaved Riemannian manifolds. Examples of these spaces are the flat torus, the orthogonal group (or any compact Lie group with a bi-invariant metric), the space of symmetric positive definite matrices, the Grassmannian, the oriented Grassmannian and the hyperbolic Grassmannian.\footnote{All these manifolds are to be regarded as Riemannian manifolds with the metric inherited from their quotient structure}

    We recall the algebraic definition of a locally symmetric space.
    \begin{definition}
        A Riemannian manifold $(M, \gm)$ is \textbf{locally symmetric} if the curvature tensor is covariantly constant, that is, $\conn R = 0$.
    \end{definition}

    Locally symmetric spaces were introduced and studied by Cartan in 1926, who also gave a complete classification of these in 1932.

    The most notable examples of locally symmetric spaces are symmetric spaces which are one of the most important families of real manifolds in Riemannian geometry. These were intensively studied by Sigurður Helgason~\parencite{helgason1978differential}.
    \begin{definition}
        A Riemannian manifold $(M, \gm)$ is a \textbf{symmetric space} if, for every point $p \in M$ there exists an involutive isometry $\sigma_p$ that fixes $p$, that is
        \[
            \sigma_p(p) = p \qquad \pa{\dif \sigma_p}_0 = - \Id.
        \]
    \end{definition}
    As the name implies, one may prove that symmetric spaces are indeed locally symmetric spaces.

    For these manifolds we have the following result.
    \begin{proposition}
        A symmetric space with bounded sectional curvature $\delta \leq \sec \leq \Delta$ is of $(\delta, \Delta, 0)$-bounded geometry.
    \end{proposition}
    \begin{proof}
        Since it is a locally symmetric space, we have that $\conn R = 0$, so the first order bound is clear. Since symmetric spaces are Riemannian homogeneous spaces, for every two points there exists an isometry taking one to the other. As such, the injectivity radius is constant throughout the manifold and thus, positive.
    \end{proof}

    From this, we get that we just need to compute bounds on the sectional curvature for these manifolds in order to give second order bounds for the exponential map. Now, bounds on the sectional curvature of these manifolds are well-known. We give here some examples that are particularly useful in the context of optimisation.

    \begin{example}[The special orthogonal group]\label{ex:son}
        The sectional curvature for a Lie group with a bi-invariant metric, after identifying any pair of tangent vector with vectors in the Lie algebra, is given for a pair of orthonormal vectors $X, Y \in \glie$
    \[
        \sec(X,Y) = \frac{1}{4}\norm{\cor{X, Y}}^2.
    \]
    In the special case of $\SO{n}$ for $n > 2$, we have that $\glie \iso \Skew{n}$.
    It is clear that the sectional curvature is non-negative for any bi-invariant metric. In the case of $\SO{n}$ this bound is tight.

    Consider the bi-invariant metric given by the Frobenius norm---the metric inherited from $\RR^n$. For the upper bounds, if we work with a matrix Lie group, it is enough to bound the norm of $XY - YX$ for matrices $X, Y$ of Frobenius norm $1$. In the general case, this inequality is called the Böttcher--Wenzel inequality and it reads
    \[
        \norm{\cor{X,Y}} \leq 2\norm{X}\norm{Y} \mathrlap{\qquad \forall X, Y \in \M{n}.}
    \]
    For a review of this inequality and a particularly clean proof see~\parencite{lu2012remarks}.

    For the case of $\SO{n}$, that is, when $X, Y$ are skew-symmetric, this inequality can be improved~\parencite{bloch2005commutators}
    \[
        \norm{\cor{X,Y}} \leq \norm{X}\norm{Y} \mathrlap{\qquad \forall X, Y \in \Skew{n}.}
    \]
    For $n = 3$ the constant can be further improved to $\frac{1}{2}$. These constants are tight.

    Wrapping all this together, we get that the bounds for $\SO{n}$ with $n > 2$ are given by
    \begin{alignat*}{2}
        0
        \leq
        \scalar{\pa{\conn\dif \exp_p}_{rv}\pa{w,w}, \dgamma(r)}
        &\leq
        \pa[\Big]{\frac{1}{r} - \frac{\sin\pa{r}}{r^2}}\norm{w}^2
        &&\qquad r < 2\pi\\
        \norm{\normal{\pa{\conn\dif \exp_p}_{rv}\pa{w,w}}}
        &\leq
        \frac{r}{9}\norm{w}^2
        &&\qquad r < 2\sqrt{2}\pi.
    \end{alignat*}
    Note that the radius of definition of the normal part of the Hessian are tight, as the conjugate radius of $\SO{n}$ is exactly $2\sqrt{2}\pi$. This can be seeing by noting that the exponential of matrices is not full rank at matrices with two eigenvalues that are $2\pi i $ apart.

    In particular, we can give a bound on the full Hessian by bounding the derivative of $\sqrt{\sigma^2 + \rho^2}$ where $\sigma$ and $\rho$ are the bounds on the tangential and normal part of the Hessian. This gives
    \[
        \norm{\pa{\conn \dif \exp_p}_{rv}(w_1,w_2)} \leq \frac{2r}{9}\norm{w_1}\norm{w_2} \qquad{r < 2\pi}.
    \]
    We can see that~\Cref{thm:full_second_order_bounds} gives us a coarser bound, but on the other hand, the bound is defined on a larger radius. In particular, we get
    \[
        \norm{\pa{\conn \dif \exp_p}_{rv}(w_1,w_2)} \leq \frac{r}{3}\norm{w_1}\norm{w_2} \qquad{r < 2\sqrt{2}\pi}.
    \]
    This radius is tight, as there are points that are $2\sqrt{2}\pi$ apart from any given point which are conjugate to it. This can be seen by computing the eigenvalues of the differential of the exponential map (see for example~\cite[Theorem $D.2$]{lezcanocasado2019cheap}).

    Better bounds are possible by using a tighter version of the bounds on the curvature tensor at the expense of having an uglier numeric constant. For example, the $\frac{1}{3}$ constant can be improved this way to $\frac{\sqrt{61}}{36}$.
    \end{example}

    These same ideas can be generalised to symmetric spaces. We will use this to compute bounds for the Hessian of the exponential map in the Grassmannian.

    \begin{example}[Sectional Curvature of a Symmetric space]
        Let $G/H$ be a symmetric space. Since a symmetric space is a Riemannian homogeneous space, we just need to bound the sectional point at one point, as the space has the same sectional curvature at every point. Denote by $\mlie = \normal{\hlie}\subset \glie$ the orthogonal complement of the Lie algebra of $H$ with respect to the metric at the identity in $G$. This complement to $\hlie$ is not a Lie algebra itself and in fact $[\mlie, \mlie] \subset \hlie$, as $G/H$ is a symmetric space. This set $\mlie$ may be identified isometrically with the tangent space at $\pi(e)$, the projection of the identity element $e \in G$, in other words, the map
        \[
        \deffun{\pa{\dif \pi}_e\vert_\mlie : \mlie -> T_{\pi(e)}G/H;}
        \]
        is a linear isometry. Via this identification we may treat vectors $X, Y \in T_{\pi(e)}G/H$ as vectors $\overline{X}, \overline{Y} \in \mlie \subset \glie$. Now, by O'Neill's formula~\parencite[Ch.~$7$, Thm.~$47$]{oneill1966fundamental}, the sectional curvature for these manifolds has a particularly simple formula for orthonormal vectors $X, Y \in T_{\pi(e)}G/H$
        \[
            \sec_{G/H}(X, Y) = \sec_G(\overline{X}, \overline{Y}) + \frac{3}{4}\norm{\cor{\overline{X}, \overline{Y}}}^2
        \]
        where we have used that $[\mlie, \mlie]\subset \hlie$.

        If $G$ is a Lie group with a bi-invariant metric we say that $G/H$ is a \textbf{normal symmetric space}. For a normal metric, the sectional curvature simplifies to
        \[
            \sec_{G/H}(X, Y) = \norm{\cor{\overline{X}, \overline{Y}}}^2.
        \]
        Trough the study of symmetric spaces of non-compact type and their duality, it is not difficult to prove that the sectional curvature on any symmetric space is either the norm or minus the norm of the Lie bracket of a pair of vectors, although we will not prove that as we will not need it.
    \end{example}

    \begin{example}[The real Grassmannian]
        The real Grassmannian as a symmetric space is given by the quotient $\Gr{n,k} = \SO{n} / \pa{\SO{k} \times \SO{n-k}}$, where the metric on $\SO{n}$ is the bi-invariant metric generated by the scalar product $\scalar{X, Y} = \frac{1}{2}\tr\pa{\trans{X}Y}$ on the Lie algebra. For this symmetric space we have that
        \begin{gather*}
            \hlie =
            \solie{k}\tensor\solie{n-k} =
            \set[\Big]{
                \begin{pmatrix}
                    B & 0 \\
                    0 & C
                \end{pmatrix}
                | B \in \Skew{k}, C \in \Skew{n-k}}\\
            \mlie = \set[\Big]{
                        \begin{pmatrix}
                            0 & A \\
                            -\trans{A} & 0
                        \end{pmatrix}
                        | A \in \M{n-k, k}}
        \end{gather*}
        Note that the metric is chosen so that for $\overline{X} \in \mlie$ we have that $\norm{\overline{X}} = \norm{X}$, where $X \in \M{n-k,k}$. This is so that this norm agrees with the usual norm in the projective plane as $\Gr{n,1} \iso \RR P^n$.

        Since the metric on $G = \SO{n}$ is bi-invariant, the Grassmannian is a normal symmetric space, so the sectional curvature is given by the norm of the commutator of elements in $\mlie$. Using the bound on the norm of the Lie bracket for skew-symmetric matrices, we would get
        \[
            0 \leq \sec(X, Y) \leq 4.
        \]
        As it can be seen by~\parencite[Lemma 2.5]{ge2014ddvv}, these bounds are not tight. They can be refined as announced in~\parencite[Theorem 3a]{wong1968sectional} and proved in~\parencite[p.292]{hildebrandt1980harmonic} via an application of Cauchy--Schwarz as
        \[
            0 \leq \sec(X, Y) \leq 2.
        \]
        Using these bounds, we get analogous bounds for the Hessian of the exponential map for the Grassmannian,
        \begin{alignat*}{2}
            0
            \leq
            \scalar{\pa{\conn\dif \exp_p}_{rv}\pa{w,w}, \dgamma(r)}
            &\leq
            \pa[\Big]{\frac{1}{r} - \frac{\sin\pa{2\sqrt{2}r}}{2\sqrt{2}r^2}}\norm{w}^2
            &&\qquad r < \frac{\pi}{\sqrt{2}}\\
            \norm{\normal{\pa{\conn\dif \exp_p}_{rv}\pa{w,w}}}
            &\leq
            \frac{8r}{9}\norm{w}^2
            &&\qquad r < \pi.
        \end{alignat*}
        Note that for the real Grassmannian $\inj = \frac{\pi}{2}$~\parencite{kozlov2000geometry}, so the radii of these equations should be enough for any practical purposes. As in the case of $\SO{n}$, it is also direct to get linear bounds on the full Hessian.
    \end{example}

\section{Convergence Rates for Dynamic Trivialisations}\label{sec:bounded_hessian}
    We now have all the necessary tools to be able to complete the program stated in the introduction. In particular, we can prove the conditional convergence of the dynamic trivialisation framework for any stopping rule. We first recap the results of the previous two sections in the following proposition.

    \begin{proposition}\label{prop:bounded_hessian}
        Let $(M, \gm)$ be a connected and complete Riemannian manifold of $(\delta, \Delta, \Lambda)$-bounded geometry and let $f$ be a function of $\alpha$-bounded Hessian on it. Fix a point $p \in M$, and let $\mathcal{X} \subset \segm{p}$ be a convex subset of $M$ with $\diam(\mathcal{X}) \leq 2r \leq 2\pi_{\frac{\Delta + \delta}{2}}$, and at least one critical point of $f$ in it. Denote the pullback of $\mathcal{X}$ under the exponential as $\overline{\mathcal{X}} \defi \exp^{-1}_p(\mathcal{X}) \subset T_pM$. Then, the map $\deffun{f \circ \exp_p : \overline{\mathcal{X}} -> \RR;}$ is of $\widehat{\alpha}_r$-bounded Hessian with constant
        \begin{gather*}
            \widehat{\alpha}_r = \alpha (C_{1,r} + C_{2,r})\\
            C_{1,r} = \max\set[\Big]{1, \lfrac{\sn_\delta\pa{r}^2}{r^2}} \qquad
            C_{2,r} =
                \lfrac{8}{3}\sn_\delta\pa[\big]{\lfrac{r}{2}}^2
                \pa{\Lambda\sn_\delta\pa[\big]{\lfrac{r}{2}}^2
                +2\max\set{\abs{\Delta}, \abs{\delta}}\sn_\delta\pa{r}}.
        \end{gather*}
    \end{proposition}
    \begin{proof}
        We will bound the Hessian of the map $f \circ \exp_p$ for an arbitrary $p \in M$. By the Leibnitz rule, we have that
        \[
            \conn\dif\pa{f \circ \exp_p}
            = \conn \pa{\dif f \circ \dif \exp_p}
            = \conn\dif f \circ \dif\exp_p +  \dif f \circ \conn\dif\exp_p.
        \]
        Taking norms, and since the left-hand side is a symmetric tensor, its maximum value---\ie, its norm---is reached at a singular vector. As such, writing the bound explicitly
        \begin{align*}
            \norm{\conn\dif\pa{&f \circ \exp_p}}_{\mathcal{X}}=\\
            & \max_{\substack{v \in \exp^{-1}(\mathcal{X})\\w \in T_pM,\ \norm{w} = 1}}
            \cor[\Big]{ \pa{\conn\dif f}_{\exp_p(v)}\pa{\pa{\dif \exp_p}_v(w), \pa{\dif \exp_p}_v(w)}
            +  \pa{\dif f}_{\exp_p(v)}\pa{ \pa{\conn\dif\exp_p}_{v}(w,w)}}.
        \end{align*}
        Since there exists a critical point of $f$ in $\segm{p}$, and by the $\alpha$-bound on the Hessian, we have that $f$ is $\alpha r^2$-Lipschitz on $\mathcal{X}$. Using this, the triangle inequality and Cauchy--Schwarz, we can bound this quantity as
        \[
            \norm{\conn\dif\pa{f \circ \exp_p}}_{\mathcal{X}} \leq
            \alpha
            \max_{\substack{v \in \exp^{-1}(\mathcal{X})\\w \in T_pM,\ \norm{w} = 1}}
            \norm{\pa{\dif\exp_p}_v(w)}^2
            +
            \alpha r^2
            \max_{\substack{v \in \exp^{-1}(\mathcal{X})\\w \in T_pM,\ \norm{w} = 1}}
            \norm{\pa{\conn\dif\exp_p}_v(w, w)}.
        \]
        From~\Cref{thm:first_order_bounds,thm:full_second_order_bounds}, we have that $C_{1,r}, C_{2,r}$ come from the bounds for the square of the norm of the differential and the norm of the Hessian of the exponential respectively.
    \end{proof}

    After the heavy work of giving bounds on the Hessian of the pullback of a function along the exponential map, we are in a position to prove convergence rates for different instances of the dynamic trivialisation framework in terms of $\widehat{\alpha}_r$. We start with one of the simplest ones, namely \textbf{static trivialisations}. This is the algorithm that comes from choosing $\code{stop} \equiv \code{False}$ in~\Cref{alg:dyn_triv}. Equivalently, this is the algorithm coming from solving
    \[
        \min_{v \in T_pM} f\pa{\exp_p(v)}
    \]
    using gradient descent on $T_pM$.

The critical assumption in this result is that iterates remain bounded inside a compact set $\mathcal{X} \subset T_pM$. This assumption is restrictive, but it is standard in previous work~\parencites{bonnabel2013stochastic}{sato2019riemannian}{tripuraneni2018averaging}{ahn2020nesterovs}.

    \begin{theorem}[Convergence of static trivialisations]\label{thm:static_trivializations}
        Let $(M, \gm)$ be a connected and complete Riemannian manifold of $(\delta, \Delta, \Lambda)$-bounded geometry and let $f$ be a function of $\alpha$-bounded Hessian on it. Fix a point $p \in M$, and let $\mathcal{X} \subset \segm{p}$ be a convex subset of $M$ with $\diam(\mathcal{X}) \leq R \leq 2\pi_{\frac{\Delta + \delta}{2}}$, and at least one critical point of $f$ in it. Consider~\Cref{alg:dyn_triv} with the stopping rule $\code{stop} \equiv \code{False}$ and fixed step-size $\eta_{i,k} = \lfrac{1}{\widehat{\alpha}_{R/2}}$ where $\widehat{\alpha}_{R/2}$ is as in~\Cref{prop:bounded_hessian}. If all the iterates of the method stay in $\mathcal{X}$, the method will find a point $v_{0, t} \in T_pM$ such that $\norm{\grad \pa{f \circ \exp_p}(v_{0,t})} < \epsilon$ in at most
        \[
        \ceil[\Big]{\lfrac{2\widehat{\alpha}_{R/2}}{\epsilon^2}\pa{f(\exp_p(v_{0,0})) - f^\ast}}
        \]
        steps, where $f^\ast$ is a lower-bound of $f$ on $\mathcal{X}$.
    \end{theorem}
    \begin{proof}
        We proved in~\Cref{prop:bounded_hessian} that the map $f \circ \exp_p$ is of $\widehat{\alpha}_{R/2}$-bounded Hessian on $\mathcal{X}$. This can be regarded as a function on a Euclidean space, for which the convergence rate is well-known (\cf, \Cref{thm:non_convex_thm}).
    \end{proof}

    \begin{remark}[The hypotheses of this convergence theorem]
        The condition $\mathcal{X}\subset\segm{p}$ might seem scary at first, but it should not be since its complement---the cut locus $\cut{p} = \MM \backslash \segm{p}$---has (Borel) measure zero (\Cref{corol:measure_zero}). Also, as explained in~\Cref{sec:riemannian_exponential}, the cut locus is, in some sense, \emph{as far as possible from $p$}, so it should not pose any problems in practice.
    \end{remark}

    As we can see, the point $p \in M$ does not play any role in the proof of~\Cref{thm:static_trivializations}, besides for the technical condition of the iterates being bounded in $\segm{p}$. Generalising this assumption, we can prove at once the convergence of the scheme of dynamic trivialisations, for an arbitrary stopping rule.

    \begin{theorem}[Convergence of dynamic trivialisations]\label{thm:dynamic_trivializations}
        Let $(M, \gm)$ be a connected and complete Riemannian manifold of $(\delta, \Delta, \Lambda)$-bounded geometry and let $f$ be a function of $\alpha$-bounded Hessian on it. Assume that in the algorithm~\Cref{alg:dyn_triv} with an arbitrary stopping rule $\code{stop}$, all the iterates $v_{i,k}$ are contained in convex sets $\mathcal{X}_i \subset \segm{p_i}$ with $\diam\pa{\mathcal{X}_i} \leq R \leq 2\pi_{\frac{\Delta + \delta}{2}}$, with at least one critical point of $f$ in each of them. Then, for the choice $\eta_{i,k} = \lfrac{1}{\widehat{\alpha}_{R/2}}$, the algorithm will find a point $v_{i, k} \in T_{p_i}M$ such that $\norm{\grad \pa{f \circ \exp_{p_i}}(v_{i,k})} < \epsilon$ in at most
        \[
            \ceil[\Big]{\lfrac{2\widehat{\alpha}_{R/2}}{\epsilon^2}\pa{f(\exp_p(v_{0,0})) - f^\ast}}
        \]
        steps, where $f^\ast$ is a lower-bound of $f$ on $\mathcal{X}$.
    \end{theorem}
    \begin{proof}
        Analogous to~\Cref{thm:static_trivializations}, as the proof does not depend on the pullback point $p$.
    \end{proof}

    This theorem can be considered as a building block to then be used in specific examples to get actual convergence rates under weaker assumptions. In plain words, this result asserts that if the algorithm is converging to a critical point, it is fine to change the trivialisation point, as the exponential map will not distort the metric too much.

    \begin{remark}[Practical considerations]
    Of course, these bounds are worst-case bounds, but one can do better in practice. As we have bounds for the distortion of the exponential map at every point, we can choose a dynamic step-length that accounts for this. For example, at a point with $\norm{v_{i,k}} = s$, we could consider
    \[
        \eta_{i,k} = \frac{1}{\widehat{\alpha}_s}.
    \]
    This is not likely to give any improvement in a real-world problem, as the constant $\alpha$ is itself an upper-bound on the norm of the Hessian of the function itself, and the exponential map may incur in large variations of this second derivative on points far from $0$.

    This is particularly true when using a dynamic stopping rule that controls how much the norm of the gradient of the pullback deviates from the actual norm of the gradient of the function. In this case, one could consider having a rule similar to the one that we outlined in the introduction, but that does not only account for the case in which the algorithm is converging to the cut locus, which will mostly happen if the sectional curvature is positive, but also accounts for the case when the gradient of the pullback is much larger than that of the function, which might happen in cases of negative curvature
    \[
        \code{stop}
        \equiv
        \pa[\Big]{\frac{\norm{\grad\pa{f \circ \exp_{p_i}}\pa{v_{i, k}}}}{\norm{\grad f\pa{x_{i,k}}}} < \epsilon}
        \ 
        \code{or}
        \ 
        \pa[\Big]{\frac{\norm{\grad\pa{f \circ \exp_{p_i}}\pa{v_{i, k}}}}{\norm{\grad f\pa{x_{i,k}}}} > \frac{1}{\epsilon}}.
    \]
    This rule can be read as ``change the trivialisation point when we detect that $\exp_p$ has deviated too much from being an isometry in the direction of the gradient''. Of course, this rule can be adapted using the information that one has a priori about the geometry of the manifold and, in particular, its cut locus, conjugate locus or injectivity radius.
    \end{remark}

\section{When Is a Retraction an Exponential Map?}
As we have seen throughout this chapter and the previous one, both the exponential and retractions may be used to pullback problems to a tangent space, although it is the exponential map for which we can have more a-priori information in the general setting.

It is clear that the exponential map associated to an affine connection---not necessarily the Levi-Civita connection of a metric---is a retraction (\Cref{def:retraction}). A natural question is the converse: When is a retraction the exponential map associated to some connection? We give an algebraic way to check this in practice.

\begin{theorem}\label{thm:retraction}
    Let $\deffun{\retr : T\MM -> \MM;}$ be a retraction on a differentiable manifold. This retraction is the exponential map associated to an affine connection on $\MM$ if and only if $\retr$ induces a flow. In symbols, writing $\gamma_{p,v}(t) \defi \retr_p(tv)$ for $(p,v) \in T\MM$, then $\retr$ comes from a connection if and only if
    \[
        \dgamma_{p,v}(t+s) = \dgamma_{\gamma_{p,v}(t), \dgamma_{p,v}(t)}(s) \mathrlap{\qquad \forall t,s > 0.}
    \]
\end{theorem}
\begin{proof}
    By definition, the exponential map is the projection of the geodesic spray flow on $TT\MM$, so the exponential map defines a flow in the sense specified above.

    For the reverse implication, assume that $\retr$ defines a flow.
    Define the curves on $T\MM$
    \[
        \sigma_{p,v} = \pa{\gamma_{p,v}, \dgamma_{p,v}}
    \]
    and define the vector field on $T\MM$ for $(p,v) \in T\MM$
    \[
        S_{p,v}
        \defi \dot{\sigma}_{p,v}(0)
        = \pa{\dgamma_{p,v}(0), \derivat{t}{0}\dgamma_{p,v}(t)}
        = \pa{v, \derivat{t}{0}\dgamma_{p,v}(t)} \in TT\MM
    \]
    where we have used that $\retr$ is a retraction on the last equality. Note that for $s > 0$
    \[
        \derivat{t}{0}\dgamma_{p,sv}(t) = s\derivat{t}{0}\dgamma_{p,v}(st) = s^2\derivat{t}{0}\dgamma_{p,v}(t)
    \]
    so the second term is $2$-homogeneous. This makes $S$ into a spray, and by the Ambrose--Palais--Singer theorem~\parencite{ambrose1960sprays}, there exists an affine connection that has $S$ as its spray with its exponential map being the projection of the flow of the spray evaluated at time $1$.

    To finish, we are just missing checking that the curves $\sigma_{p,v}$ are the integral curves of $S$, but this is direct by using that $\retr$ comes from a flow
    \begin{align*}
        \dot{\sigma}_{p,v}\pa{t}
        &= \derivat{s}{0}\sigma_{p,v}\pa{t+s}\\
        &= \pa{\dgamma_{p,v}\pa{t}, \derivat{s}{0}\dgamma_{p,v}\pa{t+s}}\\
        &= \pa{\dgamma_{p,v}\pa{t}, \derivat{s}{0}\dgamma_{\sigma_{p,v}(t)}\pa{s}}\\
        &= S_{\sigma_{p,v}(t)}.\qedhere
    \end{align*}
\end{proof}

    The result for retractions that are just defined on an open neighbourhood of the zero section of $T\MM$ is analogous.

\begin{remark}
    The connection on the previous proof is not unique as connections such as their difference is a skew-symmetric tensor induce the same spray. In particular, the connection in the proposition above can be chosen to be torsion-free.
\end{remark}

\clearpage
\chapter{GeoTorch: A Library for Optimisation on Manifolds at Scale}\label{ch:geotorch}
In this chapter, we describe some practical applications of the dynamic trivialisation framework developed in~\Cref{ch:fibred_manifolds_in_optimisation}. We show how to implement it with \gpu{} efficiency and flexibility in mind.

These ideas have led to the development of the free-software library \textbf{GeoTorch}. GeoTorch allows for coupling manifold constraints with the \gpu/\tpu{}-friendly PyTorch library  to allow performing optimisation on manifolds at scale. The library can be found under the \mitedu{} license at
\begin{center}
    \url{https://github.com/Lezcano/geotorch}
\end{center}
Some of these ideas were published in the papers~\parencites{lezcanocasado2019cheap}{lezcanocasado2019trivializations} at ICML 2019 and NeurIPS 2019.

The main idea presented in~\Cref{sec:overview_geotorch} has been submitted to be added to core PyTorch, and it is currently being discussed at
\begin{center}
    \url{https://github.com/pytorch/pytorch/pull/33344}
\end{center}
It is hoped to be accepted and merged into core PyTorch in the next release (1.9.0).

The work in~\Cref{sec:approximations_exponential} led to the implementation of the matrix exponential in core PyTorch, which was implemented by Nikitaved in
\begin{center}
    \url{https://github.com/pytorch/pytorch/pull/40161}
\end{center}
This implementation yields a $\times 12$ speed-up over its Tensorflow counterpart.

\paragraph{Outline of the chapter.}
We start by giving a detailed overview of how to implement some first-order optimisation methods on $\SO{n}$ and an introduction to orthogonal constraints in \rnns{} in~\Cref{sec:motivation}.
In~\Cref{sec:geotorch}, we show how GeoTorch compares against the other open libraries to perform optimisation on manifolds, and we go over how to implement the core features that allow for its flexibility and simplicity.
In~\Cref{sec:exprnn}, we show how to use these ideas to stabilise the training of \rnns. We also provide a literature review on the topic.
Finally, in~\Cref{sec:experiments}, we test experimentally on \rnns{} the dynamic trivialisation framework, comparing the efficiency of the methods presented in this thesis to previous methods from the literature.

\section{Orthogonal Constraints within Neural Networks}\label{sec:motivation}
We start by looking at the practical ideas that first motivated this whole thesis. In particular, we look at orthogonality constraints in the context of recurrent neural networks. These constraints help to alleviate the problems of vanishing and exploding gradients.

\subsection{Vanishing and exploding gradient problems}\label{sec:orthogonal_rnn}
Training deep neural networks presents many difficulties. Arguably the most important hindrances are the exploding and vanishing gradient problems, as first observed in~\parencite{bengio1994learning}. This problem arises from the ill-conditioning of the function defined by a neural network as the number of layers increases. This issue is particularly problematic in Recurrent Neural Networks (\rnns).

Recall the definition of an \rnn{} (\cf, \Cref{def:rnn}), for $t=1, \dots, T$
\begin{align*}
    z_t &= B h_{t-1} + Cx_t\\
    h_t &= \sigma\pa{z_t},
\end{align*}
where $B \in \M{d}$, $C \in \M{d,k}$ and $\sigma$ is a non-linearity applied coordinate-wise. Suppose that we choose the last output vector $h_T$ as the encoding of the sequence, and that we have a loss function $\loss$ that maps the encoding $h_T$ and a label in the set $\YYout$ onto the real numbers. Starting on $h_0 = 0$, we have a function
\[
    \deffun{\loss \circ \code{rnn}_{B,C} : \pa{\RR^k}^\ast \times \YYout-> \RR;
            (x_1, \dots, x_T), y -> \loss(h_T, y)}
\]
We will denote the differential of this function with respect to the hidden states as $\pderiv{\loss}{h_t}$ and the gradients of the hidden states with respect to the canonical metric on the Euclidean space as $\partialgrad{\loss}{h_t}$. Note that the adjoint with respect to this metric is simply the transpose, so this notation should be clear.

We can compute the differential iteratively using the chain rule,
\begin{equation}\label{eq:3_gradient_loss}
    \pderiv{\loss}{h_i} = \pderiv{\loss}{h_T}\pderiv{h_T}{h_{T-1}} \cdot\ldots\cdot\pderiv{h_{i+1}}{h_i} = \pderiv{\loss}{h_T}\prod_{t=i+1}^T D_tB
\end{equation}
where $D_t = \diag\pa{\sigma'\pa{z_t}}$. This formula hints at the general problem that happens in \rnns. When the matrix $B$ has eigenvalues of norm different from $1$, computing the gradient is similar to computing powers of $B$, which becomes numerically unstable. For example, if $\sigma'(x)$ is uniformly bounded by $\gamma > 0$ and the largest eigenvalue of $B$ is less than $\frac{1}{\gamma}$, the norm of $\pderiv{\loss}{h_i}$ goes to zero exponentially fast in $T$.

This is problematic when computing the gradient of the recurrent kernel $B$ which is given by
\begin{equation}\label{eq:grad_rnn}
    \partialgrad{\loss}{B}
    = \sum_{t=1}^{T} h_{t-1} \tensor \partialgrad{\loss}{z_t}
    = \sum_{t=1}^{T} h_{t-1} \tensor \pa[\Big]{D_t\trans{B}D_{t+1}\trans{B} \dots \trans{B}D_T\partialgrad{\loss}{h_T}}.
\end{equation}
Note that, for $\sigma = \relu$---or, in general, for a simple piecewise linear function---the term $D_t$ barely encodes any information about $x_t$. For the $\relu$ case $D_t$ is simply a matrix with ones in some elements of the diagonal and zeros everywhere else. As such, most of the information about $x_t$ in the gradients comes from the term $h_t$.

Since we often work with the $\relu$ activation for which $\deriv{t}\relu(t) = \mathds{1}_{\geq 0}(t)$, \textbf{let's assume for the rest of the section that $D_t = \I_n$} to simplify the discussion.

Under this assumption, the $i+1$-th term of the sum in~\eqref{eq:grad_rnn} is given by
\begin{equation}\label{eq:ith-term}
    h_i \tensor \pa{\pa{\trans{B}}^{T-i}x},
\end{equation}
for the vector $x = \partialgrad{\loss}{h_T}$. In other words, if the norms of some of the eigenvalues of $B$ are larger than $1$, the norm of the gradient will explode---\textbf{exploding gradient problem}---while if the norms of all the eigenvalues of $B$ are too small, the gradient will effectively vanish---\textbf{vanishing gradient problem}.

In practice, since the entries of the matrix $B$ are often initialised according to a normal distribution of mean zero and small variance~\parencite{glorot2010understanding}, the eigenvalues of $\trans{B}$ tend to be small, meaning that we often encounter the vanishing gradient problem in practice. This translates to a very slow training, as the \rnn{} learns the signal coming from the last elements of the sequence $\pa{\dots, x_{T-1}, x_T}$, for which $\norm{B}^{T-i}$ is not too small and the norm of~\eqref{eq:ith-term} is not too small, but it often struggles to recall the information coming from the first element $\pa{x_1, x_2, \dots}$ of the sequence.

The exploding gradient problem is often encountered when tuning the learning rate. If one uses a learning rate that is too large, the gradients may explode at the beginning of the training, which leads to overflows and the appearance of \code{NaN}s, which renders the rest of the training useless. In contrast, if the learning rate is too small, then the training will either converge very slowly or would just get stuck in a local minimum. To alleviate these problems, practitioners often resort to \textbf{learning rate schedulers} where the learning rate is changed throughout the optimisation process to try to avoid these scenarios. Vanishing gradient problems (\cf, \Cref{fig:vanish_grad}) and exploding gradient problems are often encountered in practice~\parencite{pascanu2013difficulty}.

\begin{figure*}[!tbp]
\centering
  \begin{minipage}[b]{0.5\textwidth}
      \centering
      \includegraphics[width=\columnwidth]{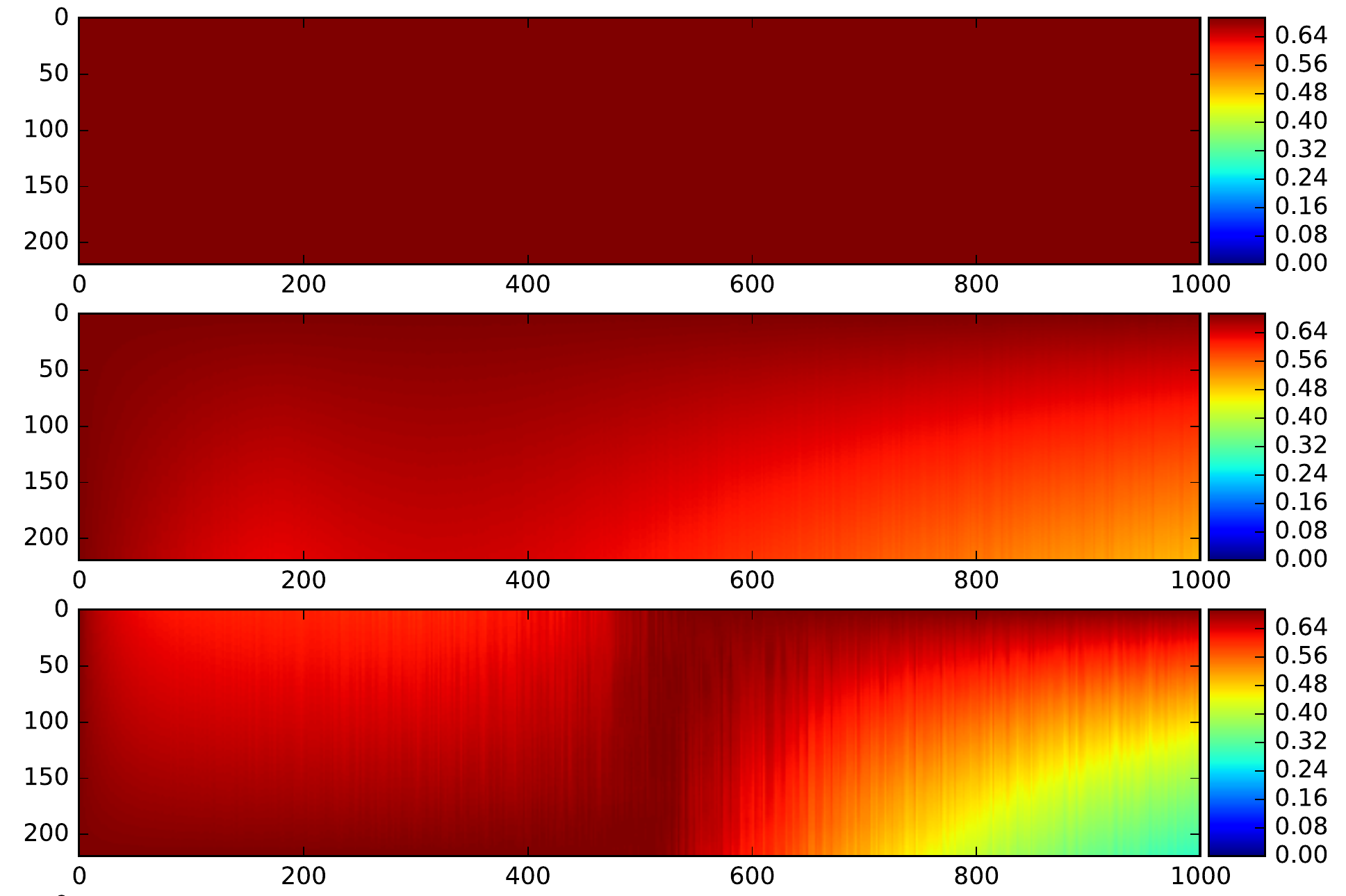}
  \end{minipage}%
  \begin{minipage}[b]{0.5\textwidth}
      \centering
      \includegraphics[width=\columnwidth]{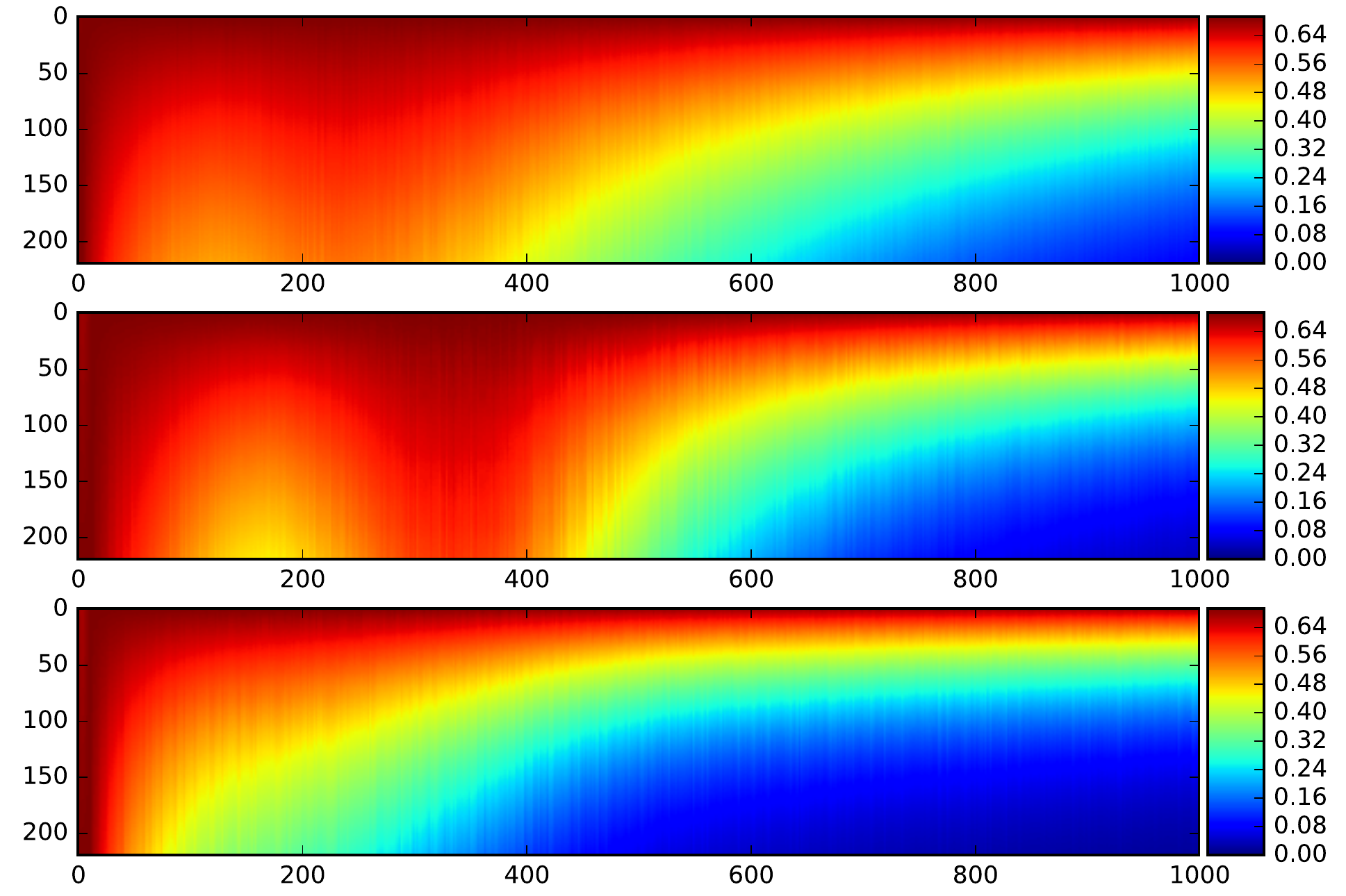}
  \end{minipage}
  \caption{Vanishing gradient problem in an \rnn. Each plot shows the evolution of the norm of the gradient $\partialgrad{\loss}{h_{T-i}}$ for $T=220$, throughout $1000$ update iterations of an \rnn{} without a non-linearity. In each plot, the recurrent kernel $B$ of the \rnn{} has singular values restricted to the interval $[1-m, 1+m]$ for $m=0, 10^{-3}, 10^{-2}$, and $m=10^{-1}, 1, \infty$ respectively. We see that, as the eigenvalues deviate from $1$, the gradient flow degrades exponentially fast. Gradient norms are normalised across the time dimension~\parencite{vorontsov2017orthogonality}.}%
  \label{fig:vanish_grad}
\end{figure*}

Another thing that can be deduced from~\Cref{eq:grad_rnn} is that we should be looking to choose a non-linearity such that $\lim_{x \to \pm\infty}\abs{\sigma'(x)} = 1$.\footnote{This is a relaxation of the stronger condition of choosing $\sigma$ to have derivative of absolute value $1$.} This condition allows the gradients to \emph{flow back} when backpropagated\footnote{The idea of the gradients ``flowing back'' during the backpropagation process---while computing the gradients using the chain rule---is often used in the deep learning context to describe that the gradients are computed from the end of the neural network all the way to the beginning of the neural network.}.

One way of mitigating the stability problems of~\Cref{eq:3_gradient_loss} is to decree $B$ to be orthogonal. This was first noted in practice in~\parencite{arjovsky2016unitary}. After this, optimisation methods over the unitary and orthogonal group have found rather fruitful applications in \rnns{}. We will do a literature review of these methods in~\Cref{sec:comparison}. In parallel to this work, there has been an increasing interest in optimisation over the orthogonal group and the Stiefel manifold in neural networks~\parencites{ozay2016optimization}{huang2018orthogonal}. As shown in these papers, orthogonal constraints in linear and \cnn{} layers can be rather beneficial for the generalisation of the network as they act as a form of regularisation of the Lipschitz constant of the neural network.

\subsection{Optimisation with orthogonal constraints}\label{sec:exp_param}
In this section, we will make explicit every detail about how to implement some first-order optimisation algorithms on $\SO{n}$. All the ideas in this section were already presented in an abstract way in~\Cref{ch:fibred_manifolds_in_optimisation}, and specialised to the Stiefel manifold in~\Cref{sec:stiefel_manifold}, so the computations should already be familiar to the reader. We will flesh out every detail of these computations for $\SO{n}$, as it is the main building block used to implement optimisation in many other different manifolds.

Let $\deffun{\total{f} : \M{n} -> \RR;}$ be a function and define $f \defi \total{f}\vert_{\SO{n}}$. We are interested in approximating a solution to the problem
\[
    \argmin_{B \in \SO{n}} f(B).
\]

The classical way to approach this problem would be to apply Riemannian gradient descent (\rgd). Let $A$ be a skew-symmetric matrix and let $B = \expm\pa{A} \in \SO{n}$. We will fix canonical metric on $\M{n}$ and the induced bi-invariant metric on $\SO{n}$. For these metrics, we have the gradient on the total space $\grad \total{f}(B) \in \M{n}$ and the gradient on the manifold $\grad f(B) \in T_B\SO{n}$.

\rgd{} works by following the geodesic defined by the direction $- \grad f(B)$ at the point $B$ for a time $\eta > 0$, that is,
\[
    B_{t+1} = \exp_{B_t}\pa{-\eta\grad f(B_t)}.
\]

The tangent space to $\SO{n}$ at a matrix $B$ is
\[
    T_B\SO{n} = \set{X \in \M{n} | \trans{B}X + \trans{X}B = 0}
\]
and, as we showed in~\Cref{sec:stiefel_manifold}, the projection onto the tangent space takes the form
\[
    \deffun{\pi_B : \M{n} -> T_B\SO{n}; X -> \frac{1}{2}\pa{X - B\trans{X}B}}
\]
Since $\SO{n}$ is an embedded submanifold of $\M{n}$, the gradient of $f$ is just the tangent component of the gradient on the ambient space
\[
    \grad f(B) = \pi_B\pa{\grad \total{f}(B)} = \frac{1}{2}\pa{\grad \total{f}(B) - B\trans{\grad \total{f}(B)}B}.
\]
Writing for short $B_t = \expm\pa{A_t}$, we have that the \rgd{} update rule for $\SO{n}$ is given by

\[
    B_{t+1} =
    \expm\pa{A_t} \expm\pa{-\eta \trans{B_t} \grad f(B_t)} =
    \expm\pa{A_t} \expm\pa{-\eta \pi_{\I_n}\pa{\trans{B_t}\grad\total{f}(B_t)}}.
\]
or more succinctly
\[
    B_{t+1} =
    \expm\pa{A_t}\expm\pa{-\eta \pi_{\I_n}\pa{\grad\pa{\total{f}\circ L_{B_t}}(\I_n)}}.
\]
It is now easy to see how this equation is implemented by the autodiff engine when using the dynamic trivialisations. Consider the Riemannian exponential on $\SO{n}$
\[
    \deffun{\exp_B = L_B \circ \expm \circ \pi_{\I_n} : \M{n} -> \SO{n};}
\]
where we have used left-invariant vector fields to identify $T_B\SO{n}$ with $\solie{n}$ and we have used $\pi_{\I_n}$ to parametrise the skew-symmetric matrices. We then have that
\[
    \grad f(B) = \grad\pa{f \circ \exp_B}(0_n) = \grad\pa{\total{f} \circ L_B \circ \expm \circ \pi_{\I_n}}(0_n).
\]

\begin{remark}[Frames]\label{rmk:frames}
    It is perhaps now easier to see the idea that we mentioned in~\Cref{sec:dyn_triv} about how, after a choice of parametrisation of the tangent space---\ie, a choice of frame $\deffun{\zeta : T_p\MM -> \RR^n;}$---the dynamic trivialisation framework automates the gradient computations. It should be now clear that different frames change the representation of the tangent space, but they do not change the optimisation algorithm. Here we are using the projection from the total space---that is, a linear Riemannian submersion---given by $\deffun{\pi_{\I_n} : \M{n} -> \solie{n};}$.

    Alternatively, we could have used the frame $\deffun{\code{frame\_skew} : \M{n} -> \Skew{n};}$ which takes the lower triangular part of a matrix and uses it to form a skew-symmetric matrix $A - \trans{A}$. This frame has the advantage that, even though we are working with matrices in $\M{n}$ to parametrise skew-symmetric matrices, we are effectively just modifying the lower $\frac{n(n-1)}{2}$ entries. This leaves the upper and diagonal $\frac{n(n+1)}{2}$ entries untouched, and they can be used to make the model more expressive without requiring any extra memory. We use this extensively in some parts of the GeoTorch library, such as when we parametrise low-rank matrices via their \svd{} using two matrices on Stiefel manifolds. In this case, one of them is parametrised using the upper triangular part of the matrix and the other one uses the lower-triangular part.
\end{remark}

Another way to attack the problem of optimisation over $\SO{n}$ is through static trivialisations as described in~\Cref{sec:change_metric}. As the matrix exponential is surjective onto $\SO{n}$, we can use it as a static trivialisation, pulling back the problem as
\begin{equation}\label{eq:exp_param}
    \min_{B \in \SO{n}} f(B) = \min_{A \in \Skew{n}} f(\expm(A)).
\end{equation}
We call this the \textbf{exponential parametrisation of $\SO{n}$}. The update rule, applying gradient descent to this problem, is then given by
\begin{align*}
    A_{t+1} &= A_t -\eta \grad\pa{f \circ \expm}(A_t)\\
    B_{t+1} &= \expm\pa{A_{t+1}}.
\end{align*}
To initialise this method, we sample a matrix $B_0 \in \SO{n}$ according to some probability measure and let $A_0 \defi \logm\pa{B_0}$.

It is also natural to consider the dynamic trivialisation framework discussed in~\Cref{sec:dyn_triv}, having, for a fixed $B_0 \in \SO{n}$, the problem
\begin{equation}\label{eq:dyn_triv_param}
    \min_{B \in \SO{n}} f(B) = \min_{A \in \Skew{n}} f(B_0\expm(A)).
\end{equation}
which has an update rule
\begin{align*}
    A_{t+1} &= A_t -\eta \grad\pa{f \circ L_{B_0} \circ \expm}(A_t)\\
    B_{t+1} &= B_0\expm\pa{A_{t+1}}.
\end{align*}
As mentioned in~\Cref{sec:dyn_triv}, we sample $B_0$ according to a probability measure on $\SO{n}$ and start at the point $A_0 = 0$. Then, we may update $B_0$ according to some stopping rule. In practice, we compute the gradient of $f \circ L_{B_0} \circ \expm \circ \code{frame\_skew}$ which is computed automatically by the autodiff engine.

\begin{remark}[\rgd{} and trivialisations agree on commutative Lie groups]
    It is quite revealing to put side by side the update rule for \rgd{} and the static trivialisation framework\footnote{Note that the gradient in~\rgd{} is the gradient taken with respect to the metric on the manifold, and the gradient in the trivialisation framework is taken with respect to the flat metric on $T_p\MM$.}
\begin{align*}
    \text{\rgd:}\qquad&         B_{t+1} = \expm\pa{A_t}\expm\pa{-\eta\grad f(\expm\pa{A_t})}\\
    \text{Static Triv.\ at $I_n$:}\qquad& B_{t+1} = \expm\pa{A_t -\eta \grad \pa{f \circ \expm}(A_t)}\\
    \text{Static Triv.\ at $B_0$:}\qquad& B_{t+1} = B_0\expm\pa{A_t -\eta \grad \pa{f \circ L_{B_0} \circ \expm}(A_t)}.
\end{align*}

We can see here that the difference between \rgd{} and the static trivialisation method at the identity stems from the fact that the Lie exponential is not a Lie group homomorphism, that is, in general we have that $\expm\pa{A_1+A_2} \neq \expm\pa{A_1}\expm\pa{A_2}$. In fact, it is not difficult to show\footnote{For example, using the Baker--Campbell--Hausdorff formula.} that the Lie exponential is a homomorphism if and only if the Lie group is commutative. Examples of commutative Lie groups are the torus or $\RR^n$---in fact, every commutative Lie group is a direct product of these two.

The geometric reason for this equality stems from the fact that a bi-invariant metric on a Lie group is flat if and only if the group is commutative. Therefore, pulling back by the Lie exponential accounts for lifting the problem to its universal cover being the Riemannian exponential the cover map, which is a Riemannian submersion in this case. This behaviour happens more generally on any Hadamard manifold as per the Cartan--Hadamard theorem (\cf, \Cref{sec:hadamard}).
\end{remark}

\begin{remark}[Different optimisation algorithms]
Both the static trivialisation and the dynamic trivialisation pullback the problem to a Euclidean space. As such, we could use any Euclidean optimisation method to solve~\eqref{eq:exp_param} and~\eqref{eq:dyn_triv_param}. This is a big advantage in practice, given that virtually all the methods developed to improve the stability and convergence of optimisation algorithms for neural networks have been devised having the Euclidean space in mind. We will go over these and other advantages of this method over \rgd{} in the context of neural networks in this next section.
\end{remark}

\section{GeoTorch}\label{sec:geotorch}
In this section, we go over some implementation details of the library \textbf{GeoTorch}. This library allows for performing optimisation on manifolds within PyTorch. This library is built upon a few core ideas:

\begin{enumerate}
	\item \gpu{} efficiency and parallelism by design
	\item Automatic computation of Riemannian gradients and Hessians
    \item Functorial design benefiting from code reuse
	\item Seamless interoperability with PyTorch
\end{enumerate}

We will show how these ideas are implemented in practice in~\Cref{sec:overview_geotorch}.

\subsection{Libraries for optimisation on manifolds}
	We start by reviewing the current available frameworks for performing optimisation on manifolds.

    At the time of writing, the most important library to perform constrained optimisation and optimisation on manifolds is \textbf{Manopt}~\parencite{boumal2014manopt}. This is a \matlab{} library maintained by Nicolas Boumal and Bamdev Mishra. This tool is widely used in papers within the field of optimisation on manifolds, being the reference when it comes to theoretical optimisation on manifolds. The main differences between Manopt and GeoTorch stem from the audience the library is targeting. Manopt targets a specialised audience that works on optimisation and may want to compare a new research line with pre-existing optimisation algorithms. In contrast, GeoTorch targets a more mainstream audience that wants to integrate constraints into a pre-existing model to regularise it, without having to change much of their initial model.

	As mentioned, Manopt is designed to provide the researcher with a framework to showcase the performance of their algorithms in simple settings. While this approach has been invaluable to the researchers doing optimisation on manifolds for the last $10$ years, it is not designed to be used in large models. Manopt has been ported to several other languages in packages like \textbf{Manopt.jl} and \textbf{PyManopt}, but these ports suffer of similar drawbacks.

    GeoTorch comes to fill in this gap. GeoTorch is built on top of PyTorch, the most widely used deep learning library. As such, it takes advantage of the efficiency improvements and \gpu{}, \tpu{}, \amp{}, distributed training support that PyTorch offers natively, among others. It also supports autodiff natively. All these are systems that GeoTorch is designed to interact with seamlessly.

    This integration with PyTorch also heavily simplifies the maintenance required for the library. Not only does one not need to implement the gradients or Hessians, but it also benefits from all the numerical analysis algorithms (\qr, \svd, Cholesky\textellipsis) together with their derivatives implemented natively on \gpu{} by the PyTorch community. We will showcase this idea in~\Cref{sec:approximations_exponential}, where we will note the advantage of a \gpu{} implementation of the matrix exponential, obtaining a $\times 2.5$ wall-clock time improvement over the non-native \gpu{} implementation and a $\times 12$ wall-clock time improvement over the implementation in Tensorflow. This implementation is now in core PyTorch as of PyTorch $1.7.0$, and it can be used by any other library or researcher that uses PyTorch.

    The first attempt to integrate constrained optimisation together with a deep learning library was \textbf{McTorch}~\parencite{meghwanshi2018mctorch}, lead by Bamdev Mishra. The project was discontinued, as their approach required forking the whole PyTorch project and adding their software on top. We will sidestep this problem by using the dynamic trivialisation framework and a number of metaprogramming tricks.

    \textbf{Geoopt} is another package that aims to be compatible with PyTorch. Its implementation relies on the user using their \code{Tensor} and \code{Parameter} classes, rather than that of PyTorch. This reason make this library difficult to use for the general audience, as they would have to commit to the use of these custom classes in their models, which is already a very big drawback. For example, if the user already uses a different tensor class in their model, it would be impossible for her to use this library. Even more, the user needs to wrap every parameter of a deep model using these special classes, which can be somewhat tedious and difficult to maintain.

    \subsection{The design of GeoTorch}\label{sec:overview_geotorch}
    The main idea driving GeoTorch is to use the static and dynamic trivialisation framework to pullback an optimisation problem from a manifold onto a Euclidean space. As described in~\Cref{ch:fibred_manifolds_in_optimisation}, if we have a map $\deffun{\triv : \total{\MM} -> \MM;}$, we may use it to pullback the problem from $\MM$ to $\total{\MM}$. In practice, we will often choose a map with domain $\total{\MM} \iso \RR^n$, such as the Riemannian exponential map or a retraction, to be able to apply Euclidean optimisation methods to the pullback problem. Even then, in this section we will not delve into the different maps used for the different manifolds, as most of these are standard in the literature. In turn, we will look at the ideas necessary for the modular implementation behind GeoTorch.

    Before delving into the implementation details, which might at times be too technical, we would like to start by showing the simplicity of GeoTorch in action.
\begin{lstlisting}[language=python,escapechar=|,caption={Example of the integration of GeoTorch with PyTorch}]
import torch
import torch.nn as nn
import geotorch

class Model(nn.Module):
    def __init__(self):
        super(Model, self).__init__()
        self.linear = nn.Linear(64, 128)
        self.cnn = nn.Conv2d(16, 32, 3)
        # Make the linear layer into a low rank layer with rank at most 10
        geotorch.low_rank(self.linear, "weight", rank=10)
        # Also works on tensors. Makes every kernel orthogonal
        geotorch.orthogonal(self.cnn, "weight")


    def forward(self, x):
        # self.linear has rank at most 10 and every 3x3 kernel in the CNN is orthogonal

# Nothing fancy from here on. Use the model as you would normally do.
model = Model()

# Any optimiser works out of the box with any parametrisation
optim = torch.optim.Adam(model.parameters(), lr=0.01)
\end{lstlisting}

    GeoTorch implements optimisation on $16$ spaces in a way that is compatible with any pre-existing PyTorch model and allows for constraining any PyTorch module with just one line of code. These spaces are listed in~\Cref{tab:manifolds_geotorch}.

    \begin{table}[t]
\centering
\scshape
\begin{tabular}{c c c}
    Real numbers & General linear group & Positive definite \\
    Symmetric & Special orthogonal group &  Positive semidefinite \\
    Skew-symmetric & Stiefel manifold & Positive semidefinite low rank \\
    Low rank & Grassmannian & Positive semidefinite fixed rank \\
    Fixed rank & Sphere & Product manifolds \\
    \multicolumn{3}{c}{Matrices with singular values in $(a,b)$} \\
\end{tabular}
\caption{List of spaces that GeoTorch currently implements.}%
\label{tab:manifolds_geotorch}
\end{table}

    We will illustrate how to implement such a simple \api{} using the example of orthogonal optimisation. We will implement a linear layer with orthogonal weights (\cf, \Cref{def:linear_layer})
    \[
        f_B(x) = Bx\mathrlap{\qquad B \in \SO{n}.}
    \]
    We will do so by pulling back the problem along the matrix exponential to a Euclidean space as we did for \exprnn. We will then point out the flaws that a naïve implementation has and we will motivate the implementation in GeoTorch by correcting these.

    \begin{example}[A naïve implementation]
    Assume that we have the following simplified implementation of a square linear layer
\begin{lstlisting}[language=python,escapechar=|,caption={Basic \code{SquareLinear} layer}]
class SquareLinear(nn.Module):
    def __init__(self, n_features):
        super().__init__()
        self.weight = nn.Parameter(torch.Tensor(n_features, n_features))

    def forward(self, x):
        return x @ self.weight  # Right-multiply as x is of size [batch, n_features]
\end{lstlisting}
    We can modify this implementation to make it into an orthogonal linear layer by pulling back the problem by the matrix exponential:
\begin{lstlisting}[language=python,escapechar=|,caption={\code{Orthogonal} v.$1$. Basic implemenation.}]
class Orthogonal(nn.Module):
    def __init__(self, n_features):
        super().__init__()
        self.weight = nn.Parameter(torch.Tensor(n_features, n_features))

    def forward(self, x):
        A = self.weight |\label{lin:weight}|
        A = A - A.T        # A is skew-symmetric
        B = A.matrix_exp() # B is special orthogonal
        return x @ B       # Right-multiply as x is of size [batch, n_features]
\end{lstlisting}
Note that here we are using the frame $\pi_{\I_n}(A) = A - \trans{A}$ (\cf, \Cref{sec:exp_param}). If we replaced~\Cref{lin:weight} with \lstinline{A = self.weight.tril()} we would have an implementation of \code{frame\_skew} as described in~\Cref{rmk:frames}.

        This \adhoc{} implementation is good enough in most scenarios, but it has three main drawbacks:
    \begin{enumerate}
        \item To implement a parametrisation on a layer we effectively need to reimplement the whole layer

        \item The parametrisation is implemented in the layer itself. We would have to repeat the implementation of the parametrisation for every layer that we want

        \item If we use the layer several times, like one does in an \rnn, the matrix exponential would be computed multiple times on the same matrix\label{enum:caching}
    \end{enumerate}

    We will solve the first and second problem in two steps. First, we will show how solve the first one by using inheritance and \textbf{Python properties}. Then, we will show how to use some advanced \textbf{metaprogramming} techniques to decouple the implementation of the class from the implementation of the constraints.

    The third one calls for the implementation of a caching mechanism, which is somewhat tricky in the context of autodiff systems. We will show a working solution, but we should warn the reader that, even though the final solution looks simple, similar-looking solutions yield problems related to the propagation of gradients. We will not mention these here to avoid making the section even more technical.
\end{example}

\subsubsection{Reusing pre-existing classes}
Let us continue with the example of the \code{Orthogonal} layer. We will show now how to use inheritance and Python properties to reuse the code from \code{SquareLinear}.

Python properties are often called \textbf{computed attributed}. They allow for using functions to define an attribute of a class. As an example, consider that we have an attribute in a class in centimetres, but we also want to have it in inches. We could implement this as follows:
\begin{lstlisting}[language=python,escapechar=|]
class Length:
    def __init__(centimetres):
        self.centimetres = centimetres

    @property
    def inches(self):
        return self.centimetres / 2.54
l = Length(2.)
print(l.centimetres) # Will print 2.0
print(l.inches)      # Will print 0.787402
\end{lstlisting}

Properties allow having computed attributes, which is exactly what trivialisations are. It is now easy to modify the code of \code{Orthogonal} to reuse the code from \code{SquareLinear}.
\begin{lstlisting}[language=python,escapechar=|,caption={\code{Orthogonal} v.$2$. Reuse code through properties and inheritance.}]
class Orthogonal(SquareLinear):
    def __init__(n_features):
        super().__init__(n_features)
        # Rename weight attribute to _weight
        self._weight = self.weight
        delattr(self, "weight")

    @property
    def weight(self):
        A = self._weight
        A = A - A.T
        return A.matrix_exp()
\end{lstlisting}

Even better, we can get rid of \code{SquareLinear} altogether and use the PyTorch implementation \code{nn.Linear} without having to reimplement it. As such, we can simply implement \code{Orthogonal} as
\begin{lstlisting}[language=python,escapechar=|,caption={\code{Orthogonal} v.$2.1$. Use the native \code{nn.Linear} from PyTorch.},label={lst:v2.1}]
class Orthogonal(nn.Linear):
    def __init__(n_features):
        super().__init__(n_features, n_features, bias=False)
        # Rename weight attribute to _weight
        self._weight = self.weight
        delattr(self, "weight")

    @property
    def weight(self):
        A = self._weight
        A = A - A.T
        return A.matrix_exp()
\end{lstlisting}

By doing this, whenever we call \code{Orthogonal.forward}, it will call the \code{nn.Linear.forward}, which will call \code{self.weight}, which will call the property and return an orthogonal matrix. Furthermore, all these operations are differentiable, so the gradients are computed automatically by the autodiff engine with no extra work from our side.

\begin{remark}[On the above implementation]
    The attentive reader will have spotted that the implementation in~\Cref{lst:v2.1} does not work in practice. This is because the parameter \code{weight} is renamed in \code{\_\_init\_\_} after the property is defined during the parsing of the class body. A correct implementation of this idea would involve the use of metaclasses. We omit this as metaclasses are quite technical and would not add much to the exposition.
\end{remark}

\subsubsection{Caching tensors}
    To solve the problem of having to compute the matrix exponential several times, we have to implement a \textbf{caching system}. A caching system simply accounts for saving an output of a deterministic function $\triv(x)$ for certain value $x$, so that if $\triv(x)$ has to be recomputed for the same $x$, we simply return the result. Note that in our case $x = \code{self.weight} \in \M{n}$ and $\triv = \expm \circ \pi_{I_n}$ is the parametrisation of $\SO{n}$ using of skew-symmetric matrices so that $\triv(x) \in \SO{n}$.

\begin{remark}[A word on automatising the caching system]
    The caching mechanism has to be aware of two particular moments in the optimisation process:
    The moment when it has to compute and cache the value, and the moment when it has to discard the result. The first is simple to detect: If we look at the cache and it does not have the value computed, we compute it and store it in the cache. Detecting when the cache is \emph{dirty}---\ie, when does it have to recompute the value in the cache---is much tricky. In particular, it accounts for detecting when the parameter $x$ has been modified.

    In optimisation problems, we know that the tensor that we are parametrising is often just going to be updated by the optimiser after a full optimisation step. A model very rarely modifies its own weights. As such, it would be natural to ask for a system that automatically detects this  and it automatically discards the cache value after the optimiser has modified the parameters.

    On the other hand, if we want to implement this in a generic way, we have to be able to provide a mechanism that works for every possible model. In particular, it should be able to detect the notably rare case when the model modifies its own parameters before the optimisation step. To do this, we would have to be able to detect when \code{self.weight} has been modified. This poses two problems. First, it would have to check for equality of the full tensor every time the tensor is accessed, which is not only very expensive, but also equality on floating-point numbers is unstable at best. Second, even though in Python it is possible to detect in some situations when a member of a class has been modified, it is not so easy to do so in PyTorch. PyTorch is mostly implemented in C++, so it is close to impossible to differentiate when a member of a class has been modified versus when it has just been accessed. To close this discussion, we quote again the \emph{Zen of Python}:
    \begin{quote}
        Explicit is better than implicit.
    \end{quote}
\end{remark}

    For all these reasons, we opted for a simpler solution along the lines of similar solutions adopted in core PyTorch. We implement a caching system that can be activated and deactivated by the user in their model via a global flag managed through a context manager. For example, assume that we have a parametrisation in place on an \rnn{} that makes the recurrent kernel orthogonal. The user would then activate the caching system as
\begin{lstlisting}[language=python,escapechar=|]
with geotorch.cached():
    for x in inputs:
        hidden_state = self.rnn(x, hidden_state)
\end{lstlisting}
This can also be done at a level of the whole model in the training loop (\cf,~\Cref{lst:feedforward}), by guarding the forward pass
\begin{lstlisting}[language=python,escapechar=|]
for idx, (batch_x, batch_y) in enumerate(loader):
    batch_x, batch_y = batch_x.to(device), batch_y.to(device)

    with geotorch.cached():
        out = model(batch_x)
    loss = F.nll_loss(out, batch_y)

    optim.zero_grad()
    loss.backward()
    optim.step()
\end{lstlisting}

This guard effectively signals the whole model when should it start caching values and when are the cached values not valid any more.

In order to take advantage of this mechanism, we would have to simply check whether one should use the cache value or not. For example, for the class \code{Orthogonal} we could write
\begin{lstlisting}[language=python,escapechar=|,caption={\code{Orthogonal} v.$3$. Integrating the caching system.},label={lst:orthogonal_v3}]
_cache_enabled = False
_cache = {}

class Orthogonal(nn.Linear):
    def __init__(n_features):
        super().__init__(n_features, n_features)
        # Rename weight attribute to _weight
        self._weight = self.weight
        delattr(self, "weight")

    def compute_weight(self):
        A = self._weight
        A = A - A.T
        return A.matrix_exp()

    @property
    def weight(self):
        global _cache
        key = id(self)
        if _cache_enabled:
            if key not in _cache:
                # If we have not cached the value, we compute it and store it
                _cache[key] = self.compute_weight()
            return _cache[key]
        else:
            # If the cache is not enabled, simply return the matrix
            return self.compute_weight()

class cached:
    def __enter__(self):
        global _cache_enabled
        _cache_enabled += 1

    def __exit__(self, exc_type, exc_value, traceback):
        global _cache_enabled
        _cache_enabled -= 1
        if not _cache_enabled:
            _cache = {}
\end{lstlisting}

Note that the implementation does not consider \code{\_cache\_enabled} to be a boolean but an integer, to be able to handle correctly nested calls to the \code{cached} context manager.

\subsubsection{Injecting properties into classes}
To approach the second and third problem of reusing code, we resort to metaprogramming. Metaprogramming is the ability of a language of treating code as data. That is, a metafunction could take some parameters and return not an instance of a class, but a whole new class.

In our case, we will use to create the code class \code{Orthogonal} given a class that just implements the orthogonality code, and the class that we want to put the parametrisation on. We would like to have an \code{Orthogonal} class that looks exactly like a PyTorch module, that is
\begin{lstlisting}[language=python,escapechar=|,caption={A class that implements the orthogonal parametrisation}]
class Orthogonal(nn.Module):
    def forward(self, A):
        A = A - A.T
        return A.matrix_exp()
\end{lstlisting}

The desideratum is to have a \emph{metafunction} that takes an instance of this class and an instance of \code{nn.Linear} and returns an instance of a class that looks like the \code{Orthogonal} layer from~\Cref{lst:orthogonal_v3}.

Even though this can be implemented in about $50$ lines of code, the details are very technical. We will summarise them here for the interested reader, but they may be regarded as a convenient implementation detail for the general reader.

In order to implement the caching system and the parametrisation system automatically, we would have to inject a property into a given object. Sadly, this is not possible in Python, as properties are stored at a class level, not at an instance level.
One way of overcoming this issue is to dynamically create a new class that inherits from \code{nn.Linear}, inject the property in this class and then replace the class from our initial object with this custom class.
This trick was suggested to us by Adam Paszke, and it is implemented in the following piece of code:
\begin{lstlisting}[language=python,escapechar=|]
def _inject_new_class(module):
    r"""Sets up the parametrization mechanism used by parametrizations.
    This works by substituting the class of the module by a class
    that extends it to be able to inject a property

    Args:
        module (nn.Module): module on which to inject the property
    """

    # We create a new class so that we can inject properties in it
    cls_name = "Parametrized" + module.__class__.__name__

    param_cls = type(
        cls_name,
        (module.__class__,),
        {
            "__qualname__": cls_name + str(id(module)),
        },
    )

    # Declare the class globally to be able to pickle it
    globals()[param_cls.__qualname__] = param_cls
    module.__class__ = param_cls
\end{lstlisting}

\subsubsection{Initialising parametrisations}
Once we have a way to change the class to one that is virtually identical but different, we may simply inject the property. While doing so, we take the chance of defining not only a getter for the property, but also a setter.

\begin{lstlisting}[language=python,escapechar=|]
def _inject_property(module, _name):
    r"""Injects a property into the class of a given module.
    It assumes that the tensor under ``module[tensor_name]``
    has already been moved

    Args:
        module (nn.Module): module on which to inject the property
        tensor_name (string): name of the property
    """

    # Define the getter
    def get_parametrized(module):
        global _cache_enabled
        global _cache

        key = _key(module, tensor_name)
        # If the _cache is not enabled or the caching was not enabled for this
        # tensor, this function just evaluates the parametrization
        if _cache_enabled and key in _cache:
            if _cache[key] is None:
                _cache[key] = module.parametrizations[tensor_name].evaluate()
            return _cache[key]
        else:
            return module.parametrizations[tensor_name].evaluate()

    # Define the setter
    def set_value(module, value):
        module.parametrizations[tensor_name].initialize_(value)

    setattr(module.__class__, tensor_name, property(get_parametrized, set_value))
\end{lstlisting}

This allows to assign to the properties provided that we implement an \code{initialize\_} method.
\begin{lstlisting}[language=python,escapechar=|]
class Skew(nn.Module):
    def forward(self, X):
        X = X.triu(1)
        return X - X.T

    def is_skew(self, X):
        # Skew modulo rounding errors
        return torch.allclose(X, -X.T)

    def initialize_(self, X):
        if not self.is_skew(X):
            raise ValueError("The input matrix is not skew-symmetric")
        return X.triu(1)

model = nn.Linear(5, 5)
geotorch.register_parametrization(model, "weight", Skew())
# Just computes `model.weight` and checks that the weight is and checks that the weight is skew-symmetric
with geotorch.cached():
    assert(torch.allclose(model.weight, -model.weight.T))
# Sample a skew matrix X and initialise the parametrised model.weight
X = torch.rand(5, 5)
model.weight = X - X.T
assert(torch.allclose(model.weight, X))
\end{lstlisting}

In most cases, the method \code{initialize\_} is simply a right-inverse of the \code{forward} method, which exists locally whenever \code{forward} is a submersion. It is at this point where the theory and practice meet to give a particularly modular design for the static and dynamic trivialisation framework.

We are missing to show how to put all these ideas together to allow composing parametrisations---\ie, calling several times the method \code{register\_parametrization} on the same tensor---and also finish the implementation of \code{register\_parametrization} in terms of the auxiliary methods presented here. We spare the reader all those details, as they are a matter of a few lines of software engineering, and they are not particularly interesting. A fully working implementation can be found in
\begin{center}
    \url{https://github.com/Lezcano/geotorch/blob/master/geotorch/parametrize.py}
\end{center}

All these ideas have been submitted to be added to core PyTorch, and it is currently being discussed in
\begin{center}
    \url{https://github.com/pytorch/pytorch/pull/33344}
\end{center}
It is hoped to be accepted and merged into core PyTorch in the next release (1.9.0).

\subsubsection{Class system}
In this section, we will touch on the class system designed for GeoTorch. The idea here is to abstract most of the patterns used in the context of optimisation on manifolds into Python constructions.

Before starting with the explanation, it is worth quickly summarising what we have accomplished so far: We have a way of injecting a function to parametrise a tensor in our model.

This may not seem like too much on its own, but it is truly powerful tool due two main reasons and a third one that allows to put them together.

\begin{enumerate}
    \item The whole PyTorch package is designed around being able to define differentiable functions as classes inheriting from \code{nn.Module}.

    \item We have detailed in~\Cref{ch:geometry,ch:fibred_manifolds_in_optimisation} how fibred manifolds provide a compositional way to pullback problems via pre-composing by frames and post-composing by submersions.

    \item Python's inheritance system is designed around pre and post-composition of functions (\emph{decoration of methods}).
\end{enumerate}

We shall now show how to put these three ideas together to implement a class system on manifolds that uses optimisation simple manifolds to perform optimisation on more complex ones. We will do it again continuing our example of orthogonal optimisation, extending it to non-square orthogonal optimisation.

\begin{example}[Stiefel manifold]
    At the moment, we have the following implementation of a map from $\M{n}$ to $\SO{n}$ via the Riemannian exponential map, which can be injected into any parameter of a model.

\begin{lstlisting}[language=python,escapechar=|,caption={Final version of the parametrisation of $\SO{n}$},label={lst:final_so}]
class SO(nn.Module):
    def __init__(self, n):
        super().__init__()
        # Initialise self.base to a matrix sampled according to the Haar measure on SO(n)
        self.register_buffer("base", nn.init.orthogonal_(torch.empty(n, n)))

    def frame(self, X):
        X = X.tril(-1)
        return X - X.T

    def exponential(self, B, A):
        return B @ torch.matrix_exp(A)

    def forward(self, X):
        A = self.frame(X)
        return self.exponential(self.base, A)
\end{lstlisting}

    Note that the map has two separate maps: A linear frame, and the computation of the Riemannian exponential. Denote the frame that we have used to map the lower-triangular part of a matrix onto a skew-symmetric matrix as
    \[
        \deffun{\code{frame\_skew} : \M{n} -> \Skew{n};}
    \]
    The \code{SO} class defined in~\Cref{lst:final_so} can then be defined mathematically denoting $\total{Q} = \code{self.base}$ as
    \[
        \deffun{L_{\total{Q}} \circ \expm \circ \code{frame\_skew} : \M{n} -> \SO{n};}
    \]

    Let us now implement optimisation on $\St{n,k}$ based on this. Recall that the tangent space of the Stiefel manifold at a point $Q \in \St{n,k}$ may be identified with $\mlie \subset \solie{n} \iso \Skew{n}$ (\Cref{sec:stiefel_manifold}), where $\mlie$ is given by
    \[
        \mlie =
        \set[\Big]{\begin{pmatrix}
            S & -\trans{A} \\
            A & 0_{n-k, n-k}
        \end{pmatrix} |
        S \in \solie{k}, A \in \M{(n-k), k}}.
    \]

    To implement optimisation on $\St{n,k}$, we just have to embed a matrix $\M{n,k}$ into $\mlie$. We may do this by simply considering the map $\deffun{\iota : \M{n,k} -> \M{n};}$ that given a matrix appends an $n \times (n-k)$ matrix of zeros to it. If we then denote by $\deffun{\pi : \SO{n} -> \St{n,k};}$ the projection of a matrix onto its first $k$ columns, we have that we may implement optimisation on $\St{n,k}$ as
    \[
        \exp^{\St{n,k}}_U = \pi \circ \exp^{\SO{n}}_{\total{U}} \circ \iota = \pi \circ L_{\total{U}} \circ \expm \circ \code{frame\_skew} \circ \iota.
    \]
    In particular, we may reuse all the implementation of $\SO{n}$ to implement the exponential on $\St{n,k}$ and simply implement $\iota$ and $\pi$.
\begin{lstlisting}[language=python,escapechar=|,caption={Stiefel manifold from $\SO{n}$},label={lst:final_stiefel}]
class Stiefel(SO):
    def __init__(self, n, k):
        super().__init__(n)
        self.n = n
        self.k = k

    def frame(self, X):
        # X \in R^{n x k}
        # We embed X into R^{n x n} by forming Y = [X, 0] \in R^{n x n}
        Y = torch.cat([X, X.new_zeros(self.n, self.n - self.k)], dim=1)
        return super().frame(Y)

    def exponential(self, B, A):
        U_total = super().exponential(B, A)
        # Project onto the first k components
        return U_total[:, :self.k]
\end{lstlisting}

    Furthermore, the Riemannian gradients at $U = \pi(\total{U})$ may be obtained by letting the autodiff engine differentiate the mapping at zero, as we have previously mentioned.
\end{example}

\begin{example}[Grassmannian]
It should be clear that there is nothing special about the Stiefel manifold in this construction. For example, we could go one step further and implement the exponential on the Grassmannian manifold in terms of that of the Stiefel manifold in just two lines. We start by recalling that, for the Grassmannian
\[
    \mlie =
    \set[\Big]{\begin{pmatrix}
        0_{k,k} & -\trans{A} \\
        A & 0_{n-k, n-k}
    \end{pmatrix} |
    A \in \M{(n-k), k}},
\]
that is, it is the same set as that of the Stiefel manifold but with the top square zeroed-out. As such, to parametrise the Grassmannian in terms of the Stiefel manifold, we simply have to implement the embedding $\iota_2$ that, given a matrix in $\M{n,k}$, zeroes out its first $k$ rows, so that $\iota \circ \iota_2$ is a map from $\M{n,k}$ onto the set $\mlie$ from the Grassmannian. This translates to the code:
\begin{lstlisting}[language=python,escapechar=|,caption={Grassmannian manifold from $\St{n,k}$},label={lst:final_grassmannian}]
class Grassmannian(Stiefel):
    def frame(self, X):
        # X \in R^{n x k}, we zero-out its first k rows
        Y = torch.cat([X.new_zeros(self.k, self.k) , X[self.k:, :]], dim=0)
        return super().frame(Y)
\end{lstlisting}

Note that these three lines of code provide a complete implementation of the Riemannian exponential map on the Grassmannian.
\end{example}

More generally, this idea of composing on the right by a linear immersion and composing on the left by a map---preferably a submersion---is a very flexible one. Consider now that we have implemented the Stiefel manifold and optimisation on $\RR_{> 0}$ (which can be done by pulling back the problem to $\RR$). We can then form the submersion
\[
    \deffun{\pi : \SO{n} \times \RR^n_{> 0} \times \SO{n} -> \GLp{n};
    U, \Sigma, V -> U\Sigma \trans{V}}
\]
The immersion for this map would simply be the map that splits a matrix into its strictly lower-triangular, strictly upper triangular and diagonal parts, while the submersion is the multiplication map.

This pattern of using fibred spaces together with product spaces allows for a rather compact implementation of optimisation on virtually any manifold. We implement a class that, given a list of manifolds, creates a product manifold. This is as simple as writing the following functor
\begin{lstlisting}[language=python,escapechar=|,caption={Implementation of a product manifold},label={lst:prod_manifold}]
class ProductManifold(nn.ModuleList):
    def forward(self, Xs):
        return tuple(mani(X) for mani, X in zip(self, Xs))
\end{lstlisting}

Through the abstraction of a \code{ProductManifold}, Python's inheritance, and the mathematical description of the Riemannian exponential and submersions as a concatenation of linear immersions and submersions, we were able to keep the implementation of every manifold to less than $20$ lines on average with no code duplication. All the manifolds support the dynamic trivialisation framework so, in particular, they implement static trivialisations and \rgd{} with no extra work.

\section{Exponential Recurrent Neural Networks (\texorpdfstring{\exprnn}{ExpRNN} / \texorpdfstring{\dtriv}{DTriv})}\label{sec:exprnn}
Given a sequence of inputs $(x_1, \dots, x_T)$ with $x_i \in \RR^k$, we define an orthogonal exponential \rnn{} (\exprnn) with hidden size $d > 0$ as
\[
    h_t = \sigma\pa{\expm(A) h_{t-1} + Cx_t}, \mathrlap{\qquad t=1, \dots, T}
\]
where $A \in \Skew{d}$, $C \in \M{d,k}$ are parameters of the model and $\sigma$ is a non-linearity. This equation is exactly that of a regular \rnn{} (\Cref{def:rnn}) using the matrix exponential to impose orthogonality constraints on the recurrent kernel.

We also define the dynamic trivialisation architecture (\dtriv), which rather than using the Lie exponential uses the Riemannian exponential to pullback the problem
\[
    h_t = \sigma\pa{B\expm(A) h_{t-1} + Cx_t}, \mathrlap{\qquad t=1, \dots, T}
\]
where $B \in \SO{n}$ is a fixed matrix---we do not optimise $B$. In general, we write \dtriv$K$ for $K \in \set{1, 2, \ldots} \cup \set{\infty}$ for the case where we update base $B$ after $K$ steps. This is an instance of the dynamic trivialisation framework (\Cref{alg:dyn_triv_bundle}) with stopping rule $\code{stop} \equiv k = K$. For the case \dtriv$\infty$, we recover the rule $\code{stop} \equiv \code{False}$, that is, \dtriv$\infty$ is effectively a static-trivialisation along $\exp_{B_0}$ for a fixed matrix $B_0 \in \SO{n}$. In particular, if $B_0 = \I_n$, \dtriv$\infty$ recovers \exprnn.

Note that generalising this architecture to the complex unitary case simply accounts for considering $A$ to be skew-Hermitian rather than skew-symmetric and letting $B \in \UU{n}$. Note that this same pattern can be used on any naturally reductive homogeneous space, lifting the problem to the linear space $\mlie$ of the Lie algebra of the total space.

This split the study of these models in three parts which will be treated in three different sections:
\begin{enumerate}
    \item The approximation of $\expm$
    \item The choice of non-linearity $\sigma$
    \item The initialisation of the model
\end{enumerate}

But before delving into these, we start by reviewing the previous approaches used in literature to enforce exact orthogonal constraints in this kind of models.

\subsection{Comparison with previous approaches}\label{sec:comparison}
A number of approaches have been presented to perform optimisation with orthogonal and unitary constraints in the context of \rnns{} previous to the introduction of the static and dynamic trivialisation framework. These approaches can be separated into two main categories: Those that use \rgd{} and those that parametrise the special orthogonal group using unconstrained parameters. The latter ones are pullback problems along maps from $\RR^k$ for some $k > 0$. As there are no submersions between $\RR^{\frac{n\pa{n-1}}{2}}$ and $\SO{n}$---since this map would be a covering map, and the universal cover of $\SO{n}$ is not trivial---we know that such maps will have singularities or will not be surjective. We will comment on these properties for some of these algorithms.

\subsubsection{Riemannian Gradient Descent (\rgd)}
This approach was used in the papers~\parencites{wisdom2016full}{vorontsov2017orthogonality}. Recall that the update step for \rgd{} in $\SO{n}$ with step-size $\eta > 0$ and a map $\deffun{\phi : \solie{n} -> \SO{n};}$ is given by
\[
    B_{t+1} = B_t\phi\pa{-\eta \trans{B_t} \grad f(B_t)},
\]
that is, we use the Lie group structure to construct the step map $r_B \defi L_B \circ \phi \circ \pa{\dif L_B}^\ast$. If $\pa{\dif\phi}_0 = \Id$, then $r_B$ is a retraction

In these two papers, they perform this update step using a variation of the Cayley map as the retraction
\[
    \deffun{\phi : \solie{n} -> \SO{n}; A -> \frac{\I_n + \frac{1}{2}A}{\I_n - \frac{1}{2}A}}.
\]

\rgd{} presents the issue of needing a step map that lies exactly on the manifold. In the case of $\SO{n}$, we need a function $\phi$ whose output lies exactly on the manifold. Our static-trivialisation approach does not require this, allowing for the use of approximations to the matrix exponential on $\Skew{n}$ whose result is not orthogonal, like truncating the Taylor series of the matrix exponential at a low degree. These approaches could be useful when the efficiency of the parametrisation is more important in comparison to the orthogonality of the resulting matrix.

A related problem of \rgd{} arises from the numerical approximations inherently present in numerical algorithms.
The Cayley map---and sometimes the matrix exponential---is often computed as a solution of a system of the form $BX = C$. For example, in the case of the Cayley map, we set $C = \I_n + \frac{1}{2}A$, $B = \I_n - \frac{1}{2}A$. These systems are solved via an \lu{} decomposition and the solution of two triangular systems, which gives an approximate solution. As such, the resulting matrix is not exactly orthogonal due to numerical and rounding errors. To account for this loss of orthogonality, implementations of \rgd{} perform a projection onto $\SO{n}$ after certain number of iterations. The orthogonal projection onto $\SO{n}$ for a matrix $B \in \M{n}$ is given by a projection onto the orthogonal component of its polar decomposition. This is sometimes implemented through the \svd{} decomposition $B = U \Sigma \trans{V}$ as
\[
    \pi_{\SO{n}}(B) = U\trans{V}.
\]
Leaving aside the fact that this is a fairly expensive operation on a \gpu, this projected gradient descent is problematic when doing non-convex optimisation. Recall that \rnns{} are particularly bad-behaved functions, as discussed in~\Cref{sec:orthogonal_rnn}. As such, the projection step often results on a great increase on the value of the objective function, often precluding the convergence of the algorithm.

\subsubsection{Parametrisations}
There are four main parametrisations of the orthogonal and unitary group presented in the literature that have been used to maintain the orthogonality of layers within a neural network.

\paragraph{Unitary \rnn{} (\urnn) {\parencite{arjovsky2016unitary}}}
This was the first paper to introduce unitary constraints to tackle the vanishing and exploding gradient problem in \rnns. It parametrises $\UU{n}$ using the following product of matrices
\[
    B = D_3T_2\mathcal{F}^{-1}D_2\Pi T_1\mathcal{F}D_1.
\]
$D_k = \diag\pa{e^{i\omega_k}}$ where $\omega_k \in \RR^n$ are parameters. $T_k = \I_n - 2 \frac{v_k\tensor v_k}{\norm{v_k}^2}$ where $v_k \in \RR^n$ are parameters. $\mathcal{F}$ and $\mathcal{F}^{-1}$ are Fourier transform and inverse Fourier transform matrices. $\Pi$ is a fixed permutation matrix. The final matrix $B$ is unitary since it is a product of unitary matrices. One of the problems of this parametrisation is that it cannot represent every unitary matrix as $\dim_{\RR}(\UU{n}) = n^2$, while the model has just $5n$ parameters.

\paragraph{Efficient Unitary \nn{} (\eunn) {\parencite{jing2017tunable}}}
This paper tried to address this surjectivity problem that \urnn{} presents by proposing a different parametrisation. Their parametrisation is based on the decomposition of a unitary matrix via Givens rotations
\[
    B = D\prod_{i=2}^N\prod_{j=1}^{i-1}R_{i,j},
\]
where $R_{i,j}$ are Givens rotation matrices and $D = \diag\pa{e^{i\omega}}$ for a parameter $\omega \in \RR^n$. They propose an efficient computation of this decomposition using ideas related to the fast Fourier transform (\fft), having $B$ as a product of $\log\pa{n}$ matrices. They do not provide any guarantee for the critical points of this parametrisation.

\paragraph{Householder reflections {\parencite{mhammedi2017efficient}}}
In this case, the authors propose to parametrise $\Ort{n}$ by using a product of $n$ Householder reflections
\[
    B = \prod_{i=1}^n \cor[\Big]{\I_n - 2\frac{v_i \tensor v_i}{\norm{v_i}^2}}.
\]
This is a surjective parametrisation of the matrices in $\Ort{n}$ with determinant $(-1)^n$. On the other hand, it presents a number of problems.

Given that this parametrisation is not locally unique, it will inherently have critical points on those points. For example, this happens at points where $v_i = v_j$ for $i \neq j$, or more generally, at points at which the matrix $V = \pa{v_1 / \norm{v_1}, \dots, v_n / \norm{v_n}}$ is singular.

Another problem that this parametrisation has is that it is not defined at $v_i = 0$. This may induce unstable gradients around those points. We believe that this is the reason why different papers have reported that they were not able to reproduce the results of the experiments presented in this paper. We also observed that this parametrisation does not work well in practice. As such, we will not include it in the comparison in the experiments.

\begin{remark}[Parametrisations as long products of matrices]
These three parametrisations have a problem in common: They are all given by a product of matrices. For this reason, whenever the parameters are updated, one has to compute the product of matrices sequentially, which is rather costly in a \gpu. This is a problem that our parametrisation does not present.
\end{remark}

\begin{remark}[Complex parametrisations]
Some of these parametrisations only work in the complex case, like those involving the \fft. This created the artificial burden of having to implement complex-valued layers. At the moment of this writing, the support for complex valued networks in PyTorch is rather limited, so putting this extraneous constraints in the model may heavily hinder the adoption of these parametrisations.
\end{remark}

\paragraph{Scaled Orthogonal / Unitary \rnn{} (\scornn{} / \scurnn)~{\parencites{helfrich2018orthogonal}{maduranga2019complex}}}
These parametrisations are the closest in spirit to ours. They use the Cayley transform as a map between $\solie{n}$ and $\SO{n}$ and pullback the problem to $\solie{n}$ (resp.~$\uulie{n}$ and $\UU{n}$). The problem with this map is that orthogonal matrices which have $-1$ as an eigenvalue are not covered by the map. On the other hand, this does not seem to be particularly problematic in practice.

Even then, they implement certain overparametrisation trick in the papers to deal with this case, although it is not clear whether the performance gain that they get comes from the soundness of their trick or the fact that, because of this trick, their model has some extra structure and is overparametrised.

\paragraph{Matrix Exponential~\parencite{hyland2017learning}}
This paper presented a basic version of the idea of using the matrix exponential to map $\uulie{n}$ surjectively onto $\UU{n}$. On the other hand, the authors write:
\begin{quote}
The matrix exponential appearing in Equation $7$ poses an issue for gradient calculations. In general, the derivative of the matrix exponential does not have a closed-form expression, so computing gradients is intractable.
\end{quote}
It turns out that this claim is not accurate, as we will show in the following section.

\subsection{The exponential map on a \texorpdfstring{\gpu}{GPU}}\label{sec:approximations_exponential}
Most of both the theoretical and practical work in this thesis uses the matrix exponential as the main building block. In this section, we will look at how to efficiently implement this map on a \gpu.

There are a myriad of methods to approximate the exponential of a matrix~\parencites{moler1978nineteen}{moler2003nineteen}. Some of the most important ones are the following:

\paragraph{Diagonal Padé approximants.}
Diagonal Padé approximants, or simply Padé approximants, are rational approximations of the form $\expm(A) \approx p_n(A)q_n(A)^{-1}$ for polynomials $p_n, q_n$ of degree $n$. A Padé approximant of degree $n$ agrees with the Taylor expansion of the exponential to degree $2n$. It is easy to show that the diagonal Padé approximant of the exponential map applied to a skew-symmetric matrix is always orthogonal. In particular, the Padé approximant of degree $1$ is the Cayley map. These methods and their implementations are described in detail in~\parencite{higham2009scaling}.

\paragraph{Taylor approximants.}
A Taylor approximant simply accounts for truncating the series defining the matrix exponential at a given term. These methods had been historically disregarded in the literature in favour of the Padé approximants as on \cpu, they are slower than their Padé counterparts.

\paragraph{Scale-squaring trick.}
The error of the Padé approximant scales as $\mathcal{O}\pa{\norm{A}^{2n+1}}$. If $\norm{A} > 1$ and we have an approximant $\psi$, the scale-squaring trick accounts for computing $\psi\pa{\frac{A}{2^k}}^{2^k}$ for the first $k=0, 1, \dots$ such that $\frac{\norm{A}}{2^k} < \frac{1}{2}$ to lower the norm of the matrix fed into the approximant.
Most types of approximants, like Padé's or Taylor's, can be coupled with the scale-squaring trick to reduce the error~\parencite{higham2009scaling}.

\paragraph{Machine-precision approximant.} Combining one of the approximants together with the scale-squaring trick, and bounding the error, it is possible to get an efficient approximation of the exponential to machine-precision. The one implemented in most standard numerical libraries is based on the paper~\parencite{al2009new}. It accounts for an efficient use of the scale-squaring trick and a Padé approximant, in the sense of trying to compute the lowest degree Padé approximant such that it gives an error less than the machine-precision.

On the other hand, we note that the fact that the Padé approximant is faster than the Taylor approximant due to its lower degree is no longer true in the context of \gpus. On a \gpu, solving a system of the form $BX = C$ with $B = q_n(A)$ and $C = p_n(A)$ is much slower than performing a few more matrix multiplications. This idea is leveraged in the paper~\parencite{bader2019computing}. The authors factor the Taylor approximants to compute them using the smallest possible number of multiplications. By doing so, together with a careful backwards analysis of the algorithm, they are able to provide an algorithm that needs of at most $5$ matrix multiplications to compute any approximant.

We implemented this algorithm in PyTorch, and we were able to get a $\times 6$ speed-up in matrices of size $1024 \times 1024$ over the Padé approximant. We also helped Nikitaved, a PyTorch core developer, to implement this algorithm natively on \cuda, obtaining a further $\times 2.5$ speed improvement, totalling a $\times 14.5$ speed-up against the Padé approximant.

We compare in~\Cref{tab:gpu_expm_taylor,tab:gpu_expm_pade} the efficiency of the Padé approximant on \gpu{} compared with an optimised implementation of the algorithm from~\parencite{bader2019computing}. We can see that, on sizes of $1024\times 1024$ we get a speed-up of $\times 14.5$ with respect to the Padé implementation.

This implementation is even faster than the usual retraction used on $\SO{n}$, the Cayley map, as we can see in~\Cref{tab:gpu_cayley}. We see that for larger sizes, the matrix exponential is still faster than the Cayley map. We see this even though the \lu{} decomposition needed for the Cayley map is implemented entirely in Magma which handles batching in parallel, while the exponential map simply handles the batch sequentially.

For reference, all the tests were implemented in the PyTorch benchmark framework to ensure the accuracy of the results.\footnote{\url{https://pytorch.org/docs/master/benchmark_utils.html}}

\begin{table}[t]
\begin{minipage}[b]{0.46\textwidth}
\centering
\scshape
\begin{tabular}{l c c c}
    \toprule
    Batch & $64 \times 64$ & $256 \times 256$ & $1024 \times 1024$\\
    \midrule
    $16$ &  $0.7$ &  $0.6$ &  $11.1$ \\
    $32$ &  $0.7$ &  $1.0$ &  $22.2$ \\
    $64$ &  $0.7$ &  $1.7$ &  $44.3$ \\
    $128$ & $0.7$ &  $3.1$ &  $88.9$ \\
    \bottomrule
\end{tabular}
\caption{Machine-precision matrix exponential on \gpu. Used the current implementation in PyTorch $1.7.0$. Times in milliseconds.}%
\label{tab:gpu_expm_taylor}
\end{minipage}%
\hspace{0.08\columnwidth}%
\begin{minipage}[b]{0.46\textwidth}
\centering
\scshape
\begin{tabular}{l c c c}
    \toprule
    Batch & $64 \times 64$ & $256 \times 256$ & $1024 \times 1024$\\
    \midrule
    $16$ &  $1.4$ &  $4.0$ &  $195.0$ \\
    $32$ &  $1.4$ &  $6.0$ &  $347.0$ \\
    $64$ &  $1.5$ & $10.6$ &  $557.0$ \\
    $128$ & $1.5$ & $24.0$ & $1288.0$ \\
    \bottomrule
\end{tabular}
\caption{Padé approximant on \gpu. Used the current implementation in Tensorflow $2.3.0$. Times in milliseconds.}%
\label{tab:gpu_expm_pade}
\end{minipage}%

\vspace{1em}
\centering
\begin{minipage}[b]{0.46\textwidth}
\centering
\scshape
\begin{tabular}{l c c c}
    \toprule
    Batch & $64 \times 64$ & $256 \times 256$ & $1024 \times 1024$\\
    \midrule
    $16$ &  $0.5$ &  $2.5$ &  $24.1$ \\
    $32$ &  $0.5$ &  $2.8$ &  $34.6$ \\
    $64$ &  $0.5$ &  $3.6$ &  $55.6$ \\
    $128$ & $0.6$ &  $5.7$ &  $100.0$ \\
    \bottomrule
\end{tabular}
\caption{Cayley map on \gpu{} implemented in PyTorch $1.7.0$. Times in milliseconds.}%
\label{tab:gpu_cayley}
\end{minipage}%
\end{table}

\paragraph{Machine-precision gradients.}
These algorithms give an approximation to the matrix exponential, but we still need to approximate its adjoint. We showed that this accounts for computing the differential of $\deffun{\expm : \gl{n} -> \GLp{n};}$ at the transpose (\Cref{thm:analytic}). To finish the implementation, we use the following theorem to approximate the differential of an analytic matrix function.

\begin{theorem}[Differential of a Matrix Function~{\parencite{mathias1996chain}}]\label{thm:differential}
    Let $U \subset \CC$ be open and let $\deffun{\psi : U -> \CC;}$ be an analytic function. Let $X \in \CC$ be a matrix with spectrum contained in $U$ then, for any $A \in \CC^{n \times n}$, denoting by the same letter the associated matrix function, we have that
    \[
        \psi\begin{pmatrix}
            X & A \\
            0 & X
        \end{pmatrix} =
        \begin{pmatrix}
            \psi(X) & \pa{\dif \psi}_X(A) \\
            0 & \psi(X)
        \end{pmatrix}.
    \]
\end{theorem}

From~\Cref{thm:analytic,thm:differential} we have that, if we know how to approximate an analytic matrix function, such as the matrix exponential, we can approximate its adjoint on matrices of size $n \times n$ by computing the value of the function on a $2n \times 2n$ matrix.

This way, we can use the fast matrix exponential implementation to implement its adjoint.
We helped to implement the matrix exponential in PyTorch. It is now part of PyTorch $1.7.0$, so all these ideas are open to be used by the machine learning community.

\subsection{Non-linearities}
As we already pointed out in~\Cref{sec:orthogonal_rnn}, the non-linearity used in the recurrent neural network has a big influence on the vanishing-gradient problems. We see the interaction of the non-linearity with the gradient when we compute the differential of the loss function with respect to the hidden states
\[
    \pderiv{\loss}{h_i} = \pderiv{\loss}{h_T}\prod_{t=i}^{T-1}D_{t+1}B.
\]
where $D_t = \diag\pa{\sigma'\pa{z_t}}$ and $z_{t+1} = B h_t + Cx_{t+1}$. Given that in our case $B$ is orthogonal, there is a trade-off between the deviation of $\sigma$ from a function with $\abs{\sigma'(x)} = 1$ and some possible exploding or vanishing gradient problems suffered by the \rnn.

In the simplest case, we could choose $\sigma(x) = x$. The resulting \rnn{} works surprisingly well for many examples, but falls short when used in more complex experiments, like the \timit{} experiment presented in~\Cref{sec:experiments}.

In the paper that introduced the unitary \rnn~\parencite{arjovsky2016unitary}, the authors introduced the \code{modrelu} non-linearity
\[
	\deffun{\code{modrelu}_b : \CC -> \CC;
		z ->\frac{z}{\abs{z}}\code{relu}\pa{\abs{z} + b}}
\]
where $\relu\pa{x} := \max\pa{x, 0}$ is the rectified linear unit and $b \in \RR^n$ is a parameter. This is the non-linearity that was used for this model in all the other papers mentioned before. As such, it will be the one that we will use in our experiments.

Even then, we also propose a novel non-linearity that explicitly models the trade-off between being close to the identity and expressiveness. We define the \textbf{dynamic soft shrink} function as
\[
    \code{dynsoftshrink}_{a,b}(x) =
    \begin{cases}
        x + b(a-1) \quad &\text{for } x > b \\
        ax \quad &\text{for } x \in [-b, b] \\
        x - b(a-1) \quad &\text{for } x < -b.
    \end{cases}
\]
for two parameters $a \in \RR$, $b \in \RR_{\geq 0}$. In other words, this non-linearity is continuous, piecewise linear, and has constant derivative $1$ outside of the interval $[-b, b]$ and derivative $a$ inside of this interval. We initialise it to $b = 0.5$, $a = 1$, so that initially $\code{dynsoftshrink}_{1, 0.5}(x) = x$.

Even though we will not show experiments with this non-linearity in~\Cref{sec:experiments}, as the purpose of that section is to compare the orthogonality optimisation process, we report here that have seen large improvements when using this non-linearity instead of the \code{modrelu}. That being said, we leave the design of better recurrent models for further research.

\subsection{Initialisation}
For the initialisation of the layer with a matrix $A_0 \in \Skew{p}$, we drew ideas from~\parencite{henaff2016recurrent}. The initialisations considered sample blocks of the form
\[
\begin{pmatrix}
    0 & -s_i \\
    s_i & 0
\end{pmatrix}.
\]
for $s_i$ \iid{} with respect to some probability measure on $[-\pi, \pi]$ and then form $A_0$ as a block-diagonal skew-symmetric matrix with these blocks.\footnote{Recall that $\exp(2\pi A) = \exp(A)$ for $A \in \Skew{n}$ since the eigenvalues of a skew-symmetric matrix are purely imaginary.}

The uniform initialisation accounts for sampling $s_i \sim \unif\cor{-\pi, \pi}$. This defines a block-diagonal orthogonal matrix $\expm\pa{A}$ with uniformly distributed blocks on half of the torus of block-diagonal $2 \times 2$ rotations.

Initialising the matrix on the main torus seemed to help convergence when compared with other less spare measures such as the Haar measure and others. This was already noted in~\parencite{henaff2016recurrent}. While we do not have a theoretical explanation for this phenomenon, we hypothesise that such initialisation, that acts locally in two coordinates at a time, couples well with the element-wise non-linearities that are present in the \rnn.

We chose $h_0 = 0$ as the initial hidden vector of the \rnn{} for simplicity, as we did not observe any empirical improvement when using the initialisation given in~\parencite{arjovsky2016unitary}.

\section{Experiments}\label{sec:experiments}
In this section, we assess the effectiveness of the exponential \rnn{} (\exprnn) and dynamic trivialisations using the matrix exponential (\dtriv$K$) in the context of orthogonal optimisation in \rnns. We test the dynamic trivialisations with the basis changed every $K = 1, 100, \infty$ steps.

We compare the performance of our parametrisation for orthogonal \rnns{} with the following approaches:
\begin{itemize}
    \item Long short-term memory (\lstm)~\parencite{hochreiter1997long}.
    \item Unitary Recurrent Neural Networks (\urnn)~\parencite{arjovsky2016unitary}.
    \item Efficient Unitary Recurrent Neural Network (\eunn)~\parencite{jing2017tunable}.
    \item Real Cayley Parametrisation (\scornn)~\parencite{helfrich2018orthogonal}.
    \item Complex Cayley Parametrisation (\scurnn)~\parencite{maduranga2019complex}.
    \item Riemannian Gradient Descent (\rgd)~\parencite{wisdom2016full}.
\end{itemize}

\begin{remark}
    Note that (stochastic) $\rgd$ is equivalent to \dtriv$1$ together with the optimiser \sgd. Furthermore, \exprnn{} is equivalent \dtriv$\infty$ with $B = \I_n$ and $\phi = \expm$, while \dtriv$\infty$ has the base at a randomly sampled matrix. At the same time, \scornn{} and \scurnn{} fall into the context of dynamic trivialisations, but using $B = \I_n$ and the Cayley map as a retraction.
\end{remark}

We use three tasks that have become standard to measure the performance of \rnns{} and their ability to deal with long-term dependencies. These are the copying memory task, the pixel-permuted \mnist{} task, and the speech prediction on the \timit{} dataset~\parencites{arjovsky2016unitary}{wisdom2016full}{henaff2016recurrent}{mhammedi2017efficient}{helfrich2018orthogonal}.

\begin{remark}
    We found empirically that having a learning rate for the orthogonal parameters that is between $5$--$10$ times larger than that of the non-orthogonal parameters yields a good performance in practice.
\end{remark}

For the other experiments, we executed the code that the other authors provided with the best hyperparameters that they reported and a batch of $128$. The results for \eunn{} are those reported in~\parencite{jing2017tunable}, and for \rgd{} and \urnn{} are those reported in~\parencite{helfrich2018orthogonal}.

The code with the exact configuration and seeds to replicate these results can be found here:
\begin{center}
\url{https://github.com/Lezcano/expRNN}
\end{center}

\subsection{Copying memory task}
\begin{figure*}[!tbp]
  \begin{minipage}[b]{0.5\textwidth}
      \centering
      \includegraphics[width=\columnwidth]{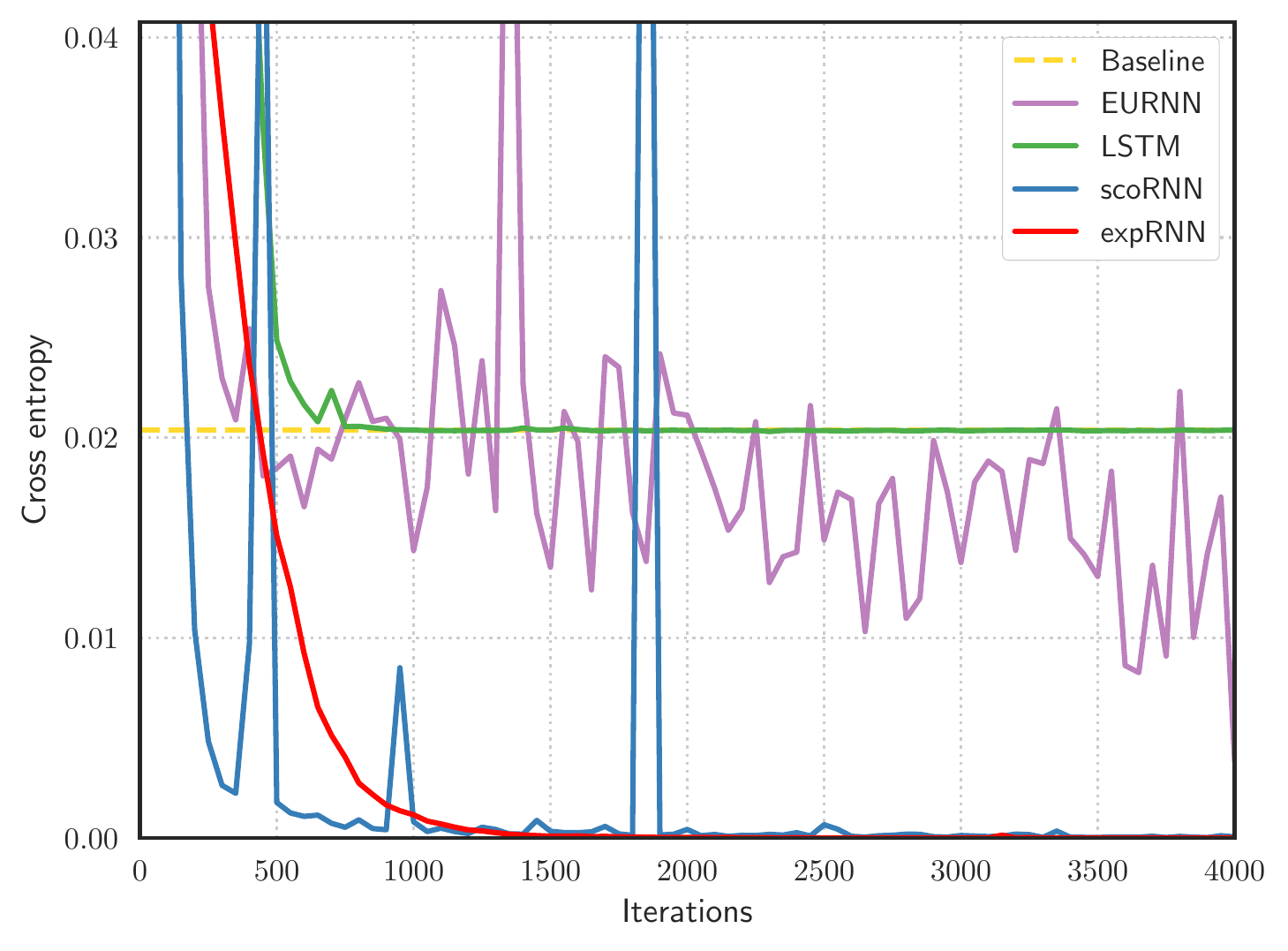}
  \end{minipage}%
  \begin{minipage}[b]{0.5\textwidth}
      \centering
      \includegraphics[width=\columnwidth]{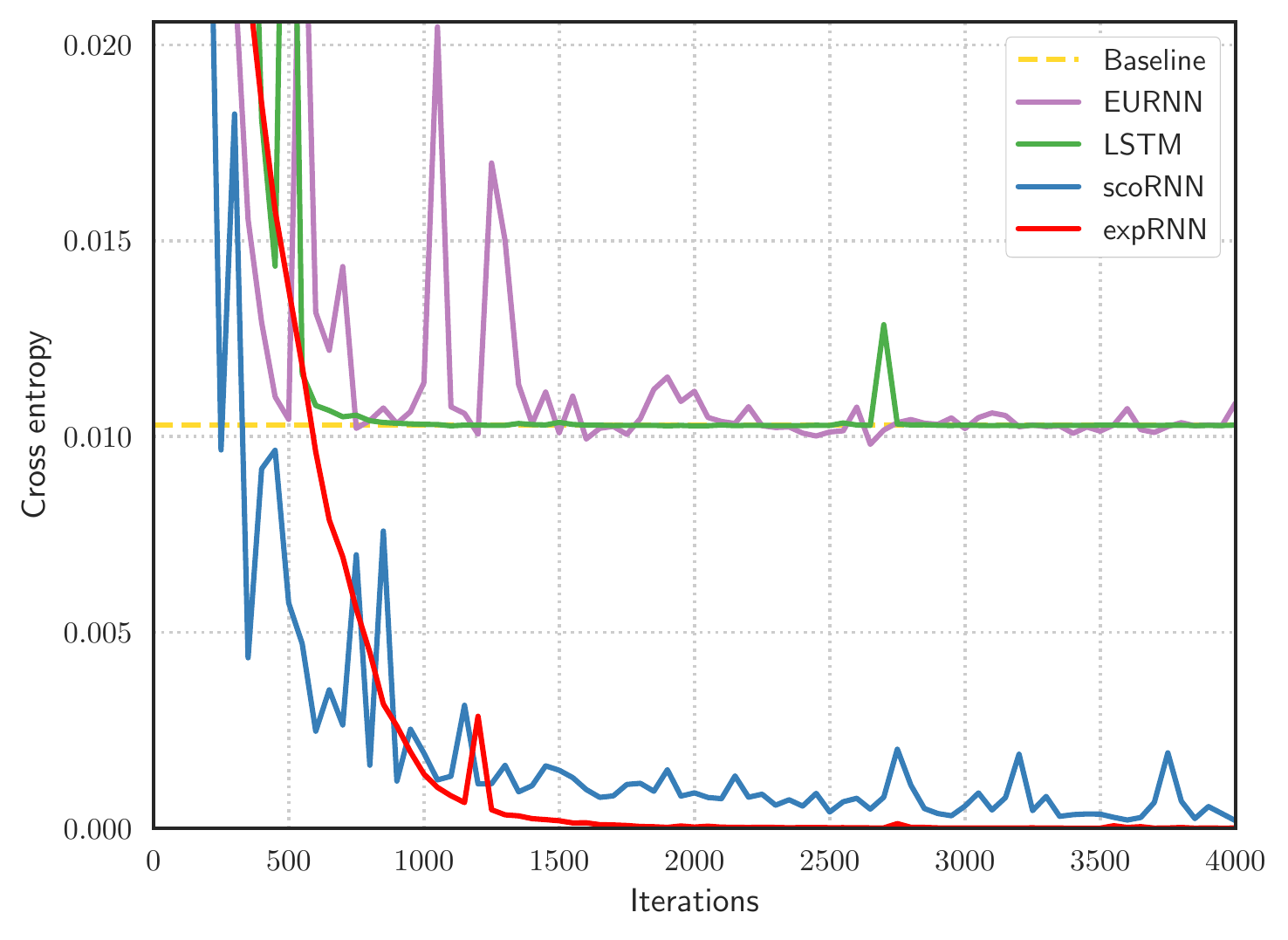}
  \end{minipage}
  \caption{Cross entropy of the different algorithms in the copying problem for $L=1000$ (left) and $L=2000$ (right).}%
  \label{fig:copying}
\end{figure*}

The copying memory task was first proposed in~\parencite{hochreiter1997long}.
The task can be defined as follows. Let $\mathcal{A} = \set{a_i}_{i=1}^A$ be an alphabet and let $\code{<blank>}$, $\code{<start>}$ be two symbols not contained in $\mathcal{A}$. For a sequence length of $S$ and a spacing of length $L$, the input sequence would be $S$ ordered characters $\pa{b_1, b_2, \dots, b_S}$ sampled \iid{} uniformly at random from $\mathcal{A}$, followed by $L$ repetitions of the character $\code{<blank>}$, the character $\code{<start>}$ and finally $S-1$ repetitions of the character $\code{<blank>}$ again. The output for this sequence would be $S+L$ times the $\code{<blank>}$ character and then the sequence of characters $\pa{b_1, b_2, \dots, b_S}$. In other words, the system has to recall the initial $S$ characters and reproduce them after detecting the input of the character $\code{<start>}$, which appears $L$ time-steps after the end of the input characters. For example, for $A = 4$, $S = 5$, $L = 8$, if we represent $\code{<blank>}$ with a dash and $\code{<start>}$ with a colon, and the alphabet $\mathcal{A} = \set{1, 2, 3, 4}$, the following sequences could be an element of the (synthetic) dataset:

{
\newcommand{\dash}{{-}}
\begin{center}
\begin{tabular}{rc}
    Input:  &\texttt{14221\dash\dash\dash\dash\dash\dash\dash\dash:\dash\dash\dash\dash} \\
    Output: &\texttt{\dash\dash\dash\dash\dash\dash\dash\dash\dash\dash\dash\dash\dash14221}
\end{tabular}
\end{center}
}

The loss function for this task is the cross entropy. The standard baseline for this task is the output of $S+L$ $\code{<blank>}$ symbols, followed by the remaining S symbols being outputted at random. This strategy yields a cross entropy of $S\log\pa{A}/ \pa{L + 2S}$.

In the experiments we set $A = 9$, $S=10$, as it was done in the previous papers.

We observe that the training of \scornn{} is unstable, which is probably due to the degeneracies explained in~\Cref{sec:comparison}.
In the follow-up paper~\parencite{maduranga2019complex}, \scurnn{} presents the same instabilities as its predecessor.
As explained in~\Cref{sec:comparison}, \exprnn{} does not suffer from this, and can be observed in our experiments as a smoother convergence.
In the more difficult problem, $L=2000$, \exprnn{} is the only architecture that is able to fully converge to the correct answer.

\subsection{Pixel-by-pixel \mnist}

\begin{figure*}[!tbp]
\centering
  \begin{minipage}[b]{0.5\textwidth}
      \centering
      \includegraphics[width=\columnwidth]{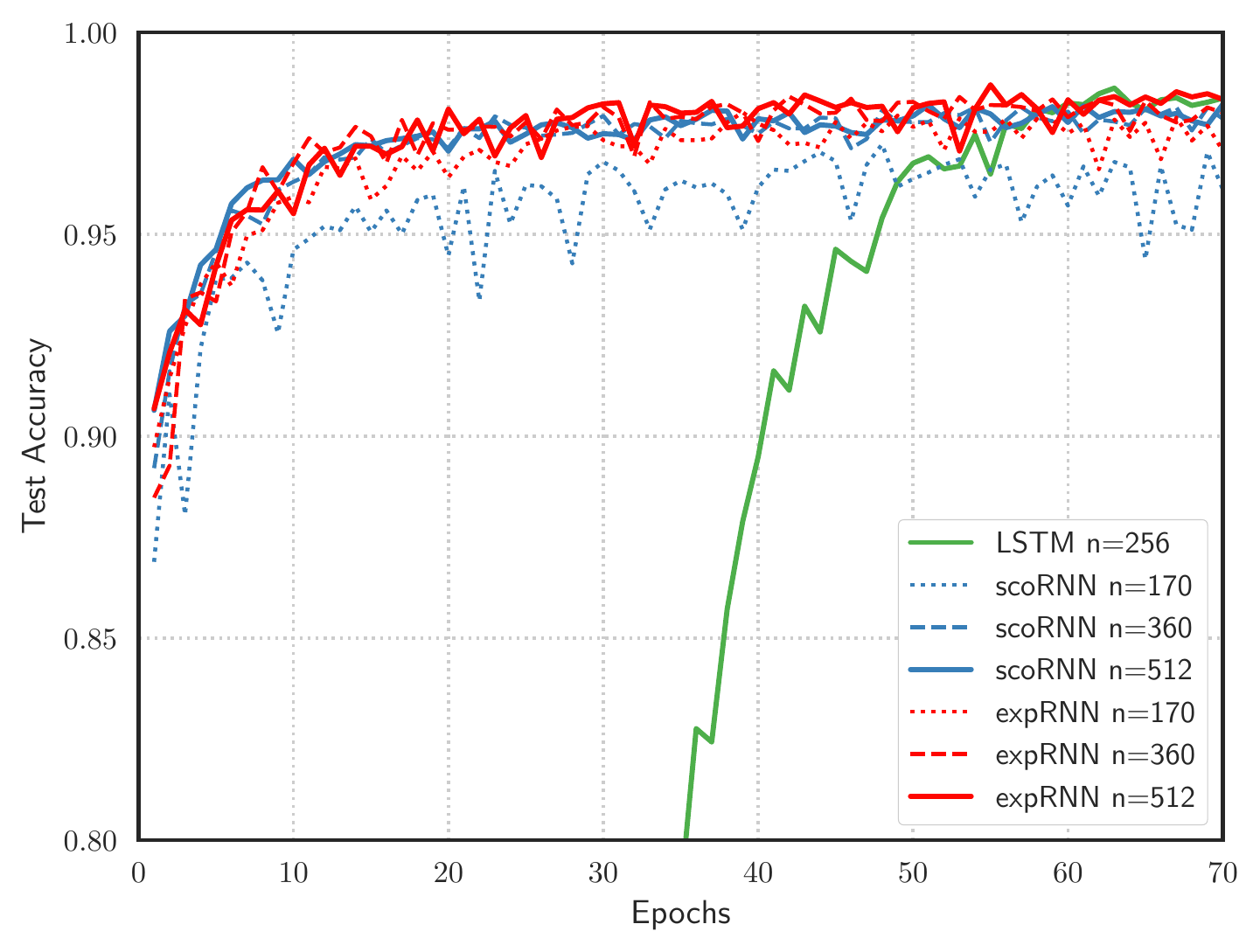}
  \end{minipage}%
  \begin{minipage}[b]{0.5\textwidth}
      \centering
      \includegraphics[width=\columnwidth]{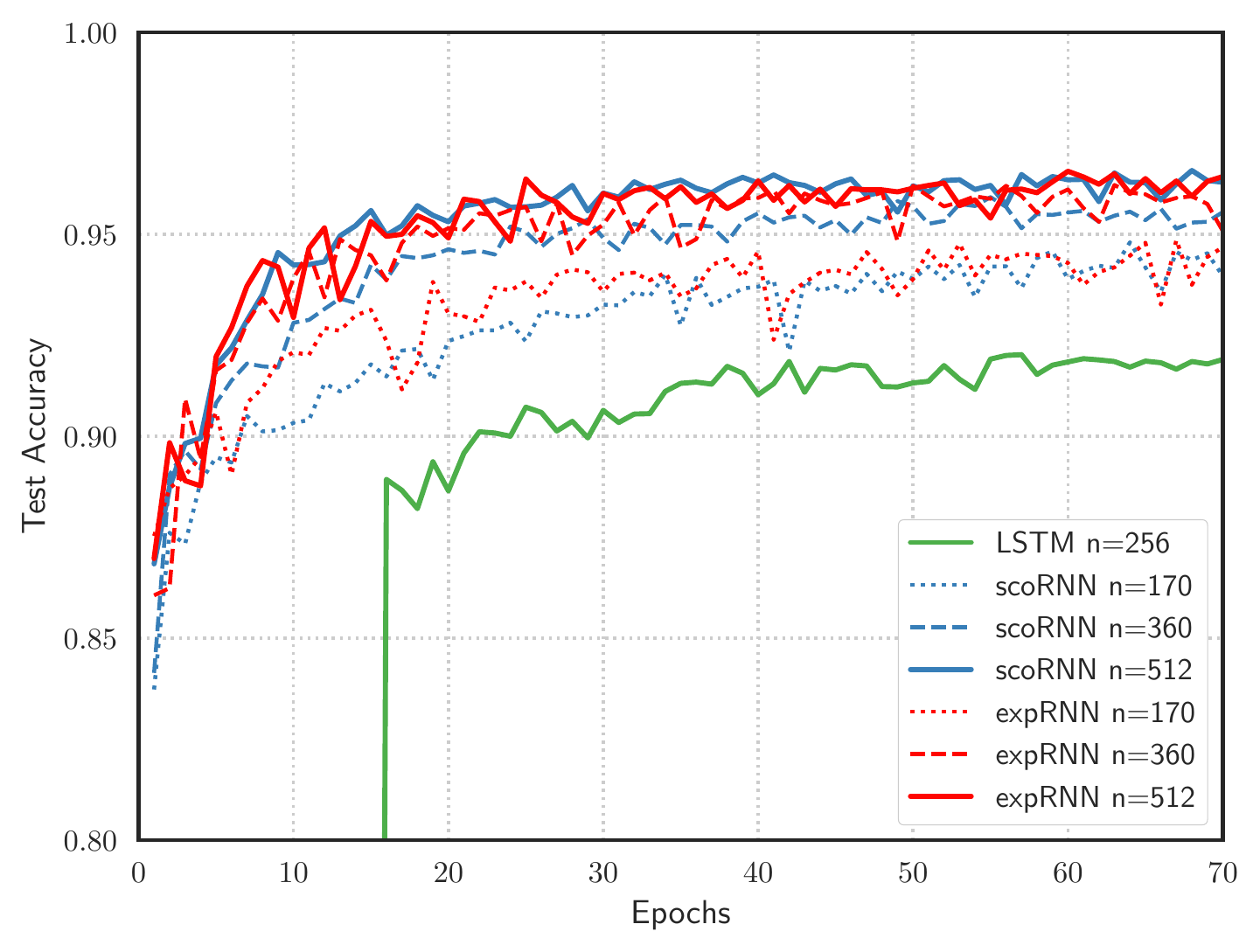}
  \end{minipage}
  \caption{Test losses for several models on pixel-by-pixel \mnist{} (left) and permuted \mnist{} (right).}%
  \label{fig:mnist}
\end{figure*}

\Cref{tab:mnist} is structured so that architectures with the same number of parameters are compared together. The number $n$ represents that we are working with a matrix of hidden size $n \times n$.

In this experiment we observed that \exprnn{} and \dtriv{} models were able to saturate the capacity of the orthogonal \rnn{} model for this task much faster than any other parametrisation, as per~\Cref{tab:mnist}. We conjecture that coupling the exponential parametrisation with an \lstm{} cell or a \gru{} cell would yield a superior architecture. We leave this for future research.

As we can see, the addition of any dynamic trivialisation to the Lie parametrisation improves the results on this experiment by $0.4\%$ out of the $1.3\%$ possible in the largest size. Moreover, it always improves the previous results, suggesting that it is always a better option to use dynamic trivialisations rather than just plain trivialisations. In general, we saw that \dtriv$100$ and \dtriv$\infty$ gave the highest stability and the best results across the experiments. We conjecture that this is due to the fact that changing too often the base is detrimental for the accumulated momentum and adaptive term that the algorithm stores internally.

\subsection{\timit{} speech dataset}
\begin{table}[t]
    \centering
\begin{minipage}[b]{0.40\columnwidth}
\caption{Best test accuracy at \mnist{} and \pmnist{}.}%
\label{tab:mnist}
\scshape
\small
\begin{tabular}{l c c c}
    \toprule
    Model & n & \mnist & \pmnist \\
    \midrule
    \midrule
    \dtriv$1$ & $170$ & $\mathbf{98.3}$ & $\mathbf{95.2}$ \\
    \dtriv$100$ & $170$ & $98.2$ & $95.1$ \\
    \dtriv$\infty$ & $170$ & $98.1$ & $95.0$ \\
    \exprnn & $170$ & $98.0$ & $94.9$ \\
    \scornn & $170$ & $97.2$ & $94.8$ \\
    \scurnn & $116$ & $97.6$ & $94.9$ \\
    \lstm & $128$ & $81.9$ & $79.5$ \\
    \rgd & $116$ & $94.7$ & $92.5$ \\
    \eunn & $512$ & $-$ & $93.7$ \\
    \urnn & $512$  & $97.6$ & $94.5$ \\
    \midrule
    \dtriv$1$ & $360$ &  $98.4$ & $96.3$ \\
    \dtriv$100$ & $360$ &  $98.8$& $96.4$ \\
    \dtriv$\infty$ & $360$ &  $\mathbf{98.9}$ & $\mathbf{96.5}$ \\
    \exprnn & $360$ & $98.4$ & $96.2$ \\
    \scornn & $360$ & $98.1$ & $95.9$ \\
    \scurnn & $250$ & $98.3$ & $96.2$ \\
    \lstm & $256$ & $88.8$ & $88.8$ \\
    \rgd & $256$ & $96.1$ & $93.9$ \\
    \urnn & $2170$  & $98.4$ & $95.3$ \\
    \midrule
    \dtriv$1$ & $512$ & $98.7$ & $96.7$ \\
    \dtriv$100$ & $512$ & $\mathbf{99.1}$ & $96.7$ \\
    \dtriv$\infty$ & $512$ & $99.0$ & $\mathbf{96.8}$ \\
    \exprnn & $512$ & $98.7$ & $96.6$ \\
    \scornn & $512$ & $98.2$ & $96.5$ \\
    \lstm & $512$ & $91.9$ & $91.8$ \\
    \rgd & $512$ & $97.3$ & $94.7$ \\
    \bottomrule
\end{tabular}
\end{minipage}%
\hspace{0.05\columnwidth}%
\begin{minipage}[b]{0.45\columnwidth}
\caption{Test \mse{} at the end of the epoch with the lowest validation \mse{} for the \timit{} task.}%
\label{tab:timit}
\small
\scshape
\begin{tabular}{l c c c c}
    \toprule
    Model & n & Val. \mse{} & Test \mse{}\\
    \midrule
    \midrule
    \dtriv$1$ & $224$ & $6.55$ & $6.54$ \\
    \dtriv$100$ & $224$ & $4.80$ & $4.77$ \\
    \dtriv$\infty$ & $224$ & $\mathbf{4.75}$ & $\mathbf{4.71}$ \\
    \exprnn & $224$ & $5.34$ & $5.30$ \\
    \scornn & $224$ & $9.26$ & $8.50$ \\
    \scurnn & $128$ & $9.42$ & $7.23$ \\
    \lstm & $84$ &  $15.42$ & $14.30$ \\
    \rgd & $128$ &  $15.07$ & $14.58$ \\
    \eunn & $158$ &  $15.57$ & $18.51$ \\
    \midrule
    \dtriv$1$ & $322$ & $4.56$ & $4.55$ \\
    \dtriv$100$ & $322$ & $3.80$ & $3.76$ \\
    \dtriv$\infty$ & $322$ & $\mathbf{3.39}$ & $\mathbf{3.76}$ \\
    \exprnn & $322$ & $4.42$ & $4.38$ \\
    \scornn & $322$ & $8.48$ & $7.82$ \\
    \lstm & $120$ & $13.93$ & $12.95$ \\
    \rgd & $192$ & $15.10$ & $14.50$ \\
    \eunn & $256$ & $15.90$ & $15.31$ \\
    \midrule
    \dtriv$1$ & $425$ & $4.21$ & $4.17$ \\
    \dtriv$100$ & $425$ & $2.02$ & $1.99$ \\
    \dtriv$\infty$ & $425$ & $\mathbf{2.00}$ & $\mathbf{1.97}$ \\
    \exprnn & $425$ & $5.52$ & $5.48$ \\
    \scornn & $425$ & $7.97$ & $7.36$ \\
    \scurnn & $258$ & $4.40$ & $3.39$ \\
    \eunn & $378$ &  $16.00$ & $15.15$ \\
    \lstm & $158$ & $13.66$ & $12.62$ \\
    \rgd & $256$ & $14.96$ & $14.69$ \\
    \bottomrule
\end{tabular}
\end{minipage}
\end{table}

We performed speech prediction on audio data with our model. We used the \timit{} speech dataset~\parencite{garofolo1993darpa} which is a collection of real-world speech recordings. The task accounts for predicting the log-magnitude of incoming frames of a short-time Fourier transform as it was first proposed in~\parencite{wisdom2016full}.

We use the separation in train / test proposed in~\parencite{garofolo1993darpa}, having $3640$ utterances for the training set, a validation set of size $192$, and a test set of size $400$. The validation / test division and the whole preprocessing of the dataset was done according to~\parencite{wisdom2016full}. The preprocessing goes as follows: The data is sampled at $8$kHz and then cut into time frames of the same size. These frames are then transformed into the log-magnitude Fourier space and finally, they are normalised according to a per-training set, test set, and validation set basis. The result of this process gives sequences of $129$ complex numbers per step, and a variable length between $61$ and $490$.

In this experiment we see a similar behaviour of the dynamic trivialisations as the one already seen in the \mnist{} and \pmnist{} experiments. We also see in this experiment that \dtriv$100$ and \dtriv$\infty$ always improve the performance of their static counterparts with base at the identity (\exprnn). They also greatly improve over vanilla \rgd.

\begin{remark}[Computing the correct loss on \timit]
As a side note, we must say that the results in this experiment should be interpreted under the following fact: We had access to two of the implementations for the tests for the other architectures regarding this experiment, and neither of them correctly handled sequences with different lengths present in this experiment. We suspect that the other implementations followed a similar approach, given that the results that they get are of the same order. In particular, the implementation released by Wisdom, which is the only publicly available implementation of this experiment, divides by a larger number than it should when computing the average \mse{} of a batch, hence reporting a lower \mse{} than the correct one. Even in this unfavourable scenario, our parametrisation is able to get results that are twice as good---the \mse{} loss function is a quadratic function---as those from the other architectures.

In the experiments in \scurnn{} they explicitly mention that they are computing the \mse{} without discarding the zeros used to pad the variable-length sequences~\parencite{maduranga2019complex}. As such, when computing the \mse, they are dividing by an incorrect number---the longest element in the batch times the elements in the batch---rather than by the correct one---the sum of the lengths of all the elements in the batch. We computed the correct validation and test loss in~\Cref{tab:timit}.
\end{remark}

\clearpage
\printbibliography[heading=bibintoc]
\end{document}